%% file: ms.tex
\documentclass{scrartcl}

\usepackage[utf8]{luainputenc}
\usepackage[USenglish]{babel}
\usepackage{csquotes}

\usepackage[a4paper,top=27mm,bottom=20mm,inner=25mm,outer=20mm]{geometry}
\usepackage[plainpages=false,pdfpagelabels,hidelinks,unicode]{hyperref}

\usepackage[%
  backend=bibtex,bibencoding=ascii,
  style=numeric-comp,
  giveninits=true, uniquename=init, 
  natbib=true,
  url=true,
  doi=true,
  isbn=false,
  backref=false,
  maxnames=99,
  ]{biblatex}
\usepackage{amsmath}
\allowdisplaybreaks
\usepackage{amssymb}
\usepackage{commath}
\usepackage{mathtools}
\usepackage{bbm}
\usepackage{nicefrac}
\usepackage{cancel}
\usepackage{siunitx}

\usepackage{amsthm}
\theoremstyle{plain}
  \newtheorem{theorem}{Theorem}
  \newtheorem{lemma}[theorem]{Lemma}
  \newtheorem{proposition}[theorem]{Proposition}

\theoremstyle{definition}
  \newtheorem{definition}[theorem]{Definition}
  \newtheorem{remark}[theorem]{Remark}

    \newtheorem{property}[theorem]{Property}
\numberwithin{theorem}{section}

\usepackage{color}
\usepackage{graphicx}
\usepackage[small]{caption}
\usepackage{subcaption}

\begingroup\expandafter\expandafter\expandafter\endgroup
\expandafter\ifx\csname pdfsuppresswarningpagegroup\endcsname\relax
\else
\fi

\usepackage{booktabs}
\usepackage{rotating}
\usepackage{multirow}

\usepackage{tablefootnote}
\usepackage{enumitem}

\usepackage{ifluatex}
\ifluatex
  \usepackage[no-math]{fontspec}
\else
  \usepackage[T1]{fontenc}
\fi
\usepackage{newpxtext,newpxmath}

\newcommand{\R}{\mathbb{R}}

\renewcommand{\Re}{\operatorname{Re}}

\DeclarePairedDelimiterX\newset[1]\lbrace\rbrace{\setaux #1||\endsetaux}
\def\setaux#1|#2|#3\endsetaux{\if\relax\detokenize{#2}\relax #1 \else #1 \;\delimsize\vert\; #2 \fi}

\let\epsilon\varepsilon
\let\phi\varphi
\let\rho\varrho

\newcommand{\dt}{\Delta t}

\renewcommand{\prod}{p}
\newcommand{\dest}{d}
\newcommand{\rest}{r}

\newcommand{\revff}[1]{#1}

\newenvironment{keywords}{\par\textbf{Keywords.}}{\par}
\newenvironment{AMS}{\par\textbf{AMS subject classification.}}{\par}

\definecolor{overColor}{rgb}{0.8, 0.4, 0.0}

\title{Issues with Positivity-Preserving Patankar-type Schemes}
\author{Davide Torlo\thanks{\href{mailto:davide.torlo@sissa.it}{davide.torlo@sissa.it},
		SISSA mathLab, Mathematics Area, SISSA, via Bonomea 265, Trieste, Italy. },
	    Philipp \"Offner\thanks{\href{mailto:poeffner@uni-mainz.de}{poeffner@uni-mainz.de}, Institut für Mathematik,
		Johannes Gutenberg Universität,
		Staudingerweg 9,
		55099 Mainz, Germany  }, Hendrik Ranocha\thanks{\href{mailto:mail@ranocha.de}{mail@ranocha.de}, Applied Mathematics, University of Münster, Orléans-Ring 10,
		48149 Münster, Germany.}}
\date{\today} 

\makeatletter
\hypersetup{pdfauthor={Davide Torlo, Philipp Öffner, Hendrik Ranocha}}
\hypersetup{pdftitle={\@title}}
\makeatother

\begin{document}

\maketitle

\begin{abstract}
	Patankar-type schemes are linearly implicit time integration methods designed to
	be unconditionally positivity-preserving. \revff{However, there are only little results on their stability or robustness.} We suggest two approaches to analyze the performance and robustness of these methods.
	In particular, we demonstrate problematic behaviors of these methods that, even on very simple linear problems, can
	lead to undesired oscillations and order reduction \revff{for vanishing initial condition}.
	Finally, we demonstrate in numerical simulations that our theoretical results for linear problems apply analogously to nonlinear stiff problems.
\end{abstract}



\section{Introduction}

Many differential equations in biology, chemistry, physics, and engineering
are naturally equipped with constraints such as the positivity of certain
solution components (e.g., density, energy, pressure) and conservation (e.g.,
total mass, momentum, energy). In particular, reaction equations are often of this
form. Typically, such reaction systems can also be stiff.
We consider such ordinary differential equations (ODEs)
\begin{equation}
	\label{eq:ode}
	u'(t)
	=
	f(u(t)),
	\quad
	u(0) = u_0,
\end{equation}
that can be written as a production destruction system (PDS) \citep{burchard2003high}
\begin{equation}
	\label{eq:pd-system}
	f_i(u)
	=
	\sum_{j\in I} (\prod_{ij}(u) - \dest_{ij}(u)),\quad \forall i \in I,
\end{equation}
where $\prod_{ij}, \dest_{ij} \geq 0$ are the production and destruction
terms, respectively. Sometimes, these terms are conveniently written as
matrices $\prod(u) = (\prod_{ij}(u))_{i,j}$ and
$\dest(u) = (\dest_{ij}(u))_{i,j}$.

\begin{definition}
	An ODE \eqref{eq:ode} is called \emph{positive}, if positive initial data $u_0 > 0$
	result in positive solutions $u(t) > 0, \forall t$. Here, inequalities for
	vectors are interpreted componentwise, i.e., $u(t) > 0$ means
	$\forall i \in I\colon u_i(t) > 0$.
	A production destruction system \eqref{eq:pd-system} is called \emph{conservative},
	if $\forall i,j \in I, \forall u\colon \prod_{ij}(u) = \dest_{ji}(u)$.
\end{definition}

A slight generalization of the PDS \eqref{eq:pd-system} is given by
the production destruction rest system (PDRS)
\begin{equation}
	\label{eq:pdr-system}
	f_i(u)
	=
	\rest_i(u) + \sum_{j \in I} (\prod_{ij}(u) - \dest_{ij}(u)),\quad \forall i \in I,
\end{equation}
where $\prod_{ij}$, $\dest_{ij}$ are as before and additional rest terms
$\rest_i$ are introduced. These can of course violate the conservative
nature of a PDS but can still result in a positive solution if
$\rest_i \geq 0$. The rest term can be interpreted as additional force/source term.

The existence, uniqueness and positivity of the solution of a PDS can be proven under the following assumptions \cite{formaggia2011positivity}.
	\begin{theorem}\label{th:ex_un_pos}
		The PDS with initial conditions $u^0 \geq 0$ has a unique solution $u\in [\mathcal{C}^1(\R^+)]^{|I|}$ and $u_i(t) >0$ if $u_i^0>0$, if
		\begin{enumerate}
			\item for all $i,j\in I$ $d_{ij}$ is locally Lipschitz continuous in $\R^{|I|}$,
			\item $d_{ij}(u)=0$ for all $i,j \in I$ if $u=0$,
			\item $d_{ij}(u)= \tilde{d}_{ij}(u)u_i$ with $\tilde{d}_{ij} \in \mathcal{C}((\overline{\R^+})^{|I|})$ and $\tilde{d}_{ij}(u)>0$ if $u>0$ and $\tilde{d}_{ij}(u)=0$ if $u=0$.
		\end{enumerate}
	\end{theorem}
In \cite{burchard2003high} the previous assumptions 2 and 3 are replaced by the condition $d_{ij}(u)\to 0$ if $u_i \to 0$.
It can be easily shown that this condition plus the Lipschitz continuity of the destruction terms lead to similar structures.
Let $C$ be the maximum of the Lipschitz continuity constants of the destruction terms and consider $u=v$ except for the $i$-th component for which $v_i=0$ and, hence, $d_{ij}(v)=0$ for the new condition. We have that
\begin{equation}
0\leq \, d_{ij}(u)=|d_{ij}(u)-d_{ij}(v) |  \leq C ||u-v||_2 = C u_i.
\end{equation}
Hence, we can define
\begin{equation}
	\tilde{d}_{ij}(u) := \frac{d_{ij}(u)}{u_i} \leq C.
\end{equation}
This condition is less restrictive and it does not guarantee the continuity of $\tilde{d}_{ij}$ in $u_i=0$. For the rest of the paper, we will consider assumptions of Theorem~\ref{th:ex_un_pos}. Also, all the physically/chemically/biologically relevant cases, of which we are aware, fall in this definition.

To ensure physically meaningful and robust numerical approximations, we would like
to preserve positivity and conservation discretely.
\begin{definition}
	A numerical method computing $u^{n+1} \approx u(t_{n+1})$ given
	$u^{n} \approx u(t_n)$ is called \emph{conservative}, if
	$\sum_i u^{n+1}_i = \sum_i u^{n}_i$.
	It is called \emph{unconditionally positive}, if $u^{n} > 0$ implies $u^{n+1} > 0$.
\end{definition}

There are several ways to study positivity of numerical methods
\cite{fekete2018positivity}, e.g., based on the concept of
strong stability preserving (SSP) \cite{gottlieb2011strong}
or adaptive Runge--Kutta (RK) methods \cite{nusslein2021positivity}.
However, general linear methods are restricted to
conditional positivity if they are at least second order accurate
\cite{bolley1978conservation}.
One way to circumvent such order restrictions is given by diagonally split
RK methods, which can be unconditionally positive
\cite{hout1996note, bellen1997unconditional, horvath1998positivity}.
However, they are less accurate
than the unconditionally positive implicit--Euler method for large step sizes in
practice \cite{macdonald2008numerical,bolley1978conservation}.

Another approach to unconditionally positivity-preserving methods is based on the
so-called Patankar trick \cite[Section~7.2-2]{patankar1980numerical}. First-
and second-order accurate conservative methods based thereon were introduced
in \cite{burchard2003high}. Later, these were extended to families of second-
and third-order accurate modified Patankar--Runge--Kutta (MPRK) methods based
on the Butcher coefficients \cite{kopecz2018order, kopecz2018unconditionally}
and the Shu--Osher form \cite{huang2019positivity, huang2019third}. Related
deferred correction (DeC) methods were proposed recently \cite{offner2019arbitrary}.
Positive but not conservative methods using the Patankar trick have been
proposed and studied in \cite{chertock2015steady}, although the connection to
Patankar methods seems to be unknown up to now.
Other related numerical schemes are inflow-implicit/outflow-explicit methods
\cite{mikula2011inflow, mikula2014inflow, frolkovic2016semi}. Ideas from
Patankar-type methods have also been used in numerical methods based on limiters
\cite{kuzmin2020entropy}.

The methods mentioned above are based on explicit RK methods. To guarantee
positivity, the schemes are modified to be linearly implicit, which seems to
introduce some stabilization mechanism. In fact, Patankar-type methods have been
applied successfully to some stiff systems \cite{kopecz2018order,
	kopecz2018unconditionally, kopecz2019comparison, chertock2015steady}.
Recently,  Patankar methods have been investigated using Lyapunov stability theory \cite{izgin2022lyapunov, izgin2021recent, izgin2022stability}. We will point out the relation between their approach and our investigations. Lately, BBKS and GeCo, two geometric integrators, have been introduce to simulate biochemistry models preserving not only positivity and conservation, but also all linear invariants of a system \cite{martiradonna2020geco,bruggeman2007second,avila2020comprehensive,broekhuizen2008improved}.

\subsection{Motivating example}
Consider the normal linear system
\begin{equation}
	\label{eq:motivating-example}
	u'(t)
	=
	10^2
	\begin{pmatrix}
		-1 & 1 \\
		1 & -1
	\end{pmatrix}
	u(t),
	\quad
	u(0) = u_0 = \begin{pmatrix} 0.1 \\ 0 \end{pmatrix},
\end{equation}
which can be written as a production destruction system with
\begin{equation}
	\prod(u)
	=
	\begin{pmatrix}
		0 & 10^2 u_2 \\
		10^2 u_1 & 0
	\end{pmatrix},
	\quad
	\dest(u)
	=
	\begin{pmatrix}
		0 & 10^2 u_1 \\
		10^2 u_2 & 0
	\end{pmatrix}.
\end{equation}
On \eqref{eq:motivating-example}, we can show different problematic behaviors.
We solve \eqref{eq:motivating-example} with several different methods. In detail, we apply
the second order method SI-RK2 of \cite{chertock2015steady},  the second- and third-order accurate modified Patankar--Runge--Kutta schemes MPRK(2,2,$\alpha$) and MPRK(4,3,$\alpha$,$\beta$) from \cite{kopecz2018order, kopecz2018unconditionally} with different parameter selections, the implicit Midpoint rule and
 fifth-order, three stage RadauIIA5 scheme \cite{hairer1999stiff}  implemented
in DifferentialEquations.jl \cite{rackauckas2017differentialequations}
in Julia \cite{bezanson2017julia}. The solutions are shown in Figure~\ref{fig:motivating-example}. It can be recognized that even for this simple test case, most of the methods are oscillating for the selected time step but with different amplitudes while RadauIIA5 results in an oscillation--free approximation. We will see that there is a connection between positivity and oscillation--free linear schemes. \\
Another problem rises if we use other Patankar schemes.
These methods are constructed for strictly positive PDS, therefore we have to substitute the zero initial condition with something very small, e.g.  $u_2(0)=10^{-250}$. We observe in  Figure~\ref{fig:inconsistency-motivational} that some of the methods replicate the initial condition for some time steps while others do not leave it at all in the considered time interval. On the other side classical implicit Runge--Kutta method as well as other modified Patankar schemes do not show this behavior and their first time step approaches quickly the steady state value. This issue is linked with a loss of accuracy in the limit for an initial condition approaching zero. \\
In our investigation, we want to find the Patankar methods that have those undesirable behaviors and avoid them.
\begin{figure}[!htb]
	\centering
	\begin{subfigure}{0.49\textwidth}
		\centering
		\includegraphics[width=\textwidth]{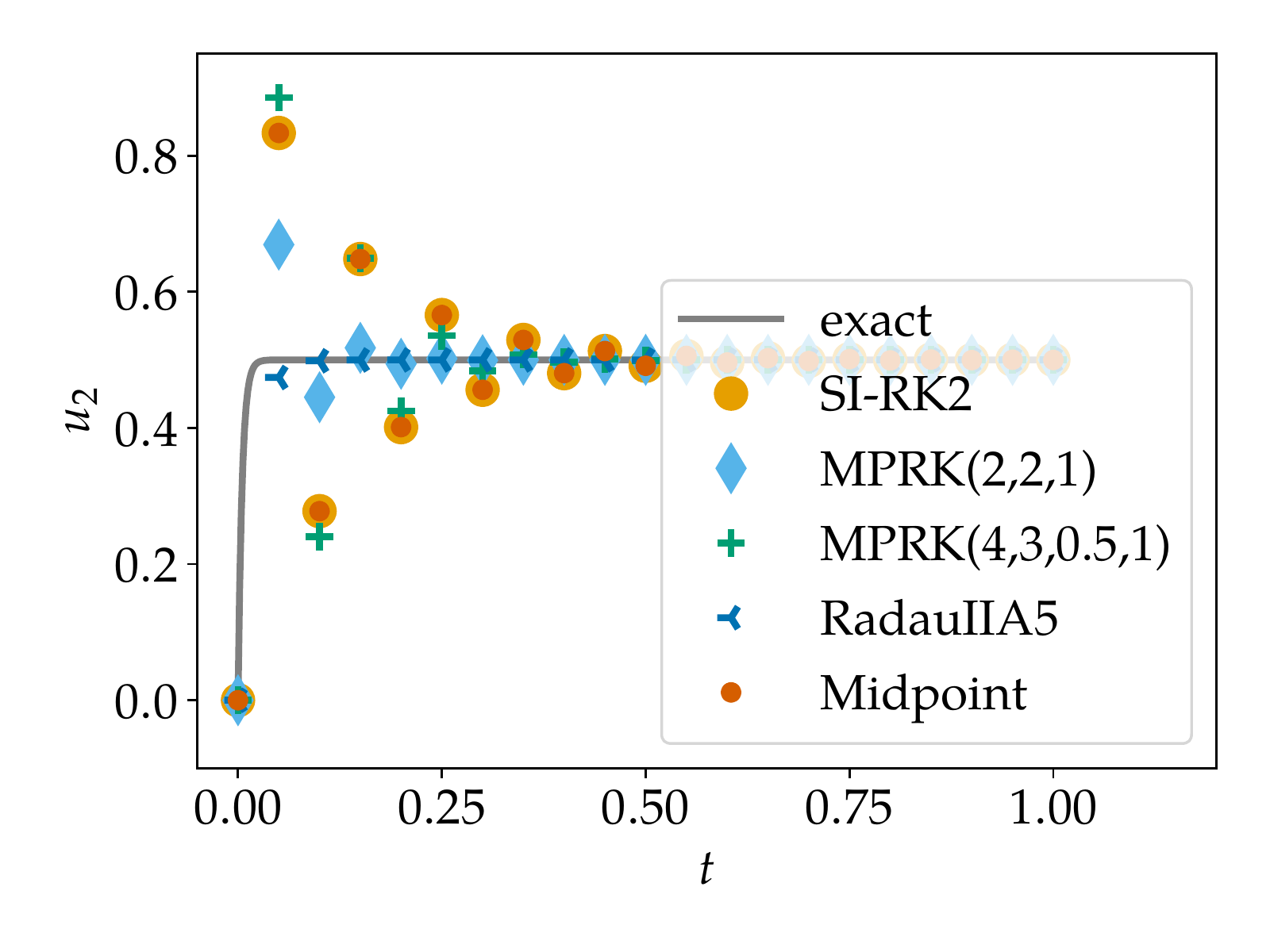}
		\caption{$\dt = 0.05$\label{fig:motivating-example}}
	\end{subfigure}%
	~
	\begin{subfigure}{0.49\textwidth}
		\centering
		\includegraphics[width=\textwidth]{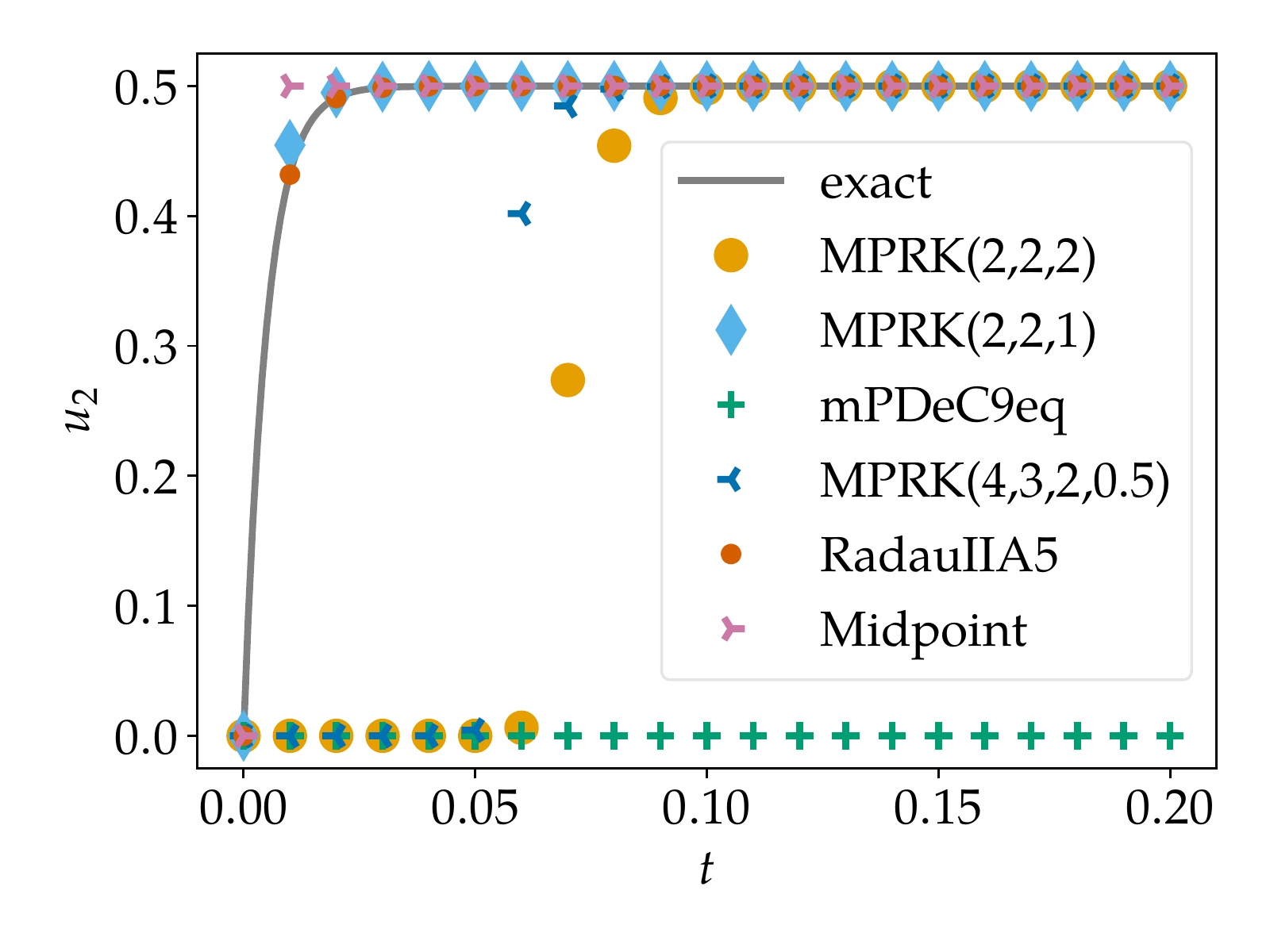}
		\caption{$\dt = 0.01$\label{fig:inconsistency-motivational}}
	\end{subfigure}%
	\caption{Numerical solutions of the normal linear system
		\eqref{eq:motivating-example} with real and non-positive
		eigenvalues obtained using different Patankar-type schemes as well as
		two implicit Runge--Kutta methods (only second component depicted)  with initial condition $u_0=(1,10^{-250})^T$. \label{fig:intro_motivation}}
\end{figure}

\begin{remark}\label{re_stability}
	A stability theory for Patankar type methods is still under development and only few preliminary results have been presented. Recently, in \cite{izgin2022lyapunov, izgin2021recent, izgin2022stability} a promising ansatz to investigate the behavior of conservative and positivity preserving methods has been proposed. In their work, the main idea is to use the center manifold theory corresponding to fixed-point investigations.  First applied on  $2\times 2$ systems in  \cite{izgin2022lyapunov}, the theory has been extended to general $n\times n$ systems in \cite{izgin2022stability}. The main idea is the following: a generic linear system $y'=Ay$ with $A\in \mathbb{R}^{n\times n}$ with initial condition $y_0>0$ possessing
	$k>0$ linear invariants is considered.
	In such a case, zero is always an eigenvalue of $A$ which implies the existence of nontrivial steady state solutions, cf. \cite{izgin2022stability}. The steady state solutions are fixed-points for any reasonable time integration method. Due to the nonlinear  character of Patankar-type schemes (actually for all higher-order positivity preserving schemes), a nonlinear iteration process is obtained. Here, additional techniques have to be used to investigate the stability properties. The authors of
	\cite{izgin2022lyapunov, izgin2021recent, izgin2022stability}  proved a theorem based on the central manifold theorem  which gives sufficient conditions for the stability of all such methods.  It is further demonstrated that MPRK22($\alpha)$ is stable for all $\Delta t>0$, i.e., it will converge to such fixed-points at any rate. \\
	As shown in \cite{izgin2022lyapunov}, we suspect that most of modified Patankar schemes are stable in the fixed-point sense. In our investigation, we do not deal with this type of stability, but, rather, we look for some more restricted schemes that show monotone character for monotone problems and that do not completely lose the high order accuracy. A stability analysis of all the considered methods with respect to the method proposed in \cite{izgin2022stability} is work in progress. Furthermore, the connection between our observations and the obtained eigenvalues of the iterative process will be considered and compared in the future.
\end{remark}

\subsection{Scope of the article}

Motivated by our numerical examples above we are interested in concepts that detect the dominant appearance of spurious oscillations and the loss of accuracy in the limit of an initial condition going to zero.
We have focused on different types of systems (stiff, dissipative ones, etc.) and considered
several quantities like the dissipation of some norms or Lyapunov functionals, cf.
\cite{ranocha2018L2stability,
	ranocha2021strong, ranocha2020energy, sun2017stability, sun2019strong}.
However, the obtained results have not been sufficient for us to describe the properties of the schemes
in an adequate way. Thus, we will directly measure the amount of spurious oscillations using a generic $2\times 2$ linear system  as a test problem, and focus as well as on the loss of accuracy in the limit process.
Our investigation leads to a deeper understanding of the basic properties of Patankar-type methods.

The rest of the article is structured as follows. The numerical schemes studied
in this article are introduced in Section~\ref{sec:numerical-schemes}. In Section~\ref{sec:linearProblem}, we describe the linear problem on which the methods will be studied.
Thereafter, in Section~\ref{sec:stability-linear}, we show the connection between oscillations and positivity for linear problems and linear schemes, then we study the oscillation--free property for RK schemes and for a MPRK scheme.
We continue with an analytical investigation
on the loss of the order of accuracy in the limit of vanishing initial condition in Section~\ref{sec:spuriousSteadyState}. In
Section~\ref{sec:numerical-experiments}, a numerical study on linear systems derives the results on bounds on time step for oscillation--free schemes for all other Patankar schemes.
In Section~\ref{sec:nonlinear-experiements}, we extend the numerical study to
nonlinear and stiff problems. Finally, we summarize and discuss our results in
Section~\ref{sec:summary}.

\section{Numerical schemes}
\label{sec:numerical-schemes}

Here, we introduce Patankar-type methods proposed in the literature that
we will investigate later. In addition, we propose a new MPRK method and give
a heuristic on how to construct such schemes in general.

\subsection{Modified Patankar--Euler method}

The explicit Euler method
$u^{n+1} = u^{n} + \dt f(u^{n})$
can be modified by the Patankar trick \cite[Section~7.2-2]{patankar1980numerical}
for a PDR system \eqref{eq:pdr-system} to get the positive
Patankar--Euler method
\begin{equation}
	\label{eq:PE}
	u^{n+1}_i = u^{n}_i + \dt \rest_i\bigl( u^{n} \bigr)
	+ \dt \sum_j \left(
	\prod_{ij}\bigl( u^{n} \bigr)
	- \dest_{ij}\bigl( u^{n} \bigr) \frac{u^{n+1}_i}{u^{n}_i}
	\right).
\end{equation}
Indeed, given $\rest,\prod,\dest \geq 0$, the new numerical solution $u^{n+1}$ is obtained by
solving a linear system with positive diagonal entries, vanishing off-diagonal
entries, and a positive right-hand side.

Since the Patankar--Euler method \eqref{eq:PE} is not conservative,
the modified Patankar--Euler method
\begin{equation}\tag{MPE}
	\label{eq:MPE}
	u^{n+1}_i = u^{n}_i + \dt \rest_i\bigl( u^{n} \bigr)
	+ \dt \sum_j \left(
	\prod_{ij}\bigl( u^{n} \bigr) \frac{u^{n+1}_j}{u^{n}_j}
	- \dest_{ij}\bigl( u^{n} \bigr) \frac{u^{n+1}_i}{u^{n}_i}
	\right)
\end{equation}
has been introduced in \cite{burchard2003high}
(with additional rest terms $\rest$ here).
The modification of the
production terms makes the method conservative if the rest terms $\rest$ vanish.
Nevertheless, the method is still positive, because the arising linear systems
has positive diagonal entries, negative off-diagonal entries, and is strictly
diagonally dominant. Hence, the system matrix is an $M$ matrix and, since the right-hand side is positive, the solution
$u^{n+1}$ is positive
\cite[Section~6.1]{axelsson1996iterative}.
We observe that, when dealing with the scalar linear test problem $u' = \lambda u$ with $\lambda < 0$,
the Patankar--Euler method coincides with the implicit--Euler method.
Similarly, \ref{eq:MPE} coincides with the implicit--Euler method if we deal with positive and conservative linear PDS. Indeed, the destruction terms $d_i(u)=\sum_j d_{ij}(u)$ must go to 0 if $u_i \to 0$ \cite{burchard2003high}.
Since the system is linear, $d_{ij}(u^n)=\tilde{d}_{ij} u_i^n$ with $\tilde{d}_{ij}\in \R^+_0$. Exploiting the conservation properties, we have $p_{ji}(u^n)=\tilde{d}_{ij} u^n_i$. Substituting these formulae in \ref{eq:MPE} leads to the implicit--Euler method.

\subsection{MPRK methods using Butcher coefficients}

A one-parameter family of MPRK schemes based on the Butcher
coefficients of a two stage, second-order RK method was introduced
in \cite{kopecz2018order}. Given a parameter $\alpha \in [1/2, \infty)$, the
method is
\begin{equation}
	\tag{MPRK(2,2,$\alpha$)}
	\label{eq:MPRK22-family}
	\begin{aligned}
		y^1 &= u^n,
		\\
		y^2_i &= u^n_i
		+ \alpha \dt \rest_i\bigl( y^1 \bigr)
		+ \alpha \dt \sum_j \left(
		\prod_{ij}\bigl( y^1 \bigr) \frac{y^2_j}{y^1_j}
		- \dest_{ij}\bigl( y^1 \bigr) \frac{y^2_i}{y^1_i}
		\right),
		\\
		u^{n+1}_i &= u^n_i
		+ \dt \left( \frac{2\alpha-1}{2\alpha} \rest_i\bigl( y^1 \bigr) + \frac{1}{2\alpha} \rest_i\bigl( y^2 \bigr) \right)
		\\&\qquad
		+ \dt \sum_j \Biggl(
		\left( \frac{2\alpha-1}{2\alpha} \prod_{ij}\bigl( y^1 \bigr) + \frac{1}{2\alpha} \prod_{ij}\bigl( y^2 \bigr) \right) \frac{u^{n+1}_j}{(y^2_j)^{1/\alpha} (y^1_j)^{1-1/\alpha}}
		\\&\qquad\qquad\qquad
		- \left( \frac{2\alpha-1}{2\alpha} \dest_{ij}\bigl( y^1 \bigr) + \frac{1}{2\alpha} \dest_{ij}\bigl( y^2 \bigr) \right) \frac{u^{n+1}_i}{(y^2_i)^{1/\alpha} (y^1_i)^{1-1/\alpha}}
		\Biggr).
	\end{aligned}
\end{equation}
The scheme for the choice $\alpha = 1$ is based
on Heun's method and has been proposed already in
\cite{burchard2003high}. Heun's method can be also written as a strong stability preserving Runge--Kutta method (SSPRK) and we will denote it by SSPRK(2,2) \cite{gottlieb2011strong}.

A similar two-parameter family MPRK(4,3,$\alpha$,$\beta$) of four stage, third-order accurate schemes was introduced and studied in \cite{kopecz2018unconditionally, kopecz2019existence}.
The  family under consideration  can be found in the \ref{sec:appendix} for completeness.

\subsection{MPRK methods using Shu--Osher coefficients}

A two-parameter family of MPRK schemes based on the Shu--Osher coefficients of
a two stage, second-order RK method was introduced in \cite{huang2019positivity}.
Given parameters $\alpha, \beta$, the method is
\begin{equation}
	\tag{MPRKSO(2,2,$\alpha$,$\beta$)}
	\label{eq:MPRKSO22-family}
	\begin{aligned}
		y^1 &= u^n,
		\\
		y^2_i &= y^1_i
		+ \beta \dt \rest_i\bigl( y^1 \bigr)
		+ \beta \dt \sum_j \left(
		\prod_{ij}\bigl( y^1 \bigr) \frac{y^2_j}{y^1_j}
		- \dest_{ij}\bigl( y^1 \bigr) \frac{y^2_i}{y^1_i}
		\right),
		\\
		u^{n+1}_i &= (1-\alpha) y^1_i + \alpha y^2_i
		+ \dt \left( (1-\frac{1}{2\beta}-\alpha\beta) \rest_i\bigl( y^1 \bigr) + \frac{1}{2\beta} \rest_i\bigl( y^2 \bigr) \right)
		\\&\qquad
		+ \dt \sum_j \Biggl(
		\left( (1-\frac{1}{2\beta}-\alpha\beta) \prod_{ij}\bigl( y^1 \bigr) + \frac{1}{2\beta} \prod_{ij}\bigl( y^2 \bigr) \right) \frac{u^{n+1}_j}{(y^2_j)^{\gamma} (y^1_j)^{1-\gamma}}
		\\&\qquad\qquad\qquad
		- \left( (1-\frac{1}{2\beta}-\alpha\beta) \dest_{ij}\bigl( y^1 \bigr) + \frac{1}{2\beta} \dest_{ij}\bigl( y^2 \bigr) \right) \frac{u^{n+1}_i}{(y^2_i)^{\gamma} (y^1_i)^{1-\gamma}}
		\Biggr),
	\end{aligned}
\end{equation}
where the parameters are restricted to $\alpha \in [0,1]$, $\beta \in (0,\infty)$,
$\alpha \beta + \frac{1}{2\beta} \leq 1$, and
\begin{equation}
	\gamma = \frac{1 - \alpha \beta + \alpha \beta^2}{\beta ( 1 - \alpha \beta)},
\end{equation}
in order to be positive. In our simulations, we will exchange the weights of production and destruction when the coefficients are negative. In the next section we will give an example of such inversion.
An extension to four stage, third-order accurate methods MPRKSO(4,3) was developed
in \cite{huang2019third} and can be found in the \ref{sec:appendix}.

\subsection{Modified Patankar deferred correction schemes}

Arbitrarily high order conservative and positive modified Patankar deferred
correction schemes (mPDeC) were introduced in \cite{offner2019arbitrary}.
A time step $[t^n, t^{n+1}]$ is divided into $M$ sub-intervals, where $t^{n,0}=t^n$
and $t^{n,M}=t^{n+1}$. For every sub-interval, the Picard-Lindel\"of theorem is mimicked. At each sub-time step $t^{n,m}$, an approximation $y^{m}$ is calculated.
In the formulation of  \cite{abgrall2017high} an iterative procedure of $K$ correction steps improves the approximation by one
order of accuracy at each iteration.
The modified Patankar trick is introduced inside the basic scheme to guarantee  positivity and conservation of the intermediate approximations.
Using the fact that initial states $y_i^{0,(k)}=u_i^n$
are identical for any correction $k$, the mPDeC correction steps
can be rewritten for $k=1,\dots,K$, $m =1,\dots, M$ and $\forall i\in I$ as
\begin{equation}
	\tag{mPDeC}
	\label{eq:explicit_dec_correction}
	\begin{split}
	&y_i^{m,(k)}-y^0_i -\sum_{r=0}^M \theta_r^m \Delta t \rest_i \bigl(y^{r,(k-1)} \bigr)-\\
	&\sum_{l=0}^M \theta_l^m \Delta t  \sum_{j=1}
	\left( \prod_{ij}(y^{l,(k-1)})
	\frac{y^{m,(k)}_{\gamma(j,i, \theta_r^m)}}{y_{\gamma(j,i, \theta_l^m)}^{m,(k-1)}}
	- \dest_{ij}(y^{l,(k-1)})  \frac{y^{m,(k)}_{\gamma(i,j, \theta_l^m)}}{y_{\gamma(i,j, \theta_l^m)}^{m,(k-1)}} \right)=0,
	\end{split}
\end{equation}
where $\theta_r^m$ are the correction weights and the $\gamma(j,i,\theta_r^m)$ takes value $j$ if $\theta_r^m>0$ and $i$ otherwise, see \cite{offner2019arbitrary} for details. This allows to obtain always positive terms in the diagonal terms and nonpositive in the offdiagonal terms of the system matrix. Finally, the new numerical solution
is $u_i^{n+1}=y_{i}^{M,(K)}$.

The choice of the distribution and the number of sub-time steps $M$ and the number of iterations $K$ determines the order of accuracy of the scheme. In the following, we will compare equispaced and Gauss--Lobatto points. To reach order $d$, we use $M=d-1$ sub-intervals and $K=d$ corrections.
We will denote the $p$th-order mPDeC method as \ref{eq:explicit_dec_correction}$p$.
Note that \ref{eq:explicit_dec_correction}1 is equivalent to \ref{eq:MPE} and
\ref{eq:explicit_dec_correction}2 is equivalent to
\hyperref[eq:MPRK22-family]{MPRK(2,2,1)}.

\subsection{A new MPRK method}

We propose the following new three stage, second-order MPRK method based on
SSPRK(3,3):
\begin{equation}
	\tag{MPRK(3,2)}
	\label{eq:MPRK32}
	\begin{aligned}
		y^1_i &= u^{n}_i,
		\\
		y^2_i &= u^{n}_i
		+ \dt \rest_i\bigl( y^1 \bigr)
		+ \dt \sum_j \left(
		\prod_{ij}\bigl( y^1 \bigr) \frac{y^2_j}{y^1_j}
		- \dest_{ij}\bigl( y^1 \bigr) \frac{y^2_i}{y^1_i}
		\right),
		\\
		y_i^3 &= u^{n}_i\\&
		+ \dt \frac{\rest_i\bigl( y^1 \bigr) + \rest_i\bigl( y^2 \bigr)}{4}
		+ \dt \sum_j \left(
		\frac{\prod_{ij}\bigl( y^1 \bigr) + \prod_{ij}\bigl( y^2 \bigr)}{4} \frac{y^3_j}{y^2_j}
		- \frac{\dest_{ij}\bigl( y^1 \bigr) + \dest_{ij}\bigl( y^2 \bigr)}{4} \frac{y^3_i}{y^2_i}
		\right),
		\\
		u^{n+1}_i &= u^{n}_i
		+ \dt \frac{\rest_i\bigl( y^1 \bigr) + \rest_i\bigl( y^2 \bigr) + 4 \rest_i\bigl( y^3 \bigr)}{6}
		\\&
		+ \dt \sum_j \Biggl(
		\frac{\prod_{ij}\bigl( y^1 \bigr) + \prod_{ij}\bigl( y^2 \bigr) + 4 \prod_{ij}\bigl( y^3 \bigr)}{6} \frac{u^{n+1}_j}{y^2_j}
		\\&
		\qquad\qquad- \frac{\dest_{ij}\bigl( y^1 \bigr) + \dest_{ij}\bigl( y^2 \bigr) + 4 \dest_{ij}\bigl( y^3 \bigr)}{6} \frac{u^{n+1}_i}{y^2_i}
		\Biggr).
	\end{aligned}
\end{equation}
For explicitly time-dependent problems, the abscissae are the ones of
SSPRK(3,3) \cite{gottlieb2011strong}, i.e., $c = (0, 1, 0.5)$. As will be seen later, this scheme
has some desirable robustness.
\ref{eq:MPRK32} is second-order accurate. We will not provide a formal proof of the accuracy of the scheme. Nevertheless we summarize the reasons of the accuracy of each stage. The second stage \\$y_i^2 = u_i(t^{n+1}) + \mathcal{O}(\dt^2)$ is an approximation of order one and we can observe that the ratios $\frac{y^2_i}{y^1_i}=1+\mathcal{O}(\dt)$ do not further decrease the accuracy since they are multiplied by $\dt$. The third stage is as well a first order approximation
	$
	y^3_i = u_i(t^n+\dt/2) + \mathcal{O}(\dt^2).
	$
	Indeed, even if the midpoint rule is a second order quadrature formula, the ratios $\frac{y_i^3}{y_i^2}=\frac{u_i(t^n+\dt/2)+\mathcal{O}(\dt^2)}{u_i(t^n+\dt)+\mathcal{O}(\dt^2)}=1+\mathcal{O}(\dt).$
In the final stage, the Simpson rule is applied, where we get only second order
accuracy since $y^2$ and $y^3$ carry a first order error with them. Hence, $$u^{n+1}_i = u_i(t^{n+1}) +\mathcal{O}(\dt^3),$$ and this gives us ratios $\frac{u^{n+1}_i}{y^2_i}=\frac{u_i(t^{n+1}) +\mathcal{O}(\dt^3)}{u_i(t^{n+1}) +\mathcal{O}(\dt^2)}=1+\mathcal{O}(\dt^2)$ which are multiplied by $\dt$.
At the end, the scheme is second-order accurate.
\begin{remark}
	The construction of higher-order MPRK schemes can be done in a similar way.
	The basic idea is to create a method with increasing stage order, similar
	to the construction of mPDeC.
	Starting from a high order RK scheme, by applying the modified Patankar trick
	in the substeps in combination with quadrature rules should lead to high order
	modified Patankar RK schemes. Essential in the construction is the fact that
	more stages have to be applied compared to classical RK schemes. This is in
	accordance with the result of \cite{kopecz2019existence} on the existence of
	third-order, three stages MPRK schemes.
	There is work in progress to describe a general recipe to construct MPRK schemes
	of arbitrary order and to study the properties of these schemes.
\end{remark}

\subsection{Semi-implicit methods}

The semi-implicit methods of \cite{chertock2015steady} are also based on
the Shu--Osher representation of SSPRK methods, which can be decomposed
into convex combinations of the previous step value and explicit Euler steps.
Instead of introducing Patankar weights multiplying all destruction terms for
a step/stage update, a Patankar weight is introduced for the destruction terms
of each Euler stage which is used to compute the new value. Since this procedure
limits the order of accuracy of the resulting scheme to first order, an
additional function evaluation is used to correct the final solution and get
second order of accuracy.

The two methods proposed in \cite{chertock2015steady} are
\begin{equation}
\tag{SI-RK2}
\label{eq:SI-RK2}
\begin{aligned}
  y^1 &= u^n,
  \\
  y^2_i &= \frac{ u^n_i + \dt \rest_i(y^1) + \dt \sum_j \prod_{ij}(y^1) }
                { 1 + \dt \sum_j \dest_{ij}(y^1) / y^1_i },
  \\
  y^3_i &= \frac{1}{2} u^n_i + \frac{1}{2}
           \frac{ y^2_i + \dt \rest_i(y^2) + \dt \sum_j \prod_{ij}(y^2) }
                { 1 + \dt \sum_j \dest_{ij}(y^2) / y^2_i },
  \\
  u^{n+1}_i &= \frac{ y^3_i + \dt^2 \bigl(\rest_i(y^3) + \sum_j \prod_{ij}(y^3) \bigr) \sum_j \dest_{ij}(y^3) / y^3_i }
                { 1 + \bigl( \dt \sum_j \dest_{ij}(y^3) / y^3_i \bigr)^2 },
\end{aligned}
\end{equation}
which uses three stages and is based on SSPRK(2,2), and
\begin{equation}
\tag{SI-RK3}
\label{eq:SI-RK3}
\begin{aligned}
  y^1 &= u^n,
  \\
  y^2_i &= \frac{ u^n_i + \dt \rest_i(y^1) + \dt \sum_j \prod_{ij}(y^1) }
                { 1 + \dt \sum_j \dest_{ij}(y^1) / y^1_i },
  \\
  y^3_i &= \frac{3}{4} u^n_i + \frac{1}{4}
           \frac{ y^2_i + \dt \rest_i(y^2) + \dt \sum_j \prod_{ij}(y^2) }
                { 1 + \dt \sum_j \dest_{ij}(y^2) / y^2_i },
  \\
  y^4_i &= \frac{1}{3} u^n_i + \frac{2}{3}
           \frac{ y^3_i + \dt \rest_i(y^3) + \dt \sum_j \prod_{ij}(y^3) }
                { 1 + \dt \sum_j \dest_{ij}(y^3) / y^3_i },
  \\
  u^{n+1}_i &= \frac{ y^4_i + \dt^2 \bigl(\rest_i(y^4) + \sum_j \prod_{ij}(y^4) \bigr) \sum_j \dest_{ij}(y^4) / y^4_i }
                { 1 + \bigl( \dt \sum_j \dest_{ij}(y^4) / y^4_i \bigr)^2 },
\end{aligned}
\end{equation}
which uses four stages and is based on SSPRK(3,3).

The relation to Patankar schemes becomes obvious by rewriting the computation
of the stage $y^2$ of \eqref{eq:SI-RK2} as
\begin{equation}
	y^2_i
	=
	u^n_i
	+ \dt \rest_i(y^1)
	+ \dt \sum_j \left(
	\prod_{ij}(y^1)
	- \dest_{ij}(y^1) \frac{y^2_i}{y^1_i}
	\right),
\end{equation}
which is the Patankar--Euler method \eqref{eq:PE}.
As for the Patankar--Euler method, the semi-implicit methods of \cite{chertock2015steady} are not conservative, i.e., it is not guaranteed that $\sum_i u^n_i = \sum_i u^{n+1}_i$ when the system is conservative.

\subsection{ Steady state preservation}
\label{sec:steady-state}

Motivated by the investigations of \cite{chertock2015steady}, steady state
preservation for (modified) Patankar methods will be studied
here. Except for the SI-RK2 and SI-RK3 methods \cite{chertock2015steady},
such investigations cannot be found in the literature.

\begin{definition}
	A method is steady state preserving if, given a time step $\dt$ and $u^n = u^*$
	with $r_i(u^*) + \sum_j p_{ij}(u^*)-d_{ij}(u^*)=0$, then $u^{n+1} = u^n = u^*$.
\end{definition}

\begin{proposition}
	All (modified) Patankar methods described above are steady state
	preserving.
\end{proposition}
\begin{proof}
	The solution to each stage and the new step value are unique. If the
	initial condition is a steady state, this steady state is also a valid
	solution to all stage and step equations. Indeed, the Patankar weights reduce to 1 and the simple rest-production-destruction forms remains and their sum is 0 in the steady state. Hence, the steady state is
	preserved.
\end{proof}

This theorem is important, since some related modifications of explicit
Runge--Kutta methods such as IMEX methods are not necessarily steady
state preserving \cite{chertock2015steady}. For (stiff) systems with an
initial condition near a steady state, the ability to preserve this steady
state exactly is desirable and usually results in a better approximation
of solutions nearby or decaying to steady state.

In our discussion, it will be useful to check not only the preservation of the steady state, but also how this state is approached, for example, if in a monotone manner or not.

\section{The simplest production destruction system}\label{sec:linearProblem}
In order to study the issues observed in Figure~\ref{fig:intro_motivation}, we will consider the simplest production destruction system that one can build. For ODE solvers, it is always useful to study Dahlquist's  equation as any linearized (and diagonalizable) system can be recast into several of these equations. Unfortunately, Dahlquist's equation is not a PDS. We propose to use a $2\times 2$ linear system similar to
\eqref{eq:motivating-example} as test problem. This is the simplest PDS that can be considered. More precisely, we consider the general $2 \times 2$
production-destruction linear system as also done lately in similar form in \cite{izgin2022lyapunov}
\begin{equation}\label{eq:linearSystemReallyGeneral}
	\begin{pmatrix}
		u_1'\\u_2'
	\end{pmatrix} = \begin{pmatrix}
		-a & b \\ a & -b
	\end{pmatrix}
	\begin{pmatrix}
		u_1\\u_2
	\end{pmatrix}.
\end{equation}
Rescaling the time, we can simplify this system to a one parameter system setting
$a+b=1$ and $0\leq\theta = a\leq 1$, i.e.,
\begin{equation}\label{eq:linearsystem2general}
	\begin{pmatrix}
		u_1'\\u_2'
	\end{pmatrix} = \begin{pmatrix}
		-\theta & (1-\theta) \\ \theta & -(1-\theta)
	\end{pmatrix}
	\begin{pmatrix}
		u_1\\u_2
	\end{pmatrix}.
\end{equation}
We can also rescale any initial condition $u^0=(u^0_1,u^0_2)^T$ to sum up to one
(scaling by a factor $\frac{1}{u_1^0+u_2^0}$). Thus, we consider the initial condition
\begin{equation}\label{eq:linearsystemIC}
	\begin{pmatrix}
		u_1^0\\u_2^0
	\end{pmatrix} =
	\begin{pmatrix}
		1-\varepsilon\\\varepsilon
	\end{pmatrix}
\end{equation}
with $0<\varepsilon<1$.
The exact solution of the problem is
\begin{equation}
	\begin{pmatrix}
		u_1(t)\\u_2(t)
	\end{pmatrix} =
	\begin{pmatrix}
		(1-\theta)+(\theta-\varepsilon)e^{-t}\\
		\theta+(\varepsilon-\theta)e^{-t}
	\end{pmatrix},
\end{equation}
and the steady state of the system is $u^*=(1-\theta,\theta)^T$.

It is interesting to rewrite the system \eqref{eq:linearsystem2general} in its diagonal form to highlight its connection with Dahlquist's equation. To do so, let us put it into a matrix formulation\begin{equation}\label{eq:systemMatrix}
	u'=Mu=\begin{pmatrix}
		-\theta & (1-\theta)\\
		\theta & -(1-\theta)
	\end{pmatrix}u,
\end{equation}
we can obtain the diagonal form
$M = L^{-1} \Lambda L$
of the system, i.e.,
\begin{equation}\label{eq:eigenStructure}
\Lambda = \begin{pmatrix}
	-1 & 0 \\ 0& 0
\end{pmatrix}; \quad L = \begin{pmatrix}
\theta  & -(1-\theta) \\ 1 & 1
\end{pmatrix}; \quad L^{-1} = \begin{pmatrix}
1 &1-\theta \\ -1& \theta
\end{pmatrix}.
\end{equation}
So for $v=Lu$ we can write an always positive (or always negative) exact solution for the first component
\begin{equation}\label{eq:v1_equation}
v_1=\theta u_1 - (1-\theta) u_2; \qquad v_1'=-v_1; \quad v_1=e^{-t}v_1^0.
\end{equation}
Indeed, this component is the solution of Dahlquist's equation, while the second component $v_2=u_1+u_2$ fulfills $v_2' = 0$ and corresponds to the conservation property.

The positivity of $v_1=\theta u_1 - (1-\theta) u_2$ in case $\theta u_1^0 - (1-\theta) u_2^0>0$, i.e., $\frac{u_1^0}{u_2^0}>\frac{1-\theta}{\theta}$, is equivalent to say that
\begin{equation}
 	\frac{u_1(t)}{u_2(t)}>\frac{1-\theta}{\theta} \Longleftrightarrow u_2(t) <\theta = u_2^*
\end{equation}
holds true for all times. This condition means that the solution does not overshoot the asymptotic steady state. This property guarantees the monotonicity of the solution. In the next section, we will see how violating this condition leads to oscillations around the asymptotic steady state.

\section{Oscillation--free schemes for linear problems}
\label{sec:stability-linear}

Now, let us reconsider the system \eqref{eq:linearsystem2general}.
In this section we try to find schemes that do not show oscillatory behavior as the ones presented in Figure \ref{fig:motivating-example}.
This reduces to finding schemes that for every $u^n$ and every system defined through
$0<\theta<1$ have a monotone behavior and do not overshoot/undershoot the steady state solution.
In particular, we define two properties that the schemes have to fulfill not to oscillate. We focus on the case $\varepsilon<\theta$ as the opposite one can be obtained switching the two components of the system \eqref{eq:linearsystem2general}.
\begin{property}[Not overshooting the steady state]\label{prop:not_over}
	A method is not overshooting the steady state of \eqref{eq:linearsystem2general} if $u^1_2<\theta$ and  $u_1^1 >(1-\theta)$ given any initial state $u^0=(1-\varepsilon,\varepsilon)$ with $\varepsilon <\theta$, while when $\varepsilon >\theta$ the method is not overshooting the steady state if $u^1_2>\theta$ and  $u_1^1 <(1-\theta)$.
\end{property}
\begin{property}[Correct direction]\label{prop:correct}
	A method is evolving in the correct direction for system \eqref{eq:linearsystem2general} if $u^1_2>\varepsilon$ and  $u_1^1 <(1-\varepsilon)$ given any initial state $u^0=(1-\varepsilon,\varepsilon)$ with $\varepsilon <\theta$, while when $\varepsilon >\theta$ the method is evolving in the correct direction if $u^1_2<\varepsilon$ and  $u_1^1 >(1-\varepsilon)$.
\end{property}

In the following we will focus mainly on Property \ref{prop:not_over}. Indeed, a similar analysis can be conduct to check when Property~\ref{prop:correct} is preserved and we put it in \ref{app:prove_direction}. Moreover, we have observed that in very few occasions the approximation moves in the \textit{wrong} direction, i.e., if $\varepsilon<\theta$ we rarely have that $u_2^1 < \varepsilon$. The interesting condition is $u_2 < \theta$, or, equivalently $\frac{u_2}{u_1}< \frac{\theta}{(1-\theta)}$. We have already shown that this condition is equivalent to preserving the positivity of the first component of the diagonalized system \eqref{eq:v1_equation}.

\begin{proposition}[Oscillation-free and positive Runge-Kutta methods]\label{prop:equiOscPosRK}
	Consider the linear system \eqref{eq:linearsystem2general} with $\varepsilon<\theta$.
	For a linear method such as RK methods, the positivity of $v_1^n=\theta u_1^n-(1-\theta) u_2^n$ is equivalent to not overshooting Property \ref{prop:not_over}, i.e., $\theta=u_2^*> u_2^n\Longleftrightarrow (1-\theta)=u_1^* <  u_1^0$. Similarly, in case $\varepsilon>\theta,$ the negativity of $v_1^n$ is equivalent to the   Property \ref{prop:not_over} condition, i.e., $\theta=u_2^*< u_2^n\Longleftrightarrow (1-\theta)=u_1^*>  u_1^0$.
\end{proposition}
\begin{proof}
	First of all, let us notice that in case $\varepsilon<\theta$, we have that $v_1^0 = \theta u_1^0 - (1-\theta)u_2^0 = \theta (1-\varepsilon) - (1-\theta)\varepsilon >0$.
Let us consider a RK scheme in Einstein's notation, denoting the $s$th RK stage with $y_i^{(s)}$ and with $A,\,b,\,c$ the usual RK matrix and vectors \cite{hairer2000solving}. The RK method  for the system \eqref{eq:linearsystem2general} can be written as
\begin{align}
	&y^{(s)}_i = u_i^n+\dt M^j_iA^s_ky_j^{(k)}\\
	&u^{n+1}_i =  u_i^n+\dt b_s M^j_i y_j^{(s)},
\end{align}
where $M$ is the matrix in \eqref{eq:systemMatrix}.
Now, premultiplying by $L$ defined in \eqref{eq:eigenStructure} we obtain the same RK method for the diagonalized system, i.e.,
\begin{align}
	&w^{(s)}_\ell := L_\ell^i y^{(s)}_i = L_\ell^i \left( u_i^n+\dt  M^j_iA^s_ky_j^{(k)}\right) = v^n_\ell + \dt \Lambda^j_\ell A^s_k w_j^{(k)}  \\
	&v_\ell^{n+1} := L_\ell^i u^{n+1}_i = L_\ell^i  \left( u_i^n+\dt b_s M^j_i u_j^{(s)} \right) =   v_\ell^n+\dt b_s \Lambda^j_\ell w_j^{(s)} .
\end{align}
Hence, if for a certain $\dt$ we have that $v_1^1>0$ for all $\varepsilon < \theta$, then, $u_2^1< \theta$ and would not overshoot the asymptotic steady state.
The other case is proved analogously.
\end{proof}
As an example, the implicit--Euler method is unconditionally positive
and thus also unconditionally oscillations--free.\\
It is also clear how to check the positivity and, hence, \revff{Property~\ref{prop:not_over}} for all RK schemes.
\begin{proposition}
	Consider the problem \eqref{eq:linearsystem2general} and a RK method. For a given $\dt$ the method \revff{fulfills Property~\ref{prop:not_over}} if
	\begin{equation}
		R(-\dt) >0,
	\end{equation}
with $R(z):=(1+z b^T(I-zA)^{-1}\mathbbm{1})$ the stability function of the RK method.
\end{proposition}
\begin{proof}
	From Proposition~\ref{prop:equiOscPosRK} we know that we can check the positivity of $v_1^1$ for the equation $v_1'=-v_1$, with initial condition $v_1^0>0$. We have then,
	\begin{align}
		w=&v_1^0 \mathbbm{1} - \dt A w \Longleftrightarrow (I+\dt A) w = v_1^0 \mathbbm{1}  \Longleftrightarrow  w = (I+\dt A)^{-1} \mathbbm{1}v_1^0 \\
		v^1_1=&v_1^0 + \dt b^T w =v_1^0 + \dt b^T(I+\dt A)^{-1} \mathbbm{1}v_1^0 =(1+\dt b^T(I+\dt A)^{-1} \mathbbm{1})v_1^0\\
		=&R(-\dt)v_1^0.
	\end{align}
Hence, having $R(-\dt)>0$ guarantees the positivity of the scheme for $v_1$ and the condition $u_2<\theta$ on system \eqref{eq:linearsystem2general}.
\end{proof}
To check this condition is quite straightforward for most RK schemes. Indeed, $R$ is a ratio of two polynomials and checking its positivity corresponds to finding roots of some polynomials.
\begin{remark}[Positivity of RK schemes]
	One should notice that a positive RK method is not usually defined such that $R(-\dt)>0$. Indeed, it is important in many contexts that also all the stages stay positive. For this definition one should require that $(I+\dt A)^{-1}$ is a positive matrix. It has been proven \cite{bolley1978conservation,gottlieb2011strong} that among linear implicit schemes only first order schemes can be unconditionally (for all $\dt>0$) positive, while all high order schemes cannot. Nevertheless, some schemes can be unconditionally positive only in the final update. An example of such schemes is RadauIIA5, which, being fifth order accurate cannot be positive for all stages \cite{bolley1978conservation}, but it is in the final update, see Table~\ref{tab:dtImplicit}.
\end{remark}

For explicit schemes it is known that explicit Euler is positive for $\dt<1$ and for all strong-stability-preserving RK (SSPRK) schemes, which are convex combination of explicit Euler steps, the positivity is obtained for $\dt<\mathcal{C}$, where $\mathcal{C}$ is their CFL condition \cite{gottlieb2011strong}. For all these scheme the CFL coefficient is well known in literature and we do not further discuss it.
For implicit schemes this conditions seems not to have been thoroughly studied to the authors' knowledge. In Table~\ref{tab:dtImplicit} we summarize the restrictions for some of the implicit RK methods obtained with a \texttt{Mathematica} notebook available in \cite{torlo2021stabilityGit}.
\begin{table}
	\centering
	\begin{tabular}{|c|c||c|c|}\hline
		Method& Condition & Method & Condition\\ \hline\hline
		Radau IA3 & $\dt<3$ & Radau IA5 & Always \\ 
		Radau IIA3 & $\dt <3$ & Radau IIA5 & Always \\
		Lobatto IIIA2 & $\dt<2$  & Lobatto IIIA4 & Always\\
		Lobatto IIIB2 & $\dt<2$ & Lobatto IIIB4 & Always	\\
		Lobatto IIIC2 & Always & Lobatto IIIC4 & $\dt<4$ \\
		 Gauss--Legendre 4& Always & Gauss--Legendre 6& $\dt\lesssim 4.32$ \\
		implicit--Euler & Always & 	Midpoint & $\dt<2$	\\
		Trapezoid &$\dt<2$ &	Qin-Zhang DIRK2 & $\dt\neq 4$ \\
			TRBDF2 & $\dt< 1+\sqrt{2}$ &Kraaijevanger-Spijker DIRK2 & Always \\
			 \hline
	\end{tabular}
\caption{List of methods and condition on $R(-\dt)>0$}\label{tab:dtImplicit}
\end{table}

Similarly, we state a proposition for Property~\ref{prop:correct}.
	\begin{proposition}
		Consider the problem \eqref{eq:linearsystem2general} and a RK method. For a given $\dt$ the method fulfills Property~\ref{prop:correct} if
		\begin{equation}
			1-R(-\dt) >0 \Longleftrightarrow b^T(I+\dt A)^{-1}\mathbbm{1}>0,
		\end{equation}
		with $R(z):=(1+z b^T(I-zA)^{-1}\mathbbm{1})$ the stability function of the RK method.
	\end{proposition}
All the schemes presented in Table~\ref{tab:dtImplicit} enjoy Property~\ref{prop:correct} unconditionally. Moreover, every A-stable scheme enjoy Property~\ref{prop:correct}. Indeed, A-stability means that
\begin{equation}
	|R(z)|<1 \text{ for } \Re(z)<0 \Longrightarrow R(-\dt)<1 \text{ for } \dt>0.
\end{equation}

Modified Patankar methods are not linear schemes. Hence, the equivalence in Proposition~\ref{prop:equiOscPosRK} does not hold. So, even if they are unconditionally positivity preserving, they are not unconditionally oscillation--free. It is not straightforward to derive an analysis for all of them. In next section, we study the \ref{eq:MPRK22-family} with $\alpha=1$, for which it is possible to derive a condition on the time step to obtain the oscillation-free condition. For all other schemes we have to perform some numerical studies, see Section~\ref{sec:numerical-experiments}.

\subsection{Oscillatory-free restrictions of MPRK(2,2,1)}

The method \ref{eq:MPRK22-family} with $\alpha=1$ is equivalent to
\ref{eq:explicit_dec_correction}2. Since it is simple enough, a detailed analysis
for the simplified linear systems \eqref{eq:linearsystem2general} is feasible.

\begin{theorem}[Time restriction for \ref{eq:explicit_dec_correction}2 for $2\times 2$ linear systems]
	Consider the system \eqref{eq:linearsystem2general} with the initial conditions \eqref{eq:linearsystemIC}. \ref{eq:explicit_dec_correction}2 enjoys Properties \ref{prop:not_over} and \ref{prop:correct} for any initial condition $0< \varepsilon< 1$ and any system $0\leq \theta\leq 1$ under the time step restriction $\Delta t \leq 2$. For the general linear system \eqref{eq:linearSystemReallyGeneral} the time restriction is $\Delta t \leq \frac{2}{a+b}$.
	\begin{proof}
		First of all, the cases $\theta=0$ and $\theta=1$ are trivially verified as the steady state solutions are $(1,0)^T $ and $(0,1)^T$, respectively. Since the scheme is positive, $ 0 < u_1^n,u_2^n < 1$ holds for any possible initial condition and time step, verifying the \textit{oscillation-free} condition.

		Secondly, the case $\varepsilon = \theta$ implies that the initial condition is the steady state. Since all modified Patankar schemes are able to unconditionally preserve the steady state, the solution will be steady.

		In the general case, we can write the solution at the first time step as ratio of polynomials that are of degree three in $\Delta t$, and degree two in $\theta$ and $\varepsilon$. Here, for brevity we write one of the two component $u_2^1=\frac{N}{D}$, where
		\begin{align*}
			N=&2 (1-\epsilon ) \epsilon ^2+2 \dt \epsilon  (\epsilon  (1-\theta )+2 (1-\epsilon ) \theta )\\
			&+\dt^2 \left((1-\epsilon ) \epsilon \theta +3 \epsilon  (1-\theta ) \theta +2 (1-\epsilon ) \theta ^2\right)+\dt^3 \left((1-\epsilon ) \theta ^2+(1-\theta ) \theta ^2\right)>0,\\
				D=&	2 (1-\epsilon ) \epsilon+\dt (2 (1-\epsilon ) \epsilon +2 \epsilon  (1-\theta
				)+2 (1-\epsilon ) \theta )\\
				&+\dt^2 ((1-\theta ) (2 \epsilon +\theta )+(1-\epsilon ) (\epsilon +2 \theta ))  +\dt^3 (\epsilon  (1-\theta )+(1-\epsilon ) \theta )>0.
		\end{align*}
		Properties \ref{prop:not_over} and \ref{prop:correct} simplifies to $\varepsilon \geq u^1_2 \geq \theta$
		in the case $\varepsilon> \theta$ and to $\varepsilon \leq u^1_2 \leq \theta$ if
		$\varepsilon<\theta$. The inequality regarding $u^1_2$ and $\varepsilon$, i.e., Property \ref{prop:correct}, is proven in \ref{app:prove_direction} in Theorem~\ref{th:directionMPRK22} for all \ref{eq:MPRK22-family} schemes with $\alpha \leq 1$.
		To prove Property~\ref{prop:not_over} we analyze the sign of $D \theta - N$, where $N$ and $D$ are the numerator and the denominator of $u_2^1 $ respectively, clearly both positive.
		For $\varepsilon > \theta$ we want to have $D \theta -N <0$ not to overshoot the steady state, while for   $\varepsilon < \theta$ we should have $D \theta -N>0$ or, in other words, $\frac{D \theta -N}{\varepsilon -\theta}<0$.
		We have that
		\begin{equation}
		\frac{D \theta -N}{\epsilon -\theta }=	-2\epsilon (1-\epsilon )-\dt 2( \theta (1- \epsilon )+\epsilon (1-\theta ))-\dt^2
		\theta (1-\theta )+\dt^3 \theta (1-\theta )<0,
		\end{equation}
		which is a third degree polynomial inequality for $\dt$ and can be rewritten as
		\begin{equation}
			p_{\varepsilon,\theta}(\dt)=\dt^3 -\dt^2 -2 \left(\frac{\varepsilon}{\theta}+\frac{1-\varepsilon}{1-\theta}\right) \dt -2 \frac{\varepsilon(1-\varepsilon)}{\theta(1-\theta)}<0.
		\end{equation}
	There are two options for real coefficients cubic polynomials. If the discriminant $\Delta \geq 0$ then the roots are all real, while if $\Delta <0$ there are two complex conjugated roots and a real one \cite{orson1866cubic}. Only if $\Delta =0$ the roots are multiple.
	Let us consider first the case $\Delta \geq 0$.
		Denoting with $y\leq w \leq  z$ the three real roots of $p_{\varepsilon,\theta}(x)$, we see that they have to satisfy
		\begin{equation}\label{eq:polynomialSystem}
			\begin{cases}
				y+w+z = 1,\\
				yz+wz+yw = -2 \left(\frac{\varepsilon}{\theta}+\frac{1-\varepsilon}{1-\theta}\right) < -2,\\
				ywz=2 \frac{\varepsilon(1-\varepsilon)}{\theta(1-\theta)} > 0.
			\end{cases}
		\end{equation}
		Since $ywz$ is positive and $yz+wz+yw$ is negative, it is clear that only one root is positive, while the other two are negative, w.l.o.g. $y\leq w<0< z$.
		From the second equation of \eqref{eq:polynomialSystem}, we see that
		\begin{align}
			&z(w+y)< z(w+y) +wy= yz+wz+yw <-2,\\
			&w+y < -\frac{2}{z}.
		\end{align}
		Using then the first equation of \eqref{eq:polynomialSystem}, we have that
		\begin{align}
			0=z+y+w -1 < z -\frac{2}{z} -1, \quad  0 < z^2 -z -2,
		\end{align}
		which has positive solutions only for $z>2.$
		Hence, $\Delta t \leq 2$ in order to avoid oscillations for all systems \eqref{eq:linearsystem2general}.
		The bound is sharp in the sense that it can be reached for the limit polynomial
		$\lim_{\theta \to 0}\lim_{\varepsilon\to 0} p_{\varepsilon,\theta}(x)$. We can observe that when $\varepsilon\to 0$, the first and third equations in \eqref{eq:polynomialSystem} tell us that $w \to 0^-$. Hence, from the second equation we can see that $y \to -2\frac{1}{(1-\theta)z}$. Finally, the third zero will converge to
		$$
		z\to \frac{1+ \sqrt{1+\frac{8}{1-\theta}}}{2}.
		$$
		For $\theta\to 0$, $z$ goes to 2.

		If $\Delta < 0$ then there are one real root  $z$ and two complex conjugated roots $y=a+ib,\bar y=a-ib$ \cite{orson1866cubic}. These roots must verify
		\begin{equation}\label{eq:polynomialSystem2}
			\begin{cases}
				2a+z = 1,\\
				2az+a^2+b^2 = -2 \left(\frac{\varepsilon}{\theta}+\frac{1-\varepsilon}{1-\theta}\right) < -2,\\
				(a^2+b^2)z=2 \frac{\varepsilon(1-\varepsilon)}{\theta(1-\theta)} > 0.
			\end{cases}
		\end{equation}
		Since $(a^2+b^2)z$ is positive, $z$ is positive.
		From the second equation of \eqref{eq:polynomialSystem2}, we see that
		\begin{align}
			&2az< 2az+a^2+b^2 <-2,\\
			&a < -\frac{1}{z}.
		\end{align}
		Using then the first equation of \eqref{eq:polynomialSystem2}, we have that
		\begin{align}
			0=z+2a -1 < z -\frac{2}{z} -1, \quad  0 < z^2 -z -2,
		\end{align}
		which has positive solutions only for $z>2.$
	\end{proof}
\end{theorem}
\begin{remark}
			The discriminant of $p_{\varepsilon,\theta}$
	\begin{equation}
		\begin{split}
			\Delta =& 4\theta (1-\theta) \big [\epsilon
			^2 (1-\theta )^3 \theta+(1-\epsilon )^2 (1-\theta ) \theta ^3+ 8 \epsilon ^3 (1-\theta )^3 +8 (1-\epsilon )^3 \theta ^3 \\
			&+6 (1-\epsilon ) \epsilon ^2 (1-\theta )^2 \theta +6 (1-\epsilon )^2 \epsilon  (1-\theta ) \theta ^2-27 (1-\epsilon )^2 \epsilon ^2 (1-\theta ) \theta \big ]
		\end{split}
	\end{equation}
	is positive in the square $0<\varepsilon,\theta <1$. This has been verified in \texttt{MPRK\_2\_2\_1\_generalSystem.nb} in \cite{torlo2021stabilityGit}. Hence, the case $\Delta <0$ never happens for $0< \varepsilon, \theta < 1$.
\end{remark}

Unfortunately, the computational complexity increases significantly for all
other schemes considered in this article.
Thus, we will perform numerical studies for all methods, using different initial conditions
($\varepsilon$), systems ($\theta$), and step sizes ($\Delta t$) to find the
largest possible time step without oscillations in Section \ref{sec:numerical-experiments}.

\section{Loss of the order of accuracy for vanishing initial conditions}\label{sec:spuriousSteadyState}
Another particular behavior we observe for some modified Patankar schemes is the loss of accuracy when one component of the initial condition tends to zero.
In this case, available analytical results on accuracy of the schemes do not
hold as they require $u_i^0\geq \varepsilon>0$ with fixed $\varepsilon$. Nevertheless, the condition $u_i^0=\varepsilon$ with $\varepsilon\to 0$ is of general interest in many applications, where physical/chemical/biological constituents might be zero and choosing the initial condition $\varepsilon\gg 0$ might ruin the accuracy of the solution. In particular when dealing with high order schemes and expecting an error of $\mathcal{O}(\dt^r)$, we might need to require the initial error to be the same order or less than the expected precision, i.e., $\varepsilon \lesssim \dt^r$, in order not to let the initial error dominate the final error.

In this section, we show for which Patankar and modified Patankar schemes there is an order reduction for a very simple linear problem. Here, we understand the
phenomenon of order reduction similarly to what happens for stiff problems,
where two parameters are coupled in a limit process
\cite[Chapter IV.15]{hairer1999stiff}. For stiff problems, these parameters are
the time step and a stiffness parameter. In our case, these two parameters are
the time step $\dt$ and the minimum of the initial data $\varepsilon$.
We will see that the order of accuracy decreases in a certain regime
$\varepsilon \ll \dt$. Consider the order of accuracy of the
first time step, defined as the largest $r$ such that
\begin{equation}
  ||u^1 - u(t^1)|| \leq K \dt ^ {r+1}
\end{equation}
as $\dt \to 0$ while $\frac{\varepsilon}{\dt}\to 0$, similar to the stiff case \cite{hairer2000solving,hairer1999stiff}. In the first time step of the simulations, this situation is very common and will result in a loss of the order of accuracy for the first time step. As soon as $u_2\gg \varepsilon$, the classical accuracy will be restored, but the final error will be anyway influenced by this initial step or some initial steps.
\begin{remark}[Order and error at the final time]
	We have seen that a method of order $r$ has an error of $\mathcal{O}(\dt^{r+1})$ at the first time step. These errors accumulate till the final time $T$ for all time steps which are $\mathcal{O}(\frac{T}{\dt})$. This results in an error at the final time of the order of $\mathcal{O}(\dt^r)$. In the situation of order reduction at one or some of the time steps, it can happen that the error produced at the first time step dominates the final error and ruins the accuracy also at the final time.
\end{remark}
Some numerical experiments validate this study in Sections \ref{sec:numerical-experiments} and \ref{sec:nonlinear-experiements}.

\subsection{Strong loss of order accuracy for vanishing initial conditions}
	Different modified Patankar schemes behave differently for vanishing initial condition, some are not affected, some become second order accurate, some first order accurate. We tested different modified Patankar schemes and the method Rodas4, as a benchmark, on \eqref{eq:linearsystem2general} with $\theta=0.5$ comparing $\varepsilon=0.01$, $\varepsilon=10^{-16}$ and $\varepsilon = 10^{-250}$. In Figures~\ref{fig:errorDecay23} and \ref{fig:errorDecay4} we plot the error decay for these test with the error defined as
	\begin{equation}
		\texttt{err}:= \frac{1}{N_t}\sum_{n=1}^{N_t} ||u^{ex}(t^n)-u^n||_2.
	\end{equation}
 We see that some \ref{eq:MPRK22-family} and \ref{eq:MPRK43-family} fall into typical first order accuracy behaviors, while other third and fourth order schemes behave like second order ones in this situation. Moreover, the order depends on the relation between $\varepsilon$ and $\dt$.
	\begin{figure}
		\centering
		\begin{subfigure}{0.48\textwidth}
			\includegraphics[width=\textwidth]{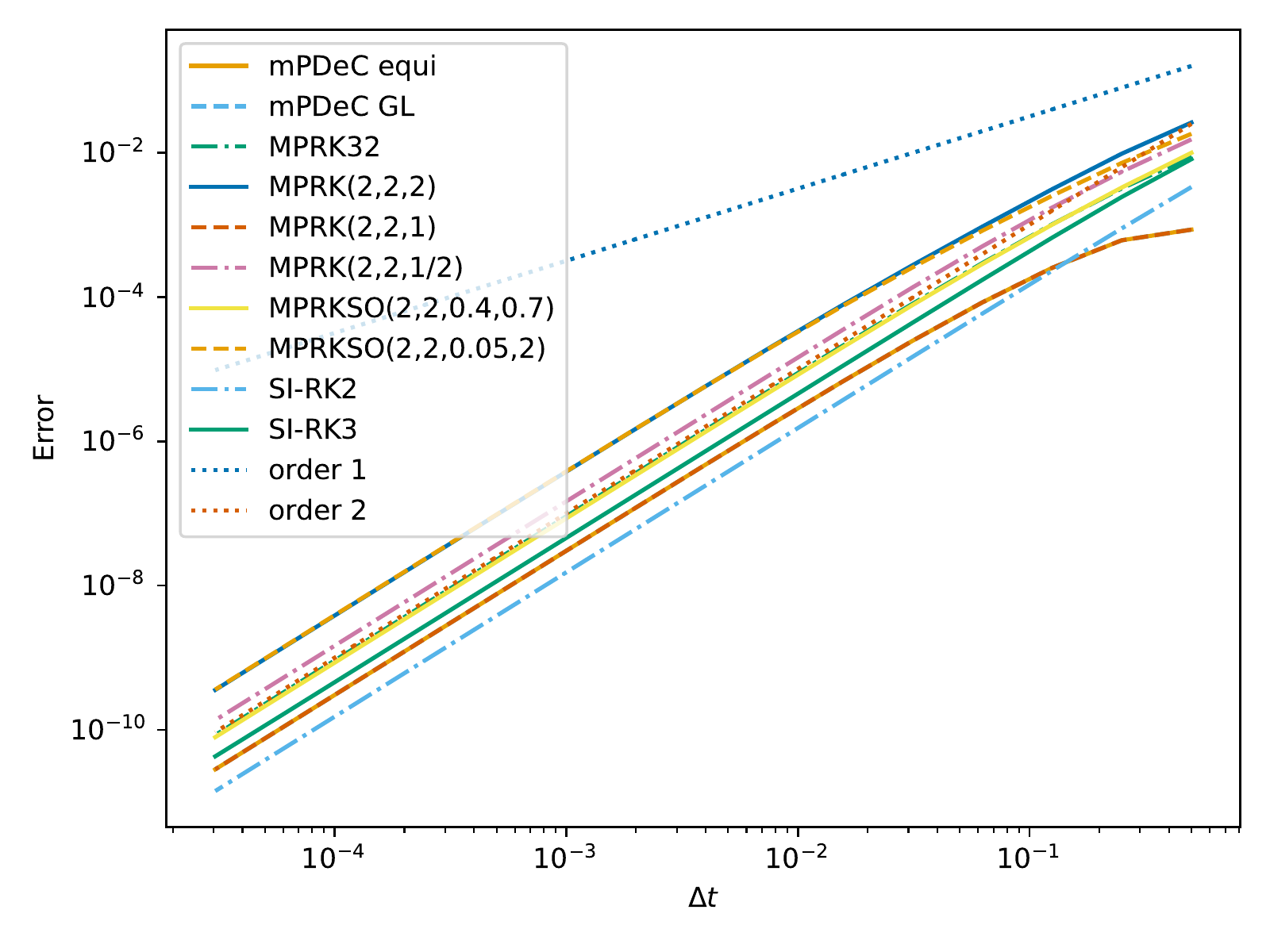}
			\caption{Order 2 schemes, $\varepsilon=10^{-2}$}
		\end{subfigure}
	\begin{subfigure}{0.48\textwidth}
	\includegraphics[width=\textwidth]{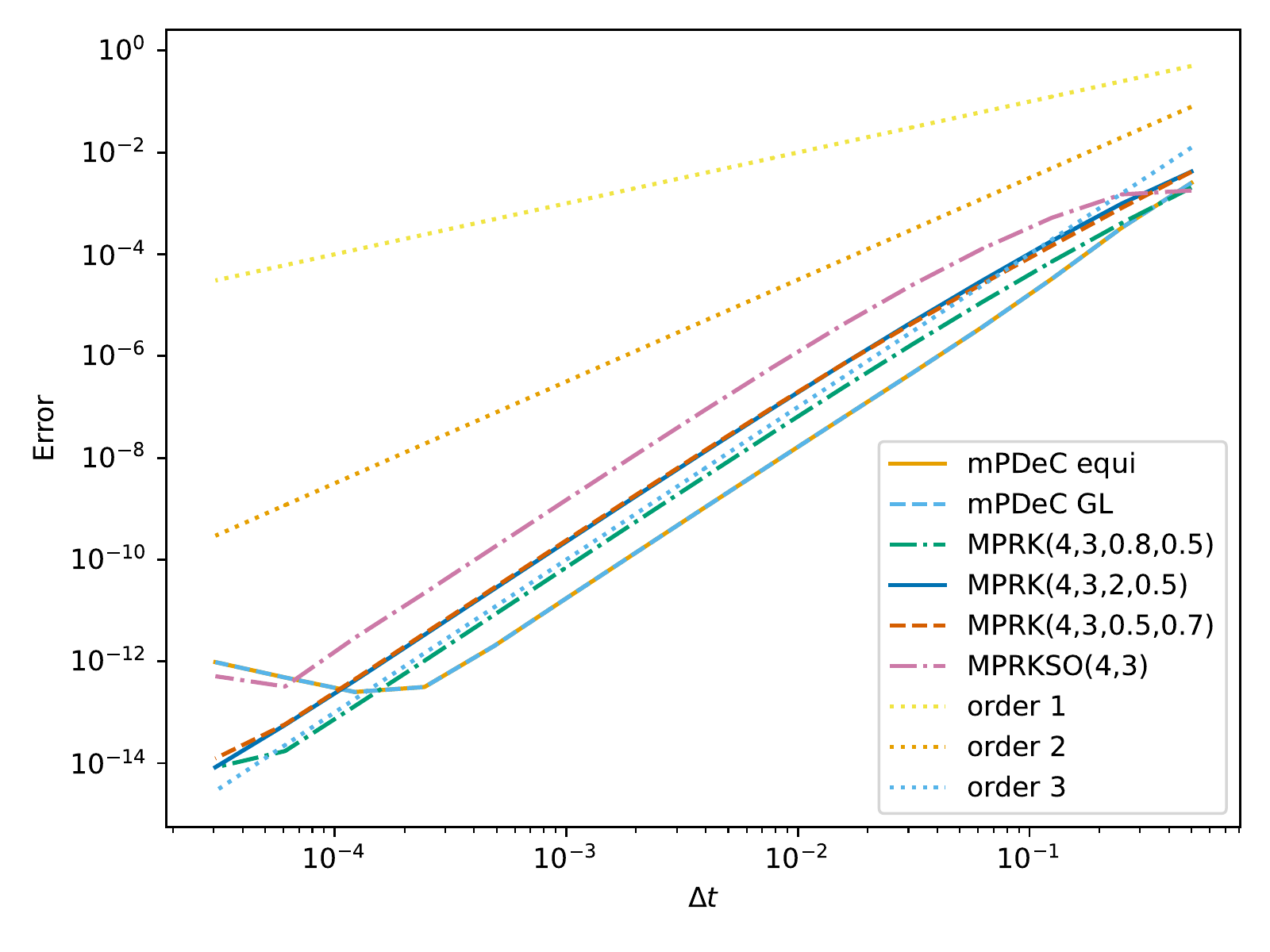}
	\caption{Order 3 schemes, $\varepsilon=10^{-2}$}
\end{subfigure}\\
		\begin{subfigure}{0.48\textwidth}
			\includegraphics[width=\textwidth]{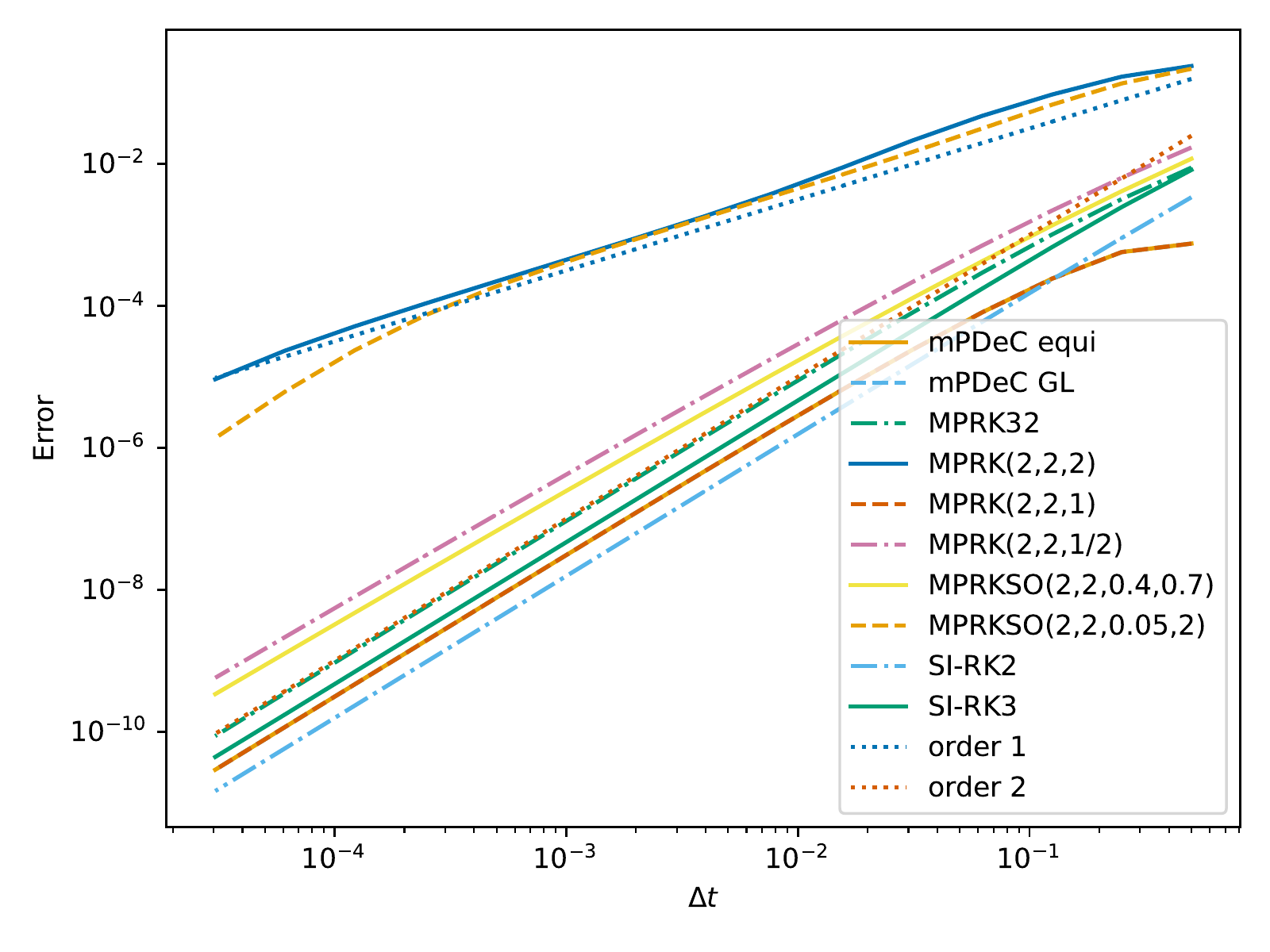}
			\caption{Order 2 schemes, $\varepsilon=10^{-16}$}
		\end{subfigure}
	\begin{subfigure}{0.48\textwidth}
	\includegraphics[width=\textwidth]{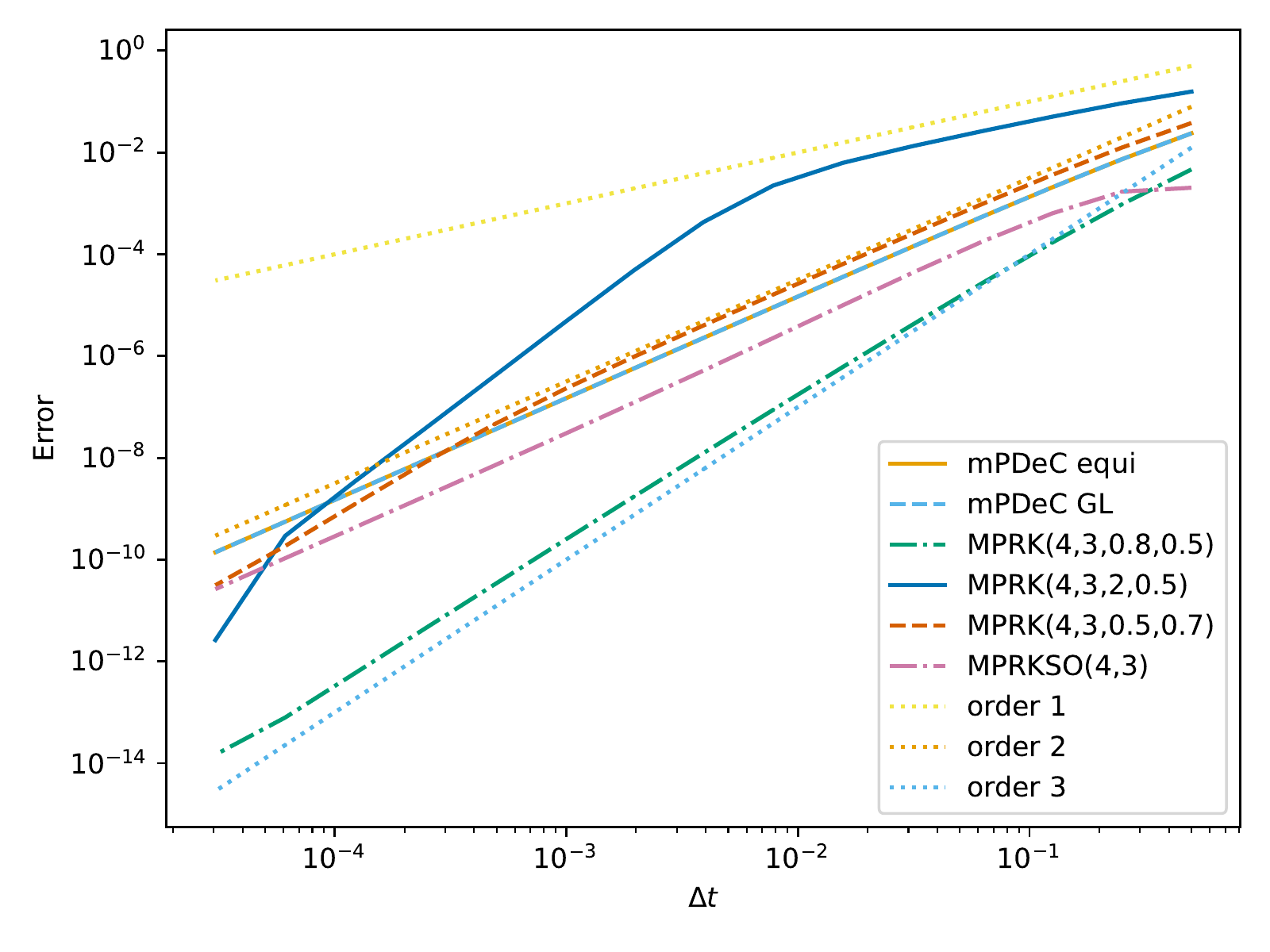}
	\caption{Order 3 schemes, $\varepsilon=10^{-16}$}
\end{subfigure}\\
\begin{subfigure}{0.48\textwidth}
\includegraphics[width=\textwidth]{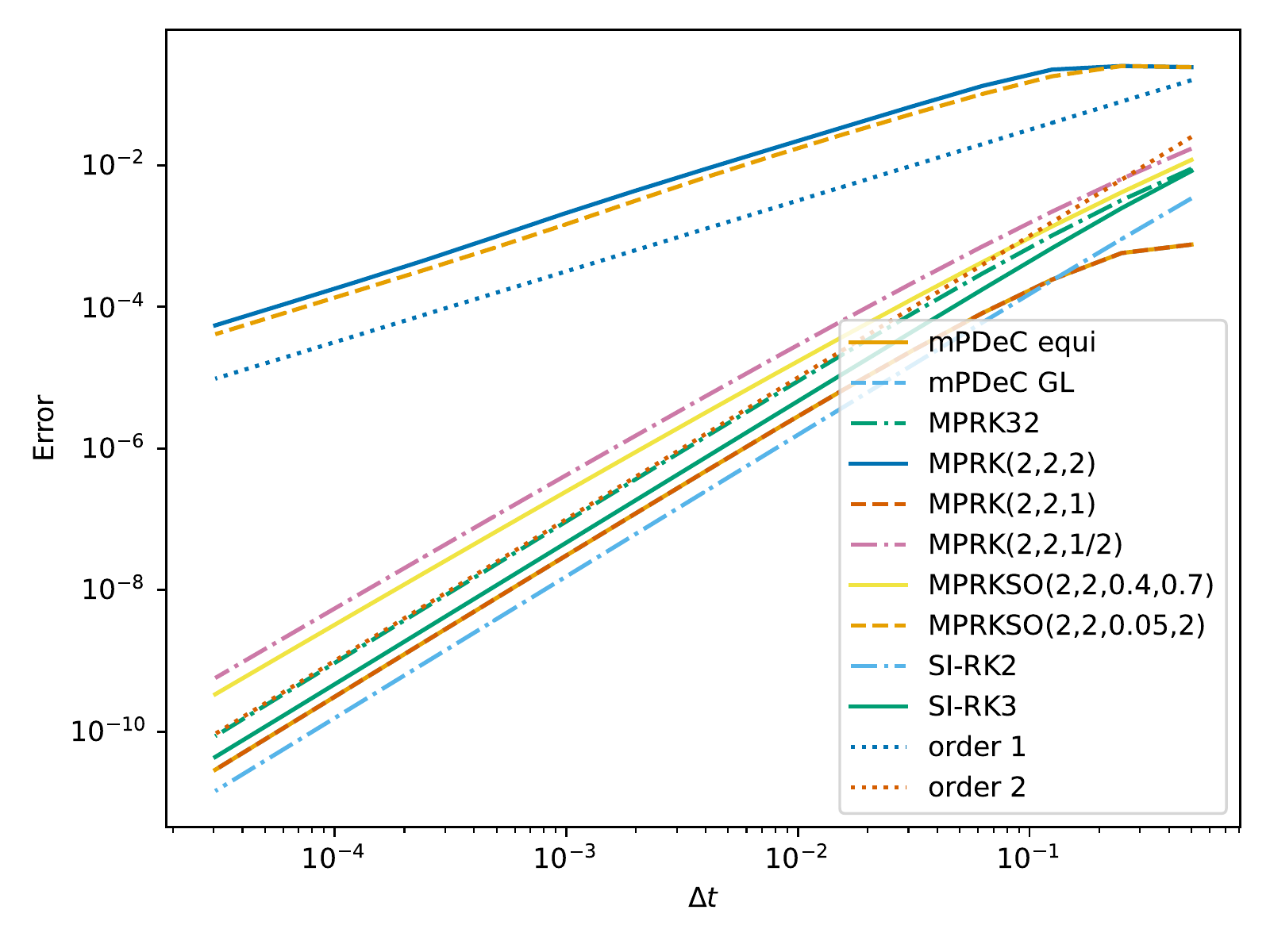}
\caption{Order 2 schemes, $\varepsilon=10^{-250}$}
\end{subfigure}
\begin{subfigure}{0.48\textwidth}
\includegraphics[width=\textwidth]{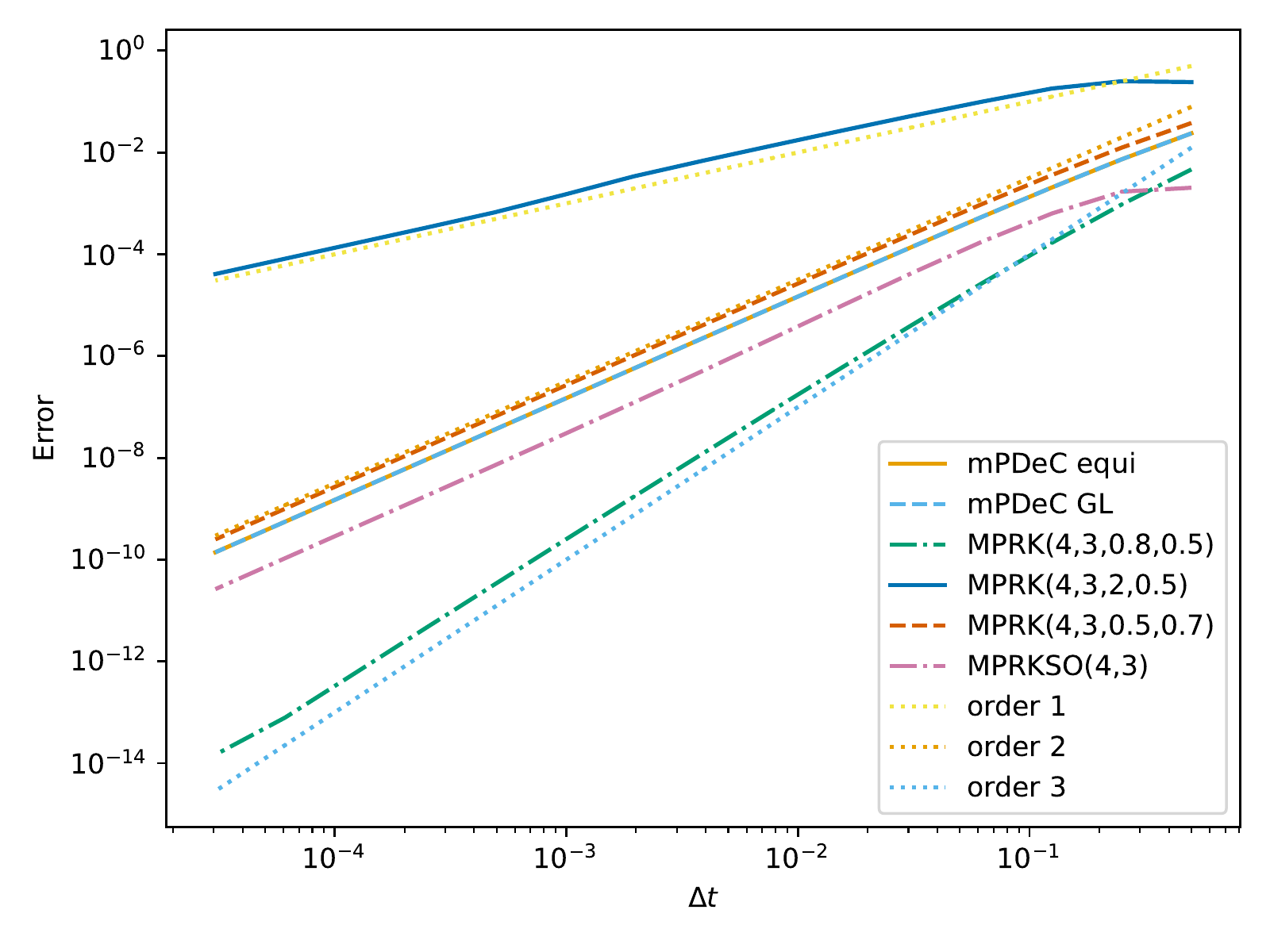}
\caption{Order 3 schemes, $\varepsilon=10^{-250}$}
\end{subfigure}
	\caption{Error decay for the system \eqref{eq:linearsystem2general} with $\theta=0.5$, at time $T=1$ with second and third order methods and different $\varepsilon$.\label{fig:errorDecay23}}
\end{figure}
	\begin{figure}
	\centering
	\begin{subfigure}{0.48\textwidth}
	\includegraphics[width=\textwidth]{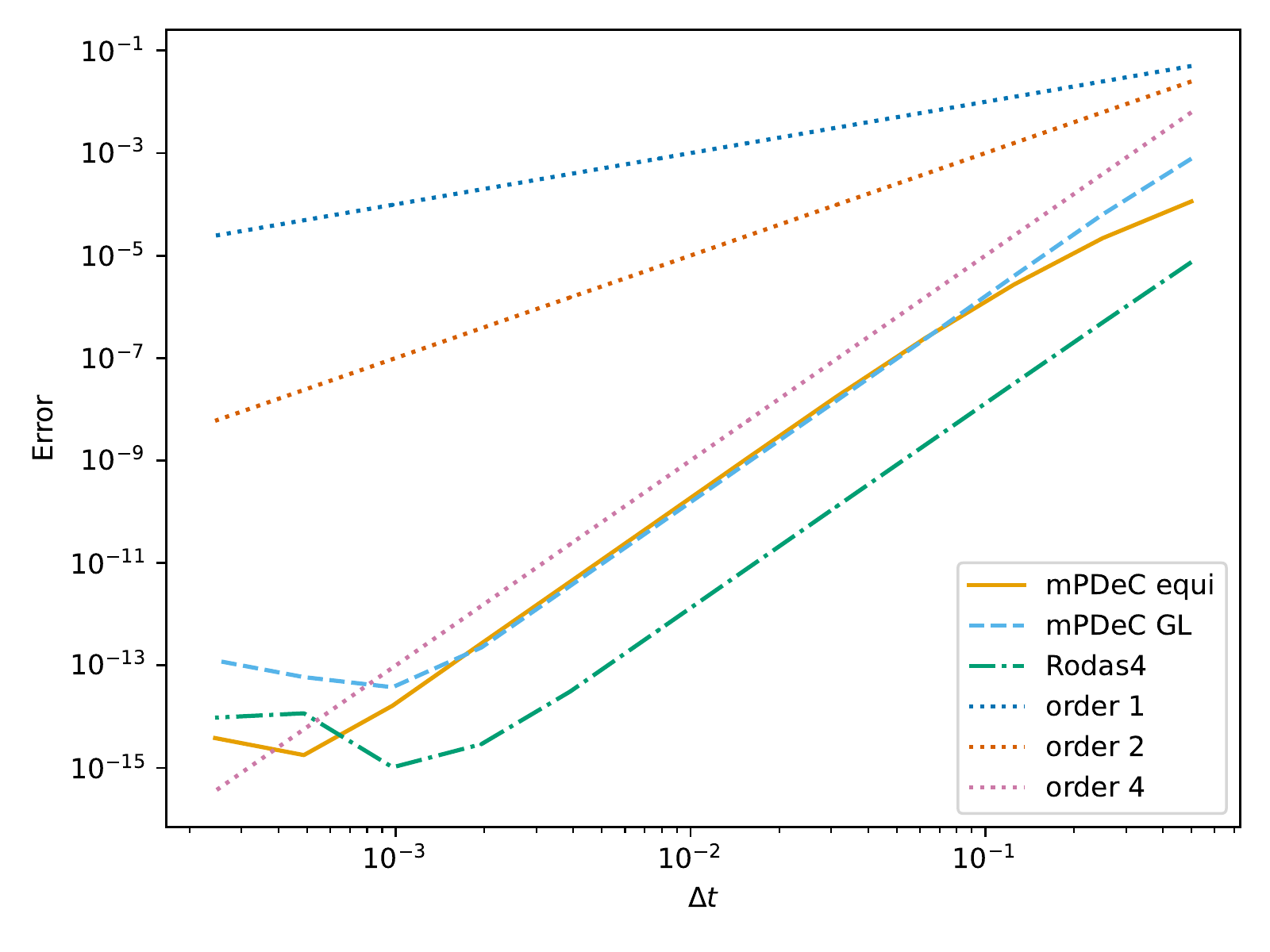}
	\caption{Order 4 schemes, $\varepsilon=10^{-2}$}
\end{subfigure}
\begin{subfigure}{0.48\textwidth}
\includegraphics[width=\textwidth]{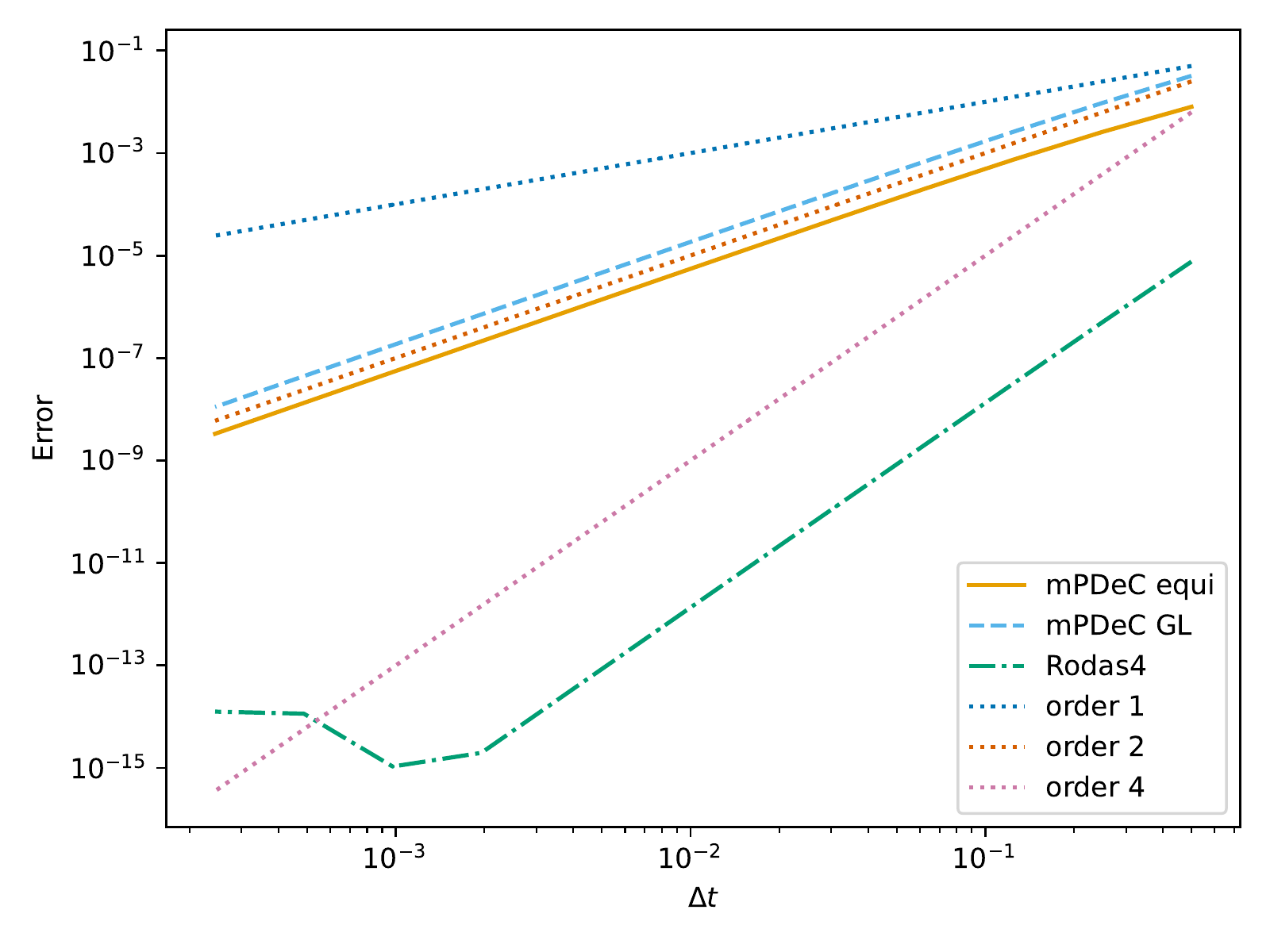}
\caption{Order 4 schemes, $\varepsilon=10^{-16}$}
\end{subfigure}
	\caption{Error decay for the system \eqref{eq:linearsystem2general} with $\theta=0.5$, at time $T=1$ with fourth order methods and different $\varepsilon$.\label{fig:errorDecay4}}
\end{figure}\\
	The fall back to first and second order is due to an error in the first time steps when one initial condition is close to 0. As soon as this component becomes large enough the error goes back to the expected one. This leaves either a shift of some $\dt$ on the solution or a first time step with a second order error. To grasp why we lose order of accuracy, we need to understand what happens in the limit of our schemes for $\varepsilon \to 0$ for the first time step.
	We remark that, in the linear system case, $\tilde{p}_{ij}$ and $\tilde{d}_{ij}$ defined in Theorem~\ref{th:ex_un_pos} as $\tilde{d}_{ij}(u)=d_{ij}(u)/u_i$ and $\tilde{p}_{ij}(u)=p_{ij}(u)/u_j$ are positive and constant.
	As an example, we can see the role of these production/destruction rates in the \ref{eq:MPE}
	\begin{align}
		&u^{1}_i = u^0_i +\dt \sum_j \left(  \tilde{p}_{ij}(u^0)\cancel{u_j^0} \frac{ u^{1}_j}{\cancel{u_j^0}} - \tilde{d}_{ij}(u^0) \cancel{u_i^0}  \frac{u^{1}_i}{\cancel{u_i^0}}\right),\\
		&u^{1}_i = \frac{u^0_i +\dt \sum_j \tilde{p}_{ij}(u^0) u^{1}_j}{1+\dt \sum_j\tilde{d}_{ij}(u^0)} =u^0_i+ \dt \sum_{j\in I} \left(\tilde{p}_{ij}(u^0)u_j^1 -\tilde{d}_{ij}(u^0)u_i^0 \right)+ \mathcal{O}(\Delta t^2).\label{eq:firstStepMPE}
	\end{align}
	Hence, we see that the method that we obtain for vanishing initial condition $\varepsilon\to 0$ is well defined and, in this case, leads to a consistent and first order scheme.

	This is not true for \ref{eq:MPRK22-family} for all $\alpha$. The first stage of the scheme is a \ref{eq:MPE} step and it does not introduce issues. The second stage depends on the coefficient $\alpha$. Let us define $\omega = \frac{1}{2\alpha}$, the second stage reads
	\begin{equation}
		u_i^{1} = u_i^0 +\dt \sum_j \left[\left( \frac{ (1-\omega) p_{ij}(y^1)+\omega p_{ij}(y^2)}{(y^2_j)^{1/\alpha}(y^1_j)^{1-1/\alpha}} \right) u^{1}_j-\left( \frac{ (1-\omega) d_{ij}(y^1)+\omega d_{ij}(y^2)}{(y^2_i)^{1/\alpha}(y^1_i)^{1-1/\alpha}} \right) u^{1}_i\right].
	\end{equation}
	Here, we cannot simplify as before the linear terms of destructions and productions. If we focus on the destruction term for the vanishing constituent, i.e., $u_i^0=\varepsilon=y^1_i\to 0$, and if we suppose that the first step is such that $y^2_i \geq C_2 \Delta t$, this is true as we have seen in the \ref{eq:MPE} step, we have that
	\begin{equation}\label{eq:limit_first_step}
		\lim_{y^1_i \to 0} \frac{ (1-\omega) d_{ij}(y^1)+\omega d_{ij}(y^2)}{(y^2_i)^{1/\alpha}(y^1_i)^{1-1/\alpha}} = \begin{cases}
			0, & \text{if } 1-1/\alpha <0\Leftrightarrow \alpha < 1,\\
			\omega \tilde{d}_{ij}(y^2), & \text{if } 1-1/\alpha =0 \Leftrightarrow \alpha = 1,\\
			\infty, & \text{if } 1-1/\alpha > 0 \Leftrightarrow \alpha >1,
		\end{cases}
	\end{equation}
	where $\tilde{d}_{ij}$ is defined in Theorem~\ref{th:ex_un_pos}.

Hence, for $\alpha >1$, when collecting the term $u_i^1$ on the left--hand side, we have that $\lim_{\varepsilon\to 0}u_i^1 = 0$. This is a zero-th order error step. Nevertheless, as one can see also in Figure~\ref{fig:errorDecay23} and Figure~\ref{fig:inconsistency-motivational}, after some steps the regime $u_2 \ll \dt $ is abandoned and the classical accuracy is restored, leading to an error of the order of $\mathcal{O}(\dt)$ at a final time $T$.

For $\alpha <1$ we have that the contribution of the destruction terms to this equation tends to 0 as $\varepsilon	\to 0$, while they where expected to give, for \eqref{eq:linearsystem2general}, a contribution of the order of $\dt$ (as $d_{21}(y^1) = (1-\theta) y^1_2 = \mathcal{O}(\dt)$). This leads to an error of $\mathcal{O}(\dt^2)$ for the first time step, i.e., a first order error. At the second step the regime $u_2 \ll \dt$ is already left, hence, the formal second order of accuracy is then restored. So, at a final time we have an error of $\mathcal{O}(\dt^2)+ \mathcal{O}(\dt^2)=\mathcal{O}(\dt^2)$. Even if in this case we do not observe an order reduction at a final time, the first time step shows order reduction and this is very common also in other higher order methods and this type of reduction would lead to a $\mathcal{O}(\dt^2)$ at the final time.

Finally, for $\alpha=1$ none of these behaviors happen, no order reduction is observed and an error of $\mathcal{O}(\dt^3)$ is formally obtained at the first time step.

We can generalize the two problematic cases that we have just explained into two lemmas. These configurations are common to many MP schemes. Hence, it will be easy then to recast each scheme to one of these cases. For the mPDeC schemes a similar issue arises from the negative coefficients and it will be discussed later.

In general, to obtain a certain order of accuracy, the RK methods build stages of increasing order of accuracy, so that the final step can perform a linear combination of functions of enough accurate stages. If the expected order of accuracy is lost in any of these stages, we might have an order reduction in that timestep update. This is why we need to study all the stages of the MP schemes to check in which of those there is an order reduction and up to which order this happens. This will lead to the understanding of the final order reduction of the method. In the following, we study a general stage and how the order reduction can happen and, then, we check which MPRK is affected in which stage by this behavior.
	 
First of all, let us write a general step of an MPRK scheme for the second component of the ODE \eqref{eq:linearsystem2general} at a certain stage $s$, exploiting the conservation property, as
	\begin{equation}\label{eq:troubled_step}
		y_2^s = u_2^0 + \dt \sum_{j<s} \gamma_j^s  \left(  p_{21}(y^j)\frac{y_1^s}{\sigma_1^j}  - d_{21}(y^j) \frac{y_2^s}{\sigma_2^j}  \right) = u_2^0 + \dt \sum_{j<s} \gamma_j^s  \left(  \theta \frac{y^j_1}{\sigma_1^j}(1-y_2^s)  - (1-\theta) \frac{y^j_2}{\sigma_2^j}y_2^s  \right),
	\end{equation}
with $\gamma_j^s$ some nonnegative RK coefficients and $\sigma_i^j$ the different denominator of the various MPRK schemes.
Now, the troubles come when there are some $\sigma_2^j$ that are an $\mathcal{O}(\varepsilon)$ or when $1/\sigma_2^j = \mathcal{O}(\varepsilon)$ and they do not match the destruction terms. These cases correspond to what observed in \ref{eq:MPRK22-family} for $\alpha >1$ and $\alpha<1$ respectively, while it is not the case of \ref{eq:MPE} where cancellation leads to a consistent approximation. To be more general, let us consider $\sigma_2^j=\mathcal{O}(\eta)$ or $1/\sigma_2^j=\mathcal{O}(\eta)$, where $\eta$ can be a power of $\varepsilon$ or a ratio between $\varepsilon$ and $\dt$. As an example, you can refer to the \ref{eq:MPRK22-family}, where at the last stage the denominator is $\sigma_2^1=\sigma_2^2=(y_2^2)^{1/\alpha}(y_2^1)^{1-1/\alpha} = \mathcal{O}(\dt^{1/\alpha}\varepsilon^{1-1/\alpha})$. This will be the case in many situations.
It will be useful to use the Big Theta Landau symbol $f(x)=\Theta(g(x))$ to indicate that $$
0<\liminf_{x\to 0} \frac{|f(x)|}{g(x)} \leq \limsup_{x\to 0} \frac{|f(x)|}{g(x)}<\infty.
$$
	First, we study the case where $\sigma^j_2=\Theta(\eta)$ which corresponds to the \ref{eq:MPRK22-family} for $\alpha>1$.
\begin{lemma}\label{lem:term_infty}
	Consider the problem \eqref{eq:linearsystem2general} with $0<\theta < 1$ and initial condition $(1-\varepsilon, \varepsilon)$ with $0<\varepsilon$. Consider the update step at the first time step given by \eqref{eq:troubled_step}. Suppose there is an $\ell<s$ with $\gamma_\ell^s >0$ such that $\sigma_2^\ell = \Theta(\eta)$, $y_2^\ell=\Theta(\dt)$ and consider the limit for $\dt\to 0$, $\frac{\eta}{\dt^2}\to 0$ and $\frac{\varepsilon}{\dt}\to 0$.  Moreover, suppose that for all stages $j$: $\frac{y_2^j}{\sigma_2^j}=\mathcal{O}\left(\frac{y_2^\ell}{\sigma_2^\ell}\right)$. Then, $$y_2^s= \Theta\left(\frac{\eta}{\dt}\right) = u_2^{ex} + \Theta(\dt),$$
	where $u_2^{ex}$ is the exact solution after the first time step.
\end{lemma}
\begin{proof}
	First of all, let us observe that $\frac{\eta}{\dt} = \frac{\eta}{\dt^2} \dt \to 0$ as both $\frac{\eta}{\dt} \to 0$ and $\dt \to 0 $.
	From \eqref{eq:troubled_step} we can write the definition of $y_2^s$ as
	\begin{equation}\label{eq:troubled_step2}
		\Bigg[ 1 + \underbrace{\dt \sum_{j<s} \gamma_j^s\theta \frac{y_1^j}{\sigma_1^j} }_{\Theta(\dt)}+ \underbrace{\dt \sum_{j<s}\gamma_j^s (1-\theta) \frac{y_2^j}{\sigma_2^j}}_{\Theta(\frac{\dt^2}{\eta})} \Bigg] y_2^s = \underbrace{y_2^0}_{=\varepsilon} + \underbrace{\dt \sum_{j<s} \gamma_j^s\theta \frac{y_1^j}{\sigma_1^j} }_{\Theta(\dt)}.
	\end{equation}
The scaling indicated by Landau symbols can be explained from hypotheses, using also $y_1=\mathcal{O}(1)$ for all stages and consequently $\sigma_1=\mathcal{O}(1)$; the initial value is $y_2^0=\varepsilon$ and all the coefficients are constant.
Then, the dominating term on the left-hand side is the $\Theta(\frac{\dt^2}{\eta})$, the only one going to infinity, and on the right side it is the $\Theta(\dt)$ as $\frac{\varepsilon}{\dt} \to 0$. So, we can write that
	\begin{align}
y_2^s =& \frac{\varepsilon + \dt\displaystyle \sum_{j<s} \gamma_j^s\theta \frac{y_1^j}{\sigma_1^j} }{ 1 +\dt \displaystyle\sum_{j<s} \gamma_j^s\theta \frac{y_1^j}{\sigma_1^j} + \dt\displaystyle \sum_{j<s}\gamma_j^s (1-\theta) \frac{y_2^j}{\sigma_2^j}}
= \frac{\frac{\eta}{\dt^2}\varepsilon + \frac{\eta}{\dt} \displaystyle\sum_{j<s} \gamma_j^s\theta \frac{y_1^j}{\sigma_1^j} }{ \frac{\eta}{\dt^2} +\underbrace{\frac{\eta}{\dt} \displaystyle\sum_{j<s} \gamma_j^s\theta \frac{y_1^j}{\sigma_1^j}}_{\Theta(\frac{\eta}{\dt})} + \underbrace{\frac{\eta}{\dt}\displaystyle \sum_{j<s}\gamma_j^s (1-\theta) \frac{y_2^j}{\sigma_2^j}}_{\Theta(1)}} \\
=&\frac{\frac{\eta}{\dt}\displaystyle \sum_{j<s} \gamma_j^s\theta \frac{y_1^j}{\sigma_1^j}}{\frac{\eta}{\dt} \displaystyle\sum_{j<s}\gamma_j^s (1-\theta) \frac{y_2^j}{\sigma_2^j}} + \Theta(1)\!\! \left( \frac{\eta \varepsilon}{\dt^2}\!-\!\frac{\eta}{\dt^2}\frac{\eta}{\dt}\displaystyle\sum_{j<s} \gamma_j^s\theta \frac{y_1^j}{\sigma_1^j}\!-\! \frac{\eta^2}{\dt^2}\left(\displaystyle\sum_{j<s} \gamma_j^s\theta \frac{y_1^j}{\sigma_1^j} \right)^2\right)\!\!+\dots\\
=&\frac{\displaystyle\sum_{j<s} \gamma_j^s\theta \frac{y_1^j}{\sigma_1^j}}{\displaystyle\sum_{j<s}\gamma_j^s (1-\theta) \frac{y_2^j}{\sigma_2^j}}
+ \mathcal{O}\left(\frac{\eta}{\dt}\frac{\varepsilon}{\dt} \right)
+\mathcal{O}\left( \frac{\eta}{\dt} \frac{\eta}{\dt^2}\right)
+\mathcal{O}\left( \frac{\eta}{\dt} \frac{\eta}{\dt}\right)= \Theta\left( \frac{\eta}{\dt} \right).
\end{align}
To obtain the previous formula is convenient to multiply numerator and denominator of $y_2^s$ by $\sigma_2^\ell$ and then, after having simplified $\dt$, at the numerator there is a $\Theta(1)$ and at the denominator the term $\frac{y_2^\ell}{\sigma_2^\ell}$ dominates the sum.
\end{proof}
This lemma shows that in the stages where the hypotheses are verified we obtain a 0-th order accurate update. Still, the value of $y_2^j = \Theta \left( \frac{\eta}{\dt}\right)$ and it is larger than $\eta$ (for small $\dt$). So, if this operation is repeated and we consider the result after a time step as a new initial condition, the new initial value $\varepsilon$ will keep increasing (and consequently $\eta$ which is proportional to $\varepsilon$), the regime $y_2 \ll \dt$ is abandoned after some time steps and the classical accuracy is restored for following time steps. Usually, in these cases an error of $\mathcal{O}(\dt)$ at a final step is observed, as a first order method.

The second situation we encounter is the opposite, when the exponents of the schemes are such that $1/\sigma_2^j $ is an $\mathcal{O}(\varepsilon)$ or one of its powers, as for \ref{eq:MPRK22-family} for $\alpha <1$.
\begin{lemma}\label{lem:term_zero}
	Consider the problem \eqref{eq:linearsystem2general} with $0<\theta < 1$ and initial condition $(1-\varepsilon, \varepsilon)$ with $0<\varepsilon$. Consider the update step at the first time step given by \eqref{eq:troubled_step}. Suppose that exists $\ell < s$ with $\gamma_\ell^s >0$ such that $1/\sigma_2^\ell = \Theta(\eta)$ and $y_2^\ell=\Theta(\dt)$  and consider the limit for $\dt\to 0$, $\frac{\eta}{\dt}\to 0$ and $\frac{\varepsilon}{\dt}\to 0$. Moreover, suppose that for all $j<s$, $\frac{y_2^j}{\sigma_2^j} =\mathcal{O}(1)$. Then, $y_2^s$ is at most an approximation of order 1 of the exact solution and the error is an $\mathcal{O}(\dt^2).$
\end{lemma}
\begin{proof}
	First we prove that $y^s_2=\Theta(\dt)$ and afterwards,  we show that it cannot be a second order approximation.
From \eqref{eq:troubled_step2} it follows that
\begin{equation}
	\Bigg[ 1 + \underbrace{\dt \sum_{j<s} \gamma_j^s\theta \frac{y_1^j}{\sigma_1^j} }_{\Theta(\dt)}+ \underbrace{\dt \sum_{j<s}\gamma_j^s (1-\theta) \frac{y_2^j}{\sigma_2^j}}_{\mathcal{O}(\dt)} \Bigg] y_2^s = \underbrace{y_2^0}_{=\varepsilon} + \underbrace{\dt \sum_{j<s} \gamma_j^s\theta \frac{y_1^j}{\sigma_1^j} }_{\Theta(\dt)},
\end{equation}
so that the dominant term on the LHS is 1 and on the RHS is the $\Theta(\dt)$. Hence, we obtain
\begin{equation}
	y_2^s = \dt \sum_{j<s}\gamma_j^s\theta \frac{y_1^j}{\sigma_1^j}  +  \mathcal{O}(\varepsilon)+  \mathcal{O}(\dt^2) = \Theta(\dt).
\end{equation}
Consider again the update equation \eqref{eq:troubled_step} and supposed that it is of order of accuracy $p>1$ in the nonvanishing initial condition regime. We see that
	\begin{align}
			y_2^s &= u_2^0 + \dt \sum_{j<s} \gamma_j^s   p_{21}(y^j)\frac{y_1^s}{\sigma_1^j}  -\dt \sum_{j<s} \gamma_j^s  d_{21}(y^j) \frac{y_2^s}{\sigma_2^j} \\
			&=u_2^0 + \dt \sum_{j<s} \gamma_j^s   p_{21}(y^j)  -\dt \sum_{j\neq \ell} \gamma_j^s  d_{21}(y^j) \frac{y_2^s}{\sigma_2^j} -\dt \gamma_\ell^s  d_{21}(y^\ell) \frac{y_2^s}{\sigma_2^\ell},
	\end{align}
Focusing on the last term we observe that
\begin{equation}
	\dt \gamma_\ell^s  d_{21}(y^\ell) \frac{y_2^s}{\sigma_2^\ell} = \dt \gamma_\ell^s  \theta y^\ell_2 \frac{y_2^s}{\sigma_2^\ell} = \dt \gamma_\ell^s \theta \underbrace{\frac{ y^\ell_2}{\sigma_2^\ell}}_{\Theta(\eta \dt )} \underbrace{y_2^s}_{\Theta(\dt)} = \Theta(\dt^3 \eta),
\end{equation}
while for the unweighted term of the original highly accurate RK method we have 
\begin{equation}
	\dt \gamma_\ell^s  d_{21}(y^\ell) = \Theta(\dt^2),
\end{equation}
hence, the difference of the two is $\Theta(\dt^2)$. Hence, the destruction term contribution related to stage $\ell$ is approximated with an error of $\Theta(\dt^2)$.
So that the error for the stage $s$ is affected mainly by this error, i.e.,
\begin{equation}
	y_2^s = y_2^{ex} + \Theta(\dt^2).
\end{equation}
\end{proof}
This proof shows that when a scheme falls in the hypotheses of this lemma, we have a first step with only accuracy order of 1, but, immediately after, the value of $u_2$ is far away from zero and the classical order of accuracy is restored. Then, at a final time the error will be a $\Theta(\dt^2)$.
We remark that the hypothesis $\frac{y_2^j}{\sigma_2^j}=\mathcal{O}(1)$ is not restrictive as it discriminates the first lemma case and second lemma case. Indeed, when this hypothesis is not fulfilled, there exists an $\ell$ such that $\sigma_2^\ell = o(y_2^\ell)$, and by an opportune definition of $\eta$ such that $\sigma_\ell^2 = \Theta(\eta)$ fulfills the hypotheses of Lemma~\ref{lem:term_infty}.

Now we can use these results to show the accuracy of all the modified Patankar schemes with positive Runge--Kutta coefficients.
\begin{theorem}[Accuracy of Patankar schemes with nonnegative RK coefficients for vanishing initial data]
	Consider the system of ODEs \eqref{eq:linearsystem2general} with $u_0=(1-\varepsilon,\varepsilon)$ with vanishing initial condition, i.e., $\frac{\varepsilon}{\dt^r}\to 0$ as $\dt \to 0$ with $r$ large enough depending on the scheme so that hypotheses of previous lemmas are met. Then, the modified Patankar schemes with positive coefficients have errors in the first time steps and at a final time as  shown in Table~\ref{tab:ordersMP} (for \ref{eq:explicit_dec_correction} we refer to Theorem~\ref{th:accuracyMPDEC}).
	\begin{table}
		\centering
		\begin{tabular}{|c|c|c|c|}\hline
			Method& Parameters & First time step error  & Final time error\\ \hline
			\ref{eq:MPRK22-family} &$\alpha=1$ & $\Theta(\dt^3) $& $\Theta(\dt^2)$\\
			\ref{eq:MPRK22-family} &$\alpha>1$ & $\Theta(\dt) $& $\Theta(\dt)$\\
			\ref{eq:MPRK22-family} &$\frac12 \leq\alpha<1$ & $\Theta(\dt^2) $& $\Theta(\dt^2)$\\
			\ref{eq:MPRK43-family} &$q>1$ & $\Theta(\dt) $& $\Theta(\dt)$\\
			\ref{eq:MPRK43-family} &$p>1$ and $q\leq 1$ & $\Theta(\dt^2) $& $\Theta(\dt^2)$\\
			\ref{eq:MPRK43-family} &$p\leq 1$ and $q\leq 1$ and $pq\neq 1$ & $\Theta(\dt^3) $& $\Theta(\dt^3)$\\
			\ref{eq:MPRK43-family} &$p=q= 1$  & $\Theta(\dt^4) $& $\Theta(\dt^3)$\\
			\ref{eq:MPRKSO22-family} &$\gamma <1$ & $\Theta(\dt) $& $\Theta(\dt)$\\
			\ref{eq:MPRKSO22-family} &$\gamma = 1$ & $\Theta(\dt^3) $& $\Theta(\dt^2)$\\
			\ref{eq:MPRKSO22-family} &$\gamma > 1$ & $\Theta(\dt^2) $& $\Theta(\dt^2)$\\
			\ref{eq:MPRKSO(4,3)} & & $\Theta(\dt^2) $& $\Theta(\dt^2)$\\
			\ref{eq:MPRK32} & & $\Theta(\dt^3) $& $\Theta(\dt^2)$\\
			\ref{eq:SI-RK2} & & $\Theta(\dt^3) $& $\Theta(\dt^2)$\\
			\ref{eq:SI-RK3} & & $\Theta(\dt^3) $& $\Theta(\dt^2)$\\
			\ref{eq:explicit_dec_correction} & Equispaced\tablefootnote{mPDeC negative $\theta^M_j$ are present only for order higher than 8.}, nonnegative $\theta^M_j$ & $\Theta(\dt^2) $& $\Theta(\dt^2)$\\
			\ref{eq:explicit_dec_correction} & Equispaced, negative $\theta^M_j$ & $\Theta(\dt) $& $\Theta(\dt)$\\
			\ref{eq:explicit_dec_correction} & Gauss-Lobatto any order & $\Theta(\dt^2) $& $\Theta(\dt^2)$\\
			\hline
		\end{tabular}
	\caption{Accuracy of  Patankar methods for vanishing initial conditions with parameters defined at the definition of each scheme, see \ref{eq:MPRK22-family}, \ref{eq:MPRK43-family}, \ref{eq:MPRKSO22-family} and \ref{eq:explicit_dec_correction} \label{tab:ordersMP}}
	\end{table}
\end{theorem}
\begin{proof}
	We analyze all the methods stage by stage.
	\begin{itemize}
		\item Let us start with \ref{eq:MPRK22-family}. The first stage is an \ref{eq:MPE} step, which coincides with an implicit--Euler step for this problem and gives that $y_2^2 = u_2(\alpha \dt) + \mathcal{O}(\dt^2) = \Theta(\dt)$ for all parameters. In the last stage, we have that the critical factor is $\sigma_2^2= (y_2^2)^{1/\alpha}(y_2^1)^{1-1/\alpha} =\Theta(\dt^{1/\alpha}(\varepsilon)^{1-1/\alpha}) $ at the denominator, while $y_2^2=\Theta(\dt)$ being at the numerator.
		\begin{itemize}
			\item When $\alpha=1$, then $\sigma_2^2=\Theta(\dt)$ and this does not arise problems, hence the classical accuracy is restored and we have an error of $\mathcal{O}(\dt^3)$ for the first time step.
			\item For $\alpha>1$ we have that $1-1/\alpha>0$ and Lemma~\ref{lem:term_infty} applies with $\eta = \varepsilon^{1-1/\alpha} \dt^{1/\alpha}$ when
			$$\frac{\eta}{\dt^2} = \varepsilon^{1-1/\alpha} \dt^{1/\alpha -2} \to 0 \text{ and } \frac{\varepsilon}{\dt} \to 0$$ as $\dt \to 0$. Hence, $u^{1}_2=\Theta\left(\left(\frac{\varepsilon}{\dt}\right)^{1-1/\alpha}\right)= u_2(t^1)+\Theta(\dt)$. It must be noticed that $\varepsilon=o\left(\left(\frac{\varepsilon}{\dt}\right)^{1-1/\alpha}\right) $, hence, each time step is moving away from the region $u_2^n \ll \dt$. After a certain number of time steps the regime $\eta \ll \dt$ will be lost and classical accuracy will be restored. The first errors of $\Theta(\dt)$ will dominate the final error.
			\item For $\frac12\leq \alpha <1$ we have that $-1\leq 1-1/\alpha<0$ and Lemma~\ref{lem:term_zero} applies with $\eta = \Theta(\dt^{1/\alpha}\varepsilon^{1/\alpha-1})$ when $\frac{\eta}{\dt} = \varepsilon^{1/\alpha-1}\dt^{1/\alpha-1} \to 0$ as $\dt\to 0$. This means that for the first time step it holds that $u^{1}_2 = u_2^{ex}+\Theta(\dt^2) = \Theta(\dt)$. So, from the second time step classical error $\mathcal{O}(\dt^3)$ accumulates at each time step, leading to an error of $\mathcal{O}(\dt^2)$ at a final time.
		\end{itemize}
	\item \ref{eq:MPRK43-family} has a first stage of \ref{eq:MPE}, so $y^2_2=y_2^{ex}(\alpha\dt)+\mathcal{O}(\dt^2)=\Theta(\dt)$. Again, according to $p$ and $q$, exactly as for the \ref{eq:MPRK22-family} we have three situations.
	\begin{equation}
		\begin{cases}
			y_2^3 = y_2(\beta\dt)+ \mathcal{O}(\dt^3) = \Theta (\dt), & p=1,\\
			y_2^3 = y_2(\beta\dt)+ \mathcal{O}(\dt^2) = \Theta (\dt), & p<1,\\
			y_2^3 = y_2(\beta\dt)+ \mathcal{O}(\dt) = \Theta \left(\left(\frac{\varepsilon}{\dt}\right)^{1-1/p}\right), &p>1,
		\end{cases}
	\end{equation}
and
\begin{equation}
	\begin{cases}
		\sigma_2 = y_2(\dt)+ \mathcal{O}(\dt^3) = \Theta (\dt), & q=1,\\
		\sigma_2 = y_2(\dt)+ \mathcal{O}(\dt^2) = \Theta (\dt), & q<1,\\
		\sigma_2 = y_2(\dt)+ \mathcal{O}(\dt) = \Theta \left(\left(\frac{\varepsilon}{\dt}\right)^{1-1/p}\right), &q>1.
	\end{cases}
\end{equation}
These are obtained with the previous lemmas exactly as in the case of \ref{eq:MPRK22-family}.
\begin{itemize}
\item Now, if $q>1$ and $\sigma_2=  \Theta \left(\left(\frac{\varepsilon}{\dt}\right)^{1-1/q}\right)$ and it verifies the hypotheses of Lemma~\ref{lem:term_infty}, i.e. $\varepsilon^{1-1/q} \dt^{1/q-2}\to 0$ and $\frac{\varepsilon}{\dt} \to 0 $ as $\dt \to 0$, then, for Lemma~\ref{lem:term_infty}, we have that $u_2^1 = \Theta \left(\frac{\varepsilon^{1-1/q}}{\dt^{2-1/q}}\right)$. This means that at the first time step we have an error of $\Theta(\dt)$. Again, we see that $\varepsilon=o \left(\frac{\varepsilon^{1-1/q}}{\dt^{2-1/q}}\right)$ and this means that only few time steps will verify the hypotheses of Lemma~\ref{lem:term_zero}. Afterwards, the original third order accuracy will be restored, leading to an overall error of $\mathcal{O}(\dt)$ at a final time. 
\item If $q< 1$ then $\sigma_2 = u_2(\dt)+ \mathcal{O}(\dt^2)$ for Lemma~\ref{lem:term_zero} with $\eta=\dt^{1/q}\varepsilon^{1-1/q}$ when $(\dt\varepsilon)^{1/q-1}\to 0$ and $\frac{\varepsilon}{\dt}\to 0$ for $\dt \to 0$. Then, $\frac{u_2(t^1)}{\sigma_2} =1+ \Theta (\dt^2)$ which leads to an error of $\Theta(\dt^3)$ at the first time step.
\item If $q=1$ none of the lemmata apply  and the weighting factor should be of the expected third order accuracy.
\item If $p>1$ then $y_2^3=u_2(\beta\dt) + \Theta(\dt)$ for Lemma~\ref{lem:term_zero} when  $\varepsilon^{1-1/p} \dt^{1/p-2}\to 0$ and $\frac{\varepsilon}{\dt} \to 0 $ as $\dt \to 0$, hence this brings in the production and destruction terms of the final update an error of $\Theta(\dt^2)$.
\item If $p<1$ then $y_2^3=u_2(\beta\dt) + \Theta(\dt^2)$ for Lemma~\ref{lem:term_zero} when $\frac{\varepsilon}{\dt}\to 0$, which brings in the final update an error of $\Theta(\dt^3)$.
\item If $p=1$ this would not contribute to errors larger than the accuracy order of the scheme.
\end{itemize}
Putting all the information together we obtain the errors in Table~\ref{tab:ordersMP}, recalling that, except for $q>1$, only the first time step falls in the hypotheses of the lemmas, so, it is the only time step affected by these errors, while for $q>1$ some time steps will be effected. Anyway, the error at the final solution is bounded by the third order accuracy of the scheme itself.
\item For \ref{eq:MPRKSO22-family} the same arguments of \ref{eq:MPRK22-family} apply with $\gamma$ in place of $\alpha$.
\item For \ref{eq:MPRKSO(4,3)} we have that the first stage is again an \ref{eq:MPE}  step and $y_2^1=\Theta(\dt)$. Then, $1/\rho_2 = \Theta(\frac{\varepsilon}{\dt^2})$. For the equation of $y_2^3$ Lemma~\ref{lem:term_zero} applies with $\eta = \frac{\varepsilon}{\dt^2}$ when $\frac{\eta}{\dt}=\frac{\varepsilon}{\dt^3}\to 0$ as $\dt \to 0$, hence $y_2^3 = u_2((\alpha_{20}+\alpha_{21}) \dt) + \Theta(\dt^2).$ Then, $\mu_2 = \Theta\left(\frac{\dt^s}{\varepsilon^{s-1}}\right)$ which makes the equation for $\tilde{a}_2$ fall in the hypotheses of Lemma~\ref{lem:term_zero} with $\eta = \frac{\varepsilon^{s-1}}{\dt^s}$ when $\frac{\eta}{\dt}= \frac{\varepsilon^{s-1}}{\dt^{s+1}}\to 0 $ as $\dt \to 0$. Hence, $\tilde{a}_2=u_2((\eta_1+\eta_2)\dt) +\Theta(\dt^2).$ Then, $\sigma_2 = \tilde{a}_2 +\Theta(\frac{\varepsilon^2}{\dt^2}) = u_2(\dt)+ \Theta(\dt)$. This means that $\frac{u_2(\dt)}{\sigma_2}=1+\Theta(\dt)$, which sums up to a first order error $\Theta(\dt^2)$ for the first step. From the second step on, the third order accuracy is restored. Hence, at a final time an error of a $\Theta(\dt^2)$  is observable.
\item In \ref{eq:MPRK32} the first stage exploits the cancellation between the destruction and production of the same constituents as for all the \ref{eq:MPE} steps. All the other stages never present $y_2^1$ at the denominator of the MP weights, hence, none of the cases of the previous lemmas is met. So no order reduction phenomena appear.
\item In \ref{eq:SI-RK2} and \ref{eq:SI-RK3} the cancellation $d_{21}(y)/y_2 = (1-\theta)$ is always exploited, so there is no troubled term at the denominators. Hence, no order reduction is observed.
\end{itemize}
\end{proof}
	\begin{figure}
		\centering
		\begin{subfigure}{0.48\textwidth}
			\includegraphics[width=\textwidth]{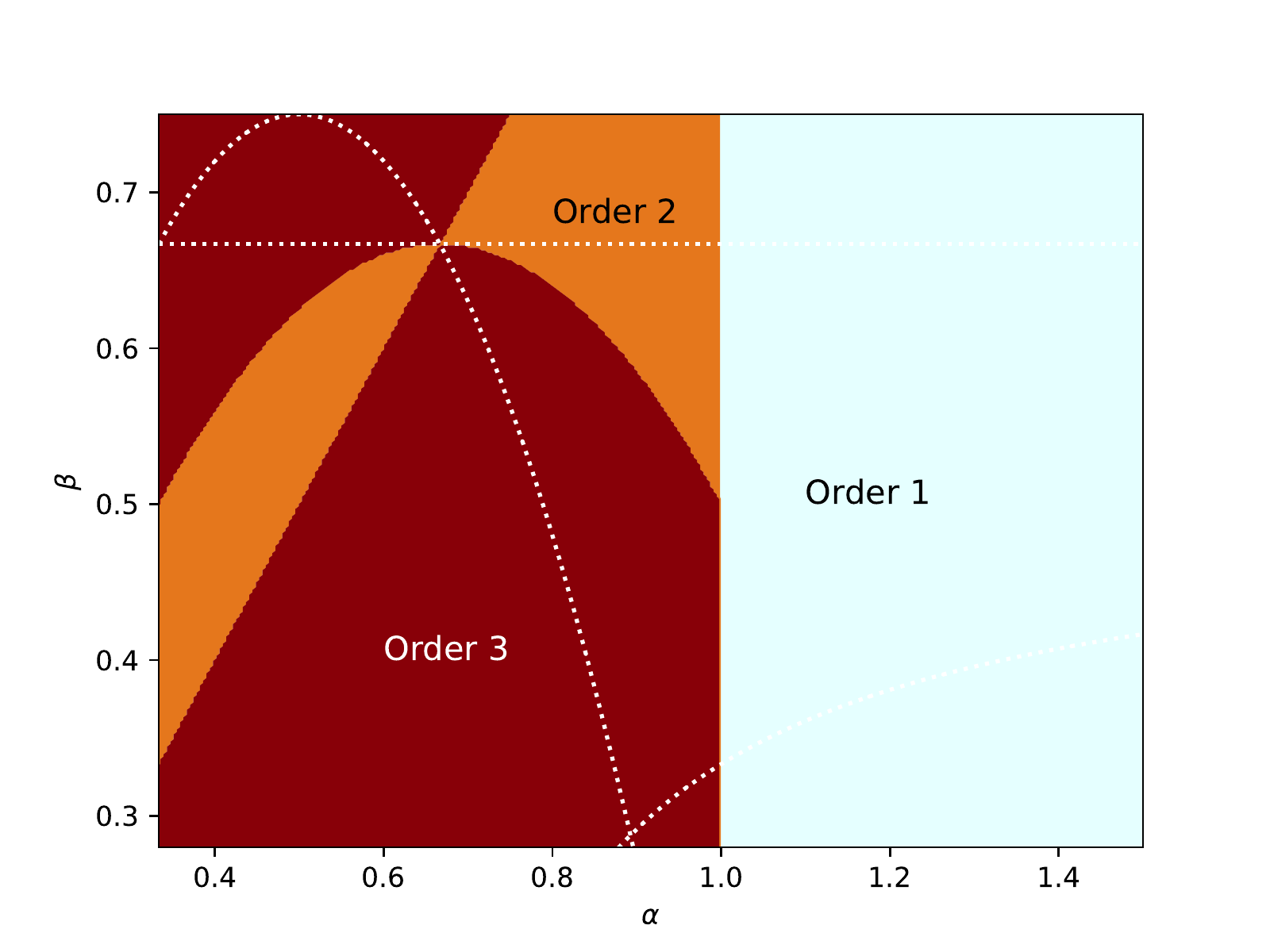}
			\caption{\ref{eq:MPRK43-family} orders: light blue first order, orange second order, brown third order\label{fig:orderMPRK43}}
		\end{subfigure}\hfill
		\begin{subfigure}{0.48\textwidth}
			\includegraphics[width=\textwidth]{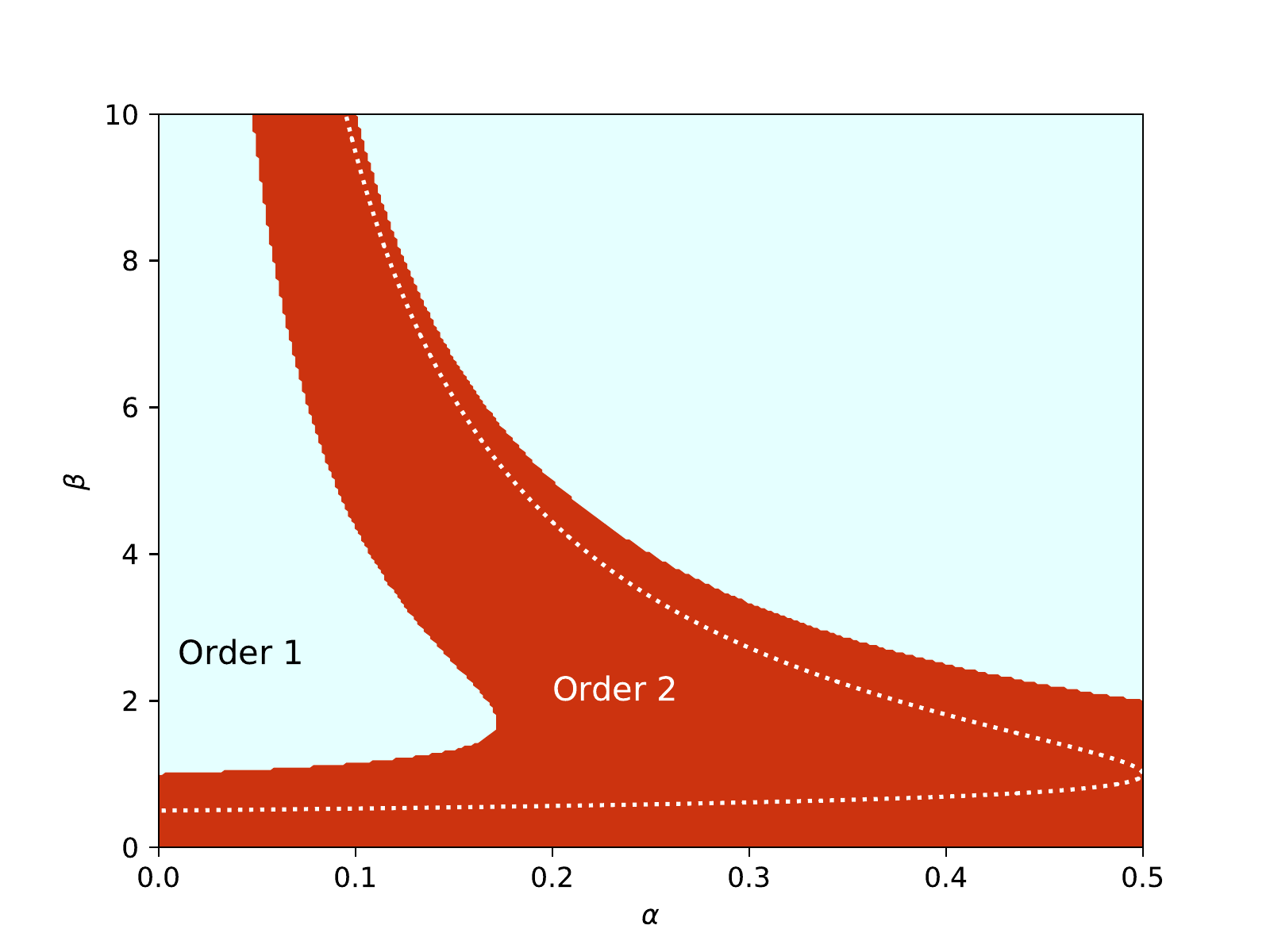}
			\caption{\ref{eq:MPRKSO22-family} orders: light blue first order, red second order \label{fig:orderMPRKSO2}}
		\end{subfigure}
		\caption{Order of accuracy of some schemes for vanishing initial conditions. The white dashed lines bound the positive RK coefficients area \cite{kopecz2018unconditionally,huang2019positivity}.}
	\end{figure}
As an example we want to focus on MPRK(3,4,2,0.5) plotted in Figure~\ref{fig:errorDecay23}. For this scheme $p=3$ and $q=2$. To verify the hypotheses of the lemmata, we need to have $\frac{\varepsilon^{1-1/2}}{\dt^{2-1/2}}\to 0$ as $\dt \to 0$, which is equivalent to $\frac{\varepsilon}{\dt^3}\to 0$ as $\dt \to 0$. Indeed, in the simulation in Figure~\ref{fig:errorDecay23}, we see that for $\dt \lesssim \varepsilon^{1/3}\approx 10^{-3.3}$ the error decays much faster than for $\dt \gtrsim 10^{-3}$. 

In Figure~\ref{fig:orderMPRK43} the order observable at a final time for \ref{eq:MPRK43-family} is summarized, while in  Figure~\ref{fig:orderMPRKSO2} it is summarized for \ref{eq:MPRKSO22-family}. For the \ref{eq:explicit_dec_correction} the order reduction comes from the negative DeC coefficients in the update formulae. In the following theorem we described the order reduction for vanishing IC.

\begin{theorem}[Loss of accuracy of \ref{eq:explicit_dec_correction} for vanishing initial data]\label{th:accuracyMPDEC}
		Consider the linear problem \eqref{eq:linearsystem2general} with IC $(1-\varepsilon, \varepsilon)^T$. For vanishing IC, i.e., $\frac{\varepsilon}{\dt}\to 0$ as $\dt \to 0$, the mPDeC is of order 2 if $\exists \theta^m_r <0$ with $m\in [\![1,M-1]\!]=\lbrace 1,\dots,M-1\rbrace$. If $\exists \theta^M_r <0$ the method is of order 1.
	\end{theorem}
	To prove the theorem let us introduce an useful proposition.
	\begin{proposition}[Carry over of the vanishing state]\label{prop:carry}
		Consider the linear problem \eqref{eq:linearsystem2general} with IC $y^{0}=(1-\varepsilon, \varepsilon)^T$. If $\exists \theta^m_r <0$ with $m\geq 1$ and if $y_2^{(k-1),m} =  \Theta(\varepsilon)$ with $\frac{\varepsilon}{\dt}\to 0$ as $\dt\to 0$, then $y_2^{(k),m} = \Theta(\varepsilon)$.
	\end{proposition}
	\begin{proof}
		Let us define $\theta^m_{-}$ the set of the negative coefficients among the $\theta^m_r$ and $\theta^m_{+}$ the set of the positive ones. We know that both sets are not empty, by hypothesis and by definition of $\theta_r^m$.
		\begin{align}
			\begin{split}
				y_i^{m,(k)}-y^0_i&-\sum_{l \in \theta^m_+} \theta_{l}^m \Delta t  \sum_{j}
				\left( \prod_{ij}(y^{l,(k-1)})
				\frac{y^{m,(k)}_{j}}{y_{j}^{m,(k-1)}}
				- \dest_{ij}(y^{l,(k-1)})  \frac{y^{m,(k)}_{i}}{y_{i}^{m,(k-1)}} \right)\\
				&-\sum_{l \in \theta^m_-} \theta_{l}^m \Delta t  \sum_{j}
				\left( \prod_{ij}(y^{l,(k-1)})
				\frac{y^{m,(k)}_{i}}{y_{i}^{m,(k-1)}}
				- \dest_{ij}(y^{l,(k-1)})  \frac{y^{m,(k)}_{j}}{y_{j}^{m,(k-1)}} \right)=0,
			\end{split}\\
			\begin{split}
				y_2^{m,(k)}-\varepsilon&-\sum_{l \in \theta^m_+} \theta_{l}^m \Delta t  		\left( \theta {y^{l,(k-1)}_1}
				\frac{y^{m,(k)}_{1}}{y_{1}^{m,(k-1)}}
				- (1-\theta) y^{l,(k-1)}_2  \frac{y^{m,(k)}_{2}}{y_{2}^{m,(k-1)}} \right)\\
				&-\sum_{l \in \theta^m_-} \theta_{l}^m \Delta t
				\left( \theta {y^{l,(k-1)}_1}
				\frac{y^{m,(k)}_{2}}{y_{2}^{m,(k-1)}}
				- (1-\theta) y^{l,(k-1)}_2 \frac{y^{m,(k)}_{1}}{y_{1}^{m,(k-1)}} \right)=0.
			\end{split}
		\end{align}
		We remind that for the conservation property of the scheme $y^{(k),r}_1=1-y^{(k),r}_2$. So, if we collect all the unknown terms in the left-hand side, we obtain
		\begin{align}
			\begin{split}
				&\left[ 1 + \Delta t \sum_{l \in \theta^m_+} \theta^m_l \left( \theta \frac{y^{l,(k-1)}_1}{y^{m,(k-1)}_1} +(1-\theta)\frac{y^{l,(k-1)}_2}{y^{m,(k-1)}_2} \right) - \Delta t \sum_{l \in \theta^m_-} \theta^m_l \left( \theta \frac{y^{l,(k-1)}_1}{y^{m,(k-1)}_2} +(1-\theta)\frac{y^{l,(k-1)}_2}{y^{m,(k-1)}_1} \right)
				\right]
				y_2^{m,(k)}\\
				&=\varepsilon+\sum_{l \in \theta^m_+} \theta_{l}^m \Delta t  \left( \theta
				\frac{y^{l,(k-1)}_{1}}{y_{1}^{m,(k-1)}} \right)
				-\sum_{l \in \theta^m_-} \theta_{l}^m \Delta t
				\left( (1-\theta) \frac{ y^{l,(k-1)}_2}{y_{1}^{m,(k-1)}} \right).
			\end{split}
		\end{align}
		Now, let us multiply the whole expression by the positive $y_2^{m,(k-1)}= \mathcal{O}(\varepsilon)$ and recalling that $y_1^{r,(k-1)}=1+ \mathcal{O}(\dt)+\mathcal{O}(\varepsilon)$. We obtain
		\begin{align}
			\begin{split}
				&\left[ \mathcal{O}(\varepsilon) + \Delta t \sum_{l \in \theta^m_+} \theta^m_l  (1-\theta){y^{l,(k-1)}_2} - \Delta t \sum_{l \in \theta^m_-} \theta^m_l  \theta y^{l,(k-1)}_1
				\right]
				y_2^{m,(k)}\\
				&=y^{m,(k-1)}_2 \left( \varepsilon+\sum_{l \in \theta^m_+} \theta_{l}^m \Delta t  \left( \theta
				\frac{y^{l,(k-1)}_{1}}{y_{1}^{m,(k-1)}} \right)
				-\sum_{l \in \theta^m_-} \theta_{l}^m \Delta t
				\left( (1-\theta) \frac{ y^{l,(k-1)}_2}{y_{1}^{m,(k-1)}} \right) \right).
			\end{split}
		\end{align}
		Now, the term $ \Delta t \sum_{l \in \theta^m_-} \theta^m_l \left( \theta y^{l,(k-1)}_1 \right)$ is the dominant in the left hand side, since $y_2^{l,(k-1)}= \mathcal{O}(\dt)$. Similarly the right hand side is dominated  by the $y_1$ terms.
		Hence, we obtain
		\begin{align}
			y^{m,(k)}_2  &= \frac{y_2^{m,(k-1)}\sum_{l \in \theta^m_+}\dt \theta_l^m \theta \frac{y_1^{l,(k-1)}}{y_1^{m,(k-1)}}  + \mathcal{O}(\varepsilon^2) + \mathcal{O}(\varepsilon \Delta t^2 )}{- \dt \sum_{l \in \theta^m_-} \theta^m_l \theta y_1^{l,(k-1)} + \mathcal{O}(\varepsilon)+ \mathcal{O}(\Delta t^2)}\\
			&= y_2^{m,(k-1)} \frac{\sum_{l \in \theta^m_+ }\theta_l^m \frac{y_1^{l,(k-1)}}{y_1^{m,(k-1)}}}{-\sum_{l \in \theta^m_-} \theta^m_l y_1^{l,(k-1)}} + \mathcal{O}(\varepsilon^2) + \mathcal{O}(\varepsilon \Delta t^2) = \Theta(\varepsilon),
		\end{align}
		because all $y_1^{l,(k-1)}$ are $\mathcal{O}(1)$. Hence, the proposition is proven.
	\end{proof}
	The proof of the theorem follows directly from this proposition.
	\begin{proof}
		If $\exists \theta_r^m <0$ with $m \in [\![ 1, M-1]\!]$, we have at the initial step all $y_2^{l,(0)}=\varepsilon$ for all $l$. Hence, by induction and using Proposition~\ref{prop:carry} we have that $y_2^{m,(K-1)}=\mathcal{O}(\varepsilon)=y_2(\beta^m \dt) + \mathcal{O}(\dt)$. Hence, computing the final update
		\begin{equation}
			y_i^{M,(K)}-y^0_i -\sum_{l} \theta_{l}^M \Delta t  \sum_{j}
			\left( \prod_{ij}(y^{l,(K-1)})
			\frac{y^{M,(K)}_{j}}{y_{j}^{m,(K-1)}}
			- \dest_{ij}(y^{l,(K-1)})  \frac{y^{M,(K)}_{i}}{y_{i}^{M,(K-1)}} \right)=0,
		\end{equation}
		the terms $d_{ij}(y^{m,(K-1)})=d_{ij}(y^{m,*}) + \mathcal{O}(\dt)$, hence an error of $\mathcal{O}(\dt^2)$ is obtained  in $y^{M,(K)}$. So, the solution at the next time iteration will be no longer a $\mathcal{O}(\varepsilon)$ and from the next time step high order errors will be restored. Hence, the approximation $\hat{y}_T$ at a certain time $T$ will be $\hat{y}_T-y(T) = \mathcal{O}(\dt^2)$.\\
		In case where $\exists \theta_r^M <0$, then $y^{M,(K)}=\mathcal{O}(\varepsilon) = y(\dt) + \mathcal{O}(\dt)$, from Proposition~\ref{prop:carry}. This condition will be left after some time steps, having brought to the method an error of $\mathcal{O}(\dt)$ at a final time $T$.
	\end{proof}
All the results in the theorems are in agreement with the motivational simulations in Figures~\ref{fig:errorDecay23} and \ref{fig:errorDecay4}.

For an automatic detection of such order reduction in the first step of the scheme, one can use symbolic tools and write, for specific problems and methods the Taylor expansion of the solution at the first time step first in $\varepsilon$ and then in $\dt$. As an example, we  show here the Taylor expansion for the error $\mathcal{E}(\varepsilon,\dt):= u_1^1(\varepsilon,\dt)-u_1^{ex}(\varepsilon,\dt)$ for \ref{eq:explicit_dec_correction}3.
Expanding first $\dt$ and then $\varepsilon$ in 0 we obtain
\begin{align*}
	\mathcal{E}(\varepsilon,\Delta t) =& \left(
	-\frac{1}{13824\varepsilon^2}
	-\frac{5}{1152\varepsilon}+
	\frac{1789}{13824}
	-\frac{1697\varepsilon}{6912}+
	\frac{7\varepsilon^2}{1536} +\mathcal{O}(\varepsilon^3) \right)\Delta t^4+\mathcal{O}(\Delta t^5),
\end{align*}
which means third order of accuracy for non vanishing $\varepsilon$, while, letting $\varepsilon\to 0$ first, we obtain
\begin{align*}
	\mathcal{E}(\varepsilon,\Delta t) =& \left(-\frac{\Delta t^2}{6} +  \mathcal{O}(\Delta t^3)\right)+ \left(112\Delta t  + \mathcal{O}(\Delta t^2)\right)\varepsilon -74880 \varepsilon^2  +\mathcal{O}(\Delta t \varepsilon^2) + \mathcal{O}(\varepsilon^3),
\end{align*}
and, hence, we have an error of $\mathcal{O}(\dt^2)$ for the first step and a global second order of accuracy.
More Taylor expansions can be found in the supplementary material \cite{torlo2021newStability} and the computations for these tests can be found in Mathematica notebooks in the accompanying reproducibility repository \cite{torlo2021stabilityGit}.

\section{Numerical experiments for simplified linear systems}
\label{sec:numerical-experiments}

As described in Section~\ref{sec:stability-linear}, we consider the simplified
$2 \times 2 $ system \eqref{eq:linearsystem2general}
with initial condition $u^0=(1-\varepsilon, \varepsilon)^T$. The goal of this
study is to find the largest time step $\Delta t$ for all possible systems
parameterized by $0\leq \theta \leq 1$ and initial conditions $0<\varepsilon<1$,
such that the properties \ref{prop:not_over} and \ref{prop:correct} are  satisfied.
To detect when the properties are fulfilled, as RadauIIA5 is doing in Figure~\ref{fig:motivating-example}, we consider the \textit{oscillation measure}
\begin{equation}\label{eq:measure}
	\texttt{osc}(u_1^0,u_1^1,u_1^*):=\begin{cases}
		\max \left \lbrace (u^1_1-u^0_1)^+, (u^*_1-u^1_1)^+\right\rbrace & \text{if } 1-\varepsilon=u_1^0>u^*_1=1-\theta,\\
		\max \left \lbrace (u^0_1-u^1_1)^+, (u^1_1-u^*_1)^+\right\rbrace & \text{if } 1-\varepsilon=u_1^0<u^*_1=1-\theta.
	\end{cases}
\end{equation}
Here, $(\cdot)^+$ denotes the positive part of a real number. This oscillation measure
vanishes for monotone schemes and increases with the amplitude of oscillations. When the initial conditions and the system taken in consideration are arbitrary, i.e., checking for all $0<\varepsilon, \theta<1$, we can use this measure to find oscillation--free schemes.
Hence, the measure \eqref{eq:measure} helps us in obtaining a very simple criterion on oscillation-free solutions studying just one time step.

Since we are interested in non-oscillatory behavior,
we need to check whether
\begin{equation}\label{eq:stability condition}
	\texttt{osc}(u_1^0,u_1^1,u_1^*) =0
\end{equation}
for every initial condition (IC) $0< \varepsilon<1$ and for every system defined through
$0\leq \theta \leq 1$.

We exploit the symmetry of the system studying only the $\varepsilon<0.5$ case,
as the other can be obtain substituting $\tilde{\varepsilon}=1-\varepsilon$ and
$\tilde{\theta}=1-\theta$.

In the following tests, we compare different methods and families presented above: \ref{eq:MPRK22-family}, \ref{eq:MPRK43-family}, \ref{eq:MPRKSO22-family}, \ref{eq:MPRKSO(4,3)}, \ref{eq:explicit_dec_correction} both for equispaced and Gauss--Lobatto sub-time steps, \ref{eq:MPRK32}, \ref{eq:SI-RK2}, and \ref{eq:SI-RK3}.

We apply all methods to a variety of $\varepsilon \in [0,0.5]$ and $\theta \in [0,0.5]$,
which are uniformly distributed in a logarithmic scale. For $\theta$, we also
consider the symmetrized values for $[0.5, 1]$.
We run the simulations for all these schemes and initial conditions for one time step $\Delta t$ of varying size, uniformly distributed in a logarithmic scale between $2^{-6}$ and $2^6$.
The maximum $\Delta t$ that gives no oscillations in the sense of \eqref{eq:stability condition} will be denoted as our bound.

\begin{figure}
	\centering
	\begin{subfigure}{0.49\textwidth}
		\centering
		\includegraphics[width=\textwidth]{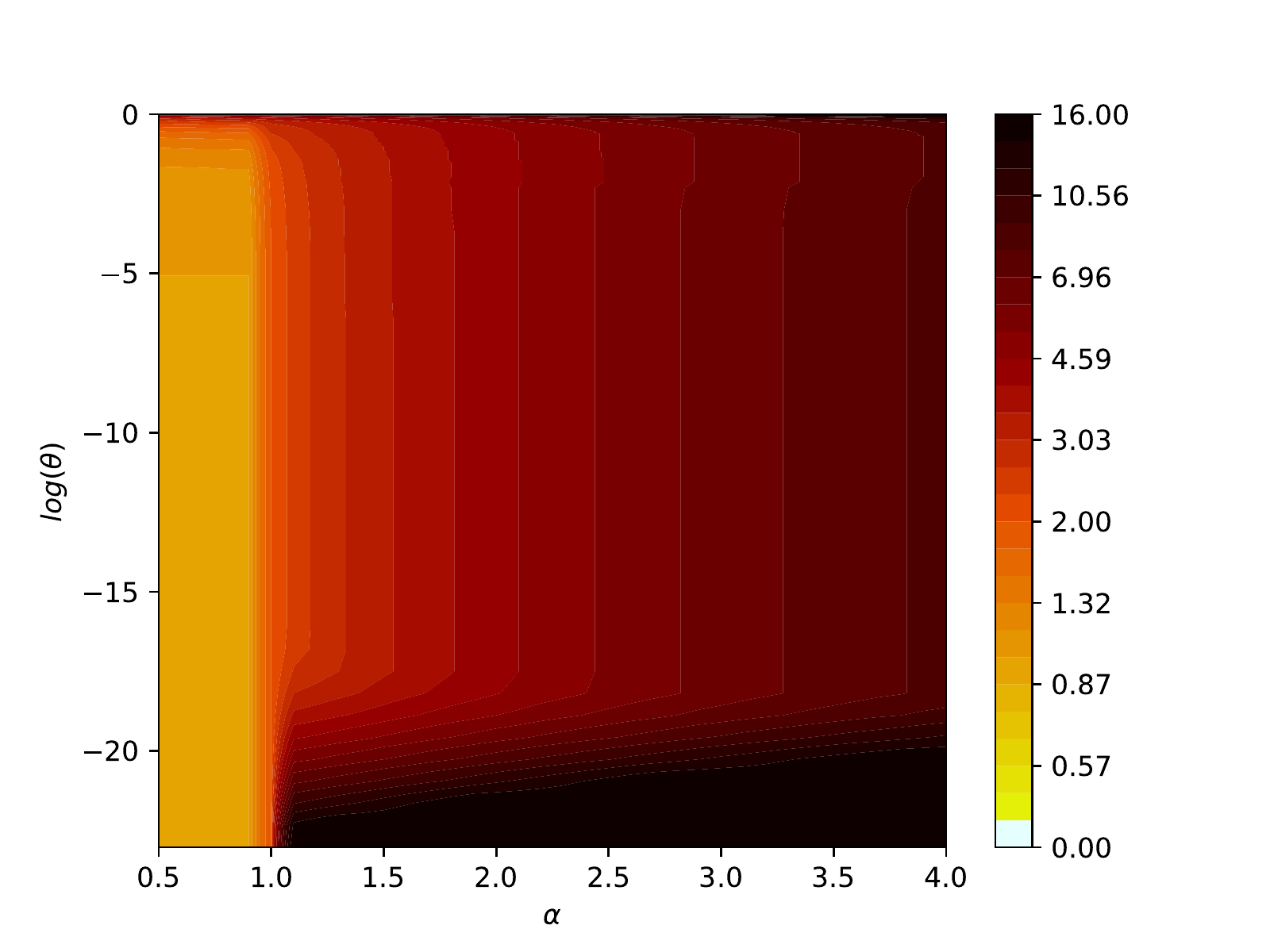}
		\caption{\ref{eq:MPRK22-family}: $\dt$ bound varying the system through $\theta$ and the method with $\alpha$.} \label{fig:systemOscillationsMPRK222}
	\end{subfigure}\hfill
	\begin{subfigure}{0.49\textwidth}
		\centering
		\includegraphics[width=\textwidth]{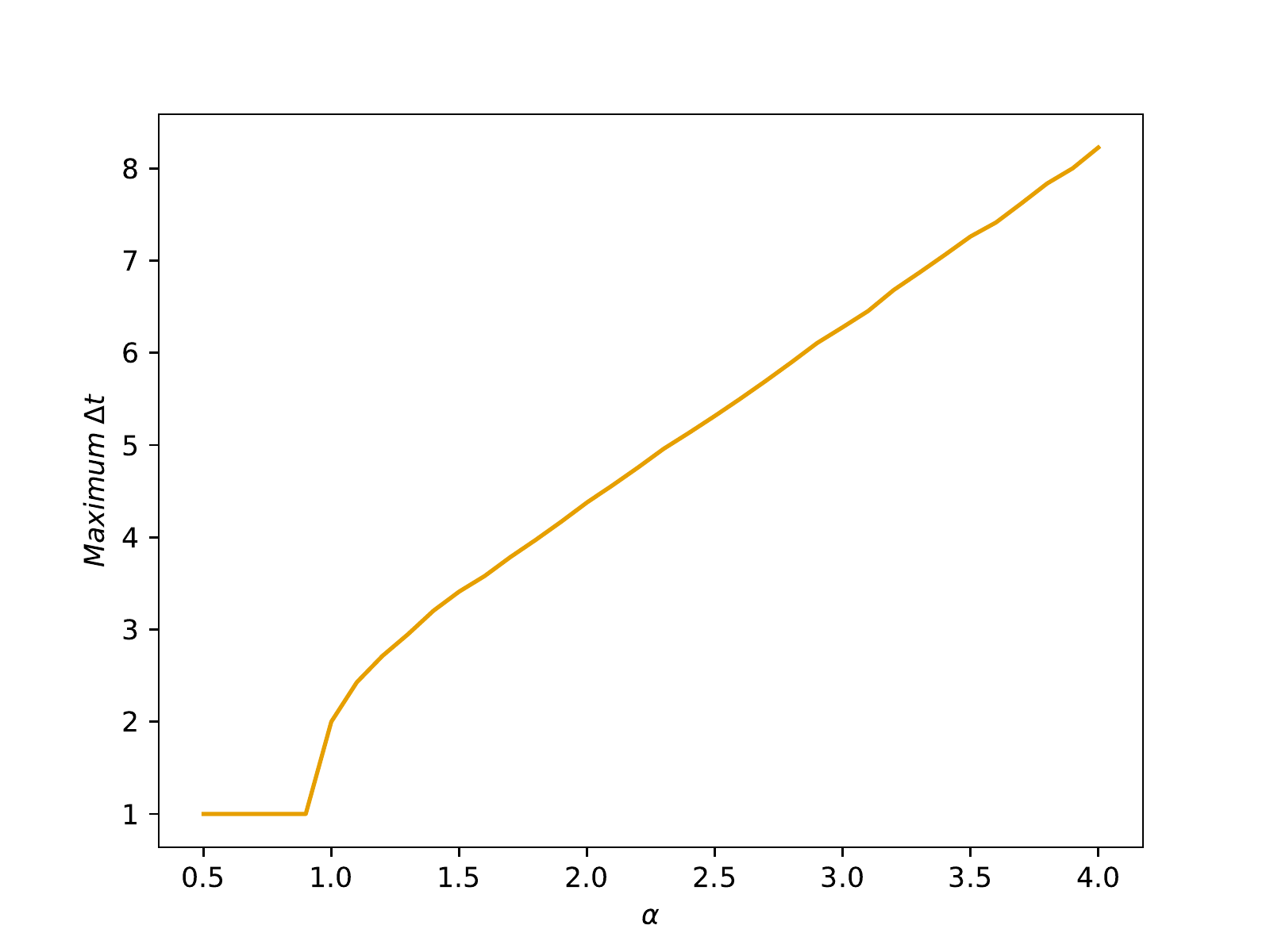}
		\caption{\ref{eq:MPRK22-family}: $\dt$ bound for all systems and initial condition varying $\alpha$.} \label{fig:systemCFLMPRK22all}
	\end{subfigure}\,
	\begin{subfigure}{0.49\textwidth}
		\centering
		\includegraphics[width=\textwidth]{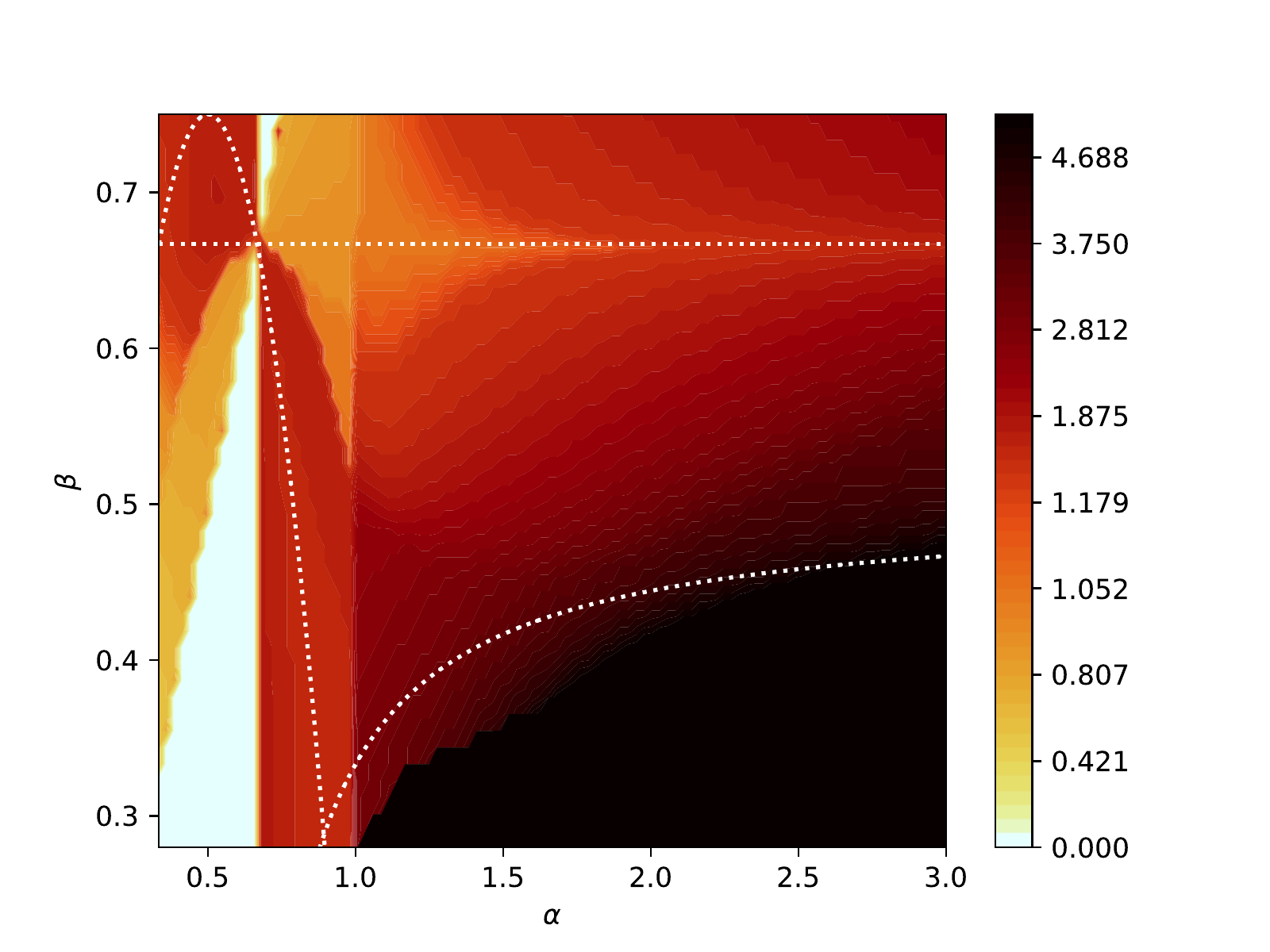}
		\caption{$\dt$ bound for \ref{eq:MPRK43-family} varying $\alpha$ and $\beta$. The white dashed lines bound the positive RK coefficients area \cite{kopecz2018unconditionally}.  \label{fig:systemCFLMPRK34allEps}}
	\end{subfigure}\,
	\begin{subfigure}{0.49\textwidth}
		\centering
		\includegraphics[width=\textwidth]{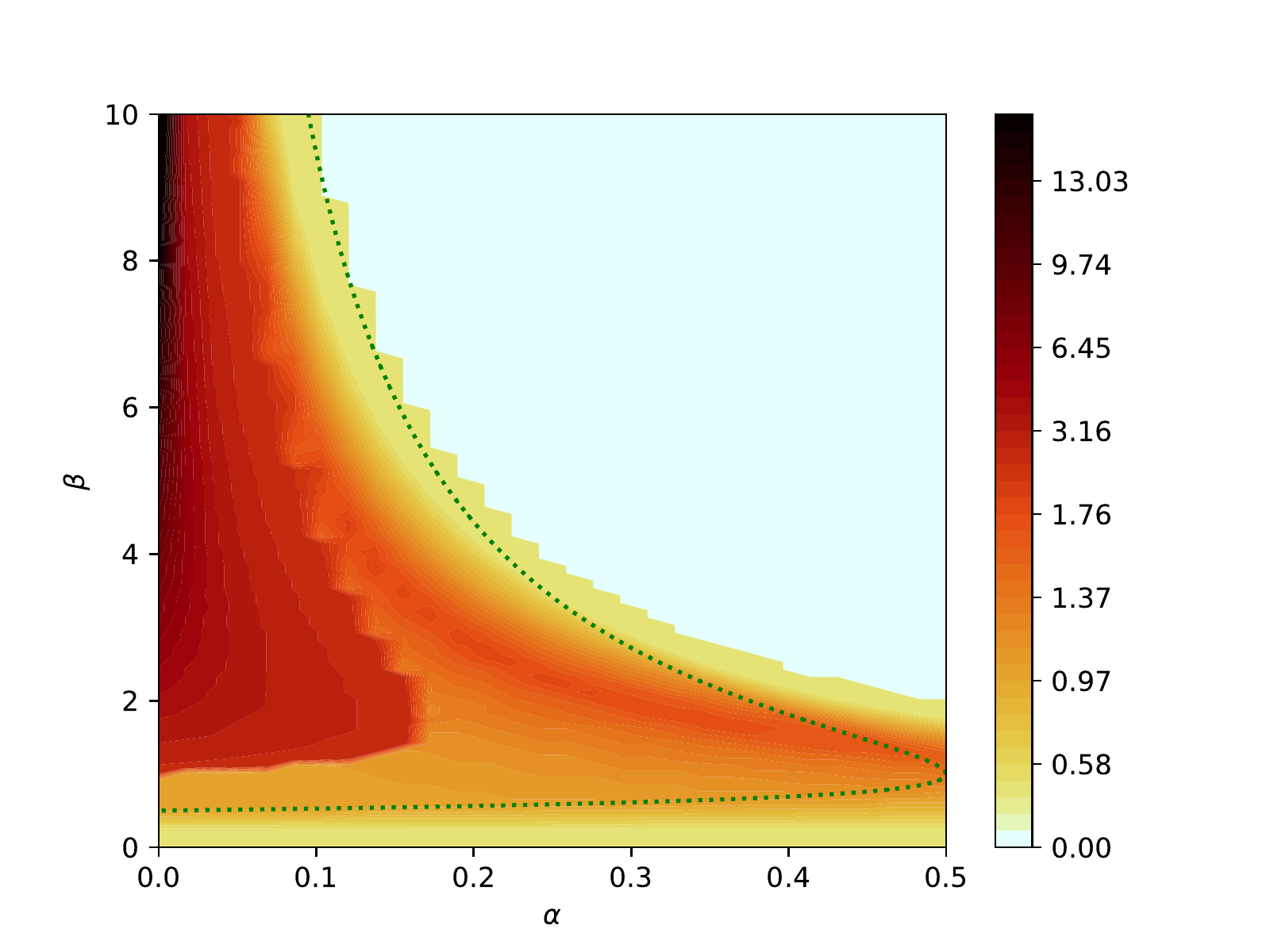}
		\caption{$\dt$ bound for \ref{eq:MPRKSO22-family} varying $\alpha$ and $\beta$. The green dashed lines bound the positive RK coefficients area \cite{huang2019positivity}. \label{fig:systemCFLMPSO2allEps}}
	\end{subfigure}	\caption{Numerical search of the $\dt$ bound for having an oscillation-free first time step, in the sense of \eqref{eq:stability condition}, for problem \eqref{eq:linearsystem2general} varying IC and system parameter $\theta$: \ref{eq:MPRK22-family}, \ref{eq:MPRK43-family} and \ref{eq:MPRKSO22-family}. \label{fig:stabilitySystemMethods}}
\end{figure}

\begin{figure}
	\begin{subfigure}{0.48\textwidth}
		\centering
		{\small
		\begin{tabular}{c}
			mPDeC\\
		\end{tabular}\\
		\begin{minipage}{0.48\textwidth}
			\begin{tabular}{|c|l|}\hline
				\multicolumn{2}{|c|}{Equispaced}\\ \hline
				\input{figures/system2/CFLDeCeqAllSystemsShort.tex}\hline
			\end{tabular}
		\end{minipage}
		\begin{minipage}{0.48\textwidth}
			\begin{tabular}{|c|l|}\hline
				\multicolumn{2}{|c|}{Gauss-Lobatto}\\ \hline
				\input{figures/system2/CFLDeCglbAllSystemsShort.tex}\hline
			\end{tabular}
		\end{minipage}
	}
		\caption{$\dt$ bound for mPDeC of order $p$ with equispaced and Gauss--Lobatto sub-time steps. In red the schemes with first order accuracy for vanishing initial conditions. \label{fig:systemCFLmPDeCallEps}}\vspace{3mm}
		\begin{tabular}{|c|l|}\hline
			Method & $\Delta t$ bound \\ \hline
			\ref{eq:MPRKSO(4,3)} & 1.31  \\
			\ref{eq:SI-RK2}        & 1.41  \\
			\ref{eq:SI-RK3}           & 1.27  \\
			\ref{eq:MPRK32}           & 16.56  \\\hline
		\end{tabular}
		\caption{Nonparamteric Patankar schemes and their $\dt$ bounds.\label{tab:osccilationsNonParam}}
	\end{subfigure} \,
\begin{subfigure}{0.48\textwidth}
	\centering
		\includegraphics[width=\textwidth]{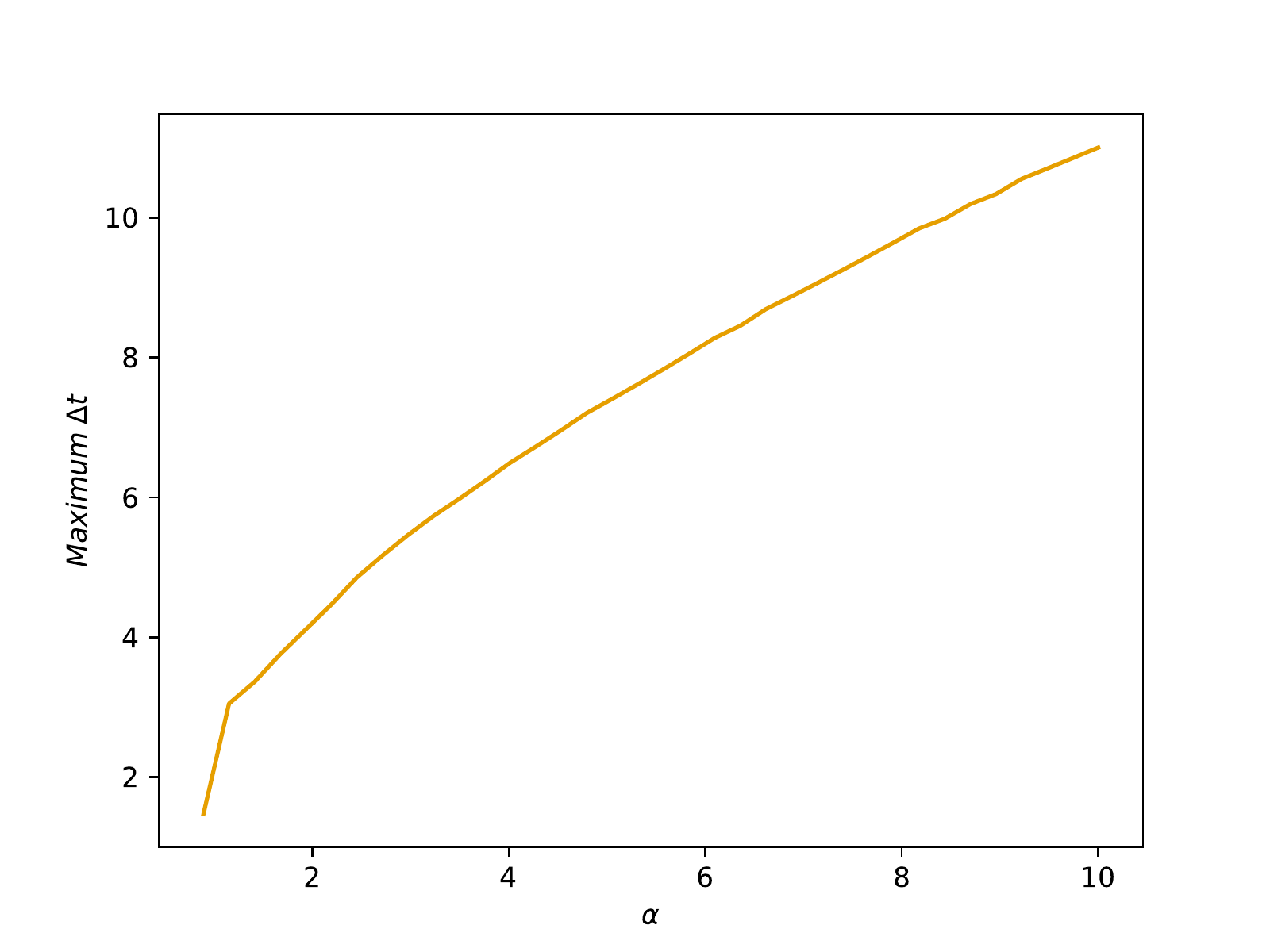}
	\caption{$\dt$ bound varying $\alpha$ for the family \ref{eq:MPRK43-family} on the curve $\beta(6\alpha-3)=3\alpha-2$ for all the systems through $\theta$  of the method. \label{fig:systemCFLMPSO3purple}}\vspace{5mm}
		\begin{tabular}{|c|l|}\hline
					\input{figures/system2/manyRK_CFL.tex}\hline
		\end{tabular}
			\caption{Other methods and their $\Delta t$ bounds.} \label{tab:systemCFL_otherRK}
	\end{subfigure}
	\caption{Numerical search of the $\dt$ bound for having an oscillation-free first time step, in the sense of \eqref{eq:stability condition}, for problem \eqref{eq:linearsystem2general} varying IC and system parameter $\theta$.}\label{fig:stabilitySystemMethods2}
\end{figure}

In Figure~\ref{fig:stabilitySystemMethods} and \ref{fig:stabilitySystemMethods2}, we present the results for the all the modified Patankar methods and for the semi-implicit Runge-Kutta methods. We highlight that the evaluation of condition \eqref{eq:stability condition} is done with a tolerance of $5  \, \times$ machine epsilon. Some tests can be sensitive to this tolerance, in particular for \eqref{eq:explicit_dec_correction} equispaced schemes with high odd order of accuracy, when the $\dt$ bound is large. There the number of stages is large and the machine error can sum up to non-negligible errors.

The second investigation of this section aims at validating the loss of accuracy of the schemes when they fall back to first order methods for $\varepsilon \to 0$. For this, we consider the system \eqref{eq:linearsystem2general} with $\theta=0.5$, and $\varepsilon=10^{-300}$ and we run the schemes for one large time step $\Delta t=1$. The exact solution at time 1 is $u_1(1)\approx 0.56$. If the approximation is such that $u_1^1 >0.999$ we say that the scheme is at most first order accurate.
By numerical experiments, we can say that this definition is robust with respect the system chosen and the tolerance on $u_1^1$. The interested reader can try different parameters in the repository code \cite{torlo2021stabilityGit}.

For \ref{eq:MPRK22-family}, we see in Figures~\ref{fig:systemOscillationsMPRK222} and \ref{fig:systemCFLMPRK22all} that the bound on $\dt$  is 1 for $\alpha < 1$,
2 for $\alpha=1$, and is increasing with $\alpha > 1$. We recall that the methods
with $\alpha > 1$ lose the order of accuracy in the limit $\varepsilon\to 0$, preserving the
initial condition as spurious steady state for few time steps. This must be kept in mind when choosing
the scheme one wants to use.
Varying the system parameter $\theta$ influences the bound on the time step,
as shown in Figure~\ref{fig:systemOscillationsMPRK222}.

For \ref{eq:MPRK43-family}, we observe areas where the $\dt$ bound reaches very low values ($\ll 1$) and other areas where it is larger than one, independently on the positivity of the RK coefficients. It must be noted that in the areas where the $\dt$ bound is large, we observe only first order accuracy for problems with $\varepsilon \to 0$ as one can compare with figure~\ref{fig:orderMPRK43}. It is noticeable that around the curve $\beta(6\alpha-3)=3\alpha-2$, which is a boundary for nonnegative coefficients \cite{kopecz2018unconditionally}, the $\dt$ bound is particularly large. Hence, in Figure~\ref{fig:systemCFLMPSO3purple} we plot the values for that specific curve, and indeed they are larger than other methods. On the other side, all the schemes given by these parameters show are only first order accurate for vanishing initial conditions.

For \ref{eq:MPRKSO22-family}, we observe that a large area of the $\alpha, \beta$
plane has $\dt$ bound around unity. The bounds increase close to the line $\alpha=0$.
For this family of methods, we also recall that as $\varepsilon \to 0$ we lose the order of accuracy for small $\alpha$ and large $\beta$. The precise area where this happens is denoted in brown in
figure~\ref{fig:orderMPRKSO2}. In the area of negative RK coefficients we observe very low $\dt$ bounds for the oscillation-free condition.

For \ref{eq:explicit_dec_correction}, we observe very different behaviors between equispaced and Gauss--Lobatto points. The two formulations coincide up to third order.
The second order \ref{eq:explicit_dec_correction} shows the $\dt=2$ bound that was
derived analytically in Section~\ref{sec:stability-linear}.
The methods based on Gauss--Lobatto nodes have a time step restriction of unity
for orders four and higher. Moreover, all the schemes reduce to order 2 when $\varepsilon\to 0$.
For equispaced nodes, we obtain larger $\dt$ bounds, in particular for schemes
with odd order of accuracy. In contrast to Gauss--Lobatto nodes, we observe also
order reduction to first order for high order schemes, more precisely for order 9 and order greater or equal to 11, when there are some negative $\theta^M_l$.

The \ref{eq:MPRKSO(4,3)} scheme has a $\dt$ bound of 1.31, as shown in
Figure~\ref{tab:osccilationsNonParam}. Moreover, it does show a reduction only to order 2 for
the numerical tests with vanishing initial conditions.
\ref{eq:MPRK32} has maybe the best conditions of all
the schemes, see Figure~\ref{tab:osccilationsNonParam}. Its $\dt$ bound is around
16 and it keeps its second order accuracy.

In Figures~\ref{tab:osccilationsNonParam}, the semi-implicit schemes are presented.
Both show similar behaviors with $\dt$ slightly larger than unity. For these
methods, there is no loss of accuracy.

In Figure~\ref{tab:systemCFL_otherRK}, we report the $\dt$ bound for some other standard time discretizations. Their implementation is available in the DifferentialEquations.jl \cite{rackauckas2017differentialequations} package in Julia \cite{bezanson2017julia}. We observe that some classical implicit schemes have a bound of around 2, while RadauIIA5 is unconditionally monotone, as predicted in Table~\ref{tab:dtImplicit}. Clearly all these methods do not suffer of order reduction for vanishing initial conditions.

A similar analysis on the $\dt$ bounds for a scalar nonlinear problem is reproduced in the supplementary material and available in \cite{torlo2021newStability}.

\section{Validation on nonlinear problems}\label{sec:nonlinear-experiements}

\subsection{Robertson problem}

The Robertson problem \cite[Section~II.10]{mazzia2008test}
with parameters
$k_1 = 0.04$,
$k_2 = 3 \cdot 10^7$,
and
$k_3 = 10^4$
is a stiff system of three nonlinear ODEs.
It can be written as a PDS \cite{kopecz2018order} with non-zero components
\begin{equation}
	\label{eq:robertson-pd}
	\begin{aligned}
		p_{12}(u)\! &=\! d_{21}(u) \!=\! k_3 u_2 u_3,
		&
		p_{21}(u)\! &=\! d_{12}(u)\! =\! k_1 u_1,
		&
		p_{32}(u)\! &=\! d_{23}(u)\! =\! k_2 u_2,
	\end{aligned}
\end{equation}
with initial conditions $u(0)=(1,0,0)^T.$
Reactions in this problem scale with different orders of magnitudes. To reasonably capture the behavior of the solution, it is necessary to use exponentially increasing time steps \cite{kopecz2018order}. To apply generic modified Patankar schemes, we have to modify the initial condition $u^0$ slightly, replacing $0$ by $\varepsilon>0$;
here, we use $\varepsilon= 10^{-180}$.

For this problem, oscillations are not so clearly defined, because the steady state $u^*=(0, 0, 1)^T$ cannot be exceeded since all the schemes are positive (and the modified Patankar also conservative).
Nevertheless, we might encounter the loss of accuracy problem  as some constituents are not present as initial conditions.
In Figure~\ref{fig:simulationsRobertson}, we observe that many methods do not catch the behavior of $u_2$ and remain close to zero. In some cases, even $u_3$ stays close to zero. All these phenomena are in accordance with the results found for the linear problem. Indeed, among the computed tests we see that \ref{eq:MPRK22-family} for $\alpha >1$, \hyperref[eq:MPRK43-family]{MPRK(4,3,10,0.5)}, \hyperref[eq:MPRKSO22-family]{MPRKSO(2,2,0.001,10)} and \ref{eq:explicit_dec_correction}11 with equispaced sub-time steps had order reduction to 1 for $\varepsilon \to 0$ and in this problem, they cannot properly describe the behavior of $u_2$ (and $u_3$).
Both semi-implicit methods \ref{eq:SI-RK2} and \ref{eq:SI-RK3} go to infinity as they do not conserve the total sum of the constituents. Hence, we are not showing their simulations.

\begin{figure}
	\centering
	\includegraphics[width=\textwidth]{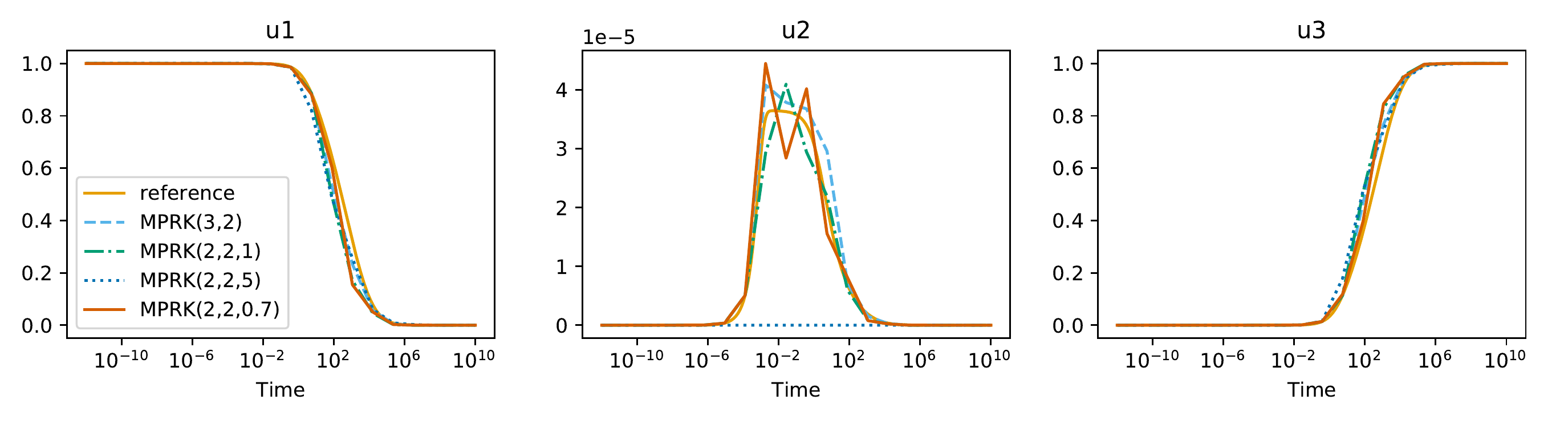}
	\includegraphics[width=\textwidth]{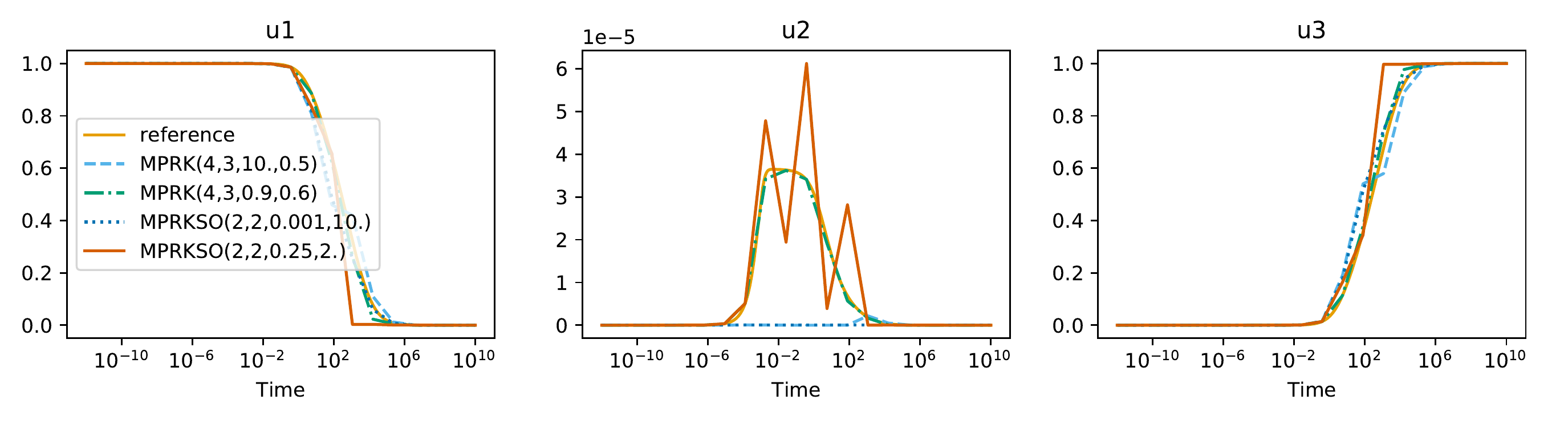}
	\includegraphics[width=\textwidth]{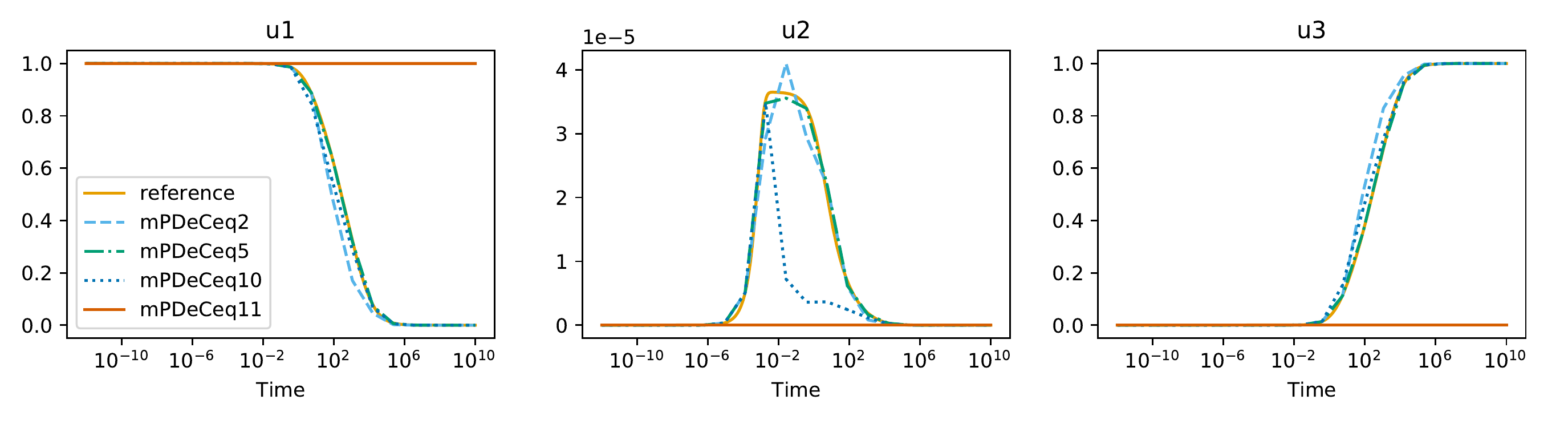}
	\includegraphics[width=\textwidth]{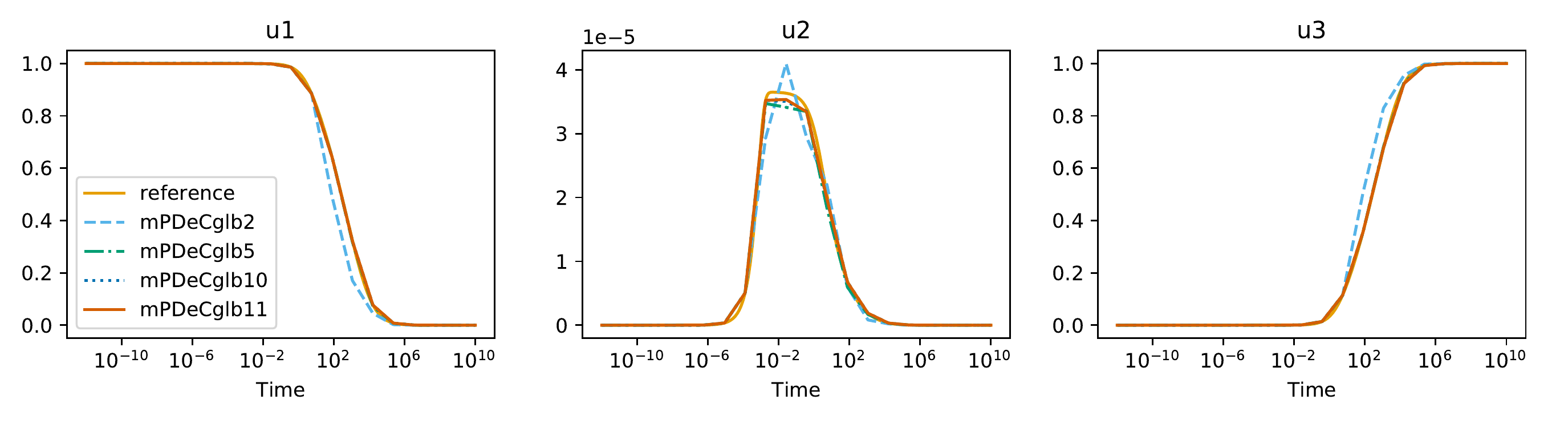}
	\caption{Robertson problem with different methods and 20 time steps.}\label{fig:simulationsRobertson}
\end{figure}

\subsection{HIRES}

We consider the ``High Irradiance RESponse'' problem (HIRES) \cite{hairer1999stiff}.
The original problem HIRES \cite[Section~II.1]{mazzia2008test} can be rewritten
as a nine-dimensional production--destruction system with
\begin{equation}
	\label{eq:HIRES-pdr}
	\begin{aligned}
		&r_1(u) = \sigma,
		&
		&d_{12}(u) = k_1 u_1,
		&
		&d_{21}(u) = k_2 u_2,
		\\
		&d_{24}(u) = k_3 u_2,
		&
		&d_{34}(u) = k_1 u_3,
		&
		&d_{31}(u) = k_6 u_3,
		\\
		& d_{43}(u) = k_2 u_4,
		&
		&  d_{46}(u) = k_4 u_4,
		&
		&  d_{56}(u) = k_1 u_5,
		\\
		&  d_{53}(u) = k_5 u_5,
		&
		&  d_{65}(u) = k_2 u_6,
		&
		&  d_{75}(u) = \frac{k_2}{2} u_7,
		\\
		&  d_{76}(u) = \frac{k_-}{2} u_7,
		&
		&  d_{79}(u) = \frac{k_*}{2} u_7,
		&
		&  d_{67}(u) = k_+ u_6 u_8,
		\\
		&  d_{87}(u) = k_+ u_6 u_8,
		&
		&  d_{78}(u) = \frac{k_- + k_*+k_2}{2} u_7,
	\end{aligned}
\end{equation}
$p_{ij}(u)=d_{ji}\, \forall \, i,j$
and parameters
\begin{equation}
	\begin{aligned}
		k_1 &= 1.71,
		&
		k_2 &= 0.43,
		&
		k_3 &= 8.32,
		&
		k_4 &= 0.69,
		&
		k_5 &= 0.035,
		\\
		k_6 &= 8.32,
		&
		k_+ &= 280,
		&
		k_- &= 0.69,
		&
		k_* &= 0.69,
		&
		\sigma &= 0.0007.
	\end{aligned}
\end{equation}
The initial condition is
$u(0) =
(
1,
0,
0,
0,
0,
0,
0,
0.0057,
0
)^T$,
where numerically we used $10^{-35}$ instead of zero for vanishing initial
constituents. The time interval is $t \in [0, 321.8122]$.

\begin{figure}
	\centering
	\ref{eq:MPRK22-family} with $\alpha=1$\\
	\includegraphics[width=\textwidth]{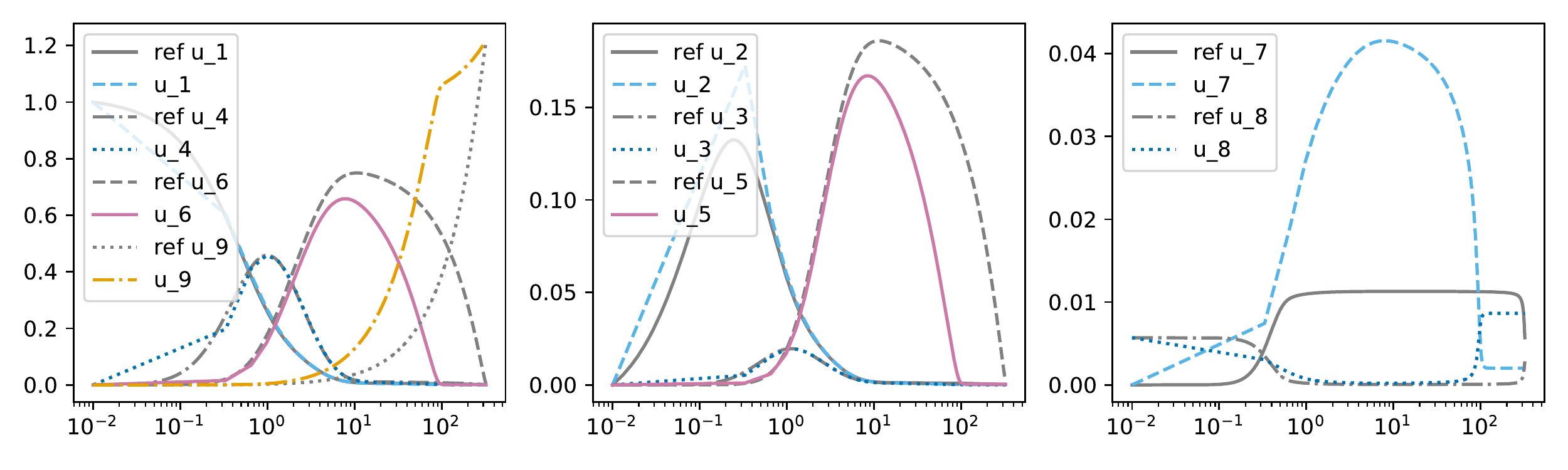}\\
	\ref{eq:MPRK22-family} with $\alpha=5$\\
	\includegraphics[width=\textwidth]{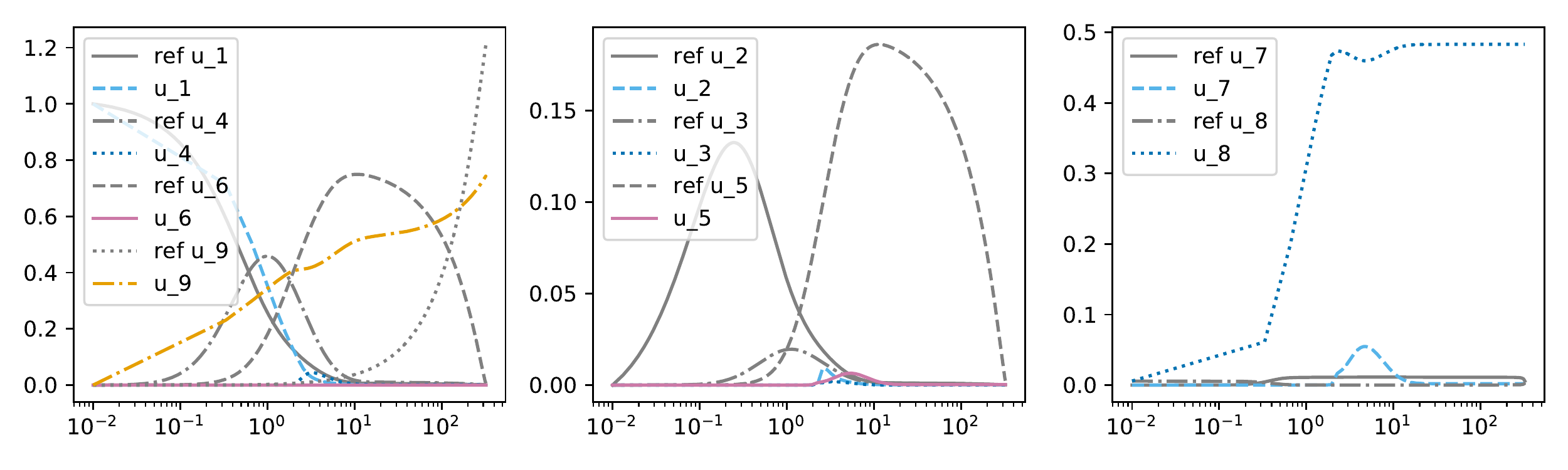}\\
	\ref{eq:explicit_dec_correction}6 with Gauss--Lobatto points\\
	\includegraphics[width=\textwidth]{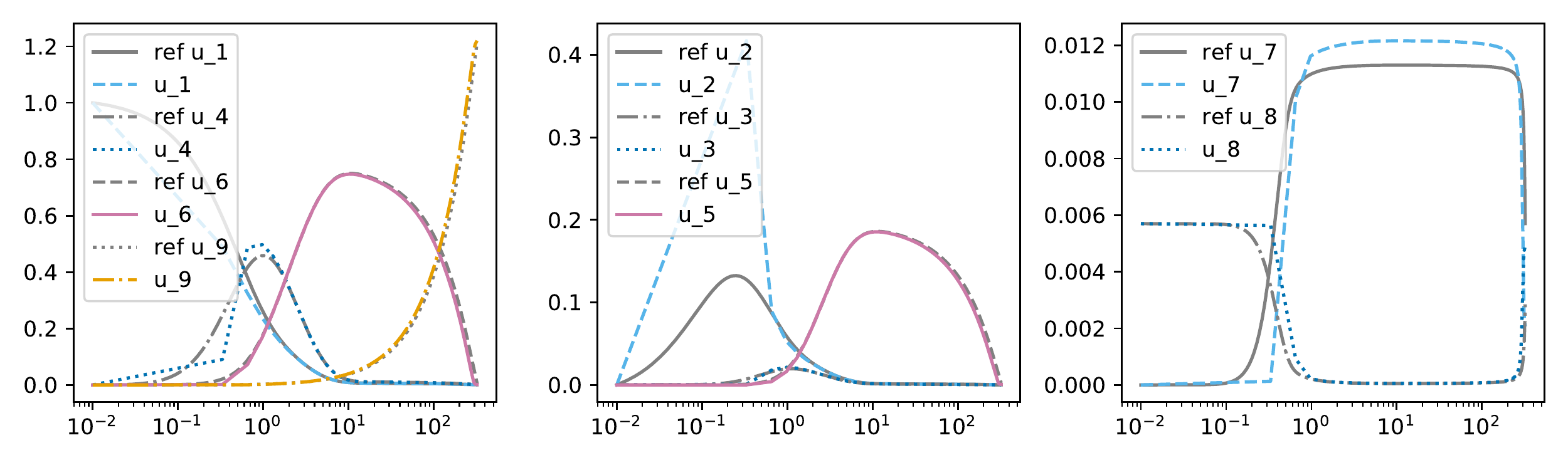}\\
	\ref{eq:MPRKSO22-family} with $\alpha=0.3$ and $\beta=2$\\
	\includegraphics[width=\textwidth]{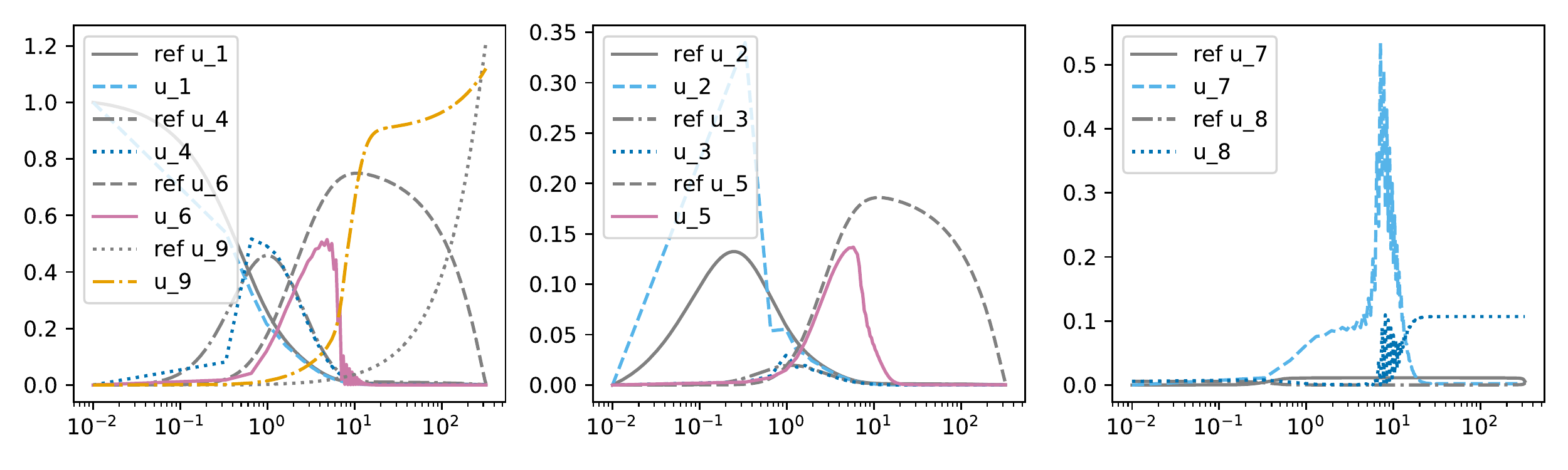}
	\caption{Simulations of HIRES problem run with different schemes with $N=10^3$ time steps, plot in logarithmic scale in time.}\label{fig:HIRES_simul}
\end{figure}
For this test, the concept of oscillation is not clear as well. Nevertheless, we can
observe inaccuracy of some methods also for this problem as some constituents are close to 0.
We compute the reference solution with $10^5$ uniform time steps using \ref{eq:explicit_dec_correction}5 with equispaced sub-time steps, which is in
accordance with the reference solution \cite{mazzia2008test} up to the fourth significant digit for all constituents.

Testing with $N=10^3$ uniform time steps, we spot troubles with the \textit{inconsistent} methods found in Section \ref{sec:numerical-experiments}.  We test the problem with many schemes presented above and we include the relative plots in the supplementary material \cite{torlo2021newStability}. For brevity, we plot in Figure~\ref{fig:HIRES_simul} just a sample.

For \ref{eq:explicit_dec_correction}, we observe the loss of accuracy only for equispaced time steps for high odd orders (9,\,11,\,13 and so on). In Figure~\ref{fig:HIRES_simul}, we see the simulation for \ref{eq:explicit_dec_correction}6 with Gauss--Lobatto points. We observe that the high accuracy helps in obtaining a good result at the end of the simulation, when $u_7$ and $u_8$ react. The moment at which this change happens is hard to catch and only high order methods are able to obtain it within this number of time steps.

We run the \ref{eq:MPRK22-family} with $\alpha\in \lbrace 1, 5\rbrace $. As for the linear case, we observe great loss of accuracy only for $\alpha>1$. This is demonstrated in
Figure~\ref{fig:HIRES_simul} for $\alpha=5$, where the evolution of some constituents is completely missed, e.g., $u_2, u_3, u_5, u_9$, while for $\alpha=1$ we obtain better results.

We test \ref{eq:MPRKSO22-family} with $\alpha=0.3,\,\beta=2$ and $\alpha=0,\,\beta=8$.
As expected, the second one shows the spurious steady state. An oscillatory behavior can be observed, though, also in the first simulation, which is shown in Figure~\ref{fig:HIRES_simul}. This is probably due to the CFL condition; refining the time discretization, the oscillations disappear.

For \ref{eq:MPRK43-family}, we test $\alpha=0.9,\,\beta= 0.6$ and $\alpha=5,\,\beta=0.5$, observing loss of accuracy only for the second one, in accordance with the linear tests. For \ref{eq:MPRKSO(4,3)}, \ref{eq:MPRK32}, \ref{eq:SI-RK2} and \ref{eq:SI-RK3}, we do not observe significant loss of accuracy, as in the linear test, nor other particular behaviors.

\section{Summary and discussion}
\label{sec:summary}

We proposed an analysis for Patankar-type schemes focused on two issues that some of these schemes present: oscillations around the steady state and loss of accuracy when a constituent is not present at the initial state. The oscillations are a property strongly linked to the positivity for linear problems and it is equivalent for linear methods. On the other side, the positivity preserving Patankar-type methods are not linear, hence, they oscillate around steady states.
Focusing on
a generic $2\times 2$ linear test problem, we introduced an oscillation measure.
Based thereon, we derived a CFL-like time step restriction avoiding oscillations
for all methods under consideration, either analytically (whenever
feasible) or numerically. Moreover, we investigated these methods near
vanishing components, discovering order reduction phenomena in many of the modified Patankar methods, even up to first order of accuracy.
Finally, we applied the methods to more challenging problems including stiff
nonlinear ones. We observed that our proposed oscillation-free and accuracy analysis
generalizes reasonably well to these other problems.

From our point of view, this is a first step toward further investigations on
Patankar-type schemes. Extensions could be based on various Lyapunov functionals instead of our oscillation measure. Moreover, different test
systems could be considered. Nevertheless, we would like to stress that our
current approach seems promising and generalizes well to other demanding problems.

As mentioned in Remark \ref{re_stability}, a stability analysis of all the considered methods with respect to  \cite{izgin2022stability} is work in progress. Furthermore, the connection between our observations and the obtained eigenvalues of the iterative process will be considered and compared in the future.

We plan also to extend our investigation to hyperbolic
conservation laws. After a spatial semidiscretizations,
we obtain ODEs that can be written as a production--destruction--rest system
\cite{meister2016positivity, huang2019third, ciallela2021}. Here, the relation
between the time step restrictions derived in this work and classical CFL conditions
will be the major focus of research.

\section*{Acknowledgments}
D. T. was funded by Team CARDAMOM in Inria--Bordeaux Sud--Ouest, France and by a SISSA Mathematical Fellowship, Italy.
P.\"O. gratefully acknowledge support of the Gutenberg Research College,  JGU Mainz and the UZH Postdoc Scholarship  (Number FK-19-104).
H. R. was supported by the German Research Foundation (DFG, Deutsche Forschungsgemeinschaft) under Germany's Excellence Strategy EXC 2044-390685587,	Mathematics M\"unster: Dynamics-Geometry-Structure.
We would like to thank Stefan Kopecz and David Ketcheson for fruitful discussion
at the beginning of this project.
This project has started with the visit by H.R.	in Zurich in 2019 which was supported by the  SNF project (Number 175784) and the	King Abdullah University of Science and Technology (KAUST).

{\small
	\appendix
	\section{Third order modified Patankar Runge--Kutta methods}\label{sec:appendix}
	In the following part, the third order accurate \ref{eq:MPRK43-family} from \cite{kopecz2018unconditionally, kopecz2019existence}
	is repeated for completeness. Please note that the investigated version is called $MPRK43I(\alpha,\beta)$ in their papers.
	It is given by
	\begin{equation}
		\tag{MPRK(4,3,$\alpha, \beta$)}
		\label{eq:MPRK43-family}
		\begin{aligned}
			y^1 &= u^n,
			\\
			y^2_i &= u^n_i
			+ a_{21} \dt \rest_i\bigl( y^1 \bigr)
			+  a_{21} \dt \sum_j \left(
			\prod_{ij}\bigl( y^1 \bigr) \frac{y^2_j}{y^1_j}
			- \dest_{ij}\bigl( y^1 \bigr) \frac{y^2_i}{y^1_i}
			\right),
			\\
			y^3_i &= u^n_i+ \Delta t \left(a_{31} \rest_i\bigl( y^1 \bigr)+
			a_{32} \rest_i\bigl( y^2 \bigr)
			\right)
			\\ & \qquad
			+\Delta t \sum_j
			\Biggl(\left(a_{31} \prod_{ij}\bigl(y^1\bigr)+ a_{32} \prod_{ij} \bigl(y^2\bigr) \right)  \frac{  y_j^3
			}{\bigl(y_j^2\bigr)^{1/p } \bigl(y_j^1\bigr)^{1-1/p} }
			\\
			& \qquad\qquad\qquad
			-\left(a_{31} \dest_{ij}\bigl(y^1\bigr)+ a_{32} \dest_{ij} \bigl(y^2\bigr) \right)  \frac{  y_i^3
			}{\bigl(y_i^2\bigr)^{1/p } \bigl(y_i^1\bigr)^{1-1/p} }
			\Biggr),
			\\
			\sigma_i &= u_i^n + \Delta t \sum_j
			\Biggl( \left( \beta_1 \prod_{ij} \bigl( y^1\bigr) +\beta_2 \prod_{ij} \bigl(y^2\bigr)  \right) \frac{\sigma_j}{\bigl(y_j^2 \bigr)^{1/q}
				\bigl(y_j^1\bigr)^{1-1/q}}
			\\ & \qquad
			-\left( \beta_1 \dest_{ij} \bigl( y^1\bigr) +\beta_2 \dest_{ij} \bigl(y^2\bigr)  \right) \frac{\sigma_i}{\bigl(y_i^2 \bigr)^{1/q}
				\bigl(y_i^1\bigr)^{1-1/q}} \Biggr)
			\\
			u^{n+1}_i &= u^n_i
			+ \dt \left( b_1 \rest_i\bigl( y^1 \bigr) + b_2\rest_i\bigl( y^2 \bigr)  + b_3\rest_i\bigl( y^3 \bigr)\right)
			\\&\qquad
			+ \dt \sum_j \Biggl(
			\left( b_1 \prod_{ij}\bigl( y^1 \bigr) +b_2\prod_{ij}\bigl( y^2 \bigr)
			+ b_3 \prod_{ij}\bigl( y^3 \bigr)
			\right) \frac{u^{n+1}_j}{\sigma_j}
			\\&\qquad\qquad\qquad
			- \left( b_1 \dest_{ij}\bigl( y^1 \bigr) +b_2\dest_{ij}\bigl( y^2 \bigr)
			+ b_3 \dest_{ij}\bigl( y^3 \bigr)
			\right) \frac{u^{n+1}_i}{\sigma_i}
			\Biggr),
		\end{aligned}
	\end{equation}
	where $p=3 a_{21}\left(a_{31}+a_{32} \right)b_3,\; q=a_{21},\;\beta_2=\frac{1}{2a_{21}}$ and $\beta_1= 1-\beta_2$. The Butcher tableaus in respect to the two parameters
	\begin{equation}
		\begin{aligned}
			\def\arraystretch{1.2}
			\begin{array}{c|ccc}
				0 &  & & \\
				\alpha & \alpha & & \\
				\beta & \frac{3\alpha\beta (1-\alpha)-\beta^2}{\alpha(2-3\alpha)}& \frac{\beta (\beta-\alpha)}{\alpha(2-3\alpha)}& \\
				\hline
				& 1+\frac{2-3(\alpha+\beta)}{6 \alpha \beta } &\frac{3 \beta-2}{6\alpha (\beta-\alpha)} & \frac{2-3\alpha}{6\beta(\beta-\alpha)}
			\end{array}
		\end{aligned}
	\end{equation}
	with positive coefficients for
	\begin{equation*}
		\begin{rcases}
			2/3 \leq \beta \leq 3\alpha(1-\alpha)\\
			3\alpha(1-\alpha)\leq\beta \leq 2/3 \\
			(3\alpha-2)/(6\alpha-3)\leq \beta \leq 2/3
		\end{rcases}
		\text{ for }
		\begin{cases}
			1/2 \leq \alpha<\frac{2}{3},\\
			2/3 \leq \alpha<\alpha_0,\\
			\alpha>\alpha_0,
		\end{cases}
	\end{equation*}
	and $\alpha_0\approx 0.89255.$ When the coefficients are negative we swap the weights of production and destruction terms as for \eqref{eq:explicit_dec_correction}.\\
	Next, also the MPRKSO(4,3) from \cite{huang2019third} is repeated.
	It is given by
	\begin{equation}
		\tag{MPRKSO(4,3)}
		\label{eq:MPRKSO(4,3)}
		\begin{aligned}
			y^1 &= u^n,
			\\
			y^2_i &= y^1_i
			+ a_{10} \dt \rest_i\bigl( y^1 \bigr)
			+  \dt \sum_j b_{10 }\left(
			\prod_{ij}\bigl( y^1 \bigr) \frac{y^2_j}{y_j^1}
			-\dest_{ij}\bigl( y^1 \bigr) \frac{y^2_i}{y_i^1}
			\right),
			\\
			\rho_i&= n_1y_i^2 +n_2 y_i^1 \left(\frac{y_i^2}{y_i^1} \right)^2
			\\
			y^3_i &= (a_{20}y^1_i +a_{21}y^2_i)
			+ \dt \left( b_{20} \rest_i  \bigl(  y^1 \bigr)+  b_{21} \rest_i  \bigl(  y^2 \bigr) \right)
			\\ \qquad
			&+  \dt \sum_j  \left(
			\left( b_{20} \prod_{ij}\bigl( y^1 \bigr)  +  b_{21} \prod_{ij}\bigl( y^2 \bigr)  \right) \frac{y^2_j}{\rho_j}
			-  \left( b_{20} \dest_{ij}\bigl( y^1 \bigr)  +  b_{21} \dest_{ij}\bigl( y^2 \bigr)  \right) \frac{y^2_i}{\rho_i}
			\right),\\
			\mu_i &=y_i^1\left( \frac{y_i^2}{y_i^1} \right)^s \\
			\tilde{a}_i&= \eta_1 y_i^1 + \eta_2 y_i^2 + \dt \sum_j \left( \left( \eta_3 \prod_{ij}\bigl( y^1 \bigr)  +\eta_4 \prod_{ij} \bigl( y^2 \bigr)  \right)  \frac{\tilde{a}_j}{\mu_j} -
			\left( \eta_3 \dest_{ij}\bigl( y^1 \bigr)  +\eta_4 \dest_{ij} \bigl( y^2 \bigr)  \right)  \frac{\tilde{a}_i}{\mu_i}  \right) \\
			\sigma_i &=\tilde{a}_i+zy_i^1 \frac{y_i^2}{\rho_i}
			\\
			u^{n+1}_i &=  \left(a_{30} y^1_i + a_{31}y^2_i +a_{32} y^3_i \right)
			+  \dt \left( b_{30}  \rest_i\bigl( y^1 \bigr) + b_{31} \rest_i \bigl( y^2 \bigr)+ b_{32} \rest_i \bigl( y^3 \bigr) \right)
			\\&\qquad
			+ \dt \sum_j \Biggl(
			\left(  b_{30}    \prod_{ij}\bigl( y^1 \bigr) +   b_{31}  \prod_{ij}\bigl( y^2 \bigr)
			+   b_{32}  \prod_{ij}\bigl( y^3 \bigr)\right) \frac{u^{n+1}_j}{\sigma_j}
			\\&\qquad\qquad\qquad
			- \left(  b_{30}    \dest_{ij}\bigl( y^1 \bigr) +   b_{31}  \dest_{ij}\bigl( y^2 \bigr)
			+   b_{32}  \dest_{ij}\bigl( y^3 \bigr)\right) \frac{u^{n+1}_i}{\sigma_i}
			\Biggr).
		\end{aligned}
	\end{equation}
	Here, the optimal SSP coefficients determined in \cite{huang2019third} will be used. They are given by
	\begin{align*}
		n_1&=2.569046025732011E-01,& n_2 &= 7.430953974267989E-01,\\
		a_{10} &=1, & a_{20} &=  9.2600312554031827E-01, \\
		a_{21} &=  7.3996874459681783E-02, &a_{31} &=2.0662904223744017E-10, \\
		b_{10} &=  4.7620819268131703E-01, &a_{30} &= 7.0439040373427619E-01, \\
		a_{32}& = 2.9560959605909481E-01, &b_{20} &= 7.7545442722396801E-02, \\
		b_{21}& =5.9197500149679749E-01, & b_{31} & = 6.8214380786704851E-10, \\
		b_{30} & = 2.0044747790361456E-01, & b_{32} &= 5.9121918658514827E-01, \\
		\eta_1 &= 3.777285888379173E-02, & \eta_2 & = 1/3, \\
		\eta_3 & = 1.868649805549811E-01,  & \eta_3 & = 2.224876040351123, \\
		z &= 6.288938077828750E-01, & s &= 5.721964308755304.
	\end{align*}

\section{Initial correct direction of Patankar schemes}\label{app:prove_direction}
As seen in Section~\ref{sec:stability-linear}, we are looking for schemes that do not oscillate.
To check this, there are two properties that must be verified. Given an arbitrary initial condition, the first step should go towards the steady state, Property~\ref{prop:correct}, and should not overshoot the steady state, Property~\ref{prop:not_over}.
In this section we investigate the direction of the first step of a method, i.e.,  Property~\ref{prop:correct}. In particular, if we know that the direction of the first step is always towards the steady state, for any initial condition, we know that oscillations are possible only around the steady state.
We will first present some theoretical results for very few schemes, then we summarize some numerical results we obtained varying $\varepsilon$ and $\theta$.

For symmetry we will check only Property~\ref{prop:correct} on the whole range of $0<\varepsilon \leq  \theta <1$.

\begin{theorem}[Direction of \ref{eq:MPE}]\label{th:directionMPE}
\ref{eq:MPE} enjoys Property~\ref{prop:correct} unconditionally, i.e., if the initial condition is above the steady state, then the first step will be below the initial condition, or, in other words,
\begin{equation}
	u_1^0>(1-\theta) \Longrightarrow u_1^0>u_1^1.
\end{equation}
\end{theorem}
\begin{proof}
	We write the \ref{eq:MPE} for the system \eqref{eq:linearsystem2general} in the first equation, making use of the conservation property and we collect all the implicit terms.
\begin{subequations}
	\begin{align}
		u^{1}_1 &= u^0_1
		+ \dt \left(
		(1-\theta)(1- u^0_1 ) \frac{1-u^{1}_1}{1-u^0_1}
		- \theta u^0_1 \frac{u^{1}_1}{u^0_1}
		\right),\\
		u^1_1 &= u^0_1
		+  \dt \left(
		(1-\theta)(1- u^{1}_1 )
		- \theta u^{1}_1
		\right),\\
		u^{1}_1 (1+\dt  ) &= y^1_1
		+  \dt
		(1-\theta),\\
		u^{1}_1  &= \frac{u^0_1
			+ \dt
			(1-\theta)}{(1+\dt  )}<\frac{ u^0_i(1
			+  \dt)
		}{(1+\dt  )}=u^0_1.
	\end{align}
	Here, we have simply used the hypothesis on $u_1^0>(1-\theta)$ and we obtain the thesis of the theorem.
\end{subequations}
\end{proof}

\begin{theorem}[Direction of \ref{eq:MPRK22-family} with $\alpha \leq 1$]\label{th:directionMPRK22}
\ref{eq:MPRK22-family} for $\alpha\leq 1$ applied on the simplified system \eqref{eq:linearsystem2general} \textit{has the correct direction }of the first time step for any $\dt> 0$.
\end{theorem}
\begin{proof}
The first stage consists in a first \ref{eq:MPE} step with time step $\alpha\dt$. So we obtain that $y^2_1<y^1_1=u^0_1$.
\begin{subequations}
For the second stage we can proceed analogously, exploiting the conservation property, the system \eqref{eq:linearsystem2general}, collecting all the implicit terms and using the hypothesis $u^0_1>(1-\theta)$.
\begin{align}
	\begin{split}
		u^{1}_1 &= u^0_1
+ \dt  \Biggl(
\left( \frac{2\alpha-1}{2\alpha} (1-\theta) \bigl( 1-y^1_1 \bigr) + \frac{1}{2\alpha} (1-\theta) \bigl( 1-y^2_1 \bigr) \right) \frac{1-u^{1}_1}{(1-y^2_1)^{1/\alpha} (1-y^1_1)^{1-1/\alpha}}
\\&\qquad\qquad\qquad
- \left( \frac{2\alpha-1}{2\alpha} \theta y^1_1  + \frac{1}{2\alpha} \theta y^2_1\right) \frac{u^{1}_1}{(y^2_1)^{1/\alpha} (y^1_1)^{1-1/\alpha}}
\Biggr),
\end{split}\\
	\begin{split}
u^{1}_1 &= u^0_1
+ \dt  \Biggl(
\left( \frac{2\alpha-1}{2\alpha} (1-\theta) \left(\frac{ 1-y^1_1 }{1-y_1^2}\right)^{1/\alpha}  + \frac{1}{2\alpha} (1-\theta) \left(\frac{ 1-y^2_1 }{1-y_1^1}\right)^{1-1/\alpha} \right) (1-u^{1}_1)
\\&\qquad\qquad\qquad
- \left( \frac{2\alpha-1}{2\alpha} \theta\left(\frac{ y^1_1}{y^2_1}\right)^{1/\alpha}  + \frac{1}{2\alpha} \theta \left(\frac{ y^2_1}{y^1_1}\right)^{1-1/\alpha}\right)u^{1}_1
\Biggr),
\end{split}\\
\begin{split}
	\Biggl( 1+&\dt \left( \frac{2\alpha-1}{2\alpha} (1-\theta) \left(\frac{ 1-y^1_1 }{1-y_1^2}\right)^{1/\alpha}  + \frac{1}{2\alpha} (1-\theta) \left(\frac{ 1-y^2_1 }{1-y_1^1}\right)^{1-1/\alpha} \right)+\\
	&\dt	\left( \frac{2\alpha-1}{2\alpha} \theta\left(\frac{ y^1_1}{y^2_1}\right)^{1/\alpha}  + \frac{1}{2\alpha} \theta \left(\frac{ y^2_1}{y^1_1}\right)^{1-1/\alpha}\right)
	\Biggr)u^{1}_1=\\
	u^0_1 +&\dt  \Biggl(
	\left( \frac{2\alpha-1}{2\alpha} (1-\theta) \left(\frac{ 1-y^1_1 }{1-y_1^2}\right)^{1/\alpha}  + \frac{1}{2\alpha} (1-\theta) \left(\frac{ 1-y^2_1 }{1-y_1^1}\right)^{1-1/\alpha} \right)<\\
	u^0_1&\left(1 +\dt  \Biggl(
	\left( \frac{2\alpha-1}{2\alpha} \left(\frac{ 1-y^1_1 }{1-y_1^2}\right)^{1/\alpha}  + \frac{1}{2\alpha}  \left(\frac{ 1-y^2_1 }{1-y_1^1}\right)^{1-1/\alpha} \right) \right). \label{eq:direction_eq_last_step}
\end{split}
\end{align}
So we have that
\begin{align}
	\begin{split}
		u^{1}_1 < u_1^0\frac{N}{D}
	\end{split}
\end{align}
with $N>0$ and $D>0$ deducible from \eqref{eq:direction_eq_last_step}. If $N<D$ we have our result, or, in other words, if $N-D<0$. So, let us compute
\begin{align}
	\begin{split}
		\frac{N-D}{\dt}=& \frac{2\alpha-1}{2\alpha} \left(\frac{ 1-y^1_1 }{1-y_1^2}\right)^{1/\alpha}  + \frac{1}{2\alpha}  \left(\frac{ 1-y^2_1 }{1-y_1^1}\right)^{1-1/\alpha} -\\
		&\frac{2\alpha-1}{2\alpha} (1-\theta) \left(\frac{ 1-y^1_1 }{1-y_1^2}\right)^{1/\alpha}  - \frac{1}{2\alpha} (1-\theta) \left(\frac{ 1-y^2_1 }{1-y_1^1}\right)^{1-1/\alpha}- \\
		&\frac{2\alpha-1}{2\alpha} \theta\left(\frac{ y^1_1}{y^2_1}\right)^{1/\alpha}  - \frac{1}{2\alpha} \theta \left(\frac{ y^2_1}{y^1_1}\right)^{1-1/\alpha},
	\end{split}\\
	\begin{split}
	\frac{N-D}{\dt}=& \frac{2\alpha-1}{2\alpha}\theta \left(\frac{ 1-y^1_1 }{1-y_1^2}\right)^{1/\alpha}  + \frac{1}{2\alpha} \theta \left(\frac{ 1-y^2_1 }{1-y_1^1}\right)^{1-1/\alpha} - \\
	&\frac{2\alpha-1}{2\alpha} \theta\left(\frac{ y^1_1}{y^2_1}\right)^{1/\alpha}  - \frac{1}{2\alpha} \theta \left(\frac{ y^2_1}{y^1_1}\right)^{1-1/\alpha}=\\
	&\frac{2\alpha-1}{2\alpha} \theta\left(\left(\frac{ 1-y^1_1 }{1-y_1^2}\right)^{1/\alpha} -\left(\frac{ y^1_1}{y^2_1}\right)^{1/\alpha} \right) + \frac{1}{2\alpha} \theta \left(\left(\frac{ 1-y^2_1 }{1-y_1^1}\right)^{1-1/\alpha}- \left( \frac{ y^2_1}{y^1_1}\right)^{1-1/\alpha} \right).
\end{split}
\end{align}
Now, we know that $y_1^1>y^2_1$, hence
$$
\frac{y_1^1}{y_1^2}>1>\frac{1-y_1^1}{1-y_1^2},
$$
so, considering $0<\alpha\leq1$, we have that $1/\alpha>0$ and $1-1/\alpha \leq 0$, we have
$$
\left(\left(\frac{ 1-y^1_1 }{1-y_1^2}\right)^{1/\alpha} -\left(\frac{ y^1_1}{y^2_1}\right)^{1/\alpha} \right)<0 \text{  and } \left(\left(\frac{ 1-y^2_1 }{1-y_1^1}\right)^{1-1/\alpha}- \left( \frac{ y^2_1}{y^1_1}\right)^{1-1/\alpha} \right)<0.
$$
Hence, $\frac{N-D}{\dt}<0$ and the proof is complete.
\end{subequations}
\end{proof}
For the case with $\alpha>1$ it is not so easy to derive an estimation as the two terms have opposite signs.

\begin{theorem}[Direction of \ref{eq:MPRKSO22-family} with $\gamma \geq 1$]
	\ref{eq:MPRKSO22-family} applied on the simplified system \eqref{eq:linearsystem2general} for positive RK coefficients and for $$\gamma=\frac{1-\alpha\beta + \alpha\beta^2}{\beta(1-\alpha\beta)}\geq 1$$ has the correct direction of the first time step.
\end{theorem}
\begin{proof}
	The proof follows the same step of proof of Theorem \ref{th:directionMPRK22}. The condition on the exponent of the weights here is precisely $\gamma\geq 1$.
\end{proof}
\begin{remark}[Accuracy area]
	We want to remark that the area in the $(\alpha,\beta)$ plane where $\gamma\geq 1$ and the RK coefficients are positive is defined by
	$$
	\alpha \leq \frac{\beta-1}{2\beta^2-\beta}\text{ with }\beta\geq 1,
	$$
	and this area coincide with the second order area for vanishing IC of \ref{eq:MPRKSO22-family} found in Figure~\ref{fig:orderMPRKSO2}.
\end{remark}

\subsection{Initial direction of other schemes}
For all other schemes it is not so easy to prove directly that the direction of the first step is the correct one. Nevertheless, we checked symbolically (when feasible) and numerically (otherwise) this property. The numerical computations are included in \texttt{CheckingDirection.ipynb} in the repository \cite{torlo2021stabilityGit}, while the only theoretical result is in \texttt{MPRK\_3\_2.nb}. We summarize in the following the results we obtained.
\begin{figure}
	\begin{center}
		\includegraphics[width=0.5\textwidth]{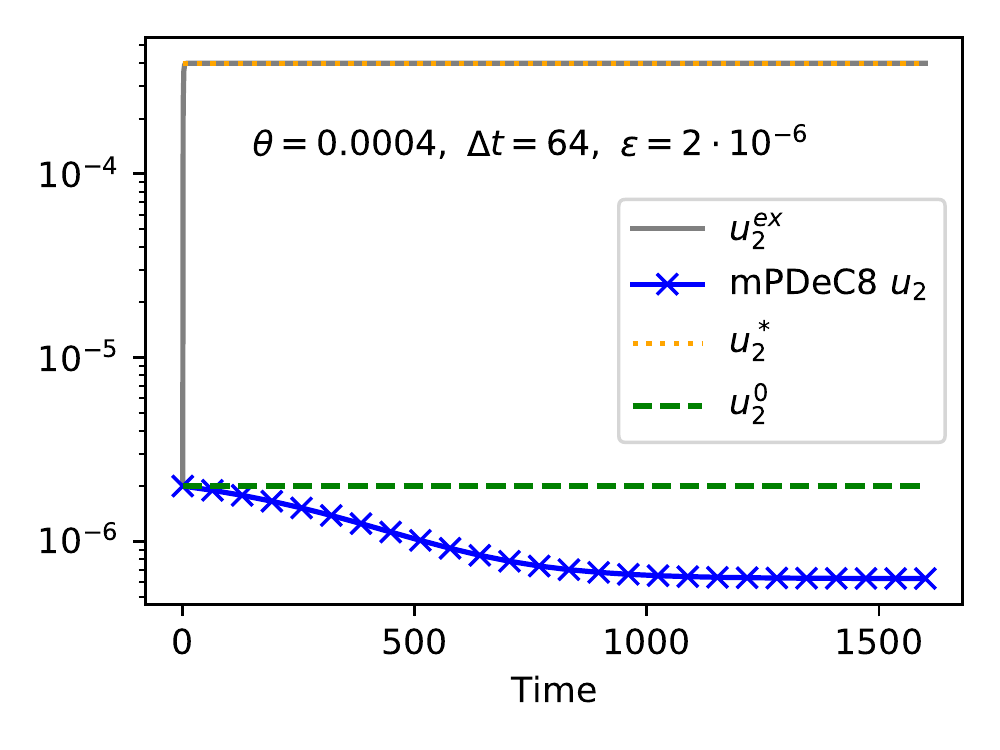}
		\caption{Simulation of \eqref{eq:linearsystem2general} with $\theta=4 \cdot 10^{-4}$ and $u_2^0=\varepsilon=2 \cdot 10^{-6}$ with \ref{eq:explicit_dec_correction}8 with equispaced points for $\dt=64$}\label{fig:wrongDirection}
	\end{center}
\end{figure}
\begin{itemize}
	\item \ref{eq:MPRK32} has the correct direction and we proved it in the Mathematica notebook \texttt{MPRK\_3\_2.nb};
	\item \ref{eq:MPRK22-family} have the correct direction for all $1/2\leq\alpha\leq 4$;
	\item \ref{eq:MPRKSO22-family} have the correct direction in an area slightly larger than the positive RK weights area displayed in Figure \ref{fig:systemCFLMPSO2allEps}, which coincide with the strictly positive $\dt$ bound area there;
	\item \ref{eq:MPRK43-family} have the correct direction except in a small area around $\alpha=2/3$ where the RK coefficients are negative;
	\item \ref{eq:MPRKSO(4,3)} has the correct direction;
	\item \ref{eq:explicit_dec_correction} with Gauss--Lobatto points have the correct direction (tested up to order 16);
	\item \ref{eq:explicit_dec_correction}  with equispaced points have the correct direction up to order 7, for order 8,~9 and 15 we found wrong directions for large $\dt(\geq 30)$ and very small initial conditions and $\theta$, all other mPDeC with orders up to 16 have the correct direction;
	\item \ref{eq:SI-RK2} and \ref{eq:SI-RK3} have the correct direction.
\end{itemize}
In Figure \ref{fig:wrongDirection}, we show an example for \ref{eq:explicit_dec_correction}8 where the correct direction is not followed. We see that even if we go away from the steady state, the scheme does not oscillate.

\printbibliography

\end{document}


\maketitle
\section{Patankar Methods}
In order to make the document self-contained, we list again the used methods.
\begin{itemize}
	\item Modified Patankar Euler method \cite{burchard2003high}
\begin{equation}\tag{MPE}
	\label{eq:MPE}
	u^{n+1}_i = u^{n}_i + \dt \rest_i\bigl( u^{n} \bigr)
	+ \dt \sum_j \left(
	\prod_{ij}\bigl( u^{n} \bigr) \frac{u^{n+1}_j}{u^{n}_j}
	- \dest_{ij}\bigl( u^{n} \bigr) \frac{u^{n+1}_i}{u^{n}_i}
	\right),
\end{equation}
\item Modified Patankar Runge--Kutta(2,2,$\alpha$) methods \cite{burchard2003high}
\begin{equation}
	\tag{MPRK(2,2,$\alpha$)}
	\label{eq:MPRK22-family}
	\begin{aligned}
		y^1 &= u^n,
		\\
		y^2_i &= u^n_i
		+ \alpha \dt \rest_i\bigl( y^1 \bigr)
		+ \alpha \dt \sum_j \left(
		\prod_{ij}\bigl( y^1 \bigr) \frac{y^2_j}{y^1_j}
		- \dest_{ij}\bigl( y^1 \bigr) \frac{y^2_i}{y^1_i}
		\right),
		\\
		u^{n+1}_i &= u^n_i
		+ \dt \left( \frac{2\alpha-1}{2\alpha} \rest_i\bigl( y^1 \bigr) + \frac{1}{2\alpha} \rest_i\bigl( y^2 \bigr) \right)
		\\&\qquad
		+ \dt \sum_j \Biggl(
		\left( \frac{2\alpha-1}{2\alpha} \prod_{ij}\bigl( y^1 \bigr) + \frac{1}{2\alpha} \prod_{ij}\bigl( y^2 \bigr) \right) \frac{u^{n+1}_j}{(y^2_j)^{1/\alpha} (y^1_j)^{1-1/\alpha}}
		\\&\qquad\qquad\qquad
		- \left( \frac{2\alpha-1}{2\alpha} \dest_{ij}\bigl( y^1 \bigr) + \frac{1}{2\alpha} \dest_{ij}\bigl( y^2 \bigr) \right) \frac{u^{n+1}_i}{(y^2_i)^{1/\alpha} (y^1_i)^{1-1/\alpha}}
		\Biggr)
		\Biggr).
	\end{aligned}
\end{equation}
\item Modified Patankar Shu--Osher Runge--Kutta(2,2,$\alpha,\beta$) methods \cite{huang2019positivity}
\begin{equation}
	\tag{MPRKSO(2,2,$\alpha$,$\beta$)}
	\label{eq:MPRKSO22-family}
	\begin{aligned}
		y^1 &= u^n,
		\\
		y^2_i &= y^1_i
		+ \beta \dt \rest_i\bigl( y^1 \bigr)
		+ \beta \dt \sum_j \left(
		\prod_{ij}\bigl( y^1 \bigr) \frac{y^2_j}{y^1_j}
		- \dest_{ij}\bigl( y^1 \bigr) \frac{y^2_i}{y^1_i}
		\right),
		\\
		u^{n+1}_i &= (1-\alpha) y^1_i + \alpha y^2_i
		+ \dt \left( (1-1/2\beta-\alpha\beta) \rest_i\bigl( y^1 \bigr) + \frac{1}{2\beta} \rest_i\bigl( y^2 \bigr) \right)
		\\&\qquad
		+ \dt \sum_j \Biggl(
		\left( (1-1/2\beta-\alpha\beta) \prod_{ij}\bigl( y^1 \bigr) + \frac{1}{2\beta} \prod_{ij}\bigl( y^2 \bigr) \right) \frac{u^{n+1}_j}{(y^2_j)^{\gamma} (y^1_j)^{1-\gamma}}
		\\&\qquad\qquad\qquad
		- \left( (1-1/2\beta-\alpha\beta) \dest_{ij}\bigl( y^1 \bigr) + \frac{1}{2\beta} \dest_{ij}\bigl( y^2 \bigr) \right) \frac{u^{n+1}_i}{(y^2_i)^{\gamma} (y^1_i)^{1-\gamma}}
		\Biggr),
	\end{aligned}
\end{equation}where the parameters are restricted to $\alpha \in [0,1]$, $\beta \in (0,\infty)$,
$\alpha \beta + 1 / 2\beta \leq 1$, and
\begin{equation}
\gamma = \frac{1 - \alpha \beta + \alpha \beta^2}{\beta ( 1 - \alpha \beta)}.
\end{equation}
\item Modified Patankar Deferred Correction methods \cite{offner2019arbitrary}
\begin{equation}
	\tag{mPDeC}
	\label{eq:explicit_dec_correction}
	y_i^{m,(k)}-y^0_i -\sum_{r=0}^M \theta_r^m \Delta t \rest \bigl(y^{r,(k-1)} \bigr)-\sum_{r=0}^M \theta_r^m \Delta t  \sum_{j=1}
	\left( \prod_{ij}(y^{r,(k-1)})
	\frac{y^{m,(k)}_{\gamma(j,i, \theta_r^m)}}{y_{\gamma(j,i, \theta_r^m)}^{m,(k-1)}}
	- \dest_{ij}(y^{r,(k-1)})  \frac{y^{m,(k)}_{\gamma(i,j, \theta_r^m)}}{c_{\gamma(i,j, \theta_r^m)}^{m,(k-1)}} \right)=0,
\end{equation}
where $\theta_r^m$ are the correction weights and the $\gamma(j,i,\theta_r^m)$ are
the indicator functions depending on  $\theta$ if the values are positive or negative, see \cite{offner2019arbitrary} for details. Finally, the new numerical solution
is $u_i^{n+1}=y^{M,(K)}$.
\item The new Modified Patankar Runge--Kutta(3,2) method based on the SSPRK(3,3)
\begin{equation}
	\tag{MPRK(3,2)}
	\label{eq:MPRK32}
	\begin{aligned}
		y^1 &= u^{n},
		\\
		y^2_i &= u^{n}
		+ \dt \rest_i\bigl( y^1 \bigr)
		+ \dt \sum_j \left(
		\prod_{ij}\bigl( y^1 \bigr) \frac{y^2_j}{y^1_j}
		- \dest_{ij}\bigl( y^1 \bigr) \frac{y^2_i}{y^1_i}
		\right),
		\\
		y^3 &= u^{n}
		+ \dt \frac{\rest_i\bigl( y^1 \bigr) + \rest_i\bigl( y^2 \bigr)}{4}
		+ \dt \sum_j \left(
		\frac{\prod_{ij}\bigl( y^1 \bigr) + \prod_{ij}\bigl( y^2 \bigr)}{4} \frac{y^3_j}{y^2_j}
		- \frac{\dest_{ij}\bigl( y^1 \bigr) + \dest_{ij}\bigl( y^2 \bigr)}{4} \frac{y^3_i}{y^2_i}
		\right),
		\\
		u^{n+1} &= u^{n}
		+ \dt \frac{\rest_i\bigl( y^1 \bigr) + \rest_i\bigl( y^2 \bigr) + 4 \rest_i\bigl( y^3 \bigr)}{6}
		\\&
		+ \dt \sum_j \left(
		\frac{\prod_{ij}\bigl( y^1 \bigr) + \prod_{ij}\bigl( y^2 \bigr) + 4 \prod_{ij}\bigl( y^3 \bigr)}{6} \frac{u^{n+1}_j}{y^2_j}
		- \frac{\dest_{ij}\bigl( y^1 \bigr) + \dest_{ij}\bigl( y^2 \bigr) + 4 \dest_{ij}\bigl( y^3 \bigr)}{6} \frac{u^{n+1}_i}{y^2_i}
		\right).
	\end{aligned}
\end{equation}
\item (Patankar) Semi Implicit Runge--Kutta(2,2) methods \cite{chertock2015steady}
\begin{equation}
	\tag{SI-RK2}
	\label{eq:SI-RK2}
	\begin{aligned}
		y^1 &= u^n,
		\\
		y^2_i &= \frac{ u^n_i + \dt \rest_i(y^1) + \dt \sum_j \prod_{ij}(y^1) }
		{ 1 + \dt \sum_j \dest_{ij}(y^1) / y^1_i },
		\\
		y^3_i &= \frac{1}{2} u^n_i + \frac{1}{2}
		\frac{ y^2_i + \dt \rest_i(y^2) + \dt \sum_j \prod_{ij}(y^2) }
		{ 1 + \dt \sum_j \dest_{ij}(y^2) / y^2_i },
		\\
		u^{n+1}_i &= \frac{ y^3_i + \dt^2 \bigl(\rest_i(y^3) + \dt \sum_j \prod_{ij}(y^3) \bigr) \sum_j \dest_{ij}(y^3) / y^3_i }
		{ 1 + \bigl( \dt \sum_j \dest_{ij}(y^3) / y^3_i \bigr)^2 }.
	\end{aligned}
\end{equation}
\item Modified Patankar Runge--Kutta(4,3,$\alpha,\beta$) methods \cite{kopecz2018unconditionally}
\begin{equation}
	\tag{SI-RK3}
	\label{eq:SI-RK3}
	\begin{aligned}
		y^1 &= u^n,
		\\
		y^2_i &= \frac{ u^n_i + \dt \rest_i(y^1) + \dt \sum_j \prod_{ij}(y^1) }
		{ 1 + \dt \sum_j \dest_{ij}(y^1) / y^1_i },
		\\
		y^3_i &= \frac{3}{4} u^n_i + \frac{1}{4}
		\frac{ y^2_i + \dt \rest_i(y^2) + \dt \sum_j \prod_{ij}(y^2) }
		{ 1 + \dt \sum_j \dest_{ij}(y^2) / y^2_i },
		\\
		y^4_i &= \frac{1}{3} u^n_i + \frac{2}{3}
		\frac{ y^3_i + \dt \rest_i(y^3) + \dt \sum_j \prod_{ij}(y^3) }
		{ 1 + \dt \sum_j \dest_{ij}(y^3) / y^3_i },
		\\
		u^{n+1}_i &= \frac{ y^4_i + \dt^2 \bigl(\rest_i(y^4) + \dt \sum_j \prod_{ij}(y^4) \bigr) \sum_j \dest_{ij}(y^4) / y^4_i }
		{ 1 + \bigl( \dt \sum_j \dest_{ij}(y^4) / y^4_i \bigr)^2 }.
	\end{aligned}
\end{equation}
\item (Patankar) Semi Implicit Runge--Kutta(2,2) methods \cite{chertock2015steady}
\begin{equation}
	\tag{MPRK(4,3,$\alpha, \beta$)}
	\label{eq:MPRK43-family}
	\begin{aligned}
		y^1 &= u^n,
		\\
		y^2_i &= u^n_i
		+ a_{21} \dt \rest_i\bigl( y^1 \bigr)
		+  a_{21} \dt \sum_j \left(
		\prod_{ij}\bigl( y^1 \bigr) \frac{y^2_j}{y^1_j}
		- \dest_{ij}\bigl( y^1 \bigr) \frac{y^2_i}{y^1_i}
		\right),
		\\
		y^3_i &= u^n_i+ \Delta t \left(a_{31} \rest_i\bigl( y^1 \bigr)+
		a_{32} \rest_i\bigl( y^2 \bigr)
		\right)
		\\ & \qquad
		+\Delta t \sum_j
		\Bigg(\left(a_{31} \prod_{ij}\bigl(y^1\bigr)+ a_{32} \prod_{ij} \bigl(y^2\bigr) \right)  \frac{  y_j^3
		}{\bigl(y_j^2\bigr)^{1/p } \bigl(y_j^1\bigr)^{1/p-1} } \\
		& \qquad\qquad\qquad
		-\left(a_{31} \dest_{ij}\bigl(y^1\bigr)+ a_{32} \dest_{ij} \bigl(y^2\bigr) \right)  \frac{  y_i^3
		}{\bigl(y_i^2\bigr)^{1/p } \bigl(y_i^1\bigr)^{1/p-1} }
		\Bigg),
		\\
		\sigma_i &= u_i^n + \Delta t \sum_j
		\left( \beta_1 \prod_{ij} \bigl( y^1\bigr) +\beta_2 \prod_{ij} \bigl(y^2\bigr)  \right) \frac{\sigma_j}{\bigl(y_j^2 \bigr)^{1/q}
			\bigl(y_j^1\bigr)^{1/q-1}}
		\\ & \qquad
		\left( \beta_1 \dest_{ij} \bigl( y^1\bigr) +\beta_2 \dest_{ij} \bigl(y^2\bigr)  \right) \frac{\sigma_i}{\bigl(y_i^2 \bigr)^{1/q}
			\bigl(y_i^1\bigr)^{1/q-1}}
		\\
		u^{n+1}_i &= u^n_i
		+ \dt \left( b_1 \rest_i\bigl( y^1 \bigr) + b_2\rest_i\bigl( y^2 \bigr)  + b_3\rest_i\bigl( y^3 \bigr)\right)
		\\&\qquad
		+ \dt \sum_j \Biggl(
		\left( b_1 \prod_{ij}\bigl( y^1 \bigr) +b_2\prod_{ij}\bigl( y^2 \bigr)
		+ b_3 \prod_{ij}\bigl( y^3 \bigr)
		\right) \frac{u^{n+1}_j}{\sigma_j}
		\\&\qquad\qquad\qquad
		-\left( b_1 \dest_{ij}\bigl( y^1 \bigr) +b_2\dest_{ij}\bigl( y^2 \bigr)
		+ b_3 \dest_{ij}\bigl( y^3 \bigr)
		\right) \frac{u^{n+1}_i}{\sigma_i}
		\Biggr),
	\end{aligned}
\end{equation}
where $p=3 a_{21}\left(a_{31}+a_{32} \right)b_3,\; q=a_{21},\;\beta_2=\frac{1}{2a_{21}}$ and $\beta_1= 1-\beta_2$. The Butcher tableaus in respect to the two parameters
\begin{equation}
	\begin{aligned}
		\def\arraystretch{1.2}
		\begin{array}{c|ccc}
			0 &  & & \\
			\alpha & \alpha & & \\
			\beta & \frac{3\alpha\beta (1-\alpha)-\beta^2}{\alpha(2-3\alpha)}& \frac{\beta (\beta-\alpha)}{\alpha(2-3\alpha)}& \\
			\hline
			& 1+\frac{2-3(\alpha+\beta)}{6 \alpha \beta } &\frac{3 \beta-2}{6\alpha (\beta-\alpha)} & \frac{2-3\alpha}{6\beta(\beta-\alpha)}
		\end{array}
	\end{aligned}
\end{equation}
with
\begin{equation*}
	\begin{rcases}
		2/3 \leq \beta \leq 3\alpha(1-\alpha)\\
		3\alpha(1-\alpha)\leq\beta \leq 2/3 \\
		(3\alpha-2)/(6\alpha-3)\leq \beta \leq 2/3
	\end{rcases}
	\text{ for }
	\begin{cases}
		1/3 \leq \alpha<\frac{2}{3},\\
		2/3 \leq \alpha<\alpha_0,\\
		\alpha>\alpha_0,
	\end{cases}
\end{equation*}
and $\alpha_0\approx 0.89255.$
\item Modified Patankar Shu--Osher Runge--Kutta(4,3) method \cite{huang2019third}
\begin{equation}
	\tag{MPRKSO(4,3)}
	\label{eq:MPRKSO(4,3)}
	\begin{aligned}
		y^1 &= u^n,
		\\
		y^2_i &= y^1_i
		+ a_{10} \dt \rest_i\bigl( y^1 \bigr)
		+  \dt \sum_j b_{10 }\left(
		\prod_{ij}\bigl( y^1 \bigr) \frac{y^2_j}{y_j^1}
		-\dest_{ij}\bigl( y^1 \bigr) \frac{y^2_i}{y_i^1}
		\right),
		\\
		\rho_i&= n_1y_i^2 +n_2 y_i^1 \left(\frac{y_i^2}{y_i^1} \right)^2
		\\
		y^3_i &= (a_{20}y^1_i +a_{21}y^2_i)
		+ \dt \left( b_{20} \rest_i  \bigl(  y^1 \bigr)+  b_{21} \rest_i  \bigl(  y^2 \bigr) \right)
		\\ \qquad
		&+  \dt \sum_j  \left(
		\left( b_{20} \prod_{ij}\bigl( y^1 \bigr)  +  b_{21} \prod_{ij}\bigl( y^2 \bigr)  \right) \frac{y^2_j}{\rho_j}
		-  \left( b_{20} \dest_{ij}\bigl( y^1 \bigr)  +  b_{21} \dest_{ij}\bigl( y^2 \bigr)  \right) \frac{y^2_i}{\rho_i}
		\right),\\
		\mu_i &=y_i^1\left( \frac{y_i^2}{y_i^1} \right)^s \\
		\tilde{a}_i&= \eta_1 y_i^1 + \eta_2 y_i^2 + \dt \sum_j \left( \left( \eta_3 \prod_{ij}\bigl( y^1 \bigr)  +\eta_4 \prod_{ij} \bigl( y^2 \bigr)  \right)  \frac{\tilde{a}_j}{\mu_j} -
		\left( \eta_3 \dest_{ij}\bigl( y^1 \bigr)  +\eta_4 \dest_{ij} \bigl( y^2 \bigr)  \right)  \frac{\tilde{a}_i}{\mu_i}  \right) \\
		\sigma_i &=\tilde{a}_i+zy_i^1 \frac{y_i^2}{\rho_i}
		\\
		u^{n+1}_i &=  \left(a_{30} y^1_i + a_{31}y^2_i +a_{32} y^3_i \right)
		+  \dt \left( b_{30}  \rest_i\bigl( y^1 \bigr) + b_{31} \rest_i \bigl( y^2 \bigr)+ b_{32} \rest_i \bigl( y^3 \bigr) \right)
		\\&\qquad
		+ \dt \sum_j \Biggl(
		\left(  b_{30}    \prod_{ij}\bigl( y^1 \bigr) +   b_{31}  \prod_{ij}\bigl( y^2 \bigr)
		+   b_{32}  \prod_{ij}\bigl( y^2 \bigr)\right) \frac{u^{n+1}_j}{\sigma_j}
		\\&\qquad\qquad\qquad
		- \left(  b_{30}    \dest_{ij}\bigl( y^1 \bigr) +   b_{31}  \dest_{ij}\bigl( y^2 \bigr)
		+   b_{32}  \dest_{ij}\bigl( y^2 \bigr)\right) \frac{u^{n+1}_i}{\sigma_i}
		\Biggr).
	\end{aligned}
\end{equation}
Here, the optimal SSP coefficients determined in \cite{huang2019third} will be used. They are given by
\begin{align*}
	n_1&=2.569046025732011E-01,& n_2 &= 7.430953974267989E-01,\\
	a_{10} &=1, & a_{20} &=  9.2600312554031827E-01, \\
	a_{21} &=  7.3996874459681783E-02, &a_{31} &=2.0662904223744017E-10, \\
	b_{10} &=  4.7620819268131703E-01, &a_{30} &= 7.0439040373427619E-01, \\
	a_{32}& = 2.9560959605909481E-01, &b_{20} &= 7.7545442722396801E-02, \\
	b_{21}& =5.9197500149679749E-01, & b_{31} & = 6.8214380786704851E-10, \\
	b_{30} & = 2.0044747790361456E-01, & b_{32} &= 5.9121918658514827E-01, \\
	\eta_1 &= 3.777285888379173E-02, & \eta_2 & = 1/3, \\
	\eta_3 & = 1.868649805549811E-01,  & \eta_3 & = 2.224876040351123, \\
	z &= 6.288938077828750E-01, & s &= 5.721964308755304.
\end{align*}
\end{itemize}
\section{Vanishing initial condition}

Here, we consider the linear initial value problem
\begin{equation}
	\label{eq:stability_system}
	u'(t)
	=
	f(u(t))
	=
	\frac{1}{2}\begin{pmatrix}
		-1 & 1 \\
		1 & -1
	\end{pmatrix}
	u(t),
	\quad
	u(0) = u^0=
	\begin{pmatrix}
		1 - \epsilon \\
		\epsilon
	\end{pmatrix}.
\end{equation}
To use the modified Patankar schemes with a generic implementation,
$\epsilon$ must be strictly positive to avoid division by zero. As recommended
in the literature \cite{kopecz2018order}, we set $\epsilon$ to the smallest
positive number that can be
represented as floating point number with given precision (usually 64 bit) whenever
we are interested in the limit $\epsilon \to 0$.
In the following we first study the behavior of some of the previously presented schemes for
$\epsilon\to 0 $, then we show where this study is meaningful and where it is less.

\subsection{Loss of accuracy for \texorpdfstring{\ref{eq:MPRK22-family}}{{MPRK(2,2,a)}}}
In order to explain the kind of computations used in the following, we start with
a simple example using again \ref{eq:MPRK22-family} with $\alpha=1$ to show how we study the accuracy of a method in the limit of $\varepsilon\to 0$.
Recall that  the system \eqref{eq:stability_system} conserves the sum of the constituents
$u_1(t) + u_2(t) = 1$ and can be formulated as
\begin{align*}
	u_1'(t)&=\frac{-u_1(t)+u_2(t)}{2}=\frac{1}{2}-u_1(t), \\
	u_2'(t)&= \; \frac{u_1(t)-u_2(t)}{2}=\frac{1}{2}-u_2(t),
\end{align*}
with steady state solution $u_1^*=u_2^*=\frac{1}{2}$ and exact solutions $$u_1(t)=\frac{1}{2}(1+e^{-t}(1-2\varepsilon)) \text{ and }u_2(t)=1-u_1(t).$$

\begin{example}[Accuracy of MPRK(2,2,1)]
	We investigate the behavior of MPRK(2,2,1) applied to \eqref{eq:stability_system}.
	Due to the conservation property and the symmetry of the problem, it suffices to
	focus on the first component $u_1$. For the first non-trivial stage, we obtain
	\begin{equation}
		y_1^2
		=
		u_1^0 + \frac{\dt}{2} \left( u_2^0 \frac{y_2^2}{u_2^0}- u_1^0 \frac{y_1^2}{u_1^0} \right)
		=
		u_1^0 + \frac{\dt}{2} \left( 1 - 2 y_1^2 \right)
		\quad\Longleftrightarrow\quad
		y_1^2
		=
		\frac{ u_1^0 + \frac{\dt}{2} }{ 1 + \dt }.
	\end{equation}
	Using this expression for $y_1^2$, the new numerical solution satisfies
	\begin{equation}
		\begin{aligned}
			u^{1}_1
			&=
			u^0_1 + \frac{\dt}{4} \Biggl(
			\left((1-u_1^0)+(1-y_1^2) \right) \frac{1-u_1^{1}}{1-y_1^2}
			- \left( u_1^0 +y_1^2\right)\frac{u_1^{1}}{y_1^2}
			\Biggr)
			\\
			&=
			u^0_1 + \frac{\dt}{4} \left(
			\frac{(1-u_1^0) (1-u_1^{1})}{1-y_1^2}
			+ 1-u_1^{1} -\frac{u_1^0u_1^{1}}{y_1^2} -u_1^{1}
			\right).
		\end{aligned}
	\end{equation}
	This can be reformulated as
	\begin{equation}
		\left(
		1 + \frac{\dt}{4} \frac{1-u_1^0}{1-y_1^2}
		+ \frac{\dt}{2} + \frac{\dt}{4} \frac{u_1^0}{y_1^2}
		\right) u_1^{1}
		=
		u_1^0 + \frac{\dt}{4}\frac{1-u_1^0}{1-y_1^2} +\frac{\dt}{4}.
	\end{equation}
	Passing to the limit $\varepsilon\to 0$ with $u_1^0=1$ and
	$\lim_{\varepsilon\to 0} y_1^2 = (1+ \dt / 2) / (1+\dt)$
	yields
	\begin{equation}
		\begin{aligned}
			& \left(1+\frac{\dt}{2} + \frac{\dt}{4} \frac{(1+\dt)}{1+\dt/2} \right)\lim_{\varepsilon\to 0} u_1^{1}
			= 1+\frac{\dt}{4}\\
			\Longleftrightarrow \qquad &\lim_{\varepsilon\to 0} u_1^{1}
			= \frac{8+6\dt+\dt^2}{8+10\dt+4\dt^2} = 1-\frac{\dt}{2}+\frac{\dt^2}{4}-\frac{\dt^3}{16}+\mathcal{O}(\dt^4).
		\end{aligned}
	\end{equation}
	This tells us that the solution is consistent with the exact one and that, as expected, the second order error is an  $\mathcal{O}(\dt^3)$ for the first time step, indeed
	\begin{equation}\label{eq:exact_sol}
		u_1(\dt) = 1-\frac{\dt}{2}+\frac{\dt^2}{4}-\frac{\dt^3}{12}+\mathcal{O}(\dt^4).
	\end{equation}
\end{example}

In general this is not true. We apply now for different $\alpha$ the same procedure on \ref{eq:MPRK22-family} for the symmetric problem \eqref{eq:stability_system}. We observe the following behaviors.
\begin{itemize}
	\item
	For $\alpha > 1$, taking the limit $\epsilon \to 0$ results in $u^1 = u^0$ and, by induction, $u^n \equiv u^0$. This can also be observed numerically if
	sufficiently high accuracy is used, e.g.\ \texttt{BigFloat} in Julia.
	For \texttt{Float64}, the first few steps do almost nothing and later steps
	result in the desired behavior. The number of steps necessary to actually do
	something increases for $\alpha \gg 1$.

	\item
	For $\alpha \in [1/2, 1)$, the schemes are consistent also for $\varepsilon \to 0$, but we lose one order of accuracy.
\end{itemize}
These are analyzed in detail in the following.

\begin{theorem}
	\label{thm:MPRK22-family-test-problem-steady-state}
	For the test problem \eqref{eq:stability_system} in the limit $\epsilon \to 0$ the initial state becomes a spurious steady state for
	\ref{eq:MPRK22-family} with $\alpha > 1$ .
\end{theorem}
\begin{proof}
	This proof makes use of explicit calculations using Mathematica
	\cite{mathematica12}. All calculations can be found in the notebook
	\texttt{MPRK\_2\_2\_alpha.nb} in the accompanying reproducibility repository \cite{torlo2021stabilityGit}.

	For $\alpha > 1$ and $\epsilon > 0$, the first step of \eqref{eq:stability_system}
	can be computed explicitly. The second component after the first step is of
	the form $u^1_2 = h_1(\epsilon) / h_2(\epsilon)$, where
	$\lim_{\epsilon \to 0} h_1(\epsilon) = \lim_{\epsilon \to 0} h_2(\epsilon) = 0$.
	Defining $\tilde{h}_i(\varepsilon):=\frac{h_i(\varepsilon)}{\varepsilon}$ for $i=1,2$, we can rewrite $u^1_2=\tilde{h}_1(\varepsilon)/\tilde{h}_2(\varepsilon)$, where
	\begin{equation}
		\label{eq:h1p}
		\lim_{\epsilon \to 0} \tilde{h}_1(\epsilon)
		= 2^{\frac{1}{\alpha}} \dt (\dt-4-4\dt\alpha)\Bigl(\frac{\dt\alpha}{1+\dt \alpha}\Bigr)^{\frac{1}{\alpha}}
	\end{equation}
	and this hold for all $\alpha>\frac{1}{2}$, while the denominator is
	\begin{equation}
		\lim_{\epsilon \to 0} \tilde{h}_2(\epsilon)
		=
		\infty, \quad \alpha > 1,
	\end{equation}
	results in $\lim_{\epsilon \to 0} u^1_2 = 0 = \lim_{\epsilon \to 0} u^0_2$
	for $\alpha > 1$.
	Since the sum of all components of $u$ is conserved,
	$\lim_{\epsilon \to 0} u^0 = (1, 0)$ is a spurious steady state.
\end{proof}

We remark that numerically this is appreciable also because we reduce to first order when $\varepsilon \to 0$. Indeed, we will have that $\lim_{\epsilon \to 0} u_1^1 = 1$, which is brings in an error of $\mathcal{O}(\dt)$ for some initial steps. Hence, the global order reduces to 1.

\begin{theorem}
	\label{thm:MPRK22-family-test-problem-CFL}
	Consider the application of \ref{eq:MPRK22-family} with $\alpha \in [0.5, 1]$
	to the test problem \eqref{eq:stability_system}  in the limit $\epsilon \to 0$.
	The order of accuracy for $\alpha<1$ reduces to 1.
\end{theorem}
\begin{proof}
	This proof makes use of explicit calculations using Mathematica
	\cite{mathematica12}. All calculations can be found in the notebook
	\texttt{MPRK\_2\_2\_alpha.nb} in the accompanying reproducibility repository \cite{torlo2021stabilityGit}.

	As in the proof of Theorem~\ref{thm:MPRK22-family-test-problem-steady-state},
	we evaluate the limit $\epsilon \to 0$ of $u^1_2 = \tilde{h}_1(\epsilon) / \tilde{h}_2(\epsilon)$. The expression of $\lim_{\varepsilon\to 0}\tilde{h}_1(\varepsilon)$ is given in \eqref{eq:h1p}, while for $\tilde{h}_2$ we have a different expression.
	In case $\alpha<1$, we have
	\begin{equation}
		\lim_{\epsilon \to 0} \tilde{h}_2(\epsilon)=\Bigl(\frac{\dt\alpha}{1+\dt \alpha}\Bigr)^{\frac{1}{\alpha}} \Bigl( -8(1+\dt\alpha)(2+\dt \alpha)^{\frac{1}{\alpha}} - 2^{\frac{1}{\alpha}} \dt (4-\dt+4\alpha \dt) \Bigr),
	\end{equation}
 	while for $\alpha=1$ we have the extra term
 	$
 	-2\frac{2+\dt}{1+\dt}\dt^2.
 	$
	Using \eqref{eq:h1p} from before and Taylor expansion in $\dt$,
	\begin{equation}
		\lim_{\epsilon \to 0} u^1_2(\epsilon)
		=
		\frac{ \lim_{\epsilon \to 0}\tilde{h}_1(\epsilon) }{ \lim_{\epsilon \to 0} \tilde{h}_2(\epsilon) }
		=
		\begin{dcases}
			1-\frac{\dt}{2}+\frac{\dt^2}{4}-\frac{\dt^3}{16}+\mathcal{O}(\dt^4),
			& \alpha = 1,
			\\
			1-\frac{\dt}{2}+\frac{\dt^2}{8}+\frac{\alpha\dt^3}{16}+\mathcal{O}(\dt^4),
			& \alpha \in [0.5, 1).
		\end{dcases}
	\end{equation}
	Hence, the term in $\dt^2$ for $\alpha<1$ does not coincide with the Taylor expansion of the exact solution \eqref{eq:exact_sol}, resulting in a first order accurate scheme for $\varepsilon\to 0$.
	Note the discontinuity at $\alpha = 1$ of $\lim_{\epsilon \to 0} u^1_2(\epsilon)$.
\end{proof}


\begin{remark}
	\label{rem:MPRK22-family-test-problem}
	This result does not demonstrate that there is an error in the proofs of the
	order of accuracy of \ref{eq:MPRK22-family} \cite{kopecz2018order}. Indeed,
	studies of the order of accuracy focus on fixed $\epsilon > 0$ and the limit
	$\dt \to 0$. Numerical experiments with $\alpha >1$ suggest that the numerical solutions stays
	approximately constant for a certain number of steps determined by $\epsilon$
	(and with less sensitivity also by $\dt$) until small changes have accumulated
	and the exponential decay of $u_1$ becomes visible. In particular, the limits
	$\dt \to 0$ and $\epsilon \to 0$ are not interchangeable.
\end{remark}

Expanding Remark~\ref{rem:MPRK22-family-test-problem}, a careful error analysis
can be conduced by constructing Taylor expansions of the error after the first
step for $\dt \to 0$ and $\epsilon \to 0$ in both possible orders.
Expanding at first around $\dt = 0$ shows that the leading order errors contain
terms proportional to $\epsilon^{-1}$ for $\alpha \neq 1$, both for $\alpha < 1$
and for $\alpha > 1$. However, these leading order terms are in agreement with the
analysis of \cite{kopecz2018order}, i.e., they are proportional to $\dt^3$.

More insights can be gained by studying the expansions first for $\epsilon \to 0$,
expanding around $\dt = 0$ afterwards. Then, the leading order terms in $\epsilon$
are $\O(\dt^3)$ for $\alpha = 1$, $\O(\dt^2)$ for $\alpha = 0.5$, and $\O(\dt)$
for $\alpha = 2$, see \texttt{MPRK\_2\_2\_alpha.nb} in \cite{torlo2021stabilityGit}. This can also be observed in numerical experiments using
\texttt{BigFloat} in Julia \cite{bezanson2017julia} as shown in
Figure~\ref{fig:MPRK22-family-test-problem}.

\begin{figure}
	\centering
	\includegraphics[width=0.5\textwidth]{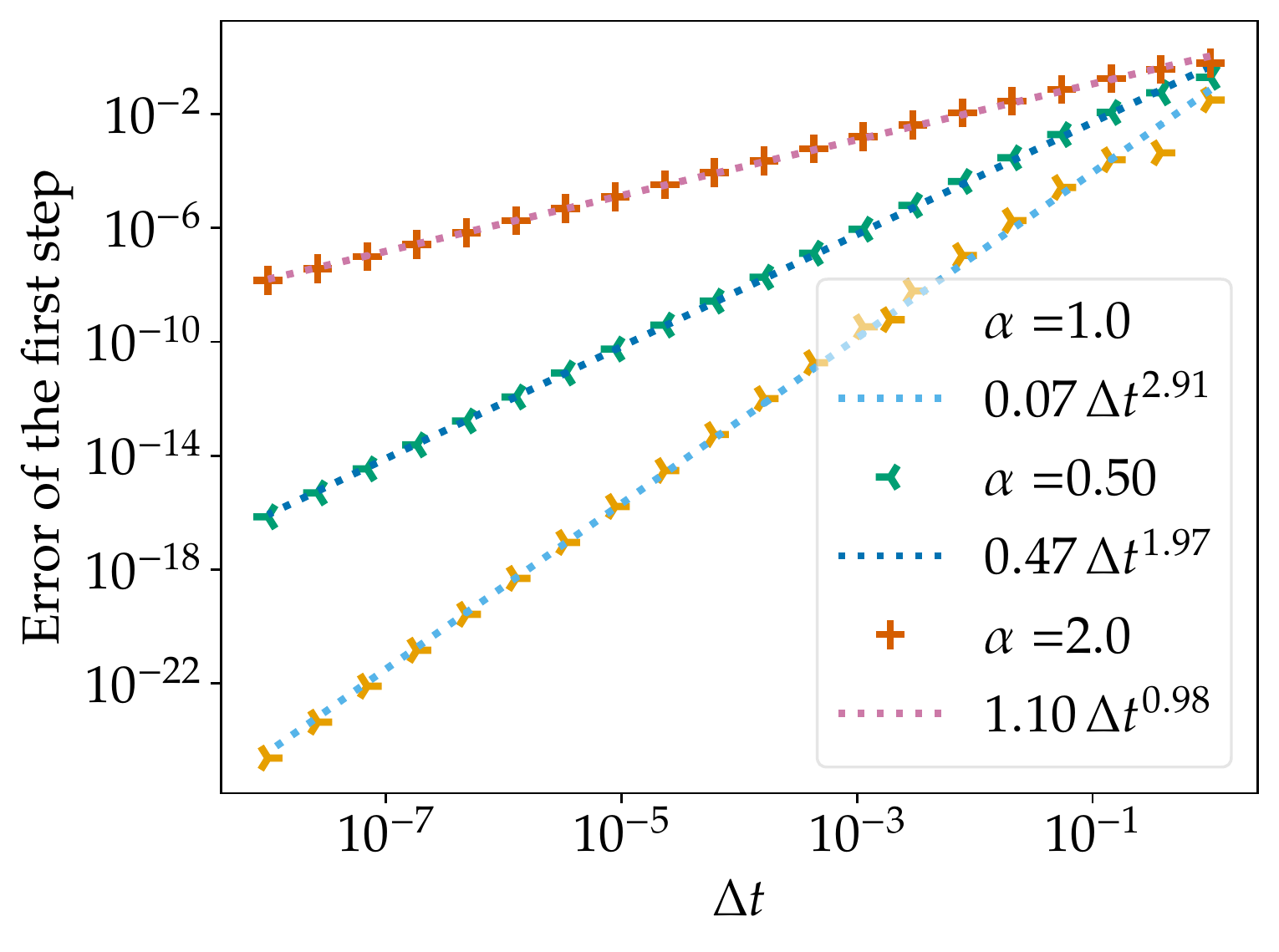}
	\caption{Convergence study of the error of the first step for different
		members of the family \ref{eq:MPRK22-family} for the test problem
		\eqref{eq:stability_system} with
		$\epsilon =$ \texttt{eps(BigFloat)}.}
	\label{fig:MPRK22-family-test-problem}
\end{figure}

\subsection{Accuracy study for other MP methods}\label{sec:orderReduction}
A similar analysis can be conducted for the \ref{eq:explicit_dec_correction} algorithm.
However, the approach we used with Mathematica was only able to give results up to third order schemes. We study the multivariate function $\eta(\varepsilon,\Delta t):=u^1_1-u_1(\Delta t)$, where the initial conditions depend on $\varepsilon$ and the time step is $\Delta t$, for different limits procedure. Letting $\dt$ go to zero faster than $\varepsilon$  and \textit{vice versa}. In practice, we compute a Taylor expansion first in $\dt $ and than in $\varepsilon$ and then the opposite in the Mathematica notebooks \texttt{mPDeC.nb} and \texttt{MPRK\_3\_2.nb} \cite{torlo2021stabilityGit}.

Since \ref{eq:explicit_dec_correction}2 is equivalent to MPRK(2,2,1), we get the
same results as before. In particular, expanding $\dt$ around 0 first and then $\varepsilon$ we obtain
\begin{align*}
	\eta(\varepsilon,\Delta t) =& \left(\frac{1}{12} - \frac{\varepsilon}{6} + \mathcal{O}(\varepsilon^3)\right)\Delta t^3+ \mathcal{O}(\Delta t^4),
\end{align*}
while, doing the opposite, we obtain
\begin{align*}
	\eta(\varepsilon,\Delta t) =& \left(\frac{\Delta t^3}{48} +  \mathcal{O}(\Delta t^4)\right)+ \left(\frac{\Delta t^2}{8} - \frac{11 \Delta t^3}{48} + \mathcal{O}(\Delta t^4)\right)\varepsilon + \\& \left(-\frac{\Delta t}{4}+\frac{\Delta t^2}{16}+\frac{3\Delta t^3}{32}+ \mathcal{O}(\Delta t^4)\right) \varepsilon^2+\mathcal{O}(\varepsilon^3).
\end{align*}
Note that $\mathcal{O}(\varepsilon)$ terms can be ignored when evaluating the order
of accuracy.
Hence, we see that in both cases we have an error of $\mathcal{O}(\Delta t^3)$
for the first step, i.e., a second-order accurate method.

For the third-order algorithm \ref{eq:explicit_dec_correction}3, we have a
different behavior and an order reduction for small $\varepsilon$: expanding first $\dt$ and then $\varepsilon$ in 0 we obtain
\begin{align*}
	\eta(\varepsilon,\Delta t) =& \left(
	-\frac{1}{13824\varepsilon^2}
	-\frac{5}{1152\varepsilon}+
	\frac{1789}{13824}
	-\frac{1697\varepsilon}{6912}+
	\frac{7\varepsilon^2}{1536} +\mathcal{O}(\varepsilon^3) \right)\Delta t^4\\
	&+\mathcal{O}(\Delta t^5),
\end{align*}
while, doing the opposite, we obtain
\begin{align*}
	\eta(\varepsilon,\Delta t) =& \left(-\frac{\Delta t^2}{6} +  \mathcal{O}(\Delta t^3)\right)+ \left(112\Delta t  + \mathcal{O}(\Delta t^2)\right)\varepsilon -74880 \varepsilon^2  \\
	&+\mathcal{O}(\Delta t \varepsilon^2) + \mathcal{O}(\varepsilon^3).
\end{align*}
If we let $\varepsilon\to 0$ before $\Delta t \to 0$, we have a reduction to first order accuracy.
The computations for these tests can be found in \texttt{MPDEC.nb} in the accompanying reproducibility repository \cite{torlo2021stabilityGit}.

For the second-order \ref{eq:MPRK32} proposed in this article, we observe a consistent
second order accuracy in the limit case $\varepsilon\to 0$, i.e., expanding first $\dt$ and then $\varepsilon$ in 0 we obtain
\begin{align*}
	\eta(\varepsilon,\Delta t) =& \left(\frac{1}{4} - \frac{\varepsilon}{2} + \mathcal{O}(\varepsilon^3)\right)\Delta t^3+ \mathcal{O}(\Delta t^4),
\end{align*}
while, doing the opposite, we obtain
\begin{align*}
\eta(\varepsilon,\Delta t) =& \left(\frac{\Delta t^3}{6} +  \mathcal{O}(\Delta t^4)\right)+ \left(\frac{\Delta t^2}{6} - \frac{71 \Delta t^3}{96} + \mathcal{O}(\Delta t^4)\right)\varepsilon + \\& \left(-\frac{\Delta t}{3}+\frac{35\Delta t^2}{96}+\frac{139\Delta t^3}{1152}+ \mathcal{O}(\Delta t^4)\right) \varepsilon^2+\mathcal{O}(\varepsilon^3).
\end{align*}
This results are computed Mathematica and the related notebook \texttt{MPRK\_3\_2.nb} is in the accompanying reproducibility repository \cite{torlo2021stabilityGit}.

\begin{remark}
	We have also analyzed MPRKSO$(2,2,\alpha,\beta)$ with selected parameters
	$\alpha$, $\beta$. We do not present these analyses here; in general, they all
	agree with the numerical studies presented in the following.
\end{remark}
\section{Oscillations on linear problems for varying systems}
In this section we briefly present the results obtained for nonparametric Patankar schemes. We plot the maximum $\dt$ for which no oscillations appear as a function of $\theta$ the parameter which governs the linear system. The study is performed checking all the possibile initial conditions. The results for \ref{eq:MPRKSO(4,3)} are depicted in Figure~\ref{fig:systemCFLMPSO3}, for \ref{eq:SI-RK2} in Figure~\ref{fig:systemOscillationsSIRK2}, for \ref{eq:SI-RK3} in Figure~\ref{fig:systemOscillationsSIRK3} and for \ref{eq:MPRK32} in Figure~\ref{fig:systemOscillationsRK32}.
\begin{figure}\centering
	\begin{subfigure}{0.48\textwidth}
		\includegraphics[width=\textwidth]{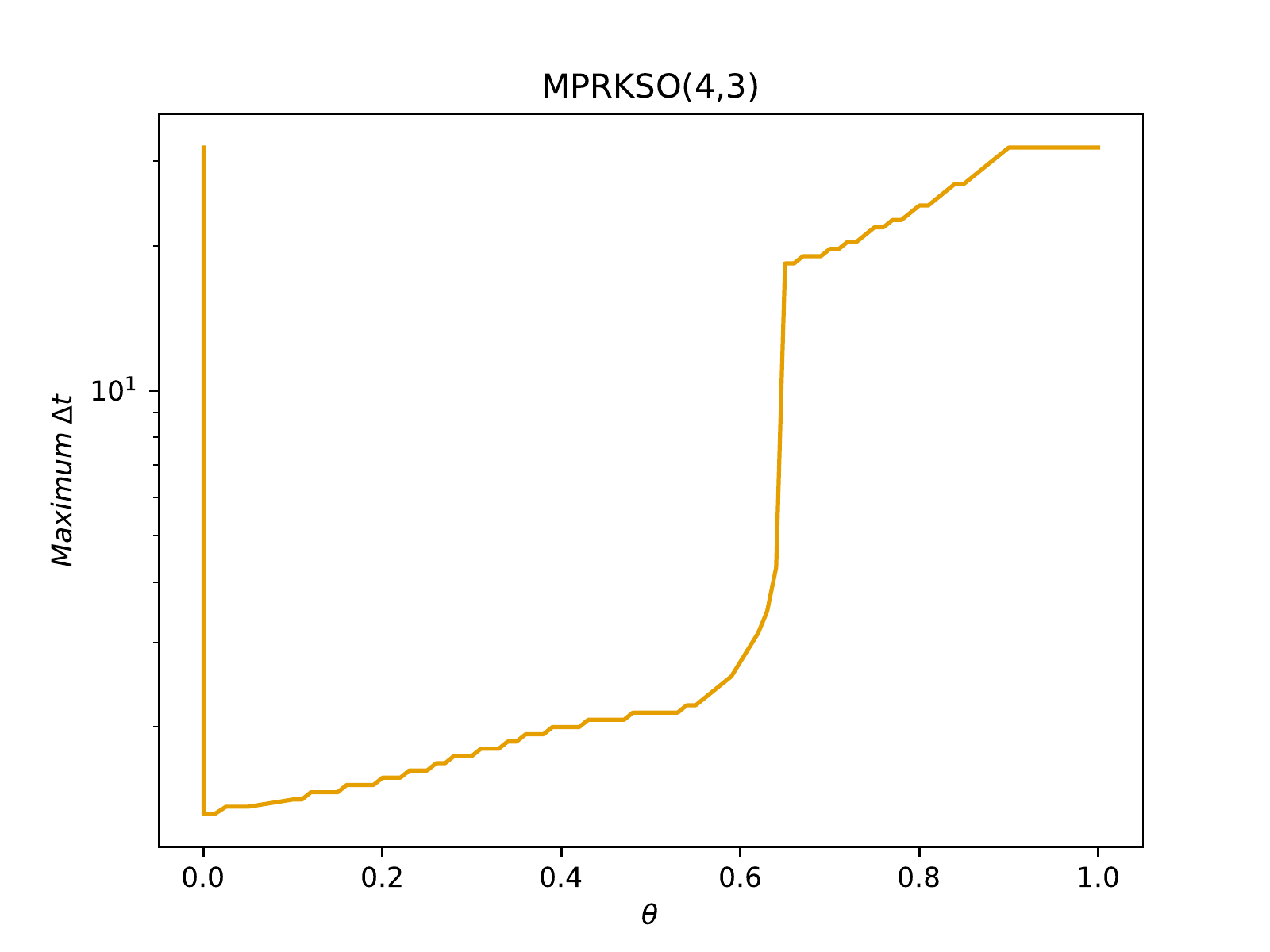}
	\caption{$\dt$ bound for \ref{eq:MPRKSO(4,3)} varying the system through $\theta$. Minimum $\dt$ is 1.31. \label{fig:systemCFLMPSO3}}
\end{subfigure}\,
	\begin{subfigure}{0.49\textwidth}
	\centering
	\includegraphics[width=\textwidth]{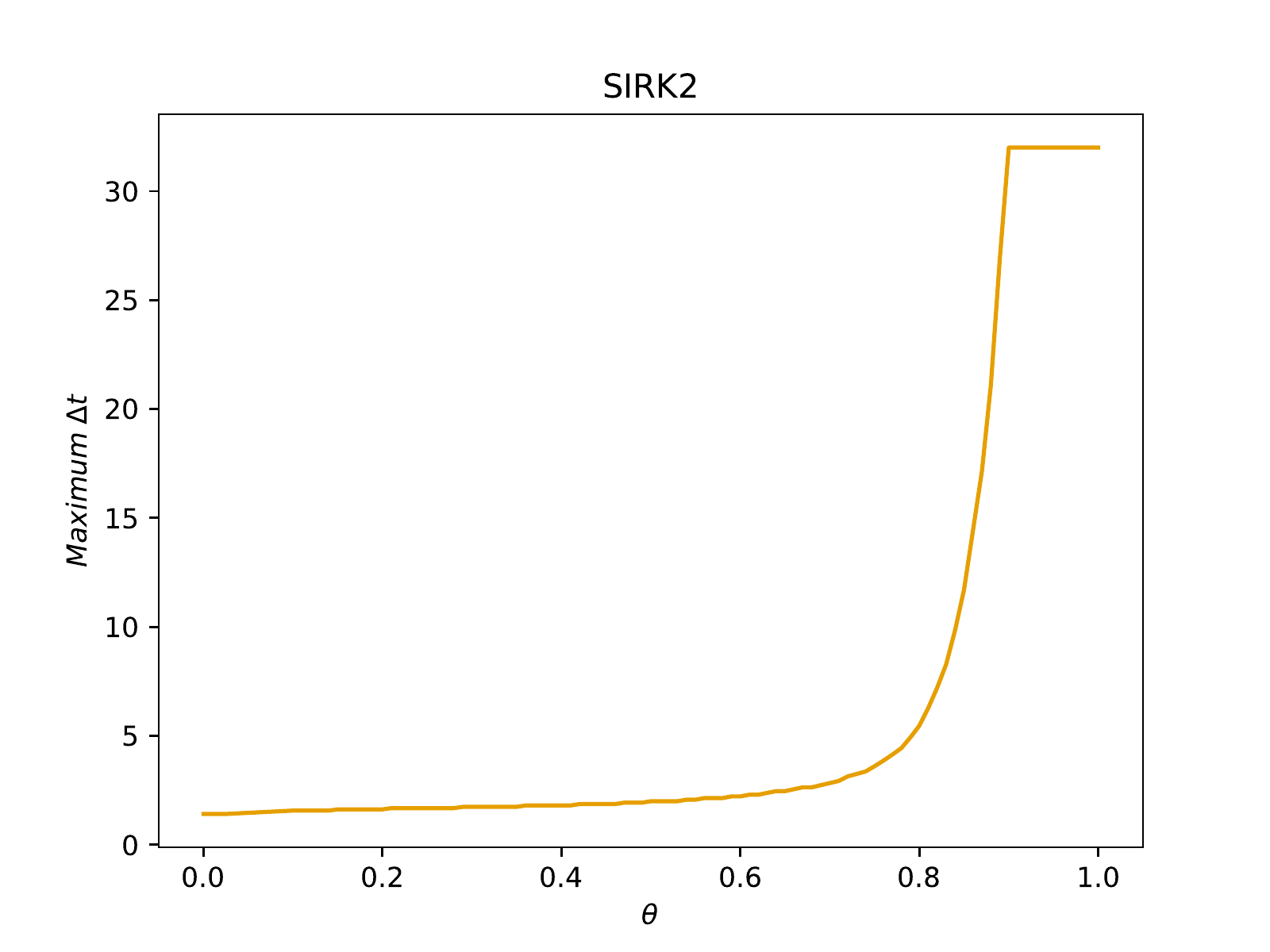}
	\caption{\ref{eq:SI-RK2}: $\dt$ bound varying $\theta$, minimum $\dt$ is  1.41.} \label{fig:systemOscillationsSIRK2}
\end{subfigure}\hfill
\begin{subfigure}{0.49\textwidth}
	\centering
	\includegraphics[width=\textwidth]{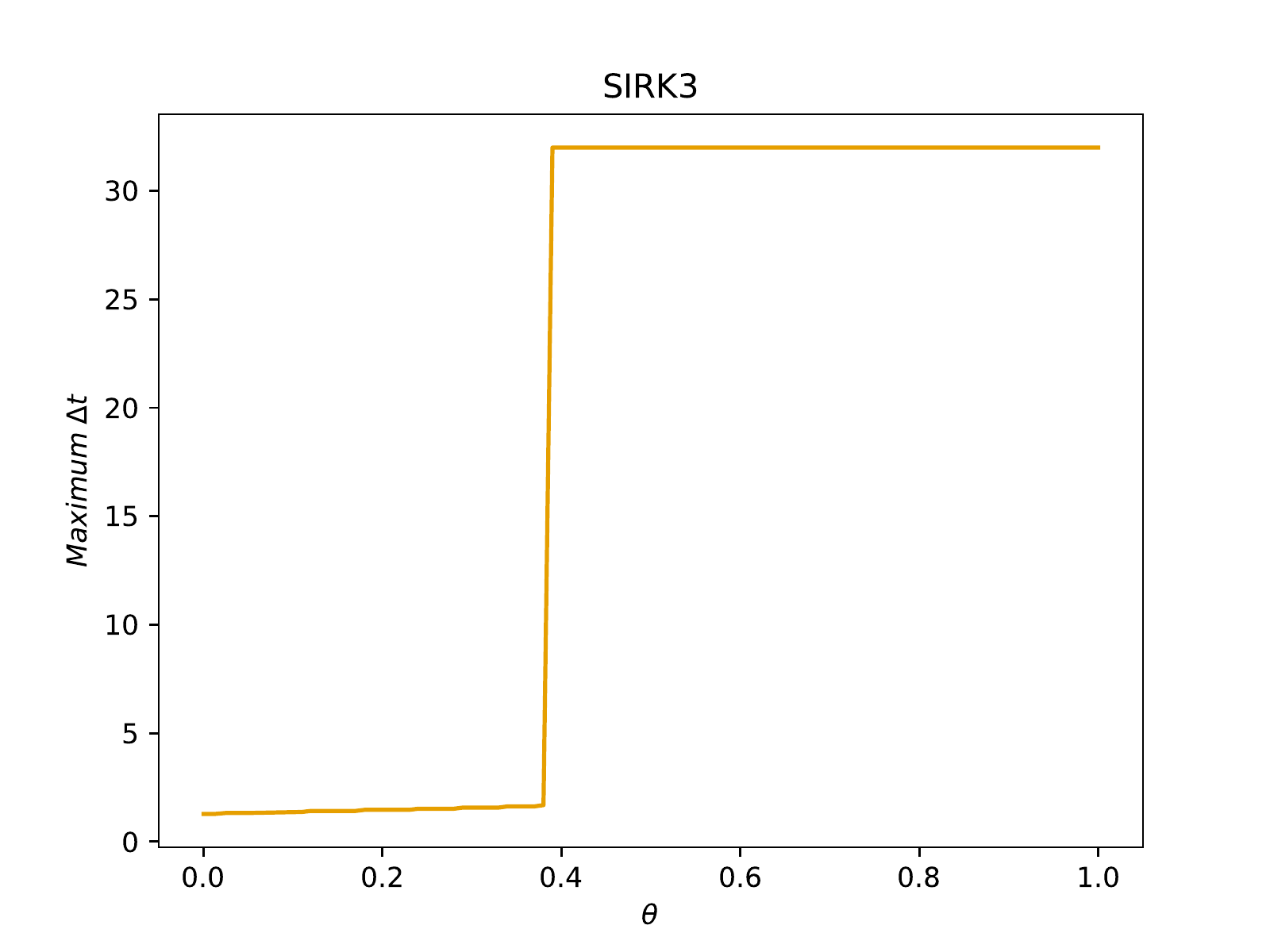}
	\caption{\ref{eq:SI-RK3}: $\dt$ bound varying $\theta$, minimum $\dt$ is 1.27. \label{fig:systemOscillationsSIRK3}}
\end{subfigure}\,
\begin{subfigure}{0.49\textwidth}
	\centering
	\includegraphics[width=\textwidth]{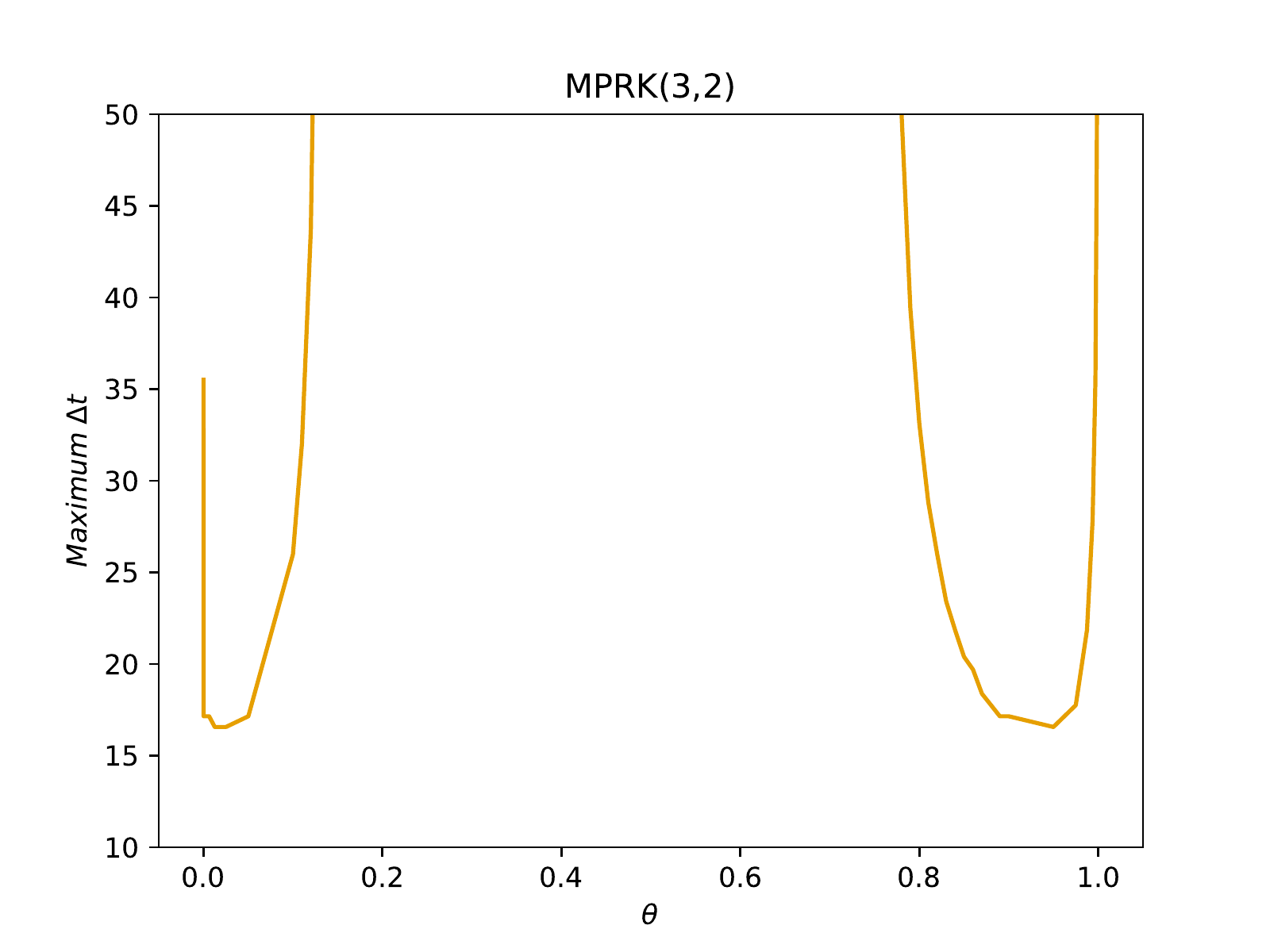}
	\caption{\ref{eq:MPRK32}: $\dt$ bound varying $\theta$. Minimum $\dt$ is 16.56. \label{fig:systemOscillationsRK32}}
\end{subfigure}\,
\end{figure}

\section{Validation on nonlinear problems}\label{sec:nonlinear-experiements}
\subsection{Scalar nonlinear problem}

The second problem on which we are testing our methods on is a scalar ODE with a
source term \cite{chertock2015steady}. Find $u:[0,0.15] \to \R$, with $u(0)=1.1 \sqrt{1/k}$, where $k>0$ is a coefficient of the problem, and
\begin{equation}\label{eq:ode-scalar}
u'=-k \vert u \vert u +1.
\end{equation}
The solution for this problem is monotonically decreasing and converging to $u_\infty = \sqrt{1/k}$.

The schemes can be applied to this problem following simple prescriptions.
\begin{itemize}
\item The source shall be integrated in time without considering the Patankar trick, simply using the coefficients of the original schemes.
\item The productions and destruction terms must be rewritten as $d_{11}=k \vert u \vert u$ and $p_{11}=0$.
\end{itemize}
We can see oscillations around the steady state produced by the schemes.

In this section, we want to validate the analysis done in the linear case, trying to understand if the $\dt$ bound we found in the previous section can be useful in the nonlinear case as well. Aiming at that, we check the first time step, which often shows overshoots with respect to the steady state, for different time steps.

In particular, we can observe that the Lipschitz constant of the right-hand side of \eqref{eq:ode-scalar} is $C(k):=\max_u k|u| = k|u_0|=1.1\sqrt{k}$. Hence, inspired by the theory for numerical PDEs, we use a CFL number in $\R^+$ through which we set the $\dt$ step as
\begin{equation}\label{eq:CFLrule}
	\dt := \frac{\text{CFL}}{C(k)}.
\end{equation}
In this way, we study the bound on $\dt$ setting a condition on th CFL number instead. Doing so, we essentially get rid of the dependence on $k$, through a rescaling factor both for time and amplitude on the solution. Hence, the CFL number should be comparable with the $\dt$ bound found in the previous sections.
We fix $k=10^4$ for the following simulations, but proportional results can be obtained for different $k$.

\begin{figure}
	\includegraphics[width=0.32\textwidth]{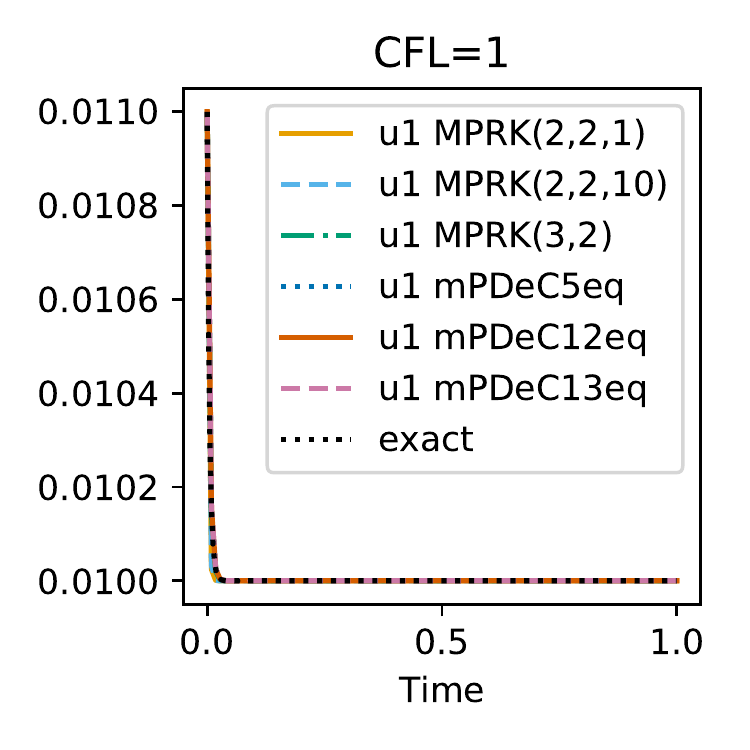}
	\includegraphics[width=0.32\textwidth]{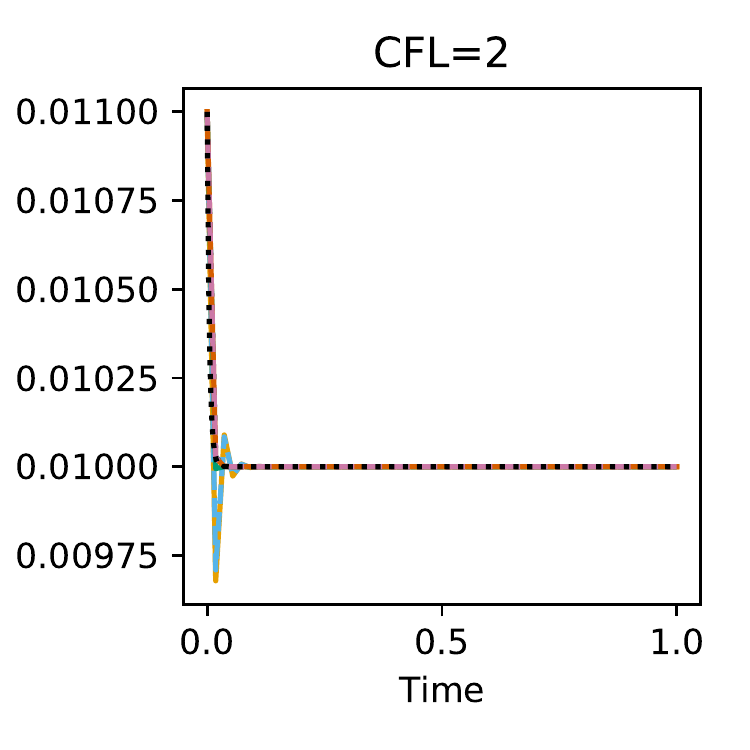}
	\includegraphics[width=0.32\textwidth]{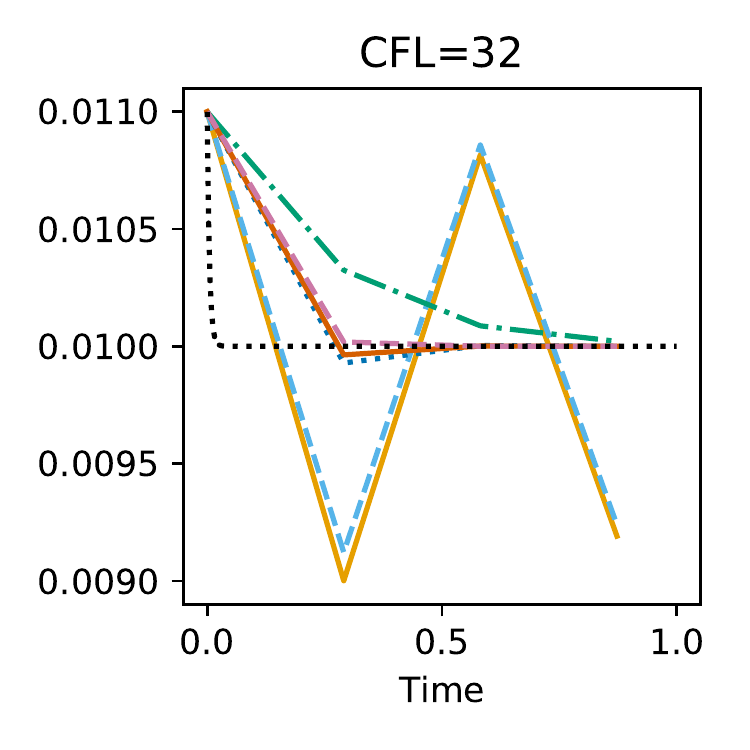}
	\caption{Simulations of \eqref{eq:ode-scalar} at different CFLs for some schemes.}\label{fig:simulationScalar}
\end{figure}

Figure~\ref{fig:simulationScalar} shows the simulations for different CFLs.
For low CFLs, we observe no oscillations for essentially all methods. Increasing
the CFL number, we observe that most of the schemes go below $u_\infty$ for the first timestep.

\begin{table}
	\small
	\centering
	\caption{Oscillation measure for problem \eqref{eq:ode-scalar} with mPDeC schemes with equispaced subtimesteps.}\label{tab:scalarMPDeCeq}
	\begin{tabular}{|l|l|l|l|l|l|l|l|l|}\hline
		\input{figures/scalar/k4/overMeasCFLs.tex}\hline
		\input{figures/scalar/k4/overMeasMPDeCeq.tex}\hline
	\end{tabular}
\end{table}

\begin{table}
	\small
	\centering
	\caption{Oscillation measure for problem \eqref{eq:ode-scalar} with mPDeC schemes with Gauss--Lobatto subtimesteps.}\label{tab:scalarMPDeCgl}
	\begin{tabular}{|l|l|l|l|l|l|l|l|l|}\hline
		\input{figures/scalar/k4/overMeasCFLs.tex}\hline
		\input{figures/scalar/k4/overMeasMPDeCgl.tex}\hline
	\end{tabular}
\end{table}

\begin{table}
	\small
	\centering
	\caption{Oscillation measure for problem \eqref{eq:ode-scalar} with \ref{eq:MPRK22-family} for few $\alpha$.}\label{tab:scalarMPRK22}
	\begin{tabular}{|l|l|l|l|l|l|l|l|l|}\hline
		\input{figures/scalar/k4/overMeasCFLs.tex}\hline
		\input{figures/scalar/k4/overMeasRK22.tex}\hline
	\end{tabular}
\end{table}

In Tables \ref{tab:scalarMPDeCeq} and \ref{tab:scalarMPDeCgl}, we list the oscillation measure for all \ref{eq:explicit_dec_correction} methods with equispaced and Gauss--Lobatto subtimesteps, respectively. Increasing the CFL, we see that many schemes overshoot the steady states. In particular, whenever we are below the $\dt$ bound found in the linear case, we do not observe oscillations. In some cases, also above this bound we do not have oscillations, but this might depend on the problem itself.

In Table~\ref{tab:scalarMPRK22}, we show similar results for \ref{eq:MPRK22-family} for some $\alpha$.  In contrast to the previous case, we observe oscillations even
if the bound is higher than the CFL tested.

In Table~\ref{tab:scalarMPRK43}, we test \ref{eq:MPRK43-family}, with some
interesting values and then on the curve $\beta(6\alpha-3)=3\alpha-2$. The first
values show oscillations according to the $\dt$ bound found in
the linear case, while, on the bottom curve, we observe no
oscillations starting from $\alpha=1$, which is slightly better then expected,
considering the (large but not so large) $\dt$ bounds of the linear case.

\begin{table}
	\small
	\centering
	\caption{Oscillation measure for problem \eqref{eq:ode-scalar} with \ref{eq:MPRK43-family} for some interesting $\alpha,\beta$. The second half of the table is on the curve $\beta(6\alpha-3)=3\alpha-2$ }\label{tab:scalarMPRK43}
	\begin{tabular}{|l|l|l|l|l|l|l|l|l|}\hline
		\input{figures/scalar/k4/overMeasCFLs.tex}\hline
		\input{figures/scalar/k4/overMeasRK43.tex}\hline\hline
		\input{figures/scalar/k4/overMeasRK43purple.tex}\hline
	\end{tabular}
\end{table}

\begin{table}
	\small
	\centering
	\caption{Oscillation measure for problem \eqref{eq:ode-scalar} with \ref{eq:MPRKSO22-family} for some interesting $\alpha,\beta$}\label{tab:scalarMPRKSO22}
	\begin{tabular}{|l|l|l|l|l|l|l|l|l|}\hline
		\input{figures/scalar/k4/overMeasCFLs.tex}\hline
		\input{figures/scalar/k4/overMeasRKSO22.tex}\hline
	\end{tabular}
\end{table}

\begin{table}
	\small
	\centering
	\caption{Oscillation measure for problem \eqref{eq:ode-scalar} with \ref{eq:MPRK32}, \ref{eq:MPRKSO(4,3)}, \ref{eq:SI-RK2}  and \ref{eq:SI-RK3}}\label{tab:scalarOthers}
	\begin{tabular}{|l|l|l|l|l|l|l|l|l|}\hline
		\input{figures/scalar/k4/overMeasCFLs.tex}\hline
		\input{figures/scalar/k4/overMeasRK32.tex}\hline
		\input{figures/scalar/k4/overMeasRKSO43.tex}\hline
		\input{figures/scalar/k4/overMeasSI2.tex}\hline
		\input{figures/scalar/k4/overMeasSI3.tex}\hline
	\end{tabular}
\end{table}

Another disappointing result comes from the schemes \ref{eq:MPRKSO22-family} in Figure~\ref{tab:scalarMPRKSO22}, where even on the line $\alpha = 0$ we do not have oscillation-free simulations with large $\dt$ as predicted by the study of the linear case. Conversely, for the other parameters we expected the oscillations for almost all CFL numbers larger than 1.

Finally, in Table~\ref{tab:scalarOthers}, we have different behaviors, except for \ref{eq:SI-RK2}. The oscillations appear for CFL neither too small nor too large. This is surprising, first of all for \ref{eq:MPRK32} of which we expected no oscillations up to CFL $\approx 16$, which shows anyway a very small oscillation (only very high order schemes have comparable oscillation amplitudes) only for CFL=2. For \ref{eq:MPRKSO(4,3)} and \ref{eq:SI-RK3} we have slightly better results than expected for large CFLs and for \ref{eq:SI-RK2} the results are exactly following the $\dt$ bounds found in the linear case.

\begin{conclusion}
For this test, most of the schemes behaves as predicted based on the linear example, with few exceptions for second-order methods. The bounds of the linear case can mostly be transferred to the considered nonlinear problem.
The linear analysis gives some meaningful results also for more challenging problems.
\end{conclusion}

\subsection{HIRES}

This problem is called HIRES after Hairer and Wanner \cite{hairer1999stiff},
referring to \glqq High Irradiance RESponse``.
The original problem HIRES \cite[Section~II.1]{mazzia2008test} can be rewritten into a nine-dimensional production--destruction system with
\begin{equation}
	\label{eq:HIRES-pdr}
	\begin{aligned}
		r_1(u) &= \sigma,
		&
		p_{21}(u) &= d_{12}(u) = k_1 u_1,
		&
		p_{12}(u) &= d_{21}(u) = k_2 u_2,
		\\
		p_{42}(u) &= d_{24}(u) = k_3 u_2,
		&
		p_{43}(u) &= d_{34}(u) = k_1 u_3,
		&
		p_{13}(u) &= d_{31}(u) = k_6 u_3,
		\\
		p_{34}(u) &= d_{43}(u) = k_2 u_4,
		&
		p_{64}(u) &= d_{46}(u) = k_4 u_4,
		&
		p_{65}(u) &= d_{56}(u) = k_1 u_5,
		\\
		p_{35}(u) &= d_{53}(u) = k_5 u_5,
		&
		p_{56}(u) &= d_{65}(u) = k_2 u_6,
		&
		p_{57}(u) &= d_{75}(u) = \frac{k_2}{2} u_7,
		\\
		p_{67}(u) &= d_{76}(u) = \frac{k_-}{2} u_7,
		&
		p_{97}(u) &= d_{79}(u) = \frac{k_*}{2} u_7,
		&
		p_{76}(u) &= d_{67}(u) = k_+ u_6 u_8,
		\\
		p_{78}(u) &= d_{87}(u) = k_+ u_6 u_8,
		&
		p_{87}(u) &= d_{78}(u) = \frac{k_- + k_*+k_2}{2} u_7.
	\end{aligned}
\end{equation}

with parameters
\begin{equation}
\begin{aligned}
  k_1 &= 1.71,
  &
  k_2 &= 0.43,
  &
  k_3 &= 8.32,
  &
  k_4 &= 0.69,
  &
  k_5 &= 0.035,
  \\
  k_6 &= 8.32,
  &
  k_+ &= 280,
  &
  k_- &= 0.69,
  &
  k_* &= 0.69,
  &
  \sigma &= 0.0007,
\end{aligned}
\end{equation}

The time interval is $t \in [0, 321.8122]$.

For this test the concept of oscillation is not clear, nevertheless, we can observe loss of accuracy also for this problem.
We compute the reference solution with 100,000 uniform timesteps. We use the \ref{eq:explicit_dec_correction}5 with equispaced subtimesteps to obtain this reference solution and we see that is in accordance with the reference solution \cite{mazzia2008test} up to the fourth significant digits for all constituents.

\begin{figure}
	\includegraphics[width=\textwidth]{figures/HIRES/MPRK221_logx.pdf}\\
	\includegraphics[width=\textwidth]{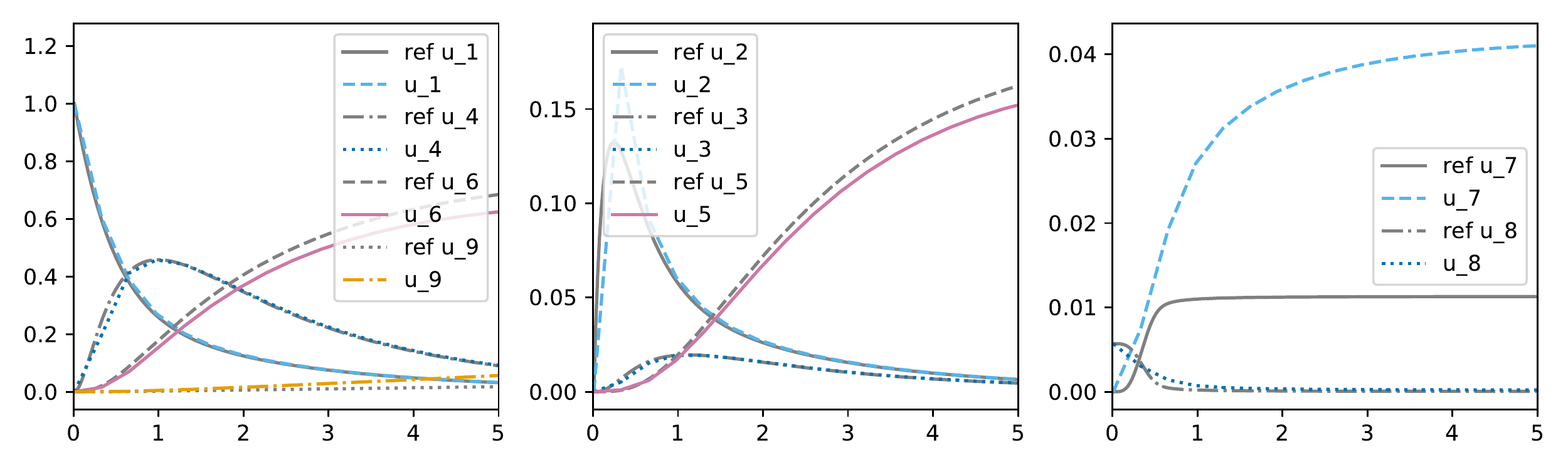}
	\caption{Simulations run with \ref{eq:MPRK22-family} with $\alpha=1$ with $N=1000$ timesteps, top logarithmic scale in time, bottom zoom on $t\in [0,5]$}\label{fig:HIRES_MPRK221}
\end{figure}

\begin{figure}
\includegraphics[width=\textwidth]{figures/HIRES/MPRK225_logx.pdf}\\
\includegraphics[width=\textwidth]{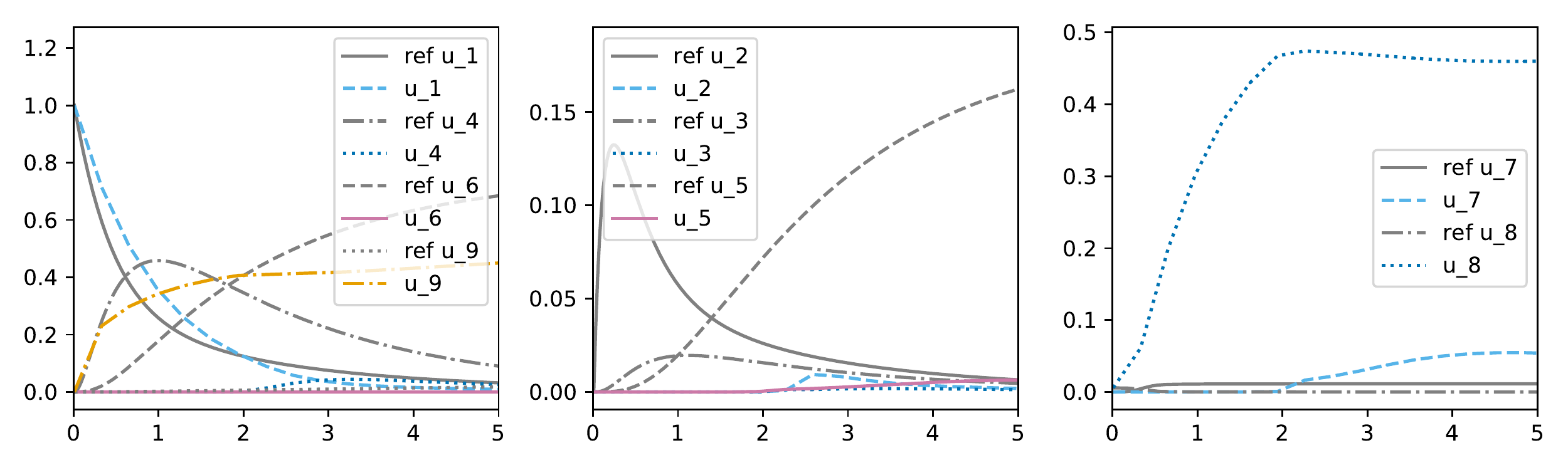}
\caption{Simulations run with \ref{eq:MPRK22-family} with $\alpha=5$ with $N=1000$ timesteps, top logarithmic scale in time, bottom zoom on $t\in [0,5]$}\label{fig:HIRES_MPRK225}
\end{figure}
\begin{figure}
	\includegraphics[width=\textwidth]{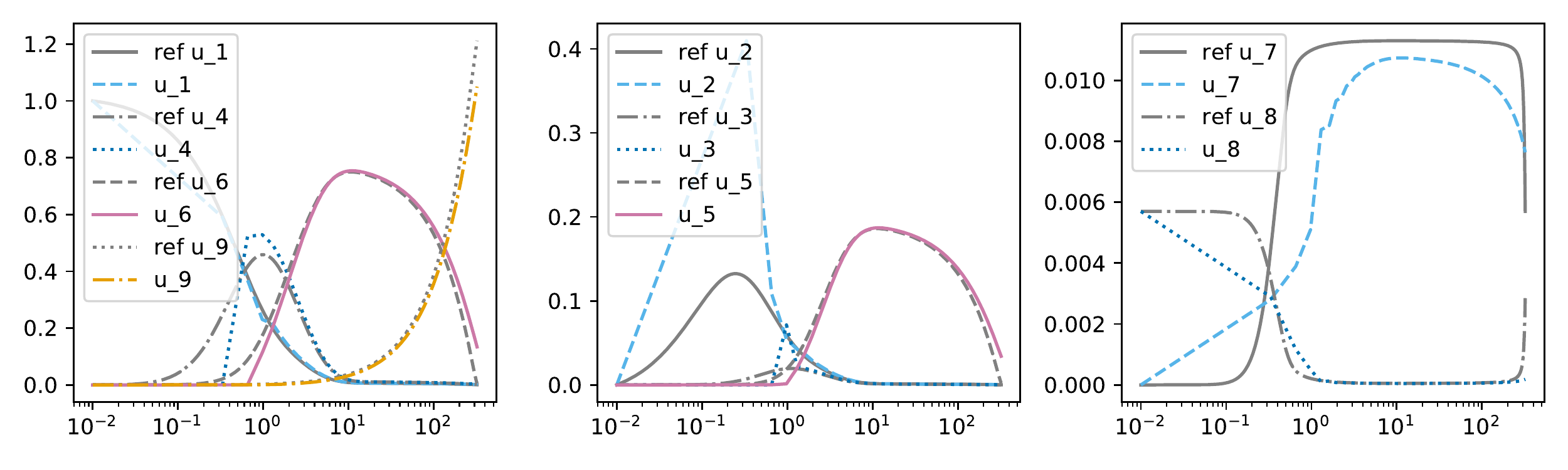}\\
	\includegraphics[width=\textwidth]{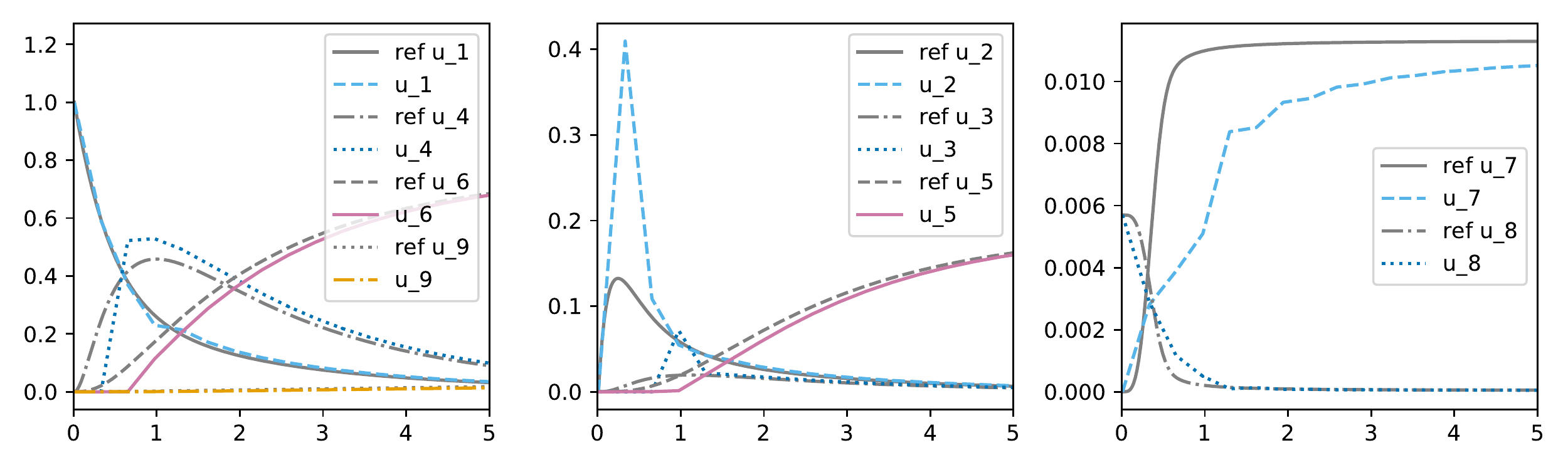}
	\caption{Simulations run with \ref{eq:MPRK22-family} with $\alpha=0.7$ with $N=1000$ timesteps, top logarithmic scale in time, bottom zoom on $t\in [0,5]$}\label{fig:HIRES_MPRK2207}
\end{figure}

\begin{figure}
	\includegraphics[width=\textwidth]{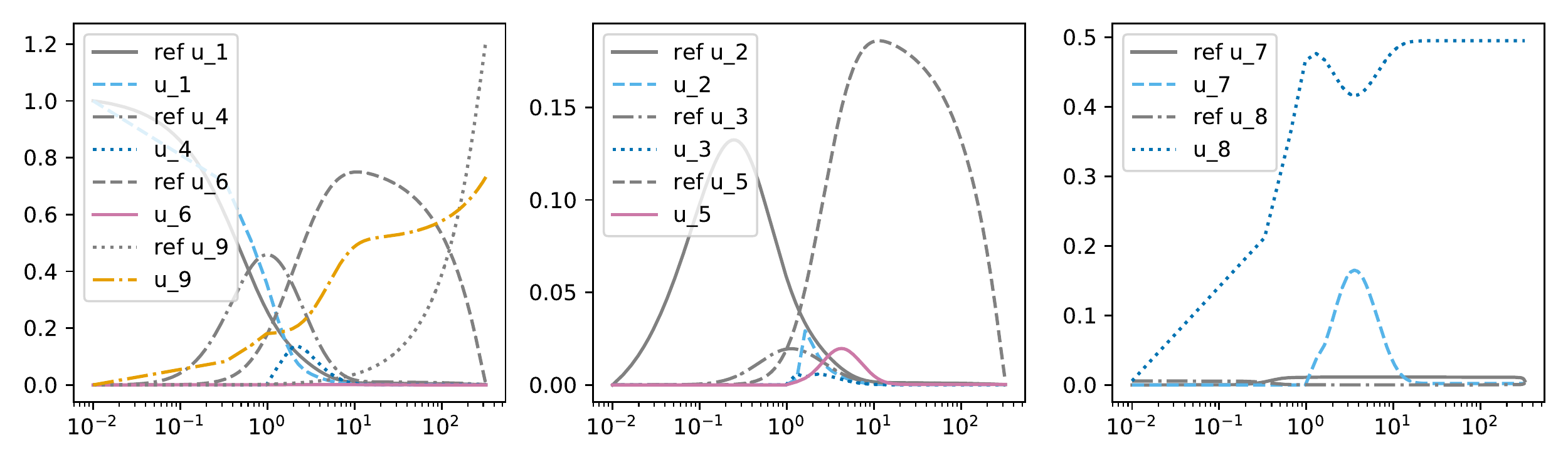}\\
	\includegraphics[width=\textwidth]{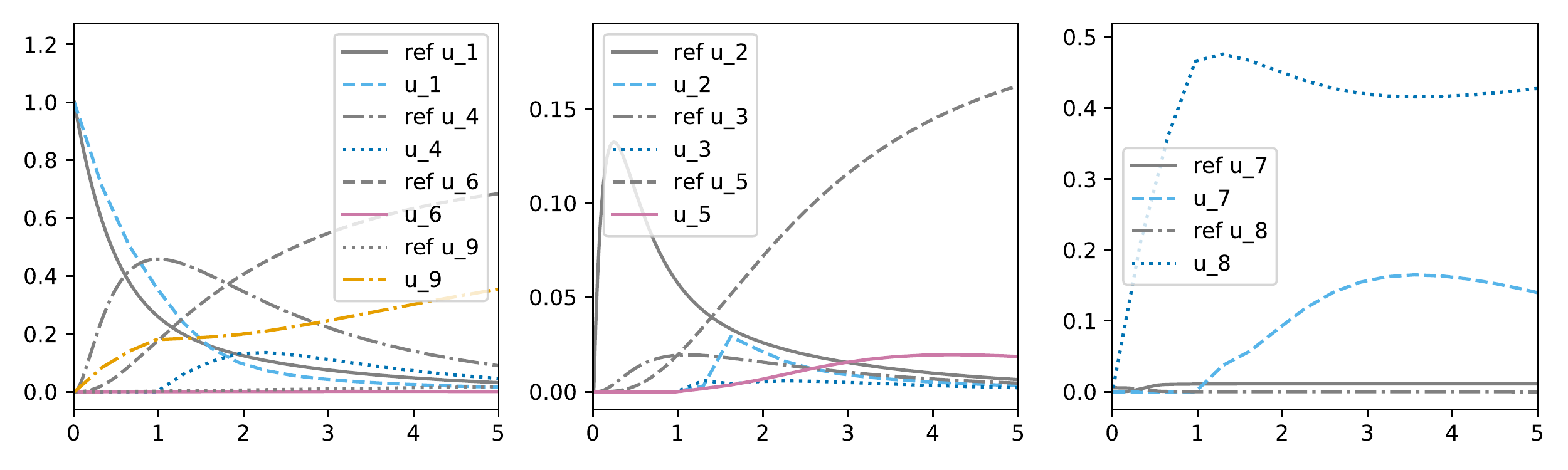}
	\caption{Simulations run with \ref{eq:MPRK43-family} with $\alpha=5$ and $\beta=0.5$ with $N=1000$ timesteps, top logarithmic scale in time, bottom zoom on $t\in [0,5]$}\label{fig:HIRES_MPRK43_5_05}
\end{figure}
\begin{figure}
	\includegraphics[width=\textwidth]{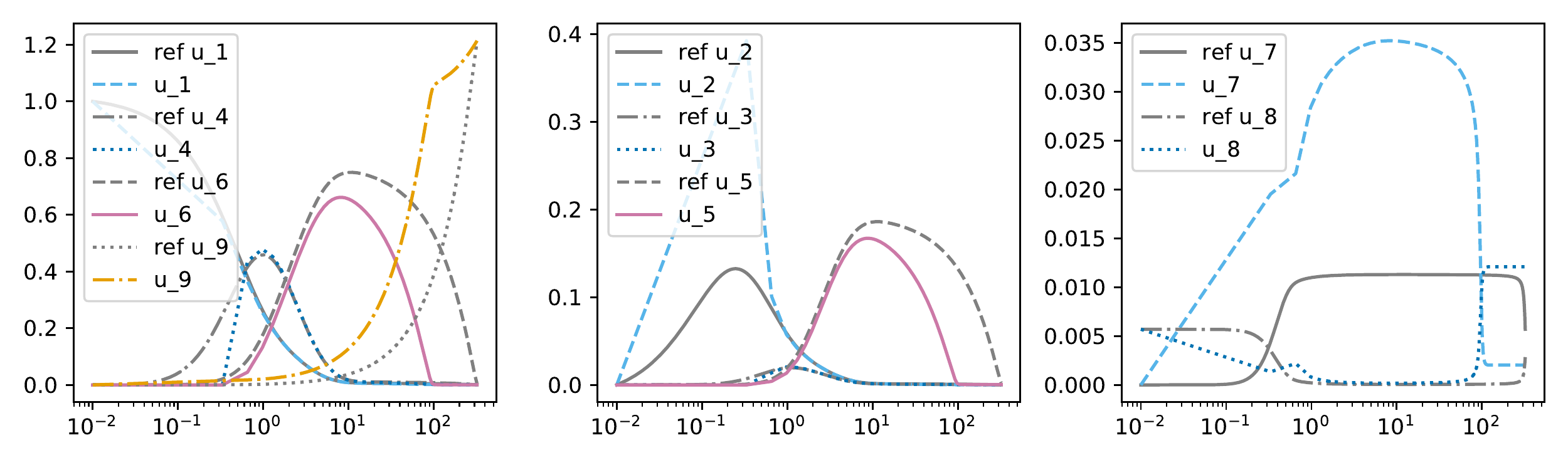}\\
	\includegraphics[width=\textwidth]{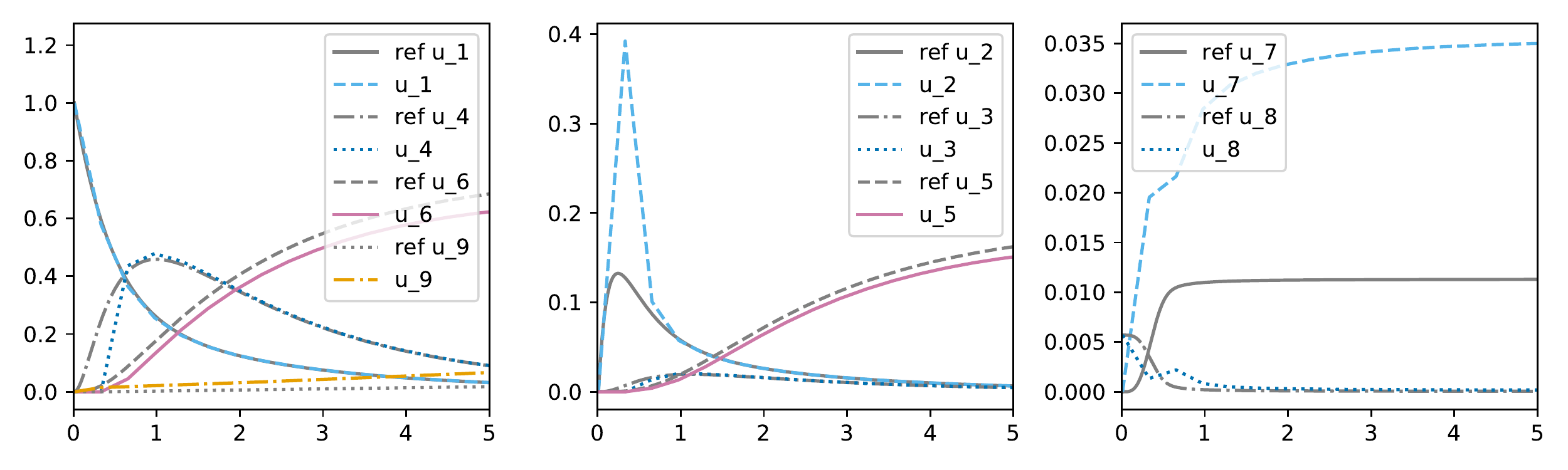}
	\caption{Simulations run with \ref{eq:MPRK43-family} with $\alpha=0.9$ and $\beta=0.6$ with $N=1000$ timesteps, top logarithmic scale in time, bottom zoom on $t\in [0,5]$}\label{fig:HIRES_MPRK43_09_06}
\end{figure}

\begin{figure}
	\includegraphics[width=\textwidth]{figures/HIRES/MPRKSO22_03_2_logx.pdf}\\
	\includegraphics[width=\textwidth]{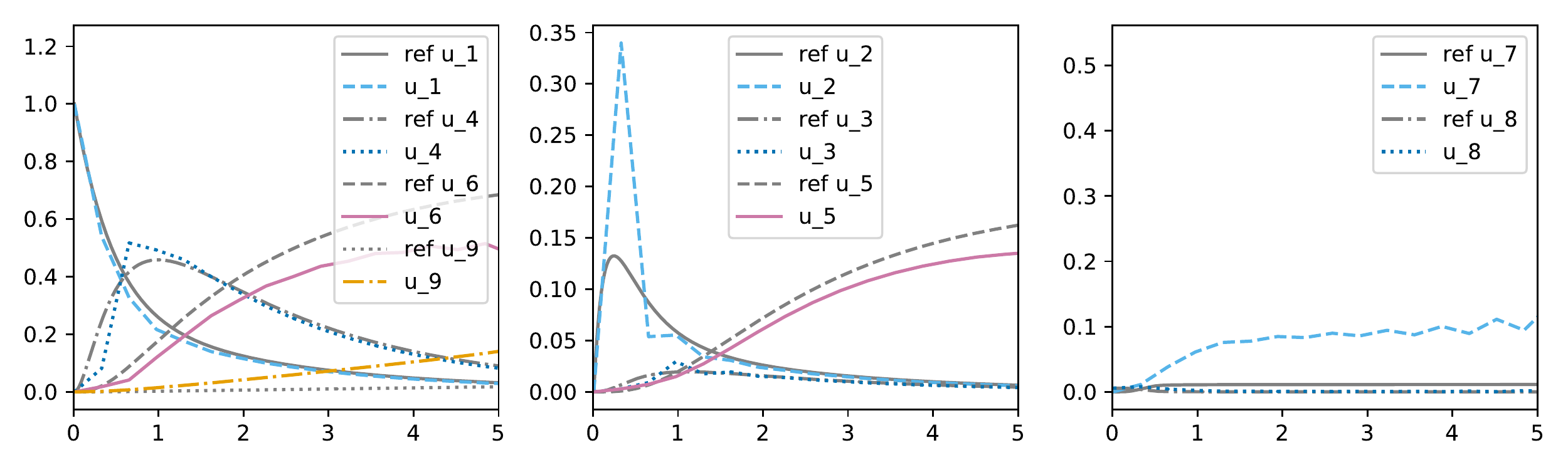}
	\caption{Simulations run with \ref{eq:MPRKSO22-family} with $\alpha=0.3$ and $\beta=2$ with $N=1000$ timesteps, top logarithmic scale in time, bottom zoom on $t\in [0,5]$}\label{fig:HIRES_MPRKSO22_03_2}
\end{figure}

\begin{figure}
	\includegraphics[width=\textwidth]{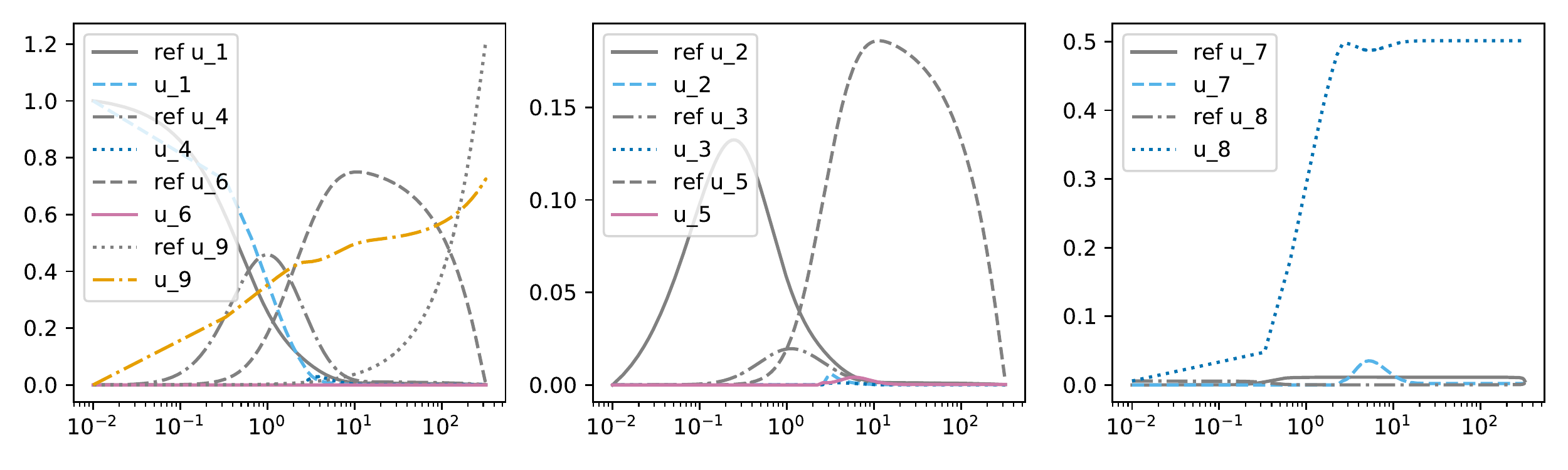}\\
	\includegraphics[width=\textwidth]{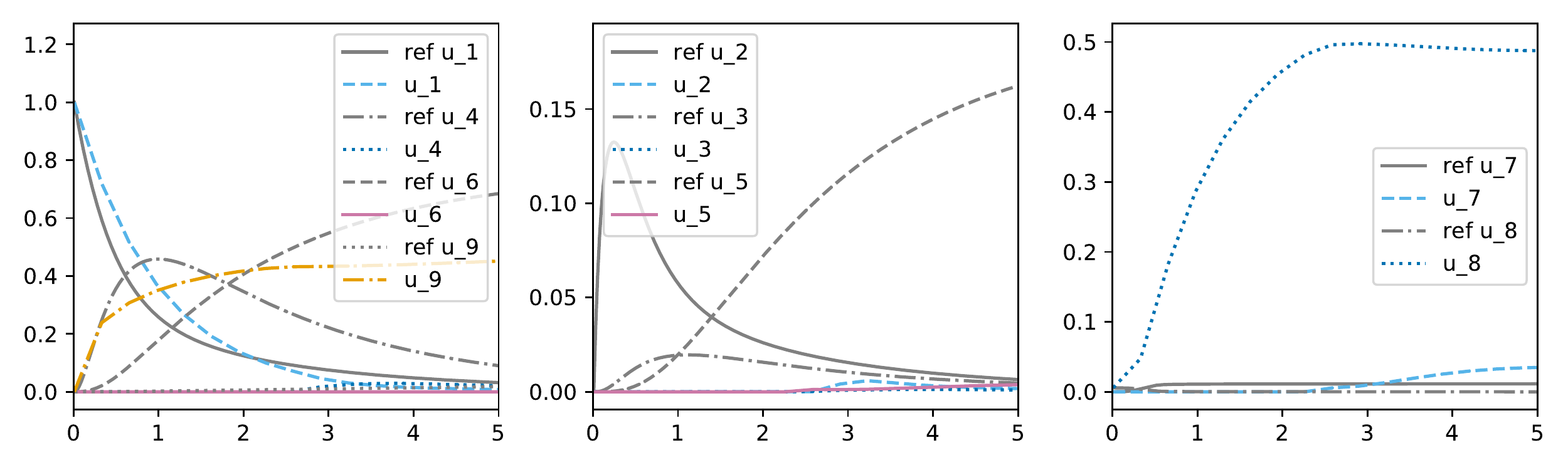}
	\caption{Simulations run with \ref{eq:MPRKSO22-family} with $\alpha=0$ and $\beta=8$ with $N=1000$ timesteps, top logarithmic scale in time, bottom zoom on $t\in [0,5]$}\label{fig:HIRES_MPRKSO22_0_8}
\end{figure}

\begin{figure}
	\includegraphics[width=\textwidth]{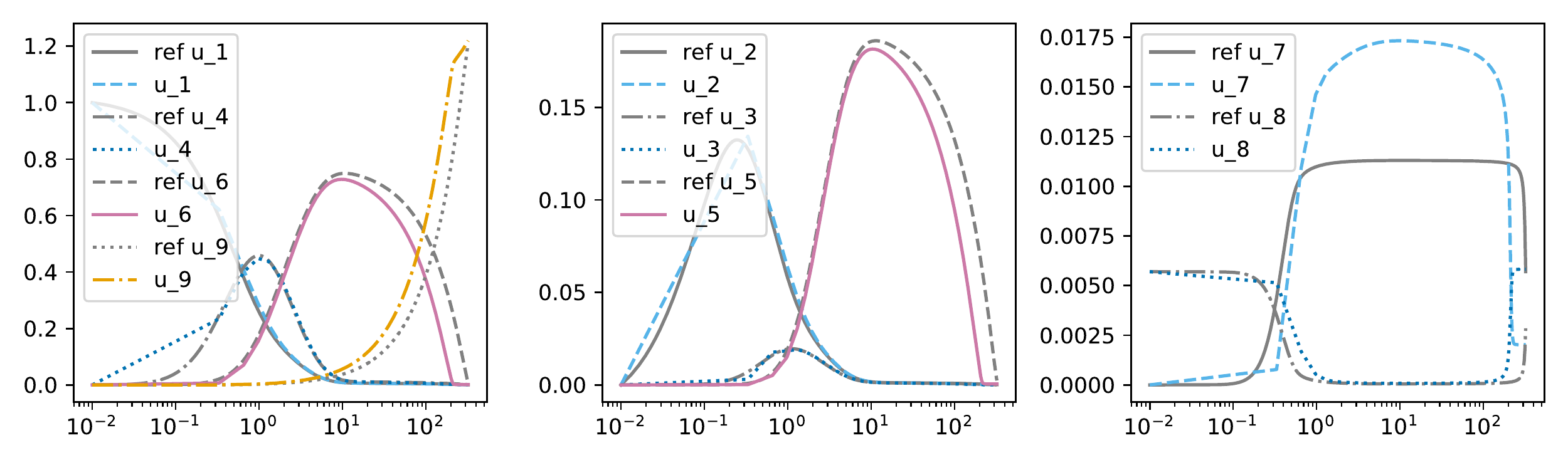}\\
	\includegraphics[width=\textwidth]{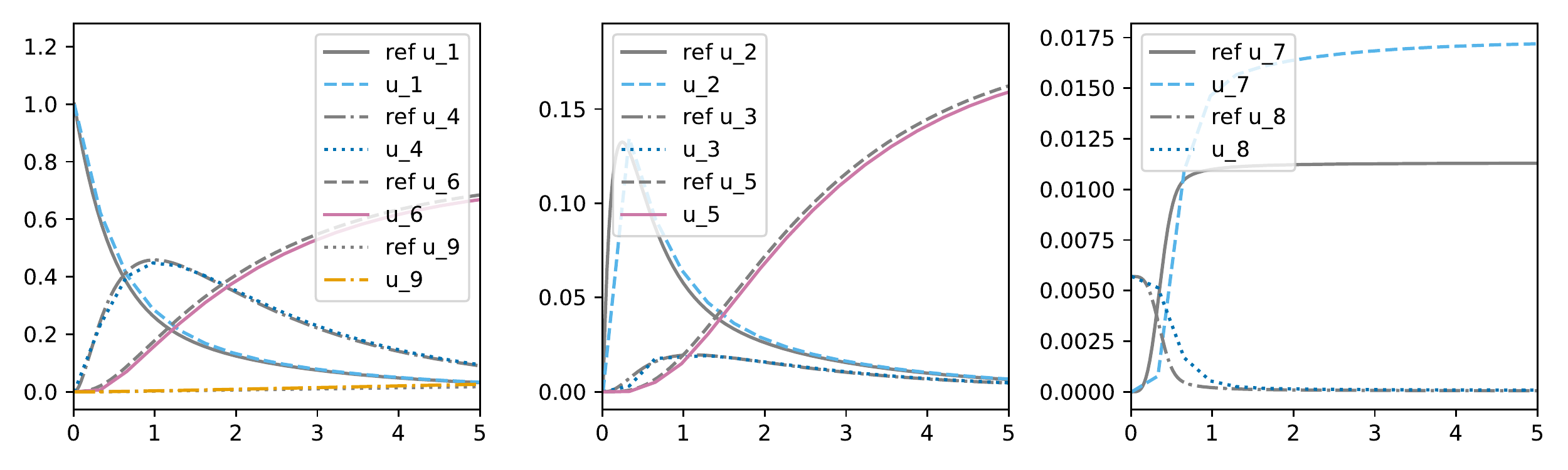}
	\caption{Simulations run with MPRKSO(4,3) with $\alpha=0$ and $\beta=8$ with $N=1000$ timesteps, top logarithmic scale in time, bottom zoom on $t\in [0,5]$}\label{fig:HIRES_MPRKSO43}
\end{figure}

\begin{figure}
	\includegraphics[width=\textwidth]{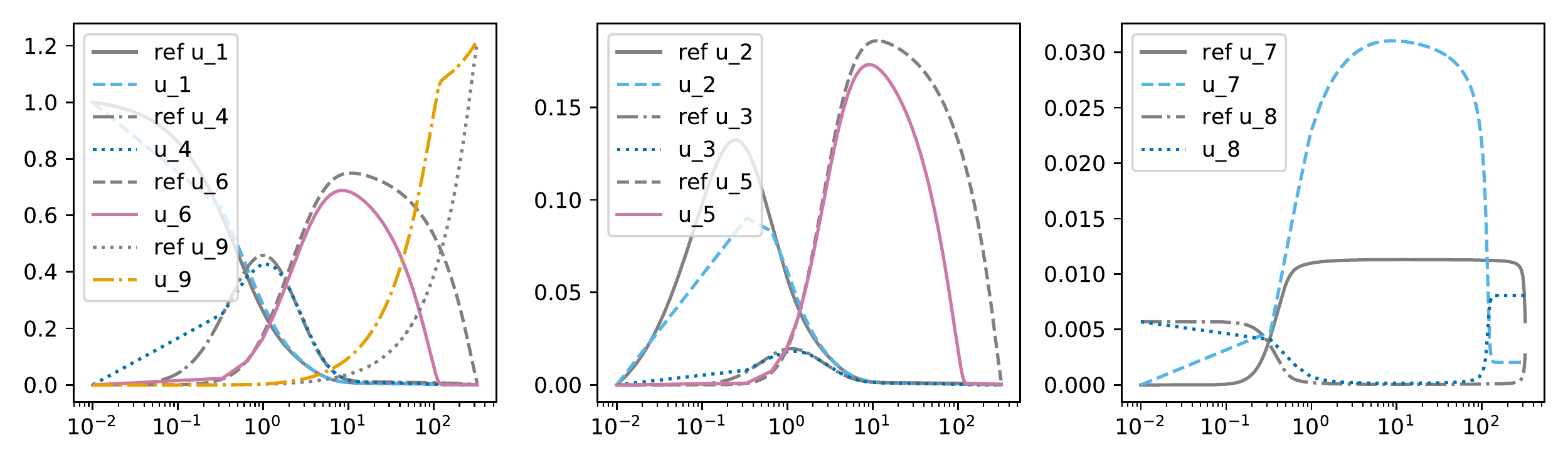}\\
	\includegraphics[width=\textwidth]{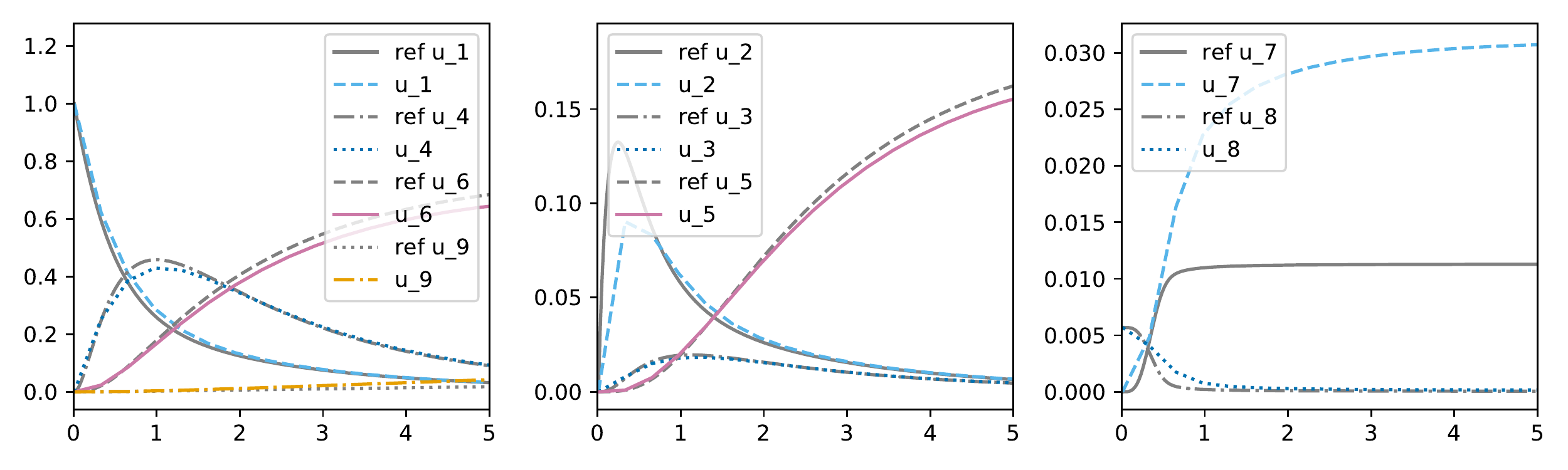}
	\caption{Simulations run with MPRK(3,2) with $N=1000$ timesteps, top logarithmic scale in time, bottom zoom on $t\in [0,5]$}\label{fig:HIRES_MPRK32}
\end{figure}

\begin{figure}
	\includegraphics[width=\textwidth]{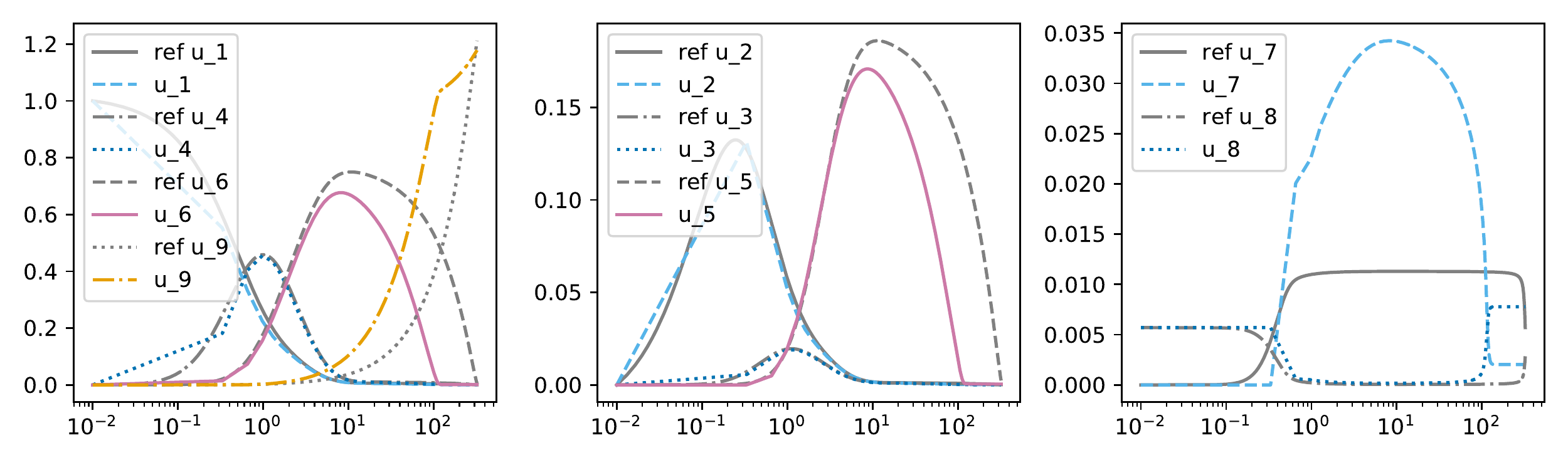}\\
	\includegraphics[width=\textwidth]{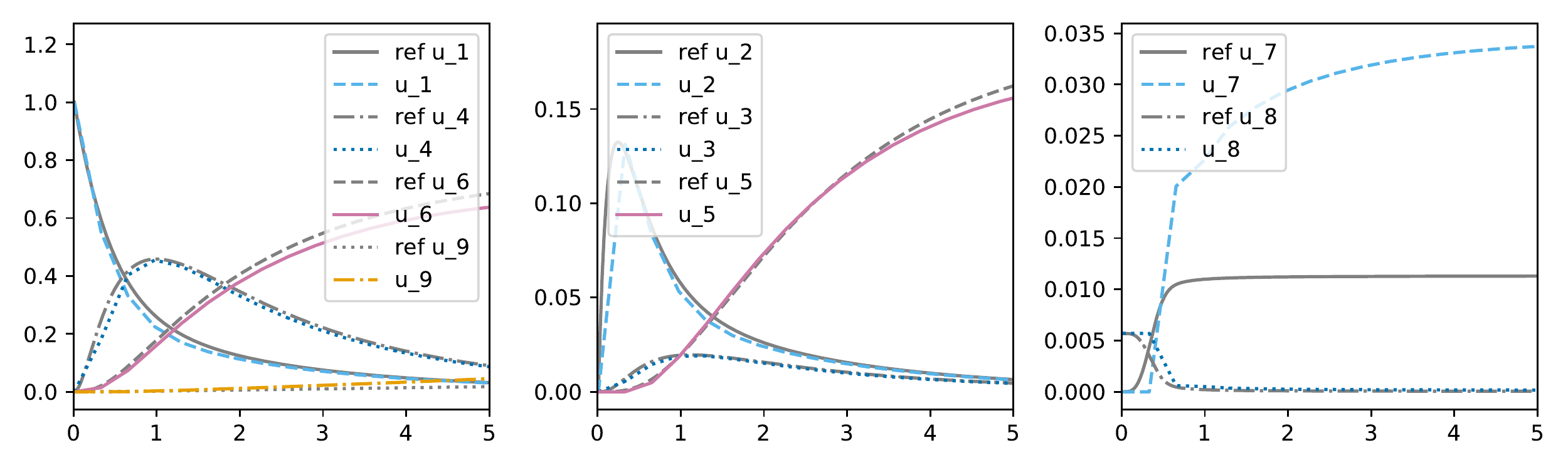}
	\caption{Simulations run with SIRK2 with $N=1000$ timesteps, top logarithmic scale in time, bottom zoom on $t\in [0,5]$}\label{fig:HIRES_SIRK2}
\end{figure}

\begin{figure}
	\includegraphics[width=\textwidth]{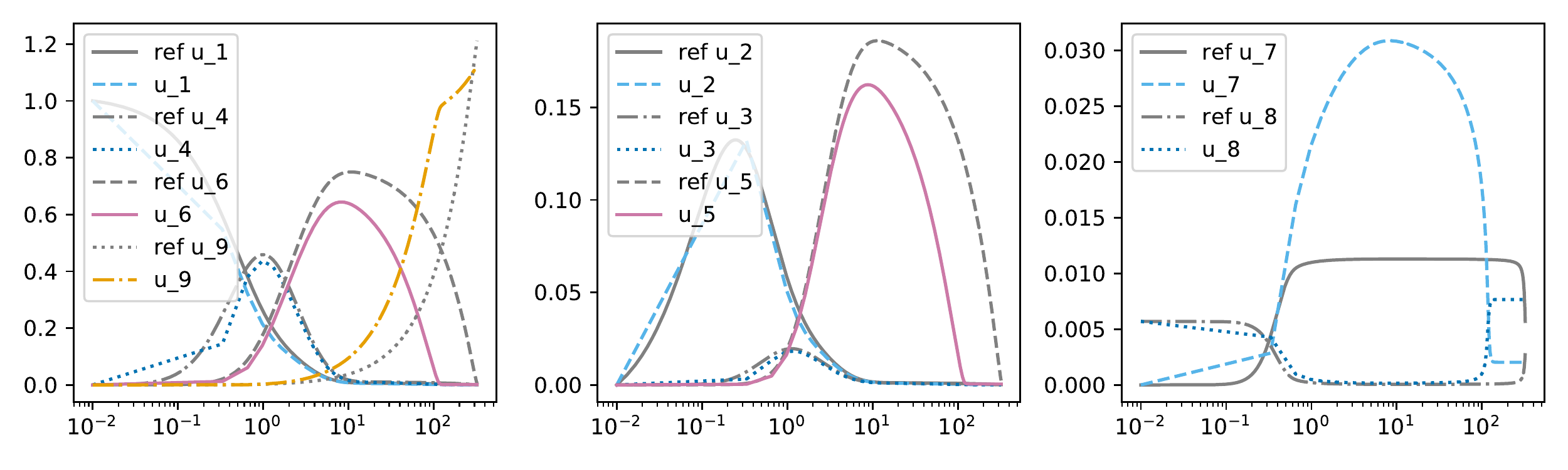}\\
	\includegraphics[width=\textwidth]{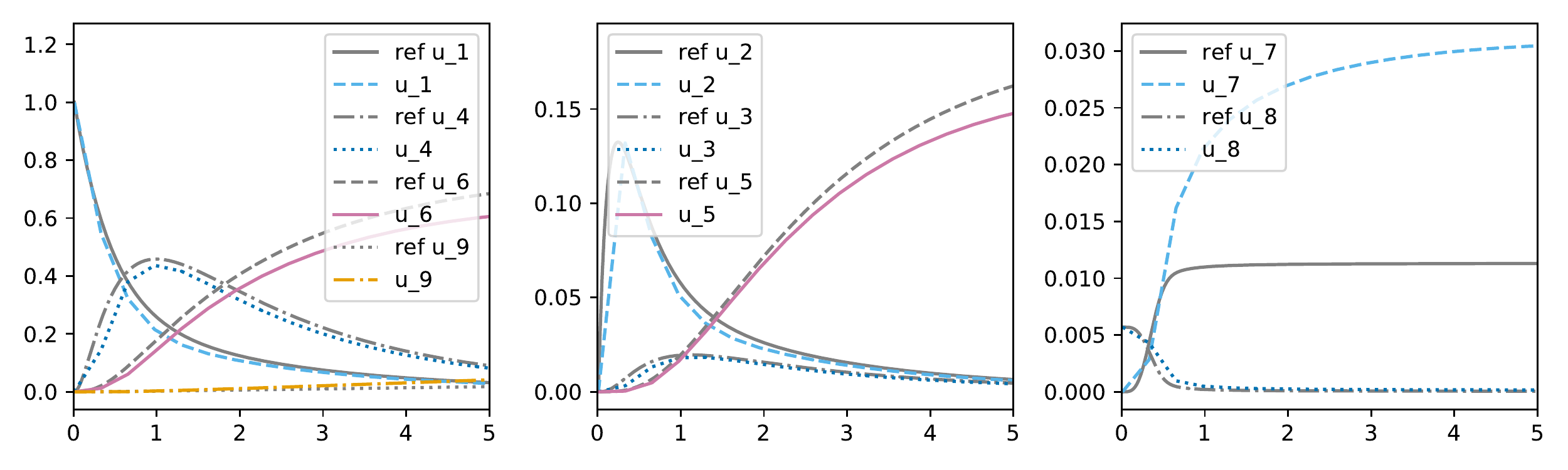}
	\caption{Simulations run with SIRK3 with $N=1000$ timesteps, top logarithmic scale in time, bottom zoom on $t\in [0,5]$}\label{fig:HIRES_SIRK3}
\end{figure}

\begin{figure}
	\includegraphics[width=\textwidth]{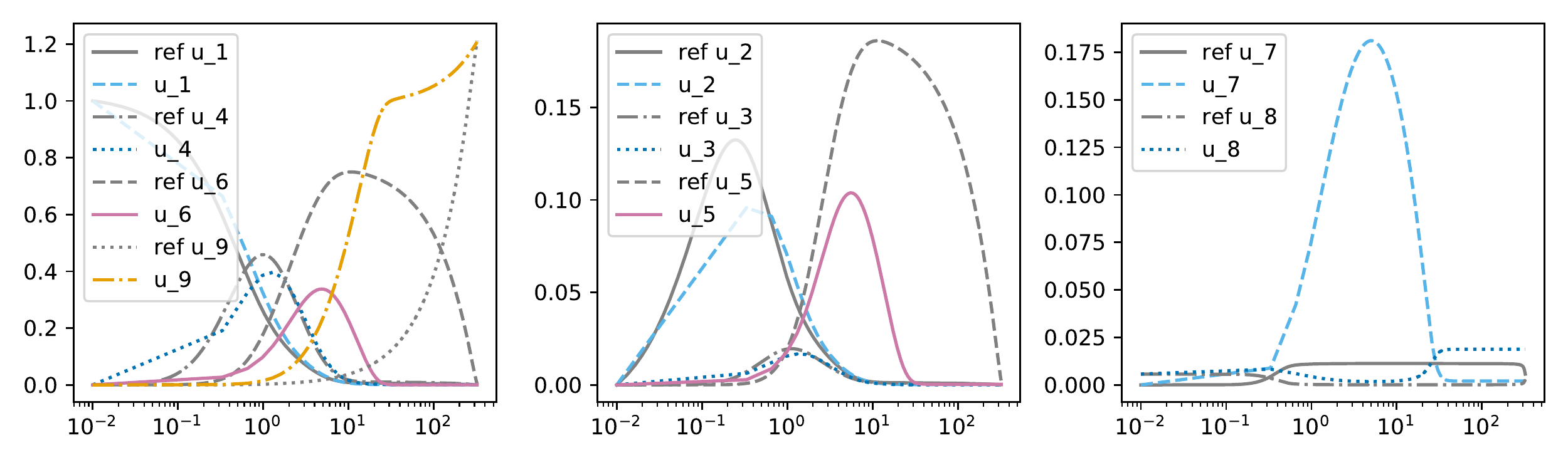}\\
	\includegraphics[width=\textwidth]{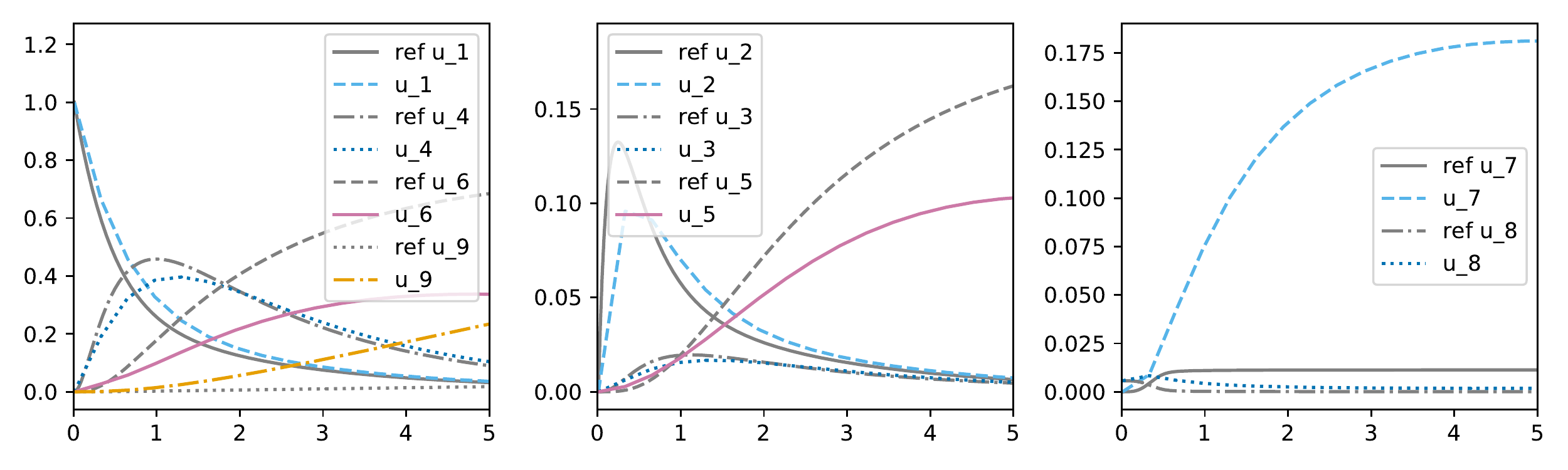}
	\caption{Simulations run with \ref{eq:explicit_dec_correction}1 with Gauss--Lobatto points with $N=1000$ timesteps, top logarithmic scale in time, bottom zoom on $t\in [0,5]$}\label{fig:HIRES_MPDeC1gl}
\end{figure}

\begin{figure}
	\includegraphics[width=\textwidth]{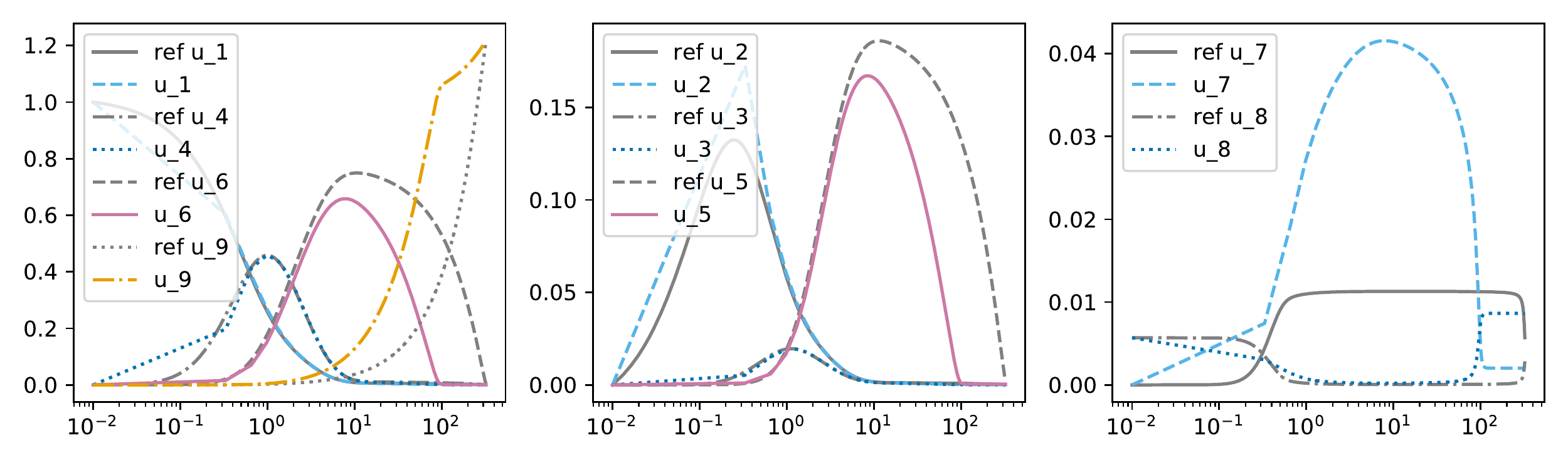}\\
	\includegraphics[width=\textwidth]{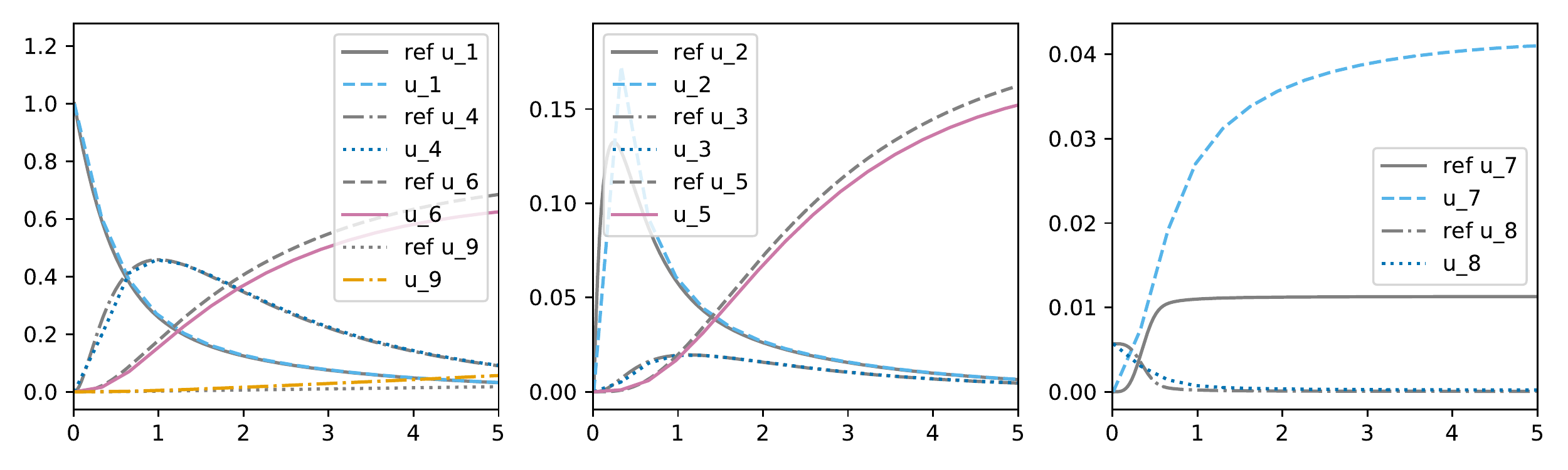}
	\caption{Simulations run with \ref{eq:explicit_dec_correction}2 with Gauss--Lobatto points with $N=1000$ timesteps, top logarithmic scale in time, bottom zoom on $t\in [0,5]$}\label{fig:HIRES_MPDeC2gl}
\end{figure}

\begin{figure}
	\includegraphics[width=\textwidth]{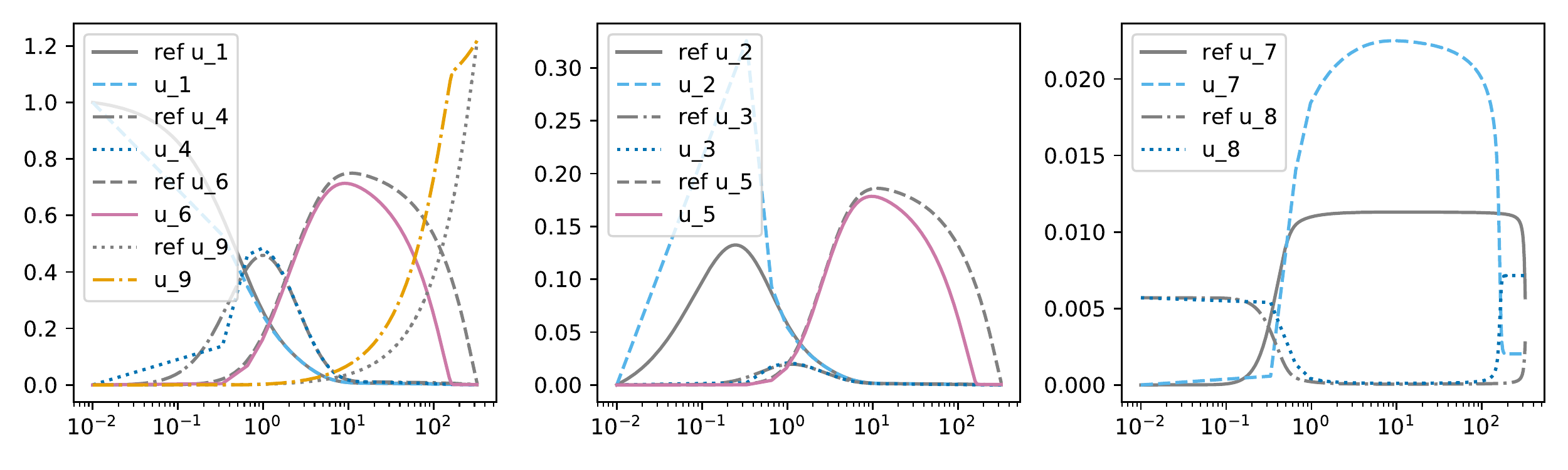}\\
	\includegraphics[width=\textwidth]{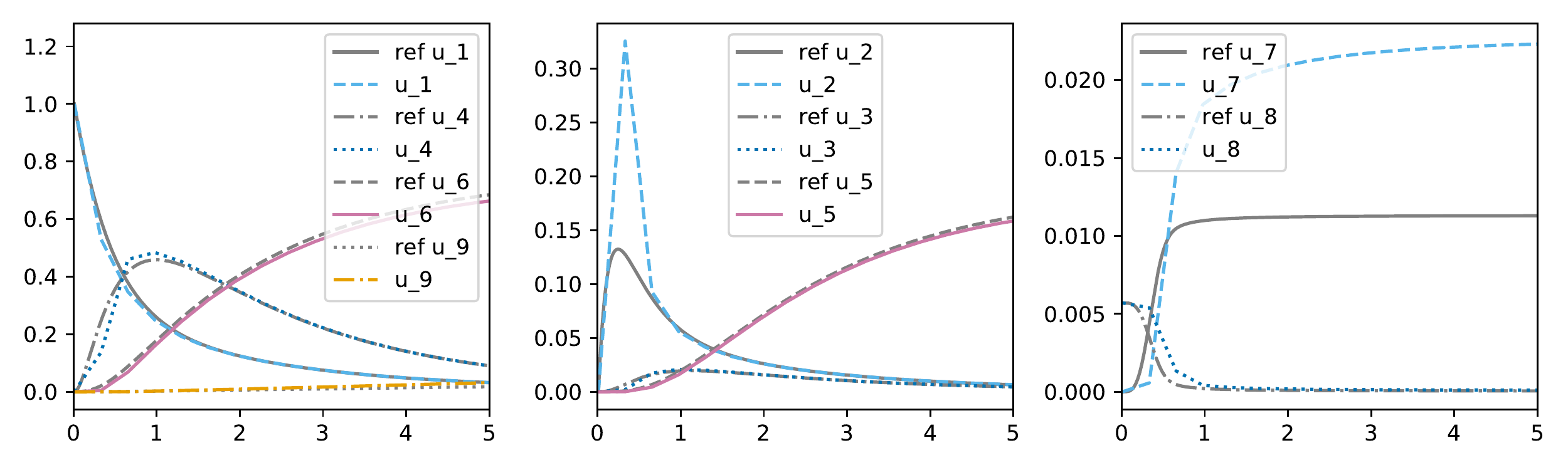}
	\caption{Simulations run with \ref{eq:explicit_dec_correction}3 with Gauss--Lobatto points with $N=1000$ timesteps, top logarithmic scale in time, bottom zoom on $t\in [0,5]$}\label{fig:HIRES_MPDeC3gl}
\end{figure}

\begin{figure}
	\includegraphics[width=\textwidth]{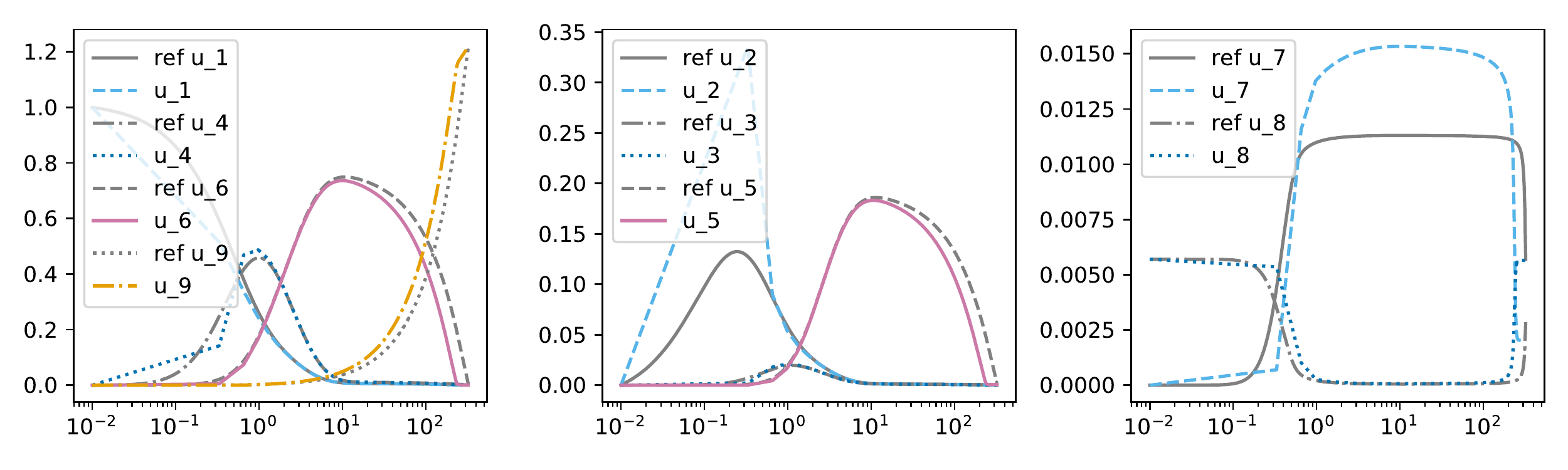}\\
	\includegraphics[width=\textwidth]{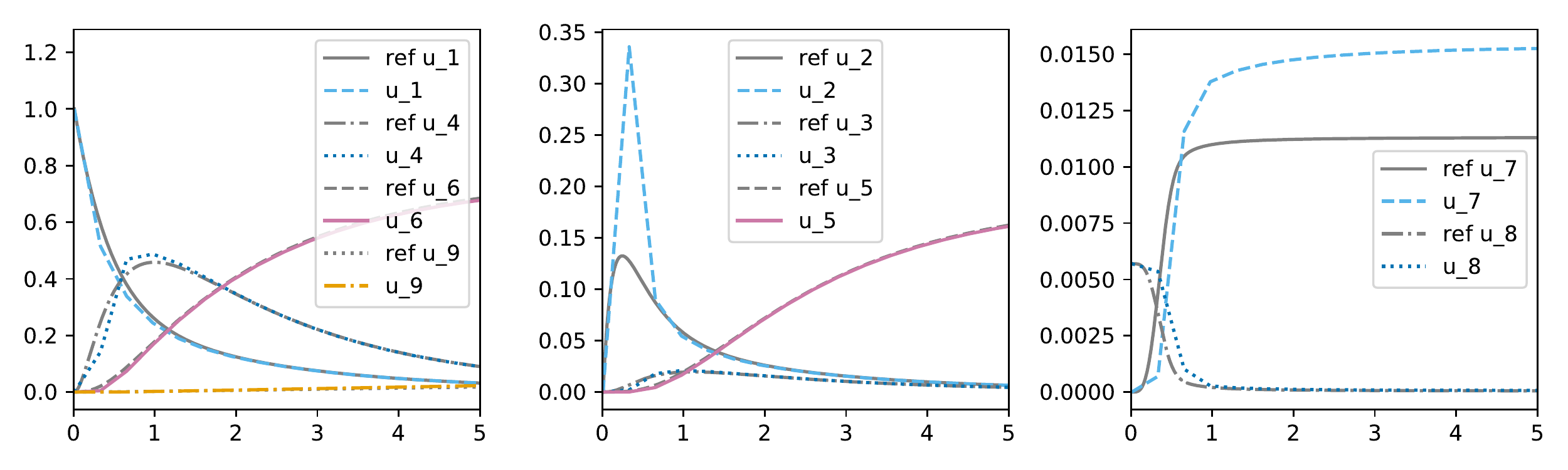}
	\caption{Simulations run with \ref{eq:explicit_dec_correction}4 with Gauss--Lobatto points with $N=1000$ timesteps, top logarithmic scale in time, bottom zoom on $t\in [0,5]$}\label{fig:HIRES_MPDeC4gl}
\end{figure}

\begin{figure}
	\includegraphics[width=\textwidth]{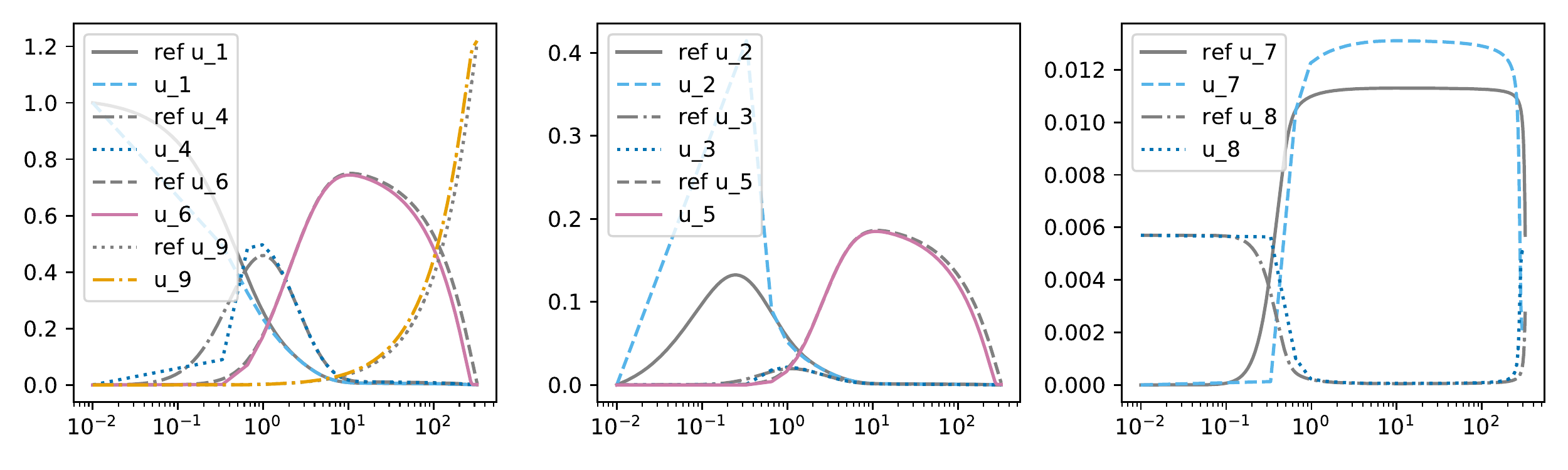}\\
	\includegraphics[width=\textwidth]{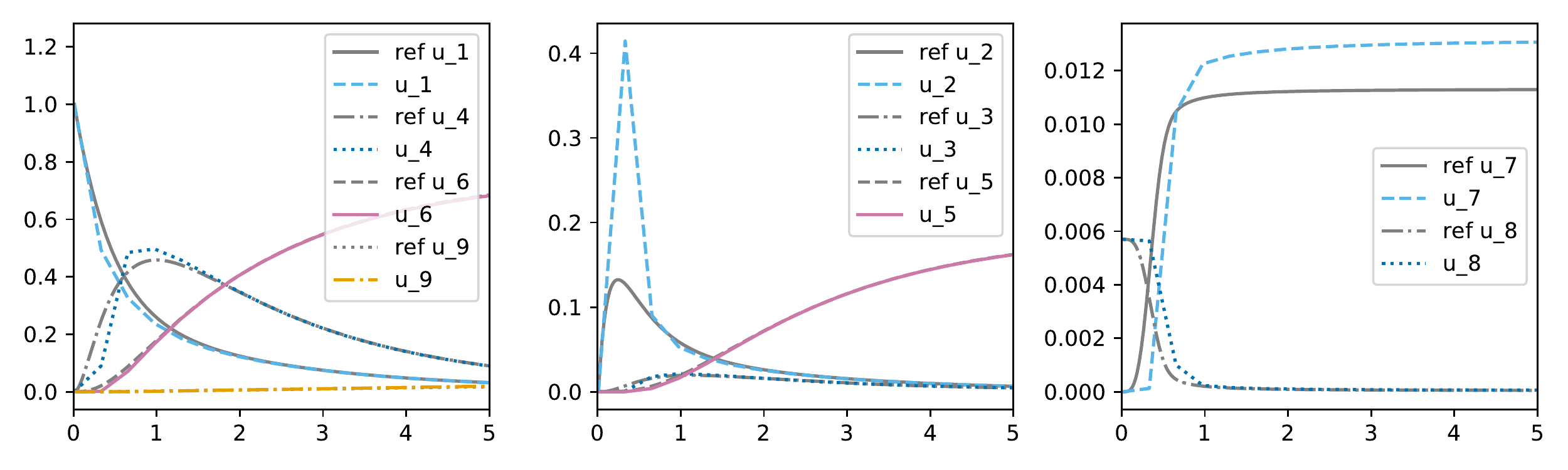}
	\caption{Simulations run with \ref{eq:explicit_dec_correction}5 with Gauss--Lobatto points with $N=1000$ timesteps, top logarithmic scale in time, bottom zoom on $t\in [0,5]$}\label{fig:HIRES_MPDeC5gl}
\end{figure}

\begin{figure}
	\includegraphics[width=\textwidth]{figures/HIRES/MPDeC6gl_logx.pdf}\\
	\includegraphics[width=\textwidth]{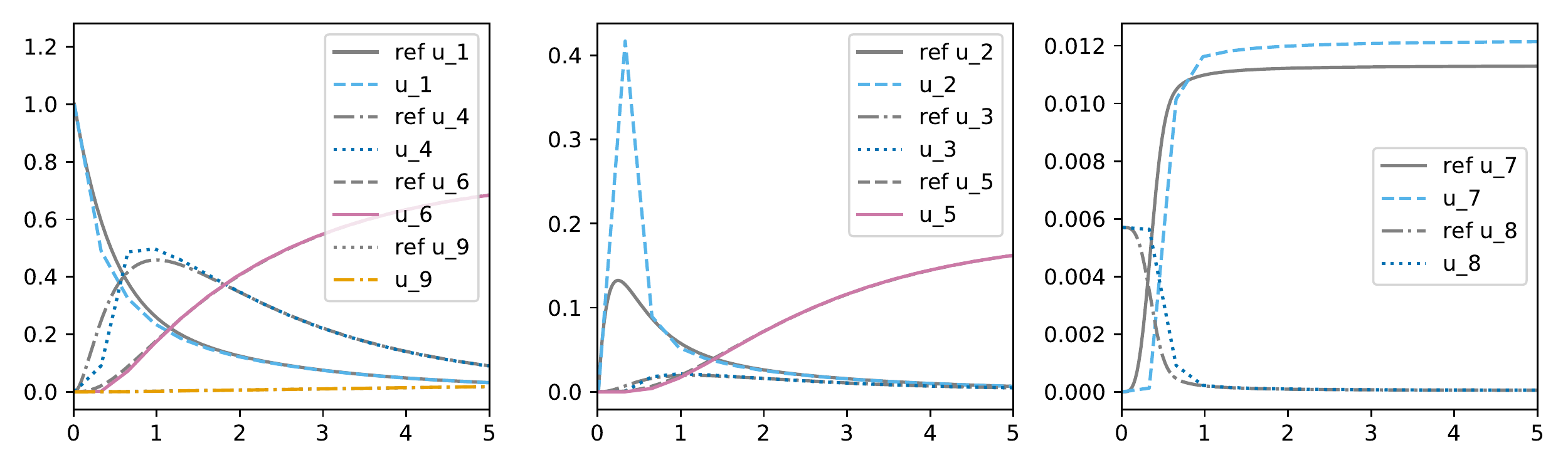}
	\caption{Simulations run with \ref{eq:explicit_dec_correction}6 with Gauss--Lobatto points with $N=1000$ timesteps, top logarithmic scale in time, bottom zoom on $t\in [0,5]$}\label{fig:HIRES_MPDeC6gl}
\end{figure}

\begin{figure}
	\includegraphics[width=\textwidth]{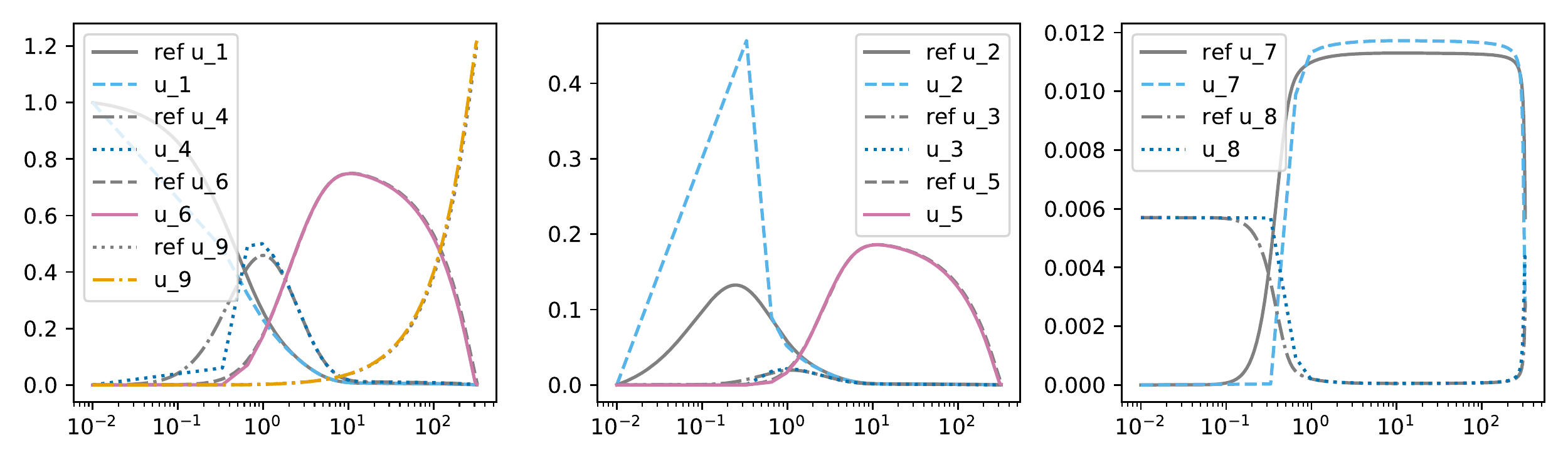}\\
	\includegraphics[width=\textwidth]{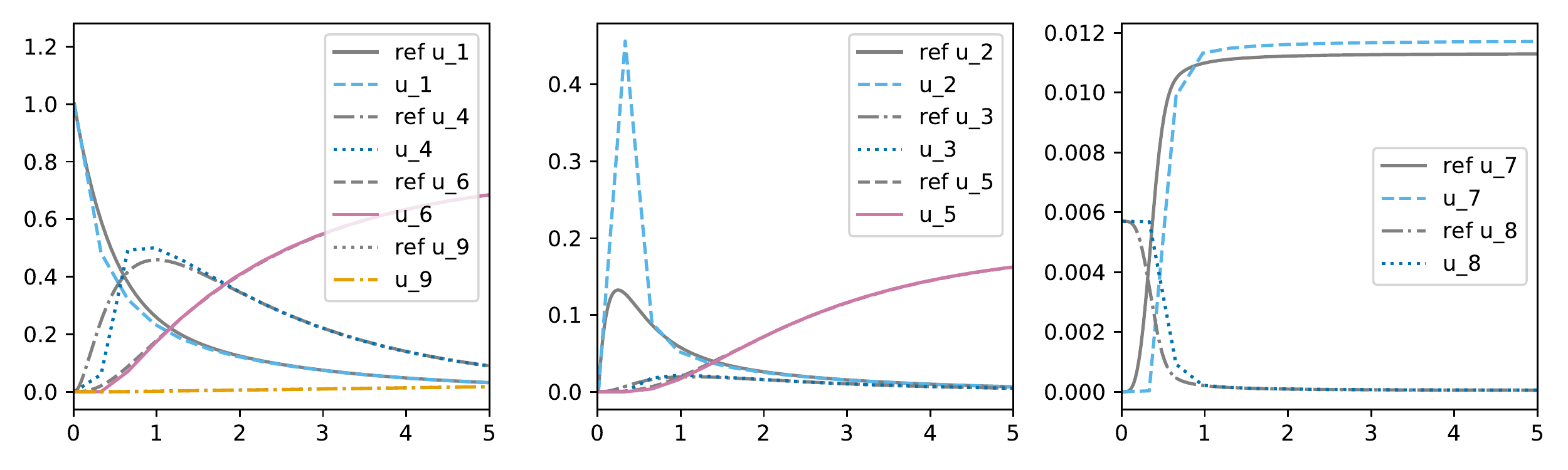}
	\caption{Simulations run with \ref{eq:explicit_dec_correction}7 with Gauss--Lobatto points with $N=1000$ timesteps, top logarithmic scale in time, bottom zoom on $t\in [0,5]$}\label{fig:HIRES_MPDeC7gl}
\end{figure}

\begin{figure}
	\includegraphics[width=\textwidth]{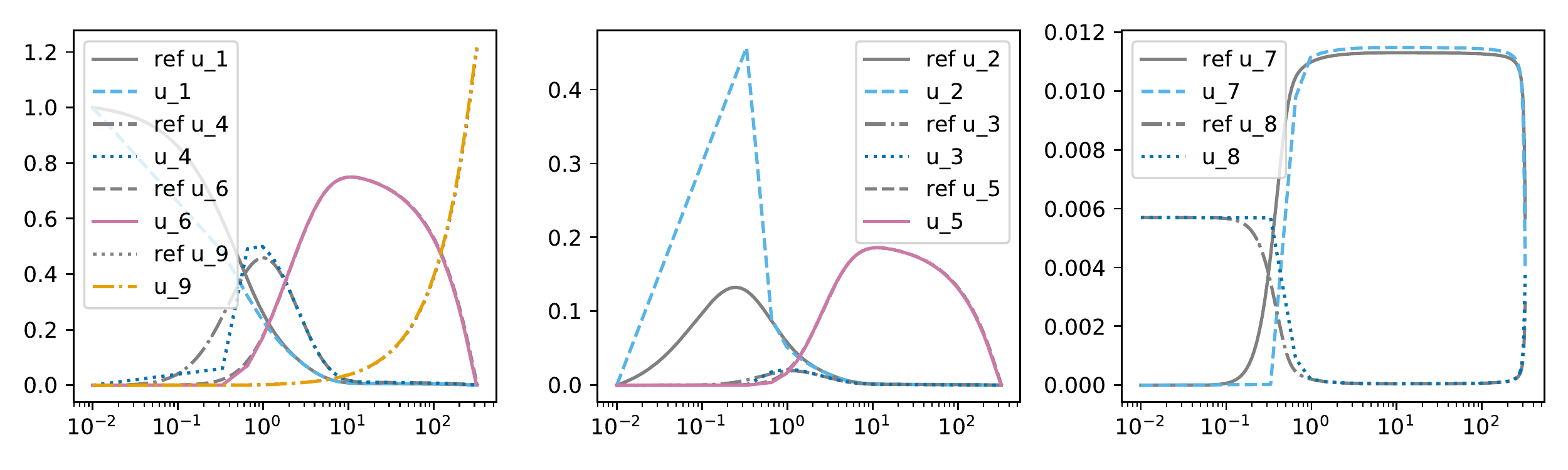}\\
	\includegraphics[width=\textwidth]{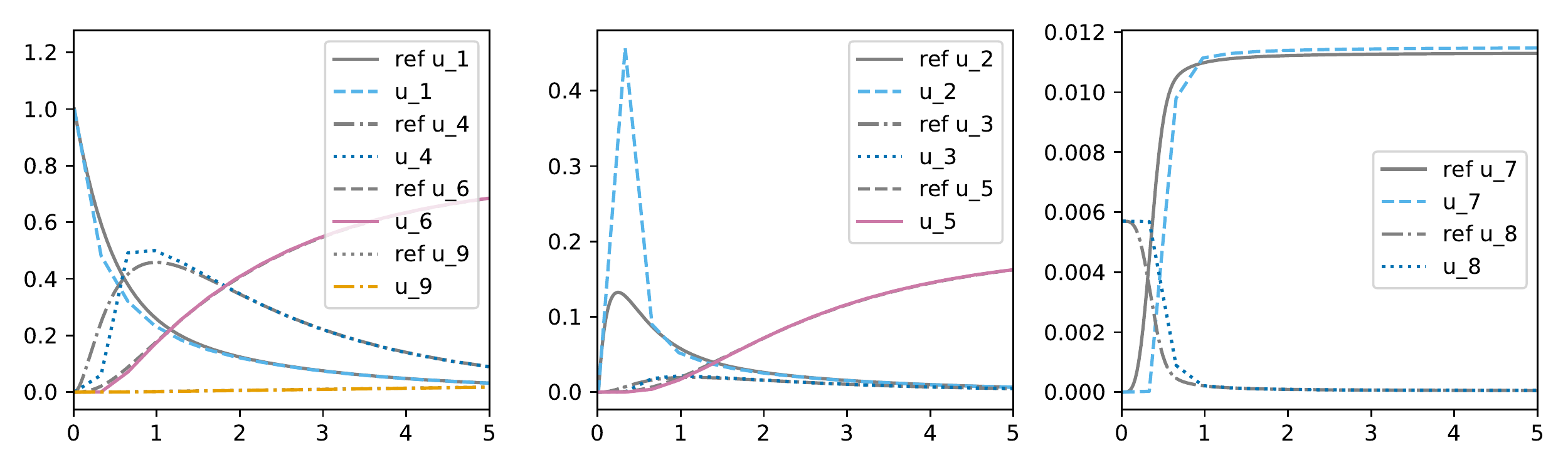}
	\caption{Simulations run with \ref{eq:explicit_dec_correction}8 with Gauss--Lobatto points with $N=1000$ timesteps, top logarithmic scale in time, bottom zoom on $t\in [0,5]$}\label{fig:HIRES_MPDeC8gl}
\end{figure}

\begin{figure}
	\includegraphics[width=\textwidth]{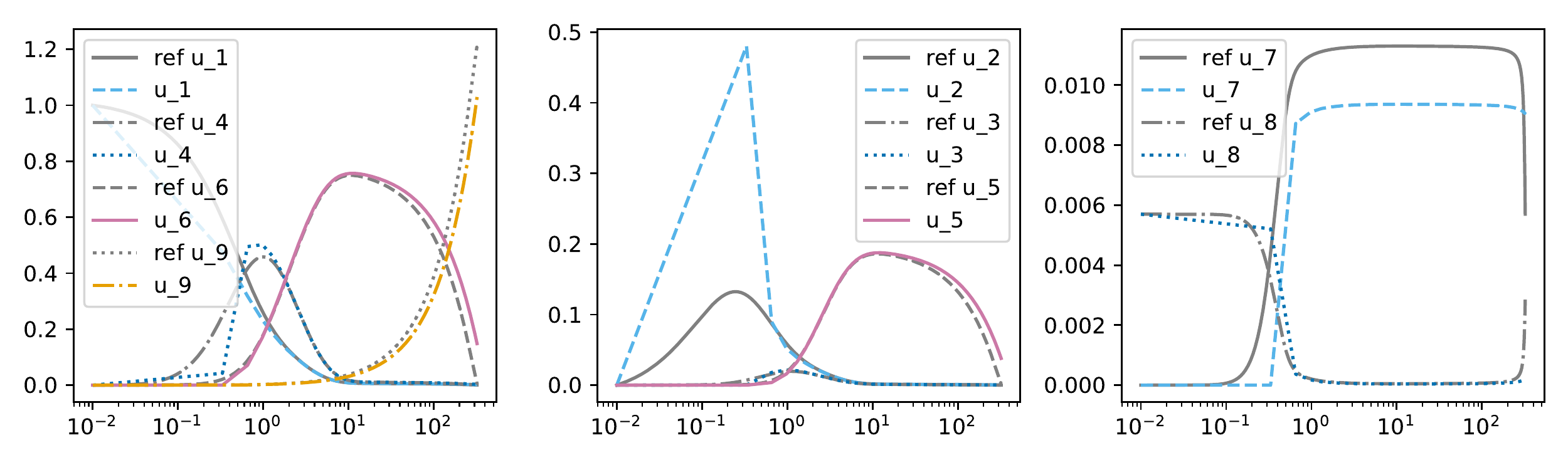}\\
	\includegraphics[width=\textwidth]{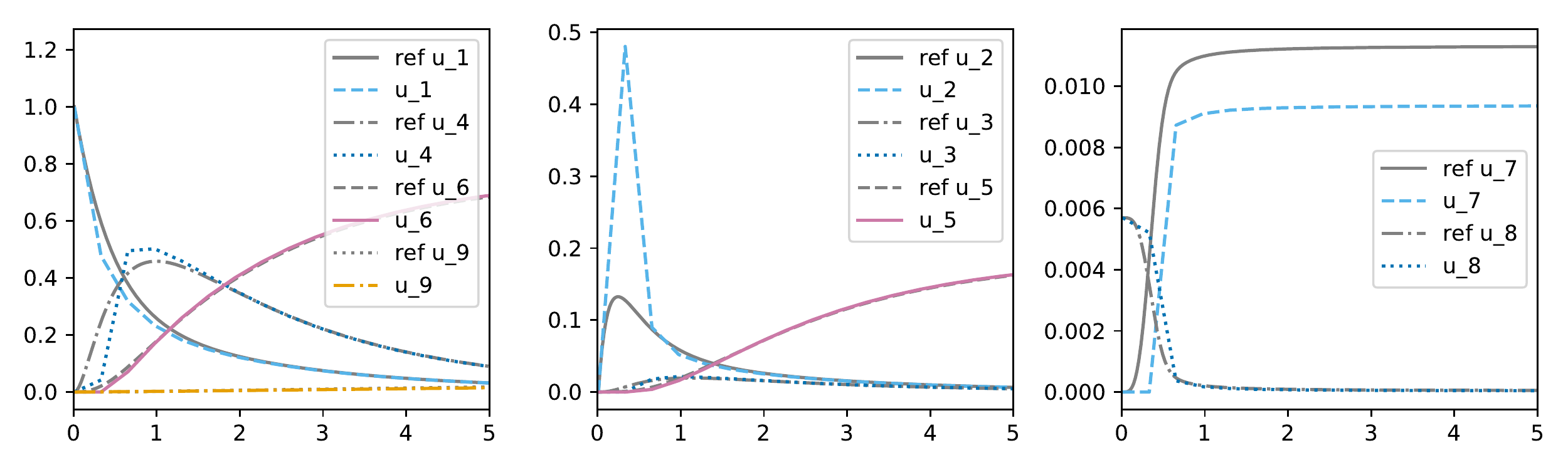}
	\caption{Simulations run with \ref{eq:explicit_dec_correction}9 with Gauss--Lobatto points with $N=1000$ timesteps, top logarithmic scale in time, bottom zoom on $t\in [0,5]$}\label{fig:HIRES_MPDeC9gl}
\end{figure}

\begin{figure}
	\includegraphics[width=\textwidth]{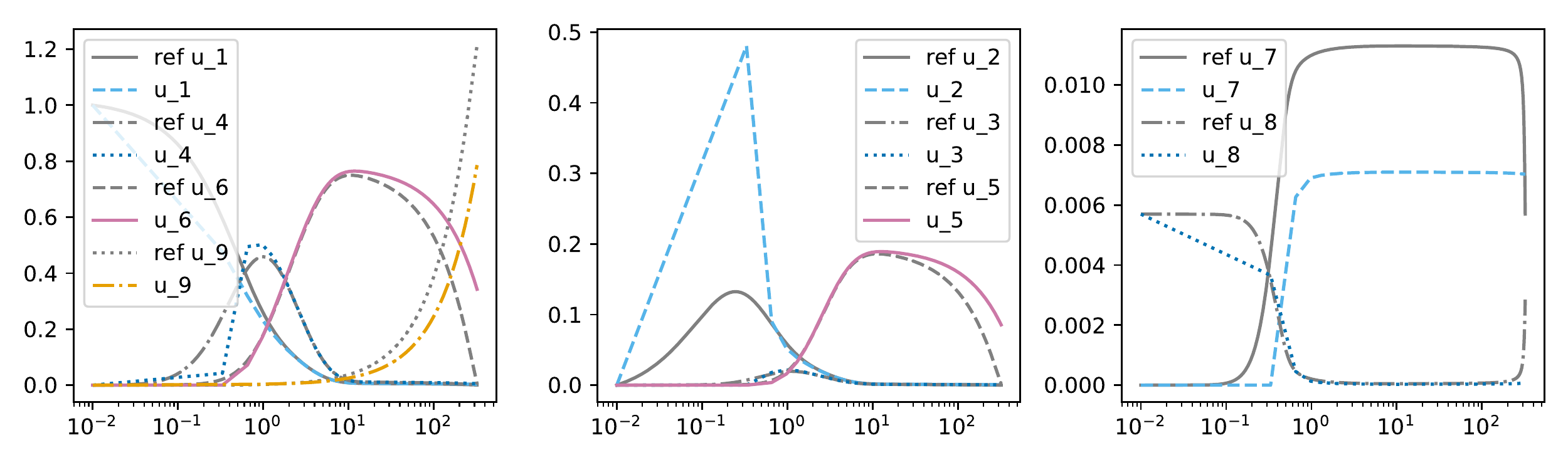}\\
	\includegraphics[width=\textwidth]{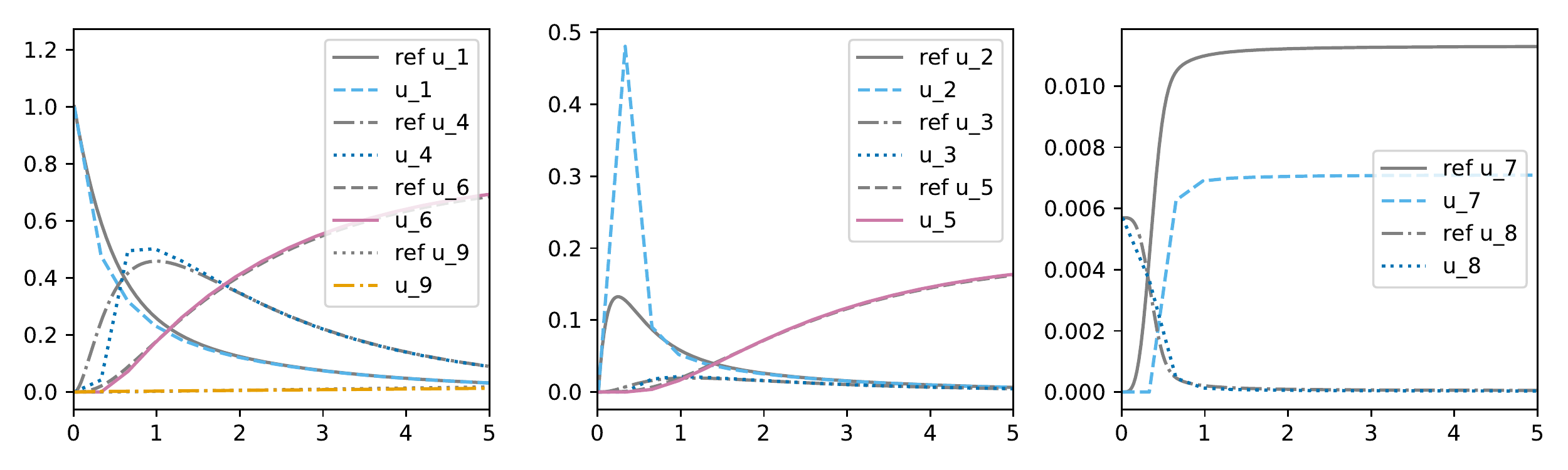}
	\caption{Simulations run with \ref{eq:explicit_dec_correction}10 with Gauss--Lobatto points with $N=1000$ timesteps, top logarithmic scale in time, bottom zoom on $t\in [0,5]$}\label{fig:HIRES_MPDeC10gl}
\end{figure}

\begin{figure}
	\includegraphics[width=\textwidth]{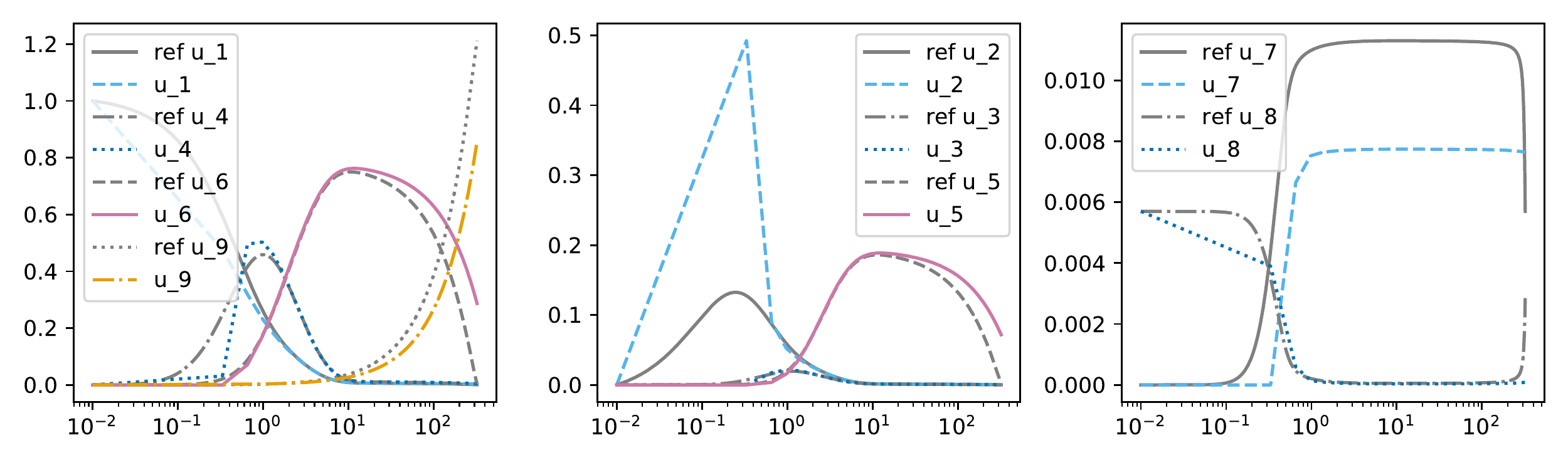}\\
	\includegraphics[width=\textwidth]{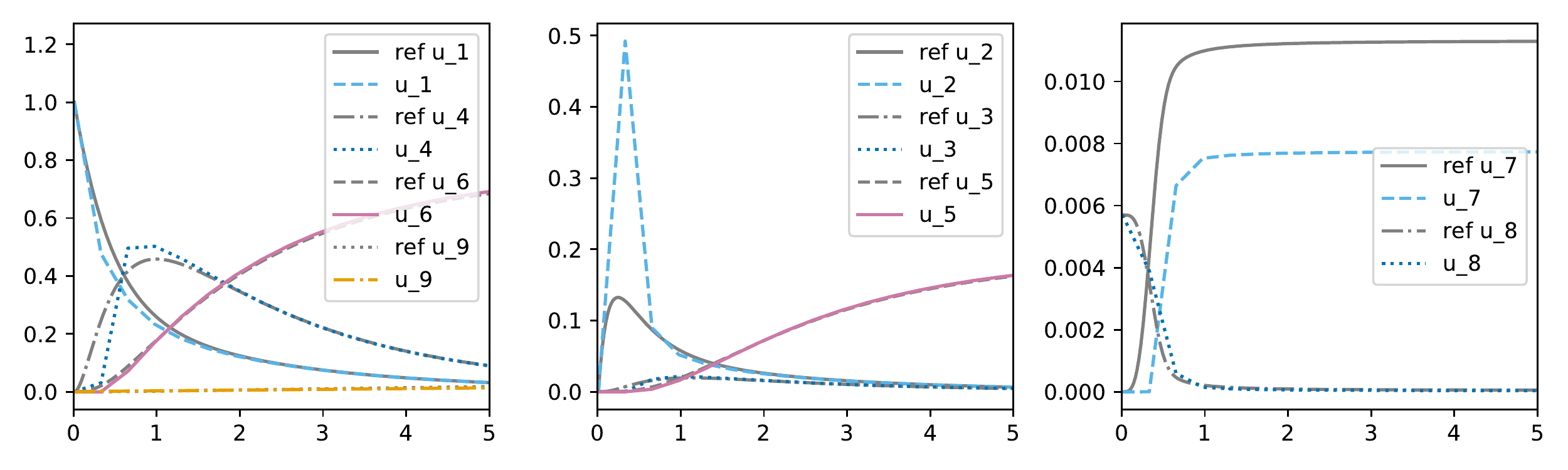}
	\caption{Simulations run with \ref{eq:explicit_dec_correction}11 with Gauss--Lobatto points with $N=1000$ timesteps, top logarithmic scale in time, bottom zoom on $t\in [0,5]$}\label{fig:HIRES_MPDeC11gl}
\end{figure}

\begin{figure}
	\includegraphics[width=\textwidth]{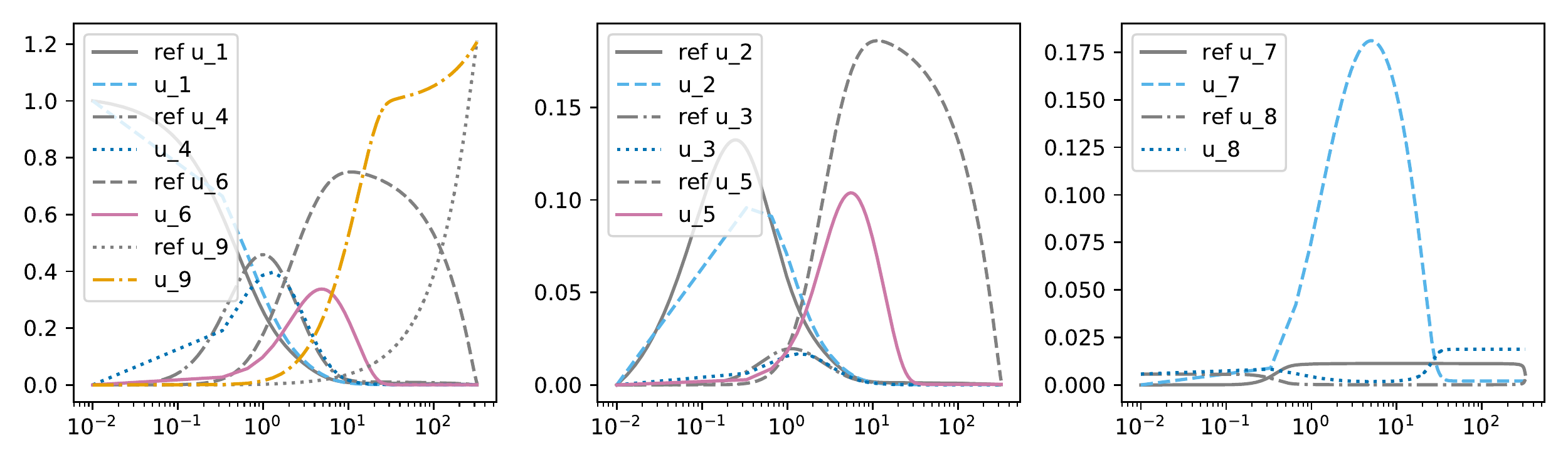}\\
	\includegraphics[width=\textwidth]{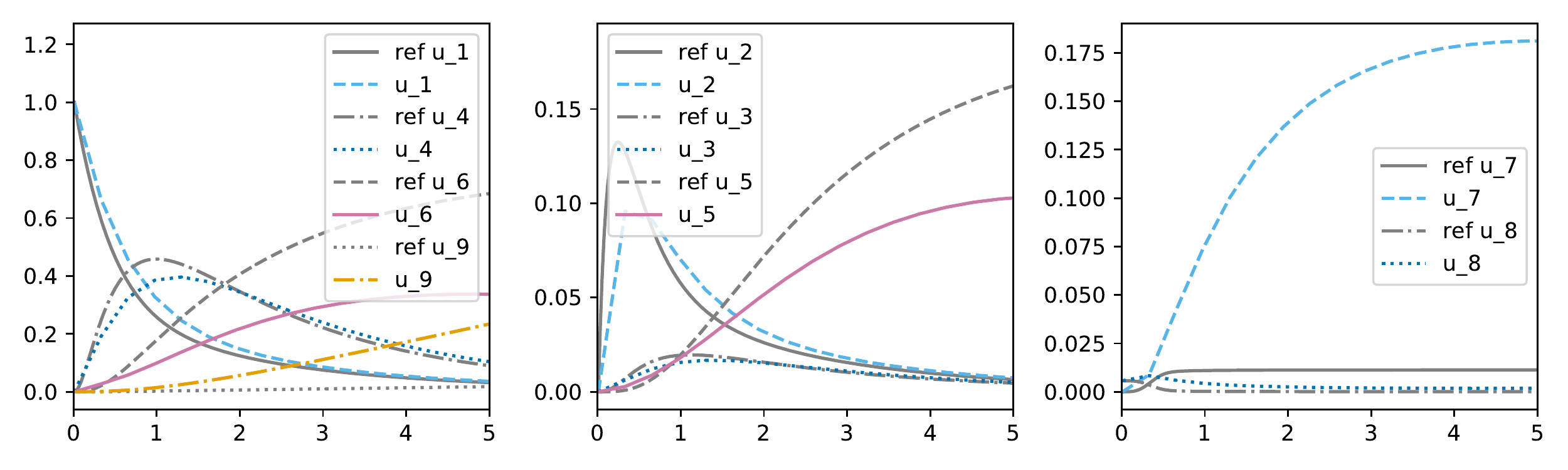}
	\caption{Simulations run with \ref{eq:explicit_dec_correction}1 with equispaced points with $N=1000$ timesteps, top logarithmic scale in time, bottom zoom on $t\in [0,5]$}\label{fig:HIRES_MPDeC1eq}
\end{figure}

\begin{figure}
	\includegraphics[width=\textwidth]{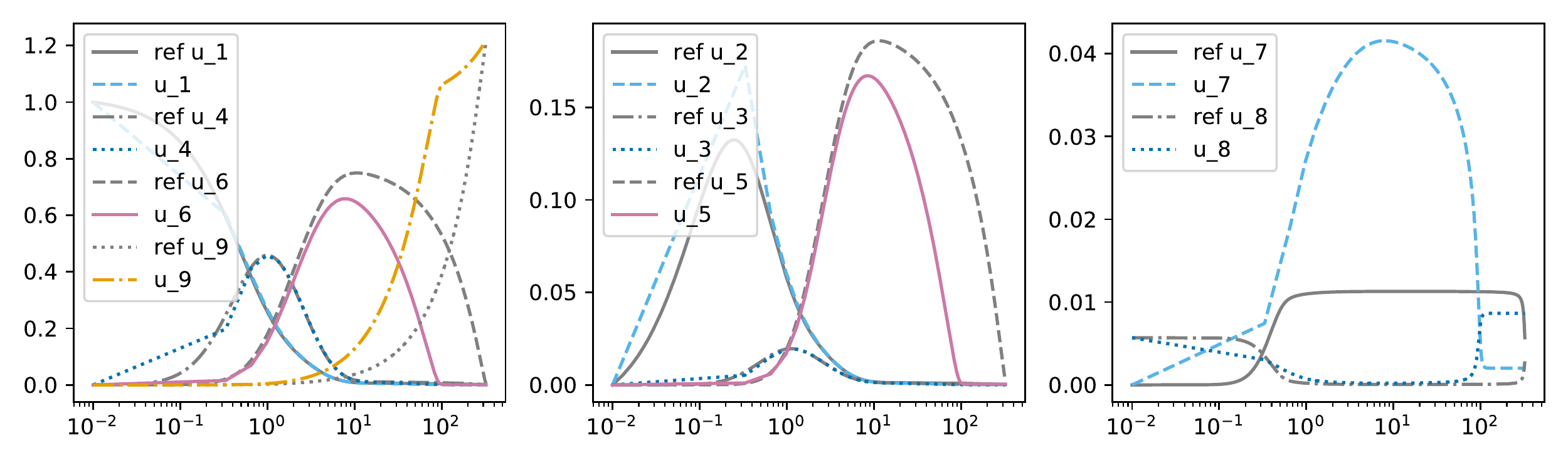}\\
	\includegraphics[width=\textwidth]{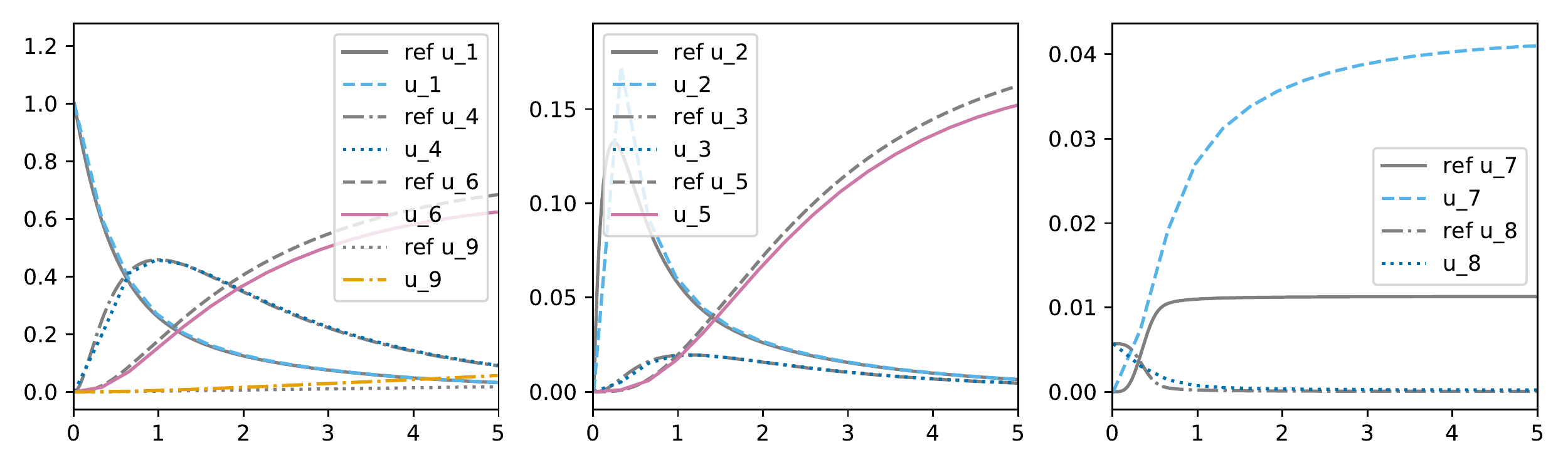}
	\caption{Simulations run with \ref{eq:explicit_dec_correction}2 with equispaced points with $N=1000$ timesteps, top logarithmic scale in time, bottom zoom on $t\in [0,5]$}\label{fig:HIRES_MPDeC2eq}
\end{figure}

\begin{figure}
	\includegraphics[width=\textwidth]{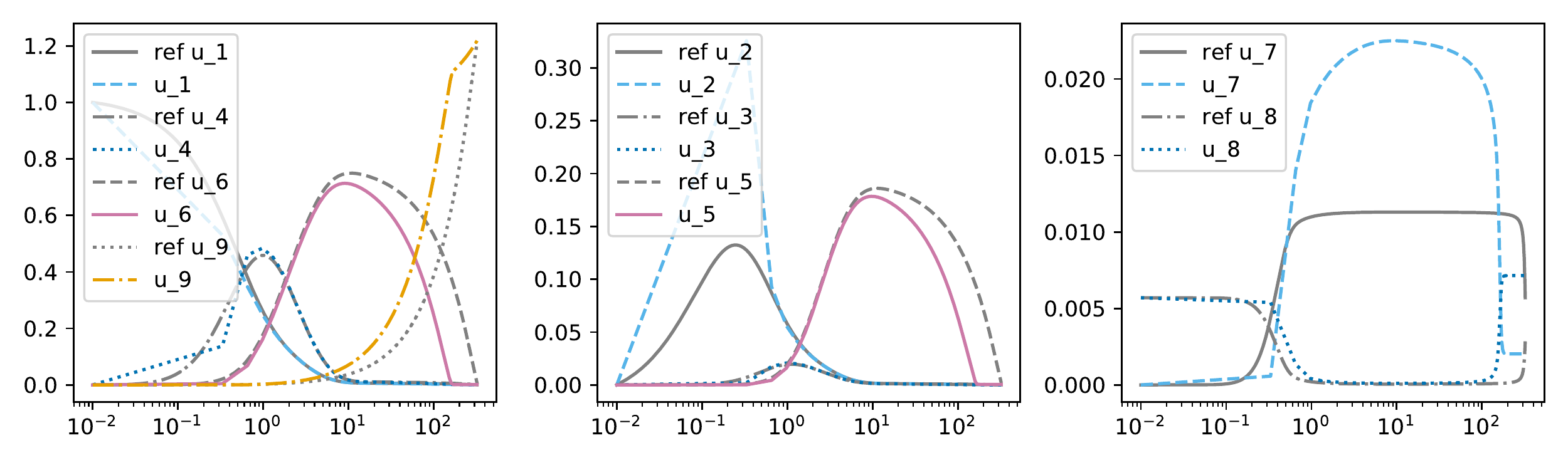}\\
	\includegraphics[width=\textwidth]{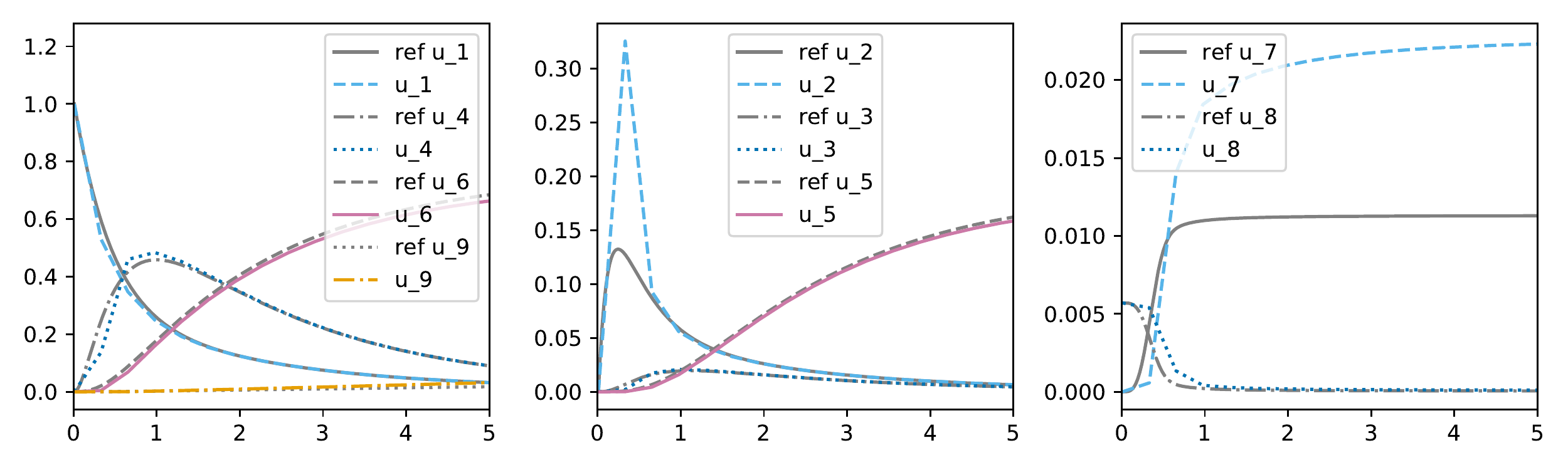}
	\caption{Simulations run with \ref{eq:explicit_dec_correction}3 with equispaced points with $N=1000$ timesteps, top logarithmic scale in time, bottom zoom on $t\in [0,5]$}\label{fig:HIRES_MPDeC3eq}
\end{figure}

\begin{figure}
	\includegraphics[width=\textwidth]{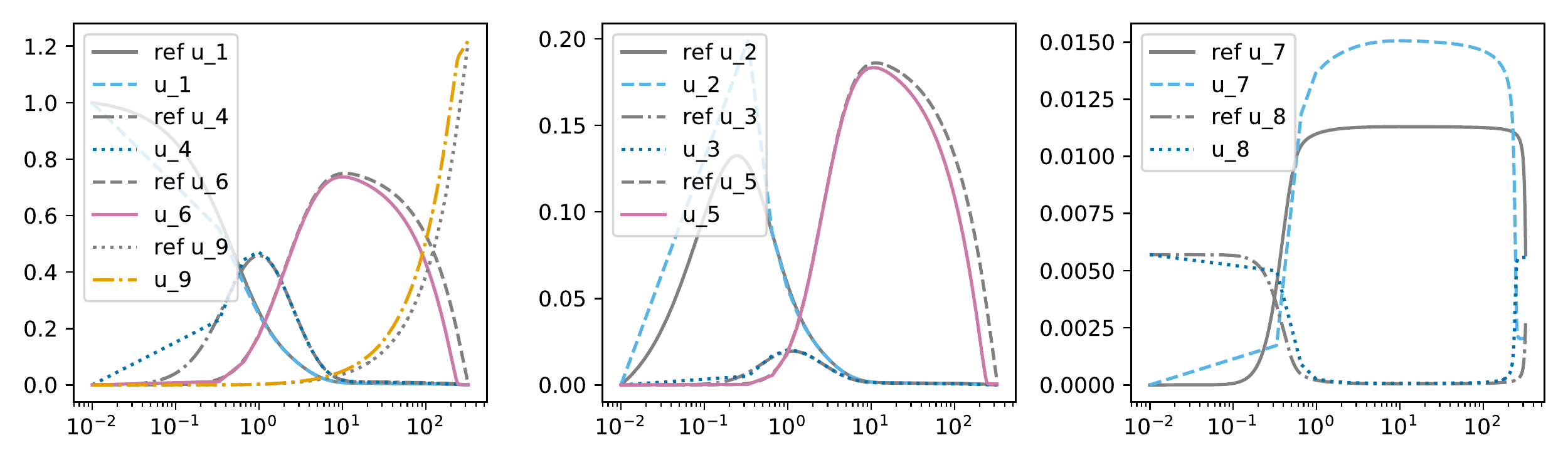}\\
	\includegraphics[width=\textwidth]{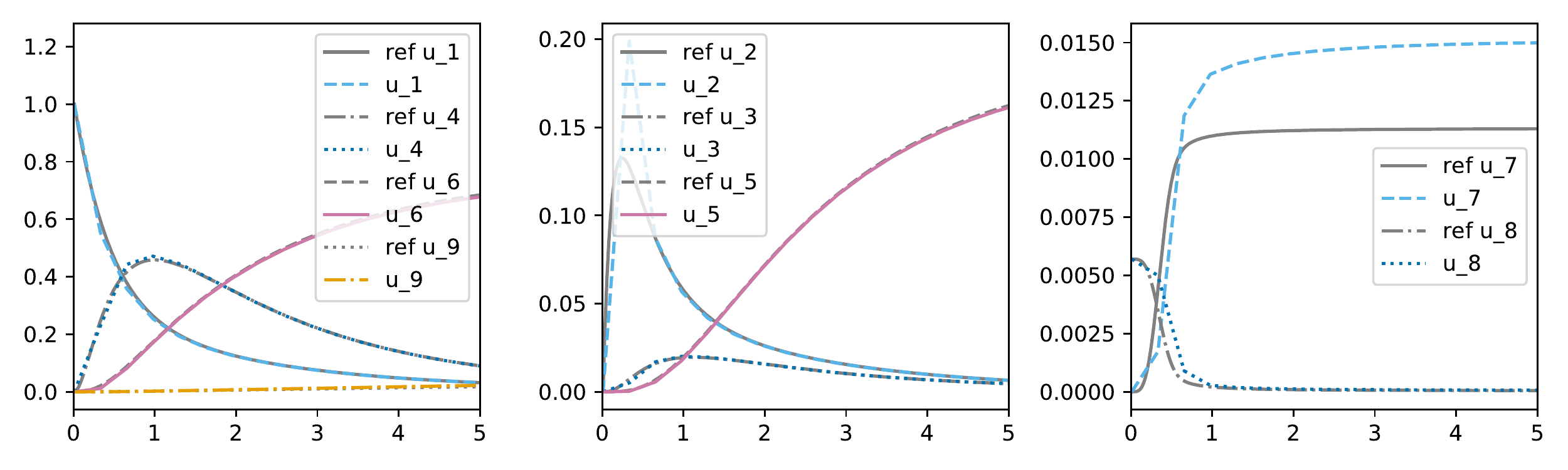}
	\caption{Simulations run with \ref{eq:explicit_dec_correction}4 with equispaced points with $N=1000$ timesteps, top logarithmic scale in time, bottom zoom on $t\in [0,5]$}\label{fig:HIRES_MPDeC4eq}
\end{figure}

\begin{figure}
	\includegraphics[width=\textwidth]{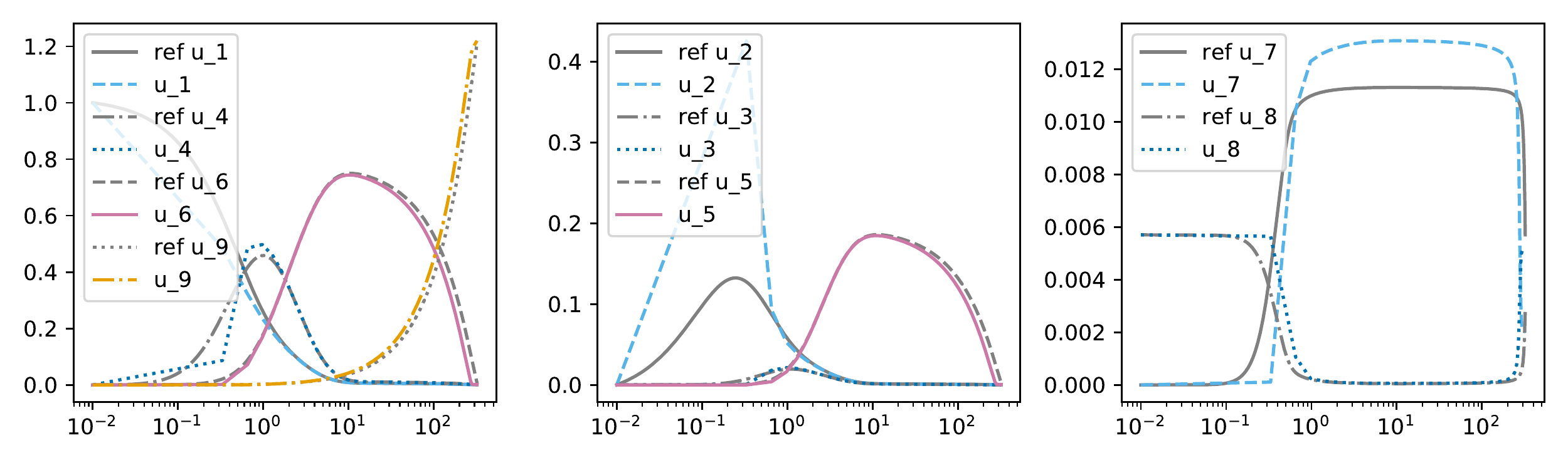}\\
	\includegraphics[width=\textwidth]{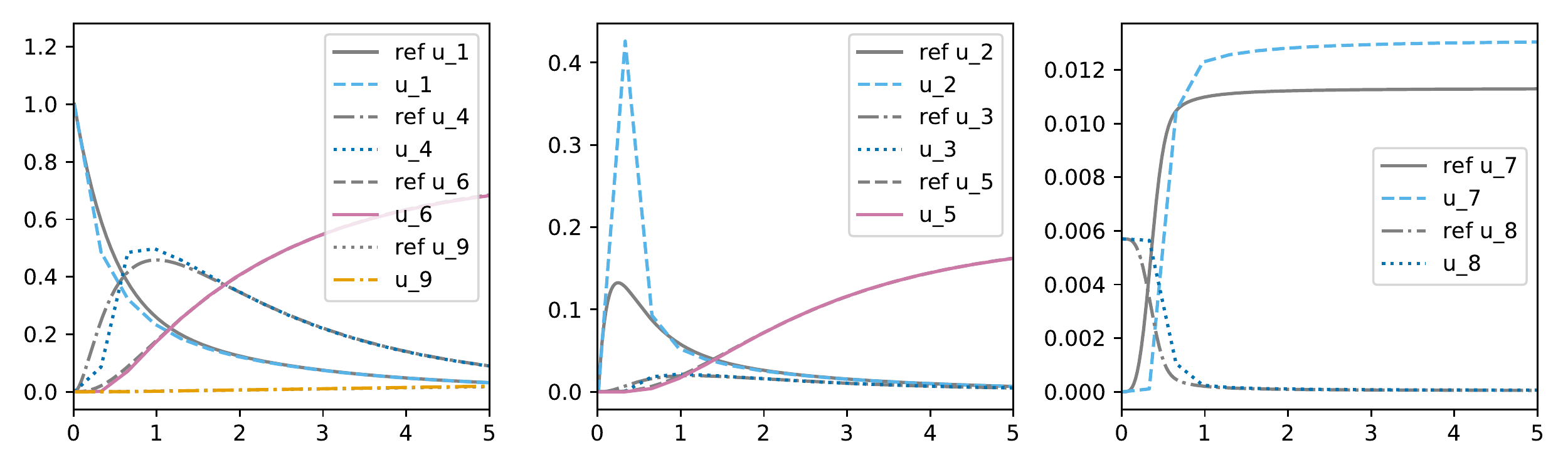}
	\caption{Simulations run with \ref{eq:explicit_dec_correction}5 with equispaced points with $N=1000$ timesteps, top logarithmic scale in time, bottom zoom on $t\in [0,5]$}\label{fig:HIRES_MPDeC5eq}
\end{figure}

\begin{figure}
	\includegraphics[width=\textwidth]{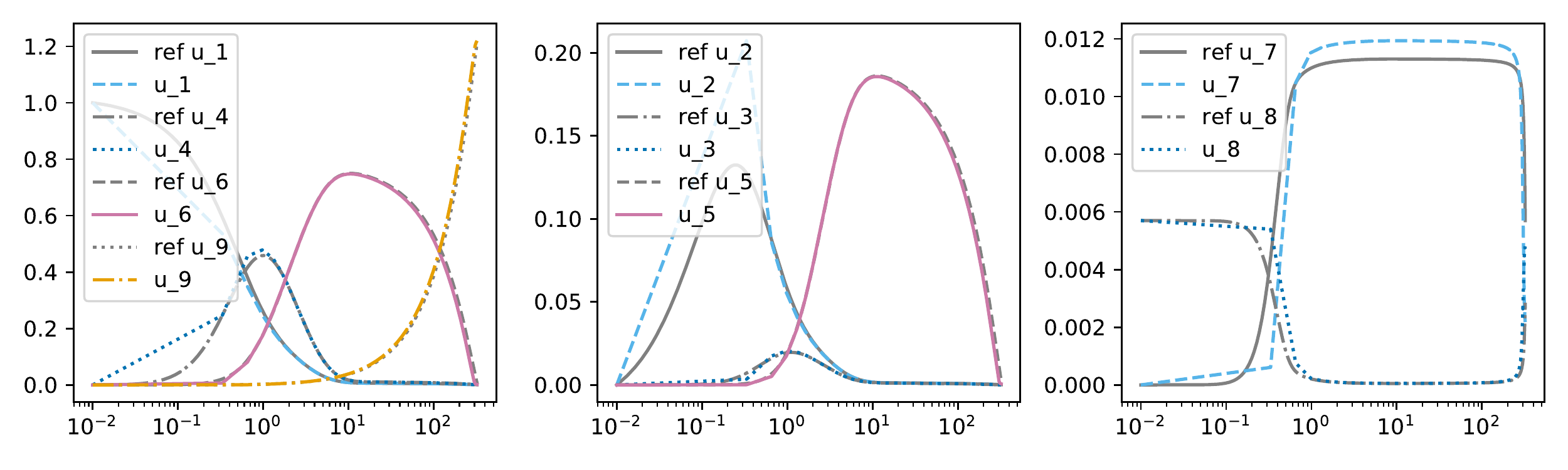}\\
	\includegraphics[width=\textwidth]{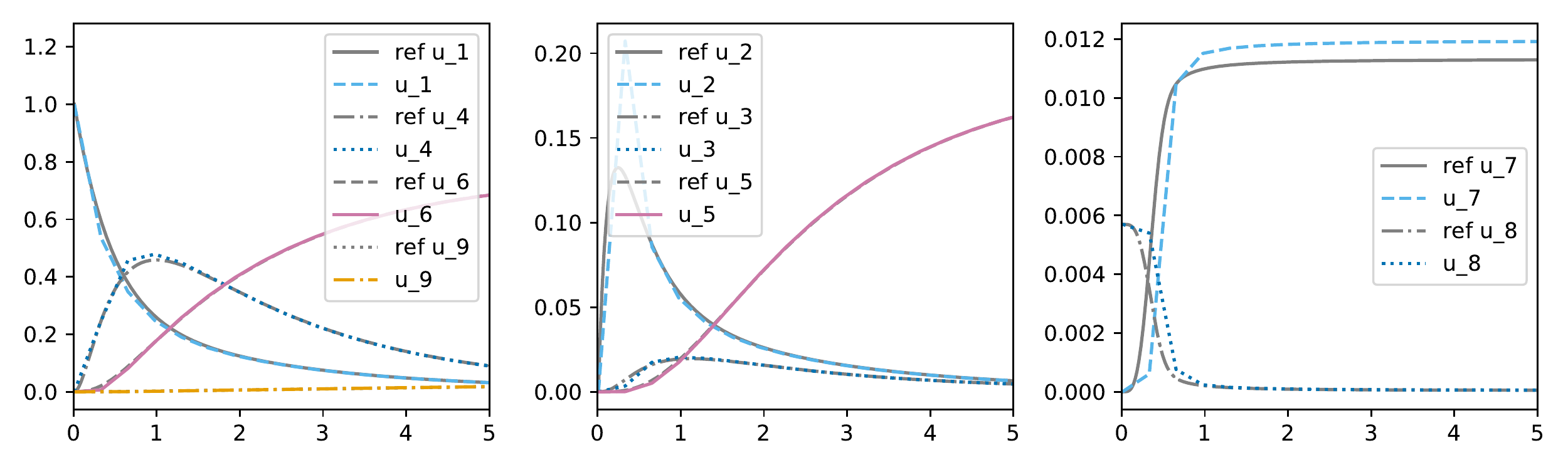}
	\caption{Simulations run with \ref{eq:explicit_dec_correction}6 with equispaced points with $N=1000$ timesteps, top logarithmic scale in time, bottom zoom on $t\in [0,5]$}\label{fig:HIRES_MPDeC6eq}
\end{figure}

\begin{figure}
	\includegraphics[width=\textwidth]{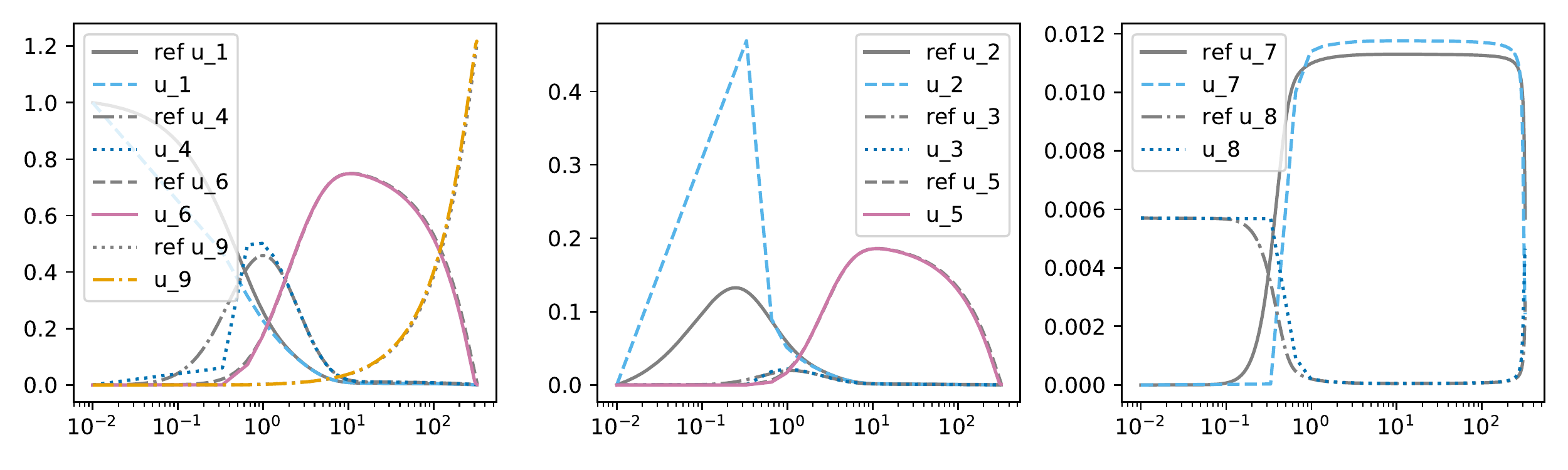}\\
	\includegraphics[width=\textwidth]{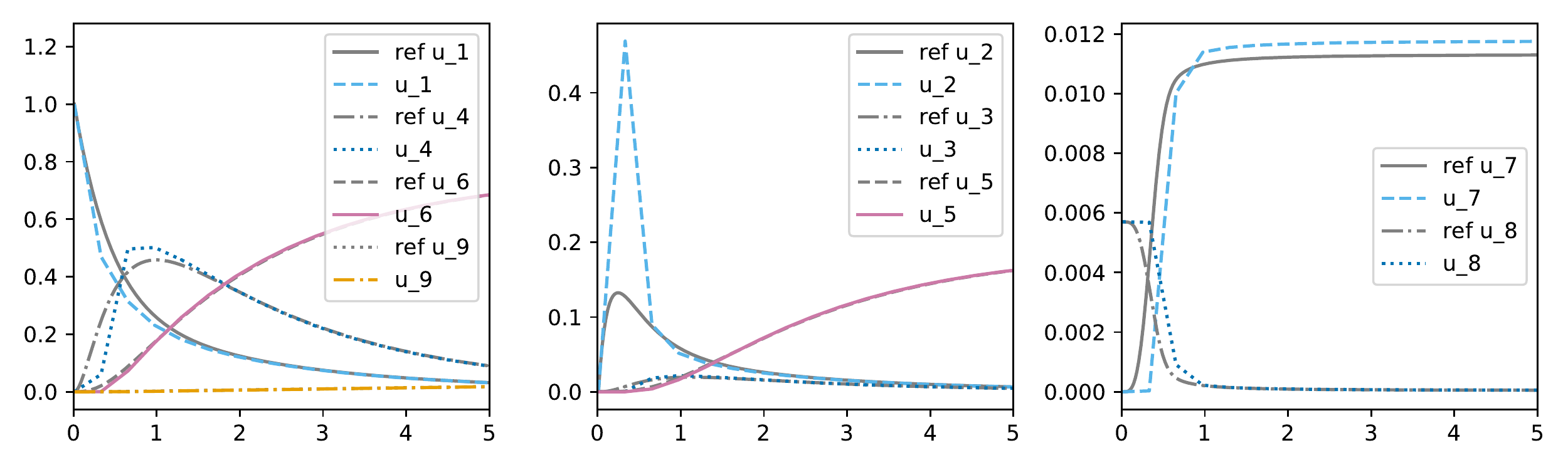}
	\caption{Simulations run with \ref{eq:explicit_dec_correction}7 with equispaced points with $N=1000$ timesteps, top logarithmic scale in time, bottom zoom on $t\in [0,5]$}\label{fig:HIRES_MPDeC7eq}
\end{figure}

\begin{figure}
	\includegraphics[width=\textwidth]{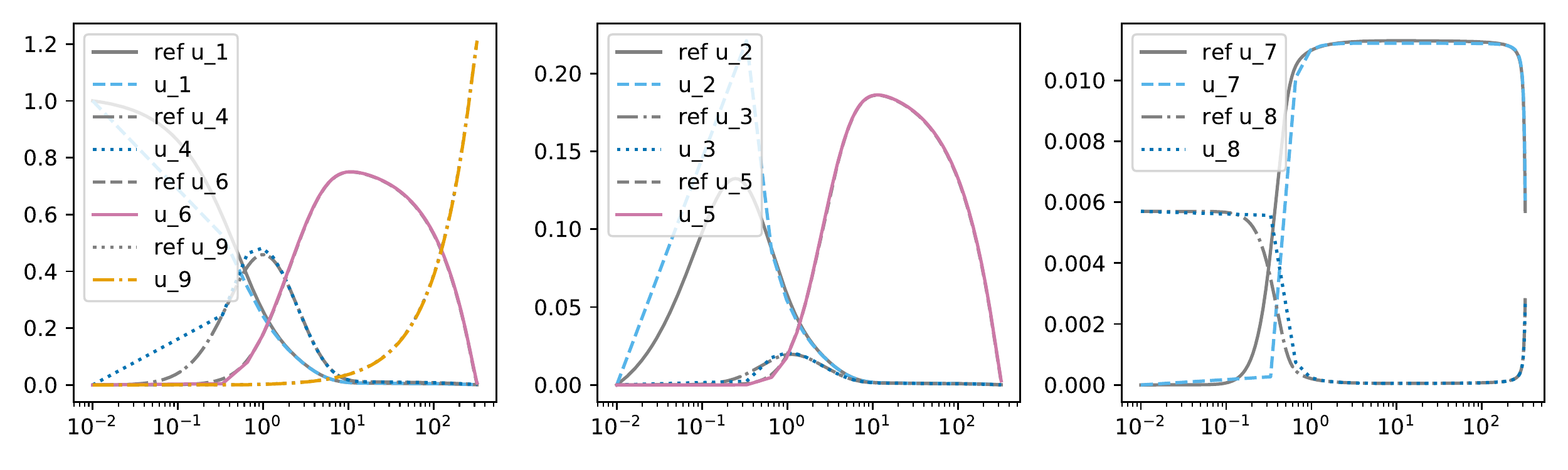}\\
	\includegraphics[width=\textwidth]{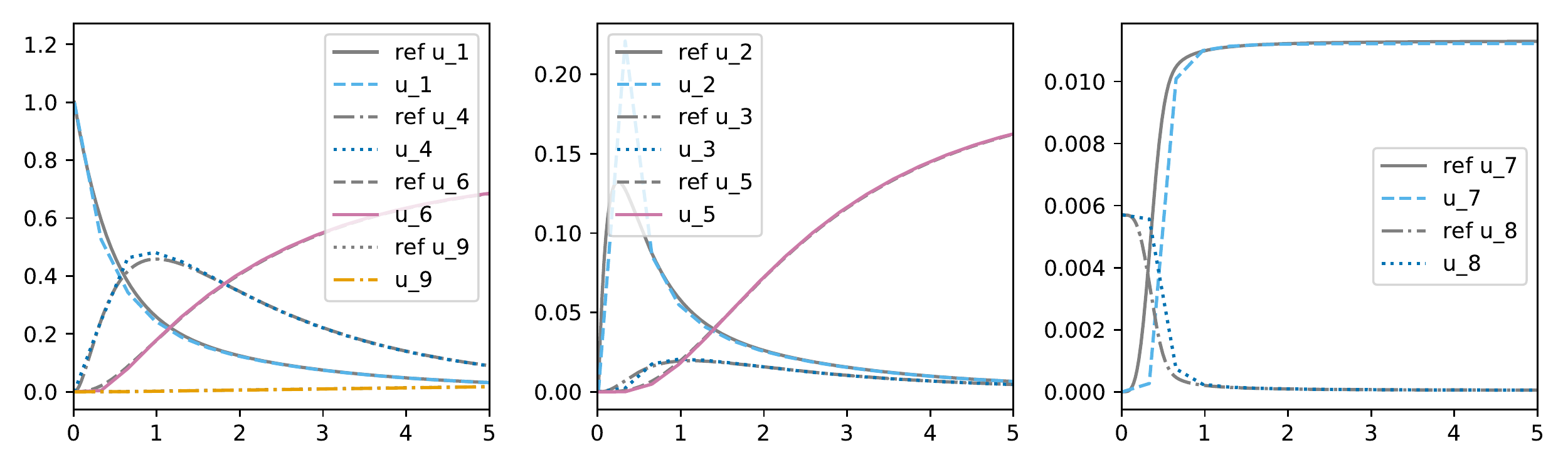}
	\caption{Simulations run with \ref{eq:explicit_dec_correction}8 with equispaced points with $N=1000$ timesteps, top logarithmic scale in time, bottom zoom on $t\in [0,5]$}\label{fig:HIRES_MPDeC8eq}
\end{figure}

\begin{figure}
	\includegraphics[width=\textwidth]{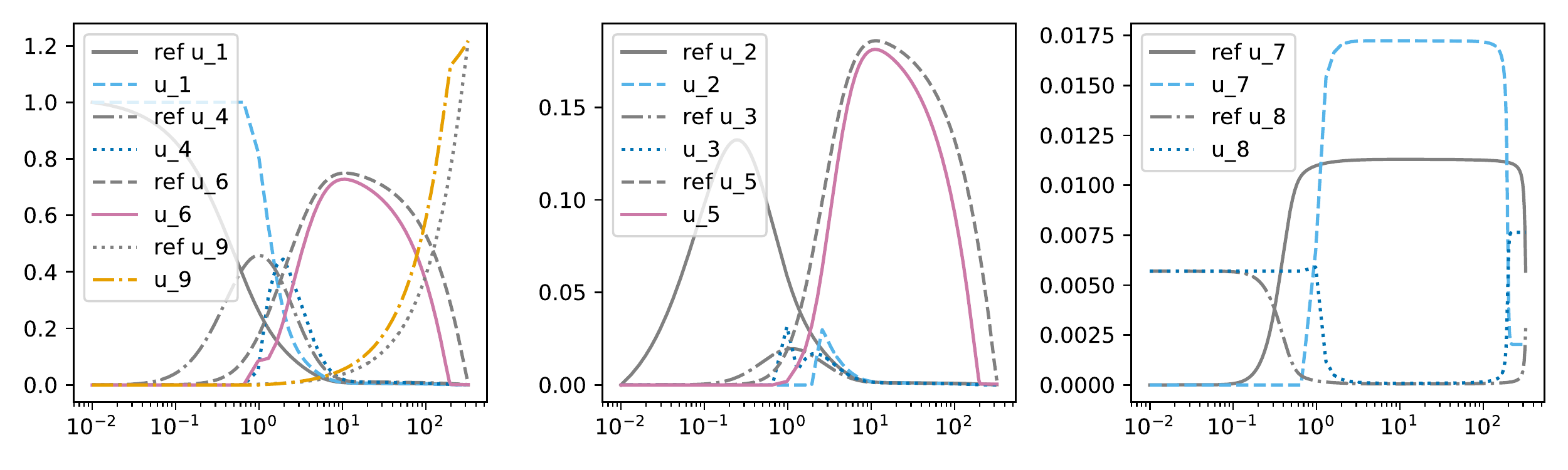}\\
	\includegraphics[width=\textwidth]{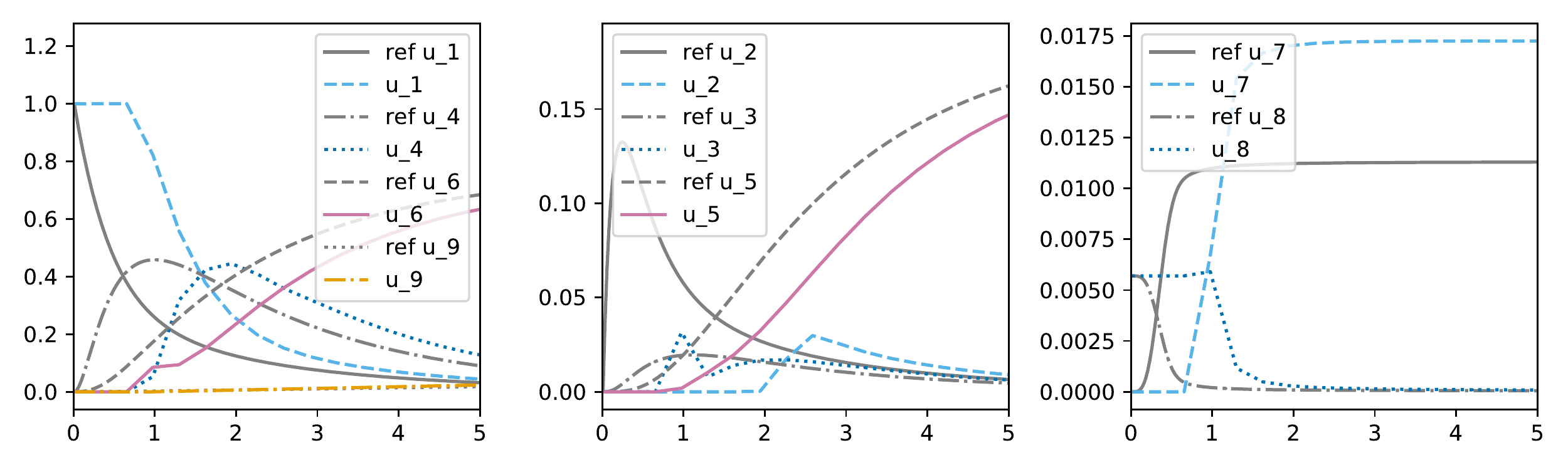}
	\caption{Simulations run with \ref{eq:explicit_dec_correction}9 with equispaced points with $N=1000$ timesteps, top logarithmic scale in time, bottom zoom on $t\in [0,5]$}\label{fig:HIRES_MPDeC9eq}
\end{figure}

\begin{figure}
	\includegraphics[width=\textwidth]{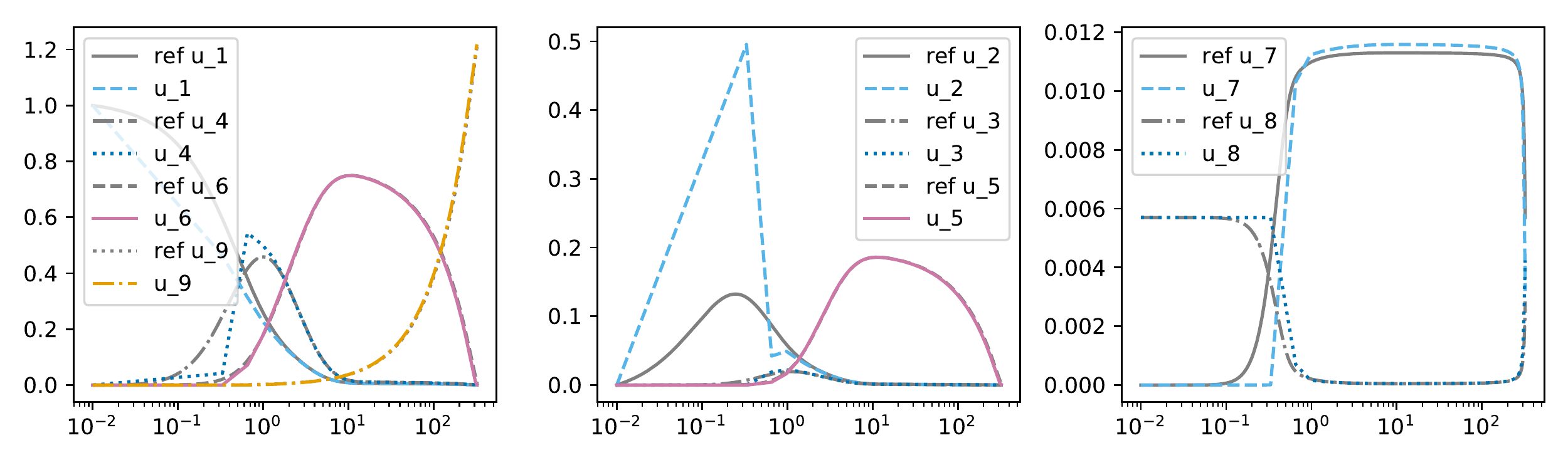}\\
	\includegraphics[width=\textwidth]{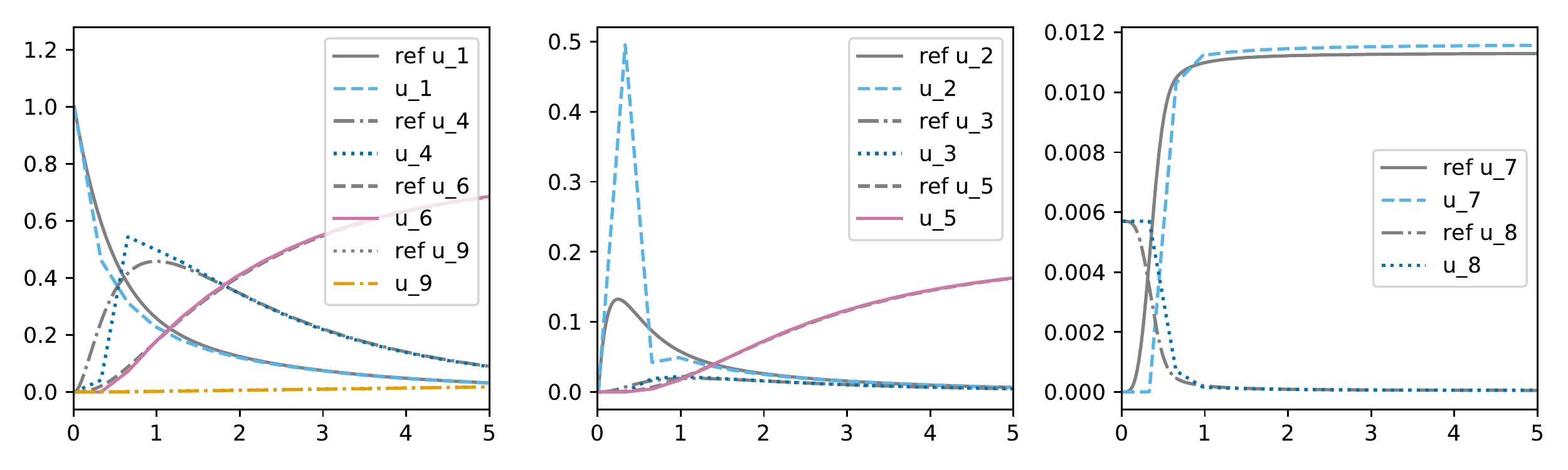}
	\caption{Simulations run with \ref{eq:explicit_dec_correction}10 with equispaced points with $N=1000$ timesteps, top logarithmic scale in time, bottom zoom on $t\in [0,5]$}\label{fig:HIRES_MPDeC10eq}
\end{figure}

\begin{figure}
	\includegraphics[width=\textwidth]{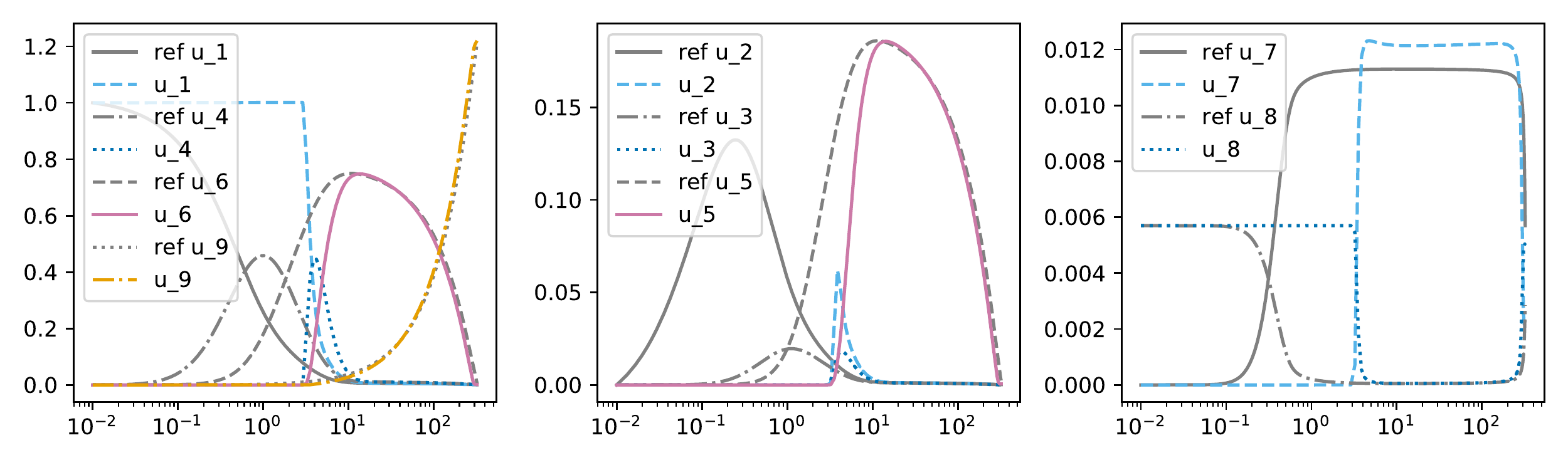}\\
	\includegraphics[width=\textwidth]{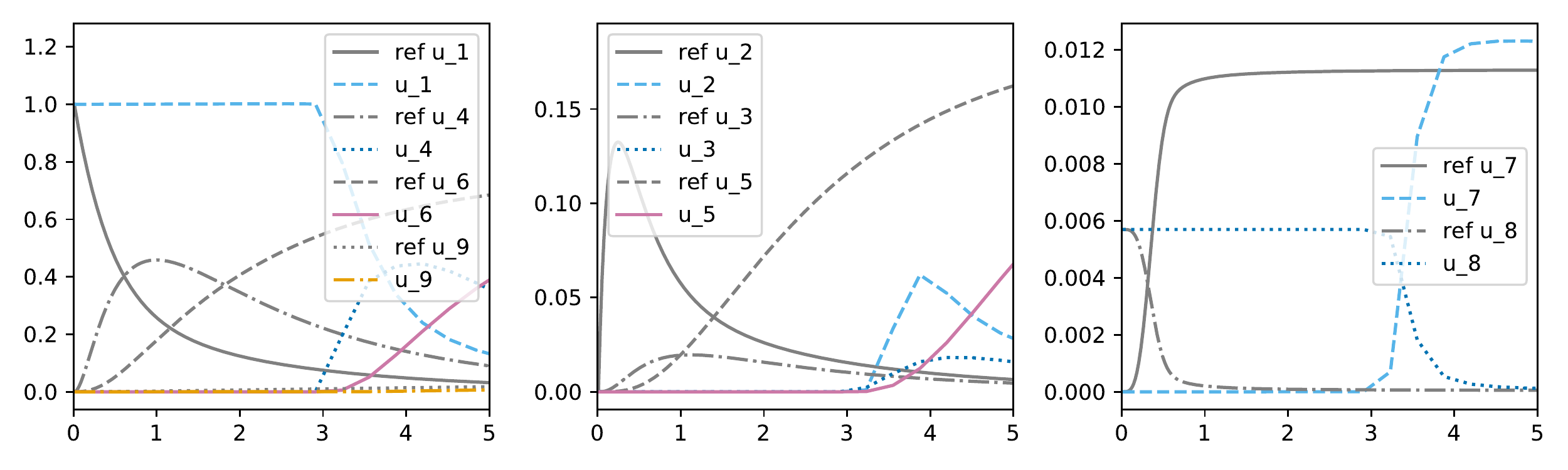}
	\caption{Simulations run with \ref{eq:explicit_dec_correction}11 with equispaced points with $N=1000$ timesteps, top logarithmic scale in time, bottom zoom on $t\in [0,5]$}\label{fig:HIRES_MPDeC11eq}
\end{figure}

Testing with $N=1000$ uniform timesteps, we spot troubles with the methods that become first order accurate.  We test the problem with many schemes presented above.
For the \ref{eq:explicit_dec_correction} we spot the first order of accuracy issue only for equispaced timesteps for high odd orders (9,\,11,\,13 and so on). In Figure~\ref{fig:HIRES_MPDeC6gl} we see the simulation for \ref{eq:explicit_dec_correction}6 with Gauss--Lobatto points. We observe that the high accuracy helps in obtaining a good result at the end of the simulation, when $u_7$ and $u_8$ react. The moment at which this change happens is really hard to catch and only high order methods are able to obtain it with this number of timesteps.

We run the \ref{eq:MPRK22-family} with $\alpha\in [0.7, 1, 5]$. As for the linear case, only for $\alpha>1$ we observe accuracy loss as it is visible in Figure~\ref{fig:HIRES_MPRK225}, where the evolution of some constituents is completely missed, e.g. $u_2, u_3, u_5, u_9$, while in Figure~\ref{fig:HIRES_MPRK221} we obtain accurate results.

We test \ref{eq:MPRKSO22-family} with $\alpha=0.3,\,\beta=2$ and $\alpha=0,\,\beta=8$ and, as expected, the second one shows the loss of accuracy. An oscillatory type behavior can be observed, though, in the first simulation, which is depicted in Figure~\ref{fig:HIRES_MPRKSO22_03_2}. This is probably due to the CFL condition as, refining the time discretization, the oscillations disappear.

For \ref{eq:MPRK43-family} we test $\alpha=0.9,\,\beta= 0.6$ and $\alpha=5,\,\beta=0.5$, observing loss of accuracy only in the second one, according with the linear tests. For \ref{eq:MPRKSO(4,3)}, \ref{eq:MPRK32}, \ref{eq:SI-RK2} and \ref{eq:SI-RK3} we do not observe inaccurate behaviors, as in the linear test.
%
%
%
%
%

\printbibliography

%% file: figures/system2/CFLDeCeqAllSystemsShort.tex
$p$ & $\Delta t$ bound \\ \hline 
   1 & $\infty$ \\
   2 & 2.0 \\ 
   3 & 1.19 \\ 
   4 & 1.11 \\ 
   5 & 1.07 \\ 
   6 & 1.04 \\ 
   7 & 1.04 \\ 
   8 & 1.37 \\ 
{\color{red} 9} &{\color{red} 6.96} \\ 
   10 & 1.0 \\ 
{\color{red}    11 } &{\color{red} 15.5 }\\ 
{\color{red}    12 } &{\color{red} 1.0 }\\ 
{\color{red}    13 } &{\color{red} 35.51} \\ 
{\color{red}    14 } &{\color{red} 1.07} \\ 
{\color{red}    15 } &{\color{red} 12.13}\\ 
{\color{red}    16 } &{\color{red} 1.80} \\ 

%% file: figures/system2/CFLDeCglbAllSystemsShort.tex
$p$ & $\Delta t$ bound \\ \hline 
1 & $\infty$ \\
   2 & 2.0 \\ 
   3 & 1.19 \\ 
   4 & 1.07 \\ 
   5 & 1.04 \\ 
   6 & 1.0 \\ 
   7 & 1.0 \\ 
   8 & 1.0 \\ 
   9 & 1.0 \\ 
   10 & 1.0 \\ 
   11 & 1.0 \\ 
   12 & 1.0 \\ 
   13 & 1.0 \\ 
   14 & 1.0 \\ 
   15 & 1.0 \\ 
   16 & 1.0 \\ 

%% file: figures/system2/manyRK_CFL.tex
Method & $\Delta t$ bound \\ \hline 
ImplicitMidpoint & 2.0  \\  
Trapezoid        & 2.0  \\ 
TRBDF2           & 2.4  \\
RadauIIA3        & 3.0  \\
RadauIIA5        & $\infty$ \\

%% file: figures/scalar/k4/overMeasCFLs.tex
CFL& 0.5& 1.0& 2.0& 4.0& 8.0& 16.0& 32.0\\ \hline 

%% file: figures/scalar/k4/overMeasMPDeCeq.tex
MPDeC1eq & 0 & 0 & {\color{overColor}2.7e-04} & {\color{overColor}5.3e-04} & {\color{overColor}7.0e-04} & {\color{overColor}8.0e-04} & {\color{overColor}8.5e-04} & {\color{overColor}8.8e-04}\\ 
MPDeC2eq & 0 & 0 & {\color{overColor}3.2e-04} & {\color{overColor}6.1e-04} & {\color{overColor}8.1e-04} & {\color{overColor}9.3e-04} & {\color{overColor}1.0e-03} & {\color{overColor}1.0e-03}\\ 
MPDeC3eq & 0 & 0 & 0 & 0 & 0 & 0 & 0 & 0\\ 
MPDeC4eq & 0 & 0 & 0 & 0 & 0 & {\color{overColor}2.8e-05} & {\color{overColor}8.5e-05} & {\color{overColor}1.3e-04}\\ 
MPDeC5eq & 0 & 0 & 0 & 0 & 0 & {\color{overColor}2.0e-05} & {\color{overColor}7.1e-05} & {\color{overColor}1.1e-04}\\ 
MPDeC6eq & 0 & 0 & 0 & {\color{overColor}2.6e-06} & {\color{overColor}4.4e-05} & {\color{overColor}1.1e-04} & {\color{overColor}1.6e-04} & {\color{overColor}1.7e-04}\\ 
MPDeC7eq & 0 & 0 & 0 & 0 & 0 & 0 & 0 & 0\\ 
MPDeC8eq & 0 & 0 & 0 & {\color{overColor}5.9e-07} & {\color{overColor}9.5e-06} & {\color{overColor}7.4e-06} & {\color{overColor}5.0e-07} & {\color{overColor}8.9e-06}\\ 
MPDeC9eq & 0 & 0 & 0 & 0 & 0 & 0 & 0 & 0\\ 
MPDeC10eq & 0 & 0 & 0 & 0 & {\color{overColor}1.2e-06} & {\color{overColor}8.2e-06} & {\color{overColor}5.4e-05} & {\color{overColor}1.2e-04}\\ 
MPDeC11eq & 0 & 0 & 0 & 0 & 0 & 0 & 0 & 0\\ 
MPDeC12eq & 0 & 0 & 0 & 0 & 0 & {\color{overColor}7.5e-06} & {\color{overColor}3.7e-05} & {\color{overColor}5.7e-05}\\ 
MPDeC13eq & 0 & 0 & 0 & 0 & 0 & 0 & 0 & 0\\ 

%% file: figures/scalar/k4/overMeasMPDeCgl.tex
MPDeC1GL & 0 & 0 & {\color{overColor}2.7e-04} & {\color{overColor}5.3e-04} & {\color{overColor}7.0e-04} & {\color{overColor}8.0e-04} & {\color{overColor}8.5e-04} & {\color{overColor}8.8e-04}\\ 
MPDeC2GL & 0 & 0 & {\color{overColor}3.2e-04} & {\color{overColor}6.1e-04} & {\color{overColor}8.1e-04} & {\color{overColor}9.3e-04} & {\color{overColor}1.0e-03} & {\color{overColor}1.0e-03}\\ 
MPDeC3GL & 0 & 0 & 0 & 0 & 0 & 0 & 0 & 0\\ 
MPDeC4GL & 0 & 0 & 0 & 0 & 0 & 0 & 0 & 0\\ 
MPDeC5GL & 0 & 0 & 0 & {\color{overColor}2.8e-06} & {\color{overColor}6.2e-05} & {\color{overColor}2.0e-04} & {\color{overColor}3.4e-04} & {\color{overColor}4.4e-04}\\ 
MPDeC6GL & 0 & 0 & 0 & {\color{overColor}2.4e-05} & {\color{overColor}1.3e-04} & {\color{overColor}3.2e-04} & {\color{overColor}5.0e-04} & {\color{overColor}6.2e-04}\\ 
MPDeC7GL & 0 & 0 & 0 & 0 & {\color{overColor}1.7e-05} & {\color{overColor}2.8e-05} & {\color{overColor}1.1e-05} & 0\\ 
MPDeC8GL & 0 & 0 & 0 & 0 & 0 & 0 & 0 & 0\\ 
MPDeC9GL & 0 & 0 & 0 & 0 & 0 & 0 & 0 & 0\\ 
MPDeC10GL & 0 & 0 & 0 & 0 & 0 & 0 & 0 & 0\\ 
MPDeC11GL & 0 & 0 & 0 & 0 & {\color{overColor}1.4e-06} & {\color{overColor}1.9e-05} & {\color{overColor}9.0e-05} & {\color{overColor}1.9e-04}\\ 
MPDeC12GL & 0 & 0 & 0 & 0 & {\color{overColor}1.6e-06} & {\color{overColor}2.5e-05} & {\color{overColor}1.1e-04} & {\color{overColor}2.1e-04}\\ 
MPDeC13GL & 0 & 0 & 0 & 0 & 0 & {\color{overColor}2.2e-06} & {\color{overColor}9.9e-06} & {\color{overColor}3.2e-06}\\ 

%% file: figures/scalar/k4/overMeasRK22.tex
MPRK(2,2,0.5) & 0 & 0 & {\color{overColor}2.7e-04} & {\color{overColor}5.3e-04} & {\color{overColor}7.0e-04} & {\color{overColor}8.0e-04} & {\color{overColor}8.5e-04} & {\color{overColor}8.8e-04}\\ 
MPRK(2,2,0.8) & 0 & 0 & {\color{overColor}3.1e-04} & {\color{overColor}6.0e-04} & {\color{overColor}8.0e-04} & {\color{overColor}9.2e-04} & {\color{overColor}9.9e-04} & {\color{overColor}1.0e-03}\\ 
MPRK(2,2,1.0) & 0 & 0 & {\color{overColor}3.2e-04} & {\color{overColor}6.1e-04} & {\color{overColor}8.1e-04} & {\color{overColor}9.3e-04} & {\color{overColor}1.0e-03} & {\color{overColor}1.0e-03}\\ 
MPRK(2,2,1.5) & 0 & 0 & {\color{overColor}3.3e-04} & {\color{overColor}6.1e-04} & {\color{overColor}8.0e-04} & {\color{overColor}9.2e-04} & {\color{overColor}9.8e-04} & {\color{overColor}1.0e-03}\\ 
MPRK(2,2,2.0) & 0 & 0 & {\color{overColor}3.2e-04} & {\color{overColor}6.0e-04} & {\color{overColor}7.9e-04} & {\color{overColor}9.0e-04} & {\color{overColor}9.6e-04} & {\color{overColor}9.9e-04}\\ 
MPRK(2,2,10.0) & 0 & 0 & {\color{overColor}2.9e-04} & {\color{overColor}5.5e-04} & {\color{overColor}7.2e-04} & {\color{overColor}8.2e-04} & {\color{overColor}8.8e-04} & {\color{overColor}9.1e-04}\\ 

%% file: figures/scalar/k4/overMeasRK43.tex
MPRK(4,3,0.50,0.70) & 0 & 0 & {\color{overColor}6.4e-05} & {\color{overColor}1.3e-04} & {\color{overColor}1.1e-04} & {\color{overColor}5.8e-05} & {\color{overColor}1.8e-05} & 0\\ 
MPRK(4,3,0.90,0.65) & 0 & 0 & {\color{overColor}7.5e-05} & {\color{overColor}1.6e-04} & {\color{overColor}1.6e-04} & {\color{overColor}1.4e-04} & {\color{overColor}1.2e-04} & {\color{overColor}1.1e-04}\\ 
MPRK(4,3,1.00,0.67) & 0 & 0 & {\color{overColor}7.9e-05} & {\color{overColor}1.7e-04} & {\color{overColor}1.8e-04} & {\color{overColor}1.6e-04} & {\color{overColor}1.5e-04} & {\color{overColor}1.3e-04}\\ 
MPRK(4,3,1.00,0.33) & 0 & 0 & 0 & 0 & 0 & 0 & 0 & 0\\ 
MPRK(4,3,1.25,0.60) & 0 & 0 & {\color{overColor}5.7e-05} & {\color{overColor}1.2e-04} & {\color{overColor}9.3e-05} & {\color{overColor}4.6e-05} & {\color{overColor}7.5e-06} & 0\\ 
MPRK(4,3,2.00,0.60) & 0 & 0 & {\color{overColor}4.5e-05} & {\color{overColor}9.4e-05} & {\color{overColor}5.7e-05} & {\color{overColor}2.5e-07} & 0 & 0\\ 
MPRK(4,3,10.00,0.60) & 0 & 0 & {\color{overColor}2.1e-05} & {\color{overColor}5.2e-05} & 0 & 0 & 0 & 0\\ 

%% file: figures/scalar/k4/overMeasRK43purple.tex
MPRK(4,3,0.50,0.75) & 0 & 0 & {\color{overColor}6.8e-05} & {\color{overColor}1.4e-04} & {\color{overColor}1.3e-04} & {\color{overColor}1.0e-04} & {\color{overColor}7.0e-05} & {\color{overColor}5.1e-05}\\ 
MPRK(4,3,0.70,0.63) & 0 & 0 & {\color{overColor}1.0e-04} & {\color{overColor}2.2e-04} & {\color{overColor}2.7e-04} & {\color{overColor}2.9e-04} & {\color{overColor}2.9e-04} & {\color{overColor}2.9e-04}\\ 
MPRK(4,3,0.80,0.48) & 0 & 0 & {\color{overColor}8.4e-05} & {\color{overColor}1.8e-04} & {\color{overColor}1.9e-04} & {\color{overColor}1.7e-04} & {\color{overColor}1.5e-04} & {\color{overColor}1.3e-04}\\ 
MPRK(4,3,0.89,0.29) & 0 & 0 & {\color{overColor}1.8e-05} & 0 & 0 & 0 & 0 & 0\\ 
MPRK(4,3,0.90,0.29) & 0 & 0 & {\color{overColor}1.4e-05} & 0 & 0 & 0 & 0 & 0\\ 
MPRK(4,3,1.00,0.33) & 0 & 0 & 0 & 0 & 0 & 0 & 0 & 0\\ 
MPRK(4,3,1.25,0.39) & 0 & 0 & 0 & 0 & 0 & 0 & 0 & 0\\ 
MPRK(4,3,2.00,0.44) & 0 & 0 & 0 & 0 & 0 & 0 & 0 & 0\\ 
MPRK(4,3,10.00,0.49) & 0 & 0 & 0 & 0 & 0 & 0 & 0 & 0\\ 

%% file: figures/scalar/k4/overMeasRKSO22.tex
MPRKSO(2,2,0.0,0.5) & 0 & 0 & {\color{overColor}2.7e-04} & {\color{overColor}5.3e-04} & {\color{overColor}7.0e-04} & {\color{overColor}8.0e-04} & {\color{overColor}8.5e-04} & {\color{overColor}8.8e-04}\\ 
MPRKSO(2,2,0.0,1.0) & 0 & 0 & {\color{overColor}3.2e-04} & {\color{overColor}6.1e-04} & {\color{overColor}8.1e-04} & {\color{overColor}9.3e-04} & {\color{overColor}1.0e-03} & {\color{overColor}1.0e-03}\\ 
MPRKSO(2,2,0.0,2.0) & 0 & 0 & {\color{overColor}3.2e-04} & {\color{overColor}6.0e-04} & {\color{overColor}7.9e-04} & {\color{overColor}9.0e-04} & {\color{overColor}9.6e-04} & {\color{overColor}9.9e-04}\\ 
MPRKSO(2,2,0.0,5.0) & 0 & 0 & {\color{overColor}3.0e-04} & {\color{overColor}5.7e-04} & {\color{overColor}7.4e-04} & {\color{overColor}8.5e-04} & {\color{overColor}9.0e-04} & {\color{overColor}9.3e-04}\\ 
MPRKSO(2,2,0.0,10.0) & 0 & 0 & {\color{overColor}2.9e-04} & {\color{overColor}5.5e-04} & {\color{overColor}7.2e-04} & {\color{overColor}8.2e-04} & {\color{overColor}8.8e-04} & {\color{overColor}9.1e-04}\\ 
MPRKSO(2,2,0.1,1.5) & 0 & 0 & {\color{overColor}3.6e-04} & {\color{overColor}6.7e-04} & {\color{overColor}8.8e-04} & {\color{overColor}1.0e-03} & {\color{overColor}1.1e-03} & {\color{overColor}1.1e-03}\\ 
MPRKSO(2,2,0.1,6.0) & 0 & {\color{overColor}3.1e-04} & {\color{overColor}9.7e-04} & {\color{overColor}1.6e-03} & {\color{overColor}2.1e-03} & {\color{overColor}2.5e-03} & {\color{overColor}2.7e-03} & {\color{overColor}2.8e-03}\\ 
MPRKSO(2,2,0.2,2.0) & 0 & {\color{overColor}5.9e-05} & {\color{overColor}5.0e-04} & {\color{overColor}9.1e-04} & {\color{overColor}1.2e-03} & {\color{overColor}1.4e-03} & {\color{overColor}1.5e-03} & {\color{overColor}1.6e-03}\\ 
MPRKSO(2,2,0.3,1.5) & 0 & {\color{overColor}2.8e-05} & {\color{overColor}4.5e-04} & {\color{overColor}8.3e-04} & {\color{overColor}1.1e-03} & {\color{overColor}1.3e-03} & {\color{overColor}1.4e-03} & {\color{overColor}1.5e-03}\\ 
MPRKSO(2,2,0.5,1.0) & 0 & 0 & {\color{overColor}2.7e-04} & {\color{overColor}5.3e-04} & {\color{overColor}7.0e-04} & {\color{overColor}8.0e-04} & {\color{overColor}8.5e-04} & {\color{overColor}8.8e-04}\\ 

%% file: figures/scalar/k4/overMeasRK32.tex
MPRK(3,2) & 0 & 0 & {\color{overColor}5.2e-06} & 0 & 0 & 0 & 0 & 0\\ 

%% file: figures/scalar/k4/overMeasRKSO43.tex
MPRKSO(4,3) & 0 & 0 & {\color{overColor}5.1e-05} & {\color{overColor}7.7e-05} & 0 & 0 & 0 & 0\\ 

%% file: figures/scalar/k4/overMeasSI2.tex
SIRK2 & 0 & 0 & {\color{overColor}2.9e-04} & {\color{overColor}5.4e-04} & {\color{overColor}7.0e-04} & {\color{overColor}8.0e-04} & {\color{overColor}8.5e-04} & {\color{overColor}8.8e-04}\\ 

%% file: figures/scalar/k4/overMeasSI3.tex
SIRK3 & 0 & 0 & {\color{overColor}1.0e-04} & {\color{overColor}3.3e-05} & 0 & 0 & 0 & 0\\ 